\documentclass[10pt]{amsart}
\usepackage[utf8]{inputenc}
\usepackage{times}
\usepackage{pdflscape} 
\usepackage{afterpage}
\usepackage{longtable}
\usepackage[dvipsnames]{xcolor}
\usepackage{todonotes}

\usepackage{etoolbox}

\makeatletter
\patchcmd{\@thm}{\let\thm@indent\indent}{\let\thm@indent\noindent}{}{}
\patchcmd{\@thm}{\thm@headfont{\scshape}}{\thm@headfont{\bfseries}}{}{}

\usepackage[T1]{fontenc}

\usepackage{hyperref}
\hypersetup{
    colorlinks = true,
    citecolor = [rgb]{0.00, 0.5, 0.00},
}

\usepackage{mathtools}

\let\U\undefined
\let\G\undefined

\usepackage{multirow}
\usepackage{longtable}

\usepackage{amsmath}
\usepackage{amsfonts}
\usepackage{amssymb}
\usepackage{float}
\usepackage{amsthm}
\usepackage{comment}
\usepackage{amscd}
\usepackage{amsxtra}
\usepackage{epsfig}
\usepackage{epigraph}
\usepackage{pgfplots}
\usepackage{enumitem}

\usepackage{listings}
\usepackage{color}

\definecolor{dkgreen}{rgb}{0,0.6,0}
\definecolor{gray}{rgb}{0.5,0.5,0.5}
\definecolor{mauve}{rgb}{0.58,0,0.82}

\usepackage{color}
\usepackage[all]{xy}
\usepackage{verbatim}

\usepackage{eucal}
\usepackage{mathrsfs}

\usepackage{graphicx}
\usepackage{tikz-cd}

\usepackage{url}
\usepackage{verbatim}

\usepackage{xy}
\xyoption{all}
\newcommand{\xar}[1]{\xrightarrow{{#1}}}

\usepackage{amsthm}
\usepackage{chngcntr}

\usepackage{caption} 
\makeatletter
\def\@tocline#1#2#3#4#5#6#7{\relax
  \ifnum #1>\c@tocdepth 
  \else
    \par \addpenalty\@secpenalty\addvspace{#2}%
    \begingroup \hyphenpenalty\@M
    \@ifempty{#4}{%
      \@tempdima\csname r@tocindent\number#1\endcsname\relax
    }{%
      \@tempdima#4\relax
    }%
    \parindent\z@ \leftskip#3\relax \advance\leftskip\@tempdima\relax
    \rightskip\@pnumwidth plus4em \parfillskip-\@pnumwidth
    #5\leavevmode\hskip-\@tempdima
      \ifcase #1
       \or\or \hskip 1em \or \hskip 2em \else \hskip 3em \fi%
      #6\nobreak\relax
    \hfill\hbox to\@pnumwidth{\@tocpagenum{#7}}\par
    \nobreak
    \endgroup
  \fi}
\makeatother

\usepackage[noabbrev,capitalize]{cleveref} 

\crefformat{nul}{(#2#1#3)}
\Crefformat{nul}{(#2#1#3)}

\crefname{section}{\S}{\S\S}
\crefname{subsection}{\S}{\S\S}
\crefname{axioms}{Axiom}{Axioms}
\crefname{exercise}{Exercise}{Exercises}
\crefname{exercisenum}{Exercise}{Exercises}
\crefname{construction}{Construction}{Constructions}
\crefname{problem}{Problem}{Problems}
\crefname{theorem}{Theorem}{Theorems}
\crefname{definition}{Definition}{Definitions}
\crefname{prop}{Proposition}{Propositions}
\crefname{lemma}{Lemma}{Lemmas}
\crefname{example}{Example}{Examples}
\crefname{examplealph}{Example}{Examples}
\crefname{corollary}{Corollary}{Corollaries}
\crefname{nonexample}{Nonexample}{Nonexamples}
\crefname{equation}{}{}
\crefname{summary}{Summary}{Summaries}
\crefname{recollection}{Recollection}{Recollections}
\Crefname{recollection}{Recollection}{Recollections}
\Crefname{nonexample}{Nonexample}{Nonexamples}
\Crefname{corollary}{Corollary}{Corollaries}
\Crefname{corollary}{Corollary}{Corollaries}
\Crefname{axioms}{Axiom}{Axioms}
\Crefname{exercise}{Exercise}{Exercises}
\Crefname{exercisenum}{Exercise}{Exercises}
\Crefname{construction}{Construction}{Constructions}
\Crefname{problem}{Problem}{Problems}
\Crefname{theorem}{Theorem}{Theorems}
\Crefname{definition}{Definition}{Definitions}
\Crefname{prop}{Proposition}{Propositions}
\Crefname{lemma}{Lemma}{Lemmas}
\Crefname{example}{Example}{Examples}
\Crefname{examplealph}{Example}{Examples}
\Crefname{section}{\S}{\S\S}
\Crefname{subsection}{\S}{\S\S}

\DeclareMathOperator{\Fun}{Fun}
\DeclareMathOperator{\Hom}{Hom}
\DeclareMathOperator{\Aut}{Aut}
\DeclareMathOperator{\Ext}{Ext}

\DeclareMathOperator{\spec}{Spec}
\DeclareMathOperator{\spev}{Spev}
\DeclareMathOperator{\spf}{Spf}

\DeclareMathOperator{\QCoh}{\mathrm{QCoh}}

\DeclareMathOperator{\tr}{Tr}
\DeclareMathOperator{\End}{End}
\DeclareMathOperator{\Pic}{Pic}

\DeclareMathOperator{\colim}{colim}

\DeclareMathOperator{\Tot}{Tot}
\DeclareMathOperator{\im}{im}

\DeclareMathOperator{\Sym}{Sym}

\numberwithin{equation}{section}
\newtheorem{lemma}{Lemma}[section]

\newtheorem{corollary}[lemma]{Corollary}

\newtheorem{theorem}[lemma]{Theorem}

\newtheorem{prop}[lemma]{Proposition}

\newtheorem*{conj-moore}{Conjecture~\ref{moore-splitting}}
\newtheorem*{conj-cent}{Conjecture~\ref{centrality-conj}}
\newtheorem*{conj-tmf}{Conjecture~\ref{tmf-conj}}
\newtheorem*{thm-main}{Theorem~\ref{main-thm}}
\newtheorem*{cor-bpn}{Corollary~\ref{bpn}}
\newtheorem*{thm-string}{Theorem~\ref{mstring}}

\theoremstyle{definition}

\newtheorem{definition}[lemma]{Definition}
\newtheorem{warning}[lemma]{Warning}
\newtheorem{example}[lemma]{Example}

\newtheorem{construction}[lemma]{Construction}
\newtheorem{remark}[lemma]{Remark}
\newtheorem{notation}[lemma]{Notation}

\newtheorem{setup}[lemma]{Setup}

\newtheorem{recall}[lemma]{Recollection}

\newtheorem{observe}[lemma]{Observation}

\newtheorem{heuristic}[lemma]{Heuristic}

\DeclareMathSymbol{A}{\mathalpha}{operators}{`A}
\DeclareMathSymbol{B}{\mathalpha}{operators}{`B}
\DeclareMathSymbol{C}{\mathalpha}{operators}{`C}
\DeclareMathSymbol{D}{\mathalpha}{operators}{`D}
\DeclareMathSymbol{E}{\mathalpha}{operators}{`E}
\DeclareMathSymbol{F}{\mathalpha}{operators}{`F}
\DeclareMathSymbol{G}{\mathalpha}{operators}{`G}
\DeclareMathSymbol{H}{\mathalpha}{operators}{`H}
\DeclareMathSymbol{I}{\mathalpha}{operators}{`I}
\DeclareMathSymbol{J}{\mathalpha}{operators}{`J}
\DeclareMathSymbol{K}{\mathalpha}{operators}{`K}
\DeclareMathSymbol{L}{\mathalpha}{operators}{`L}
\DeclareMathSymbol{M}{\mathalpha}{operators}{`M}
\DeclareMathSymbol{N}{\mathalpha}{operators}{`N}
\DeclareMathSymbol{O}{\mathalpha}{operators}{`O}
\DeclareMathSymbol{P}{\mathalpha}{operators}{`P}
\DeclareMathSymbol{Q}{\mathalpha}{operators}{`Q}
\DeclareMathSymbol{R}{\mathalpha}{operators}{`R}
\DeclareMathSymbol{S}{\mathalpha}{operators}{`S}
\DeclareMathSymbol{T}{\mathalpha}{operators}{`T}
\DeclareMathSymbol{U}{\mathalpha}{operators}{`U}
\DeclareMathSymbol{V}{\mathalpha}{operators}{`V}
\DeclareMathSymbol{W}{\mathalpha}{operators}{`W}
\DeclareMathSymbol{X}{\mathalpha}{operators}{`X}
\DeclareMathSymbol{Y}{\mathalpha}{operators}{`Y}
\DeclareMathSymbol{Z}{\mathalpha}{operators}{`Z}

\renewcommand{\AA}{\mathbf{A}}
\newcommand{\FF}{\mathbf{F}}
\newcommand{\Z}{\mathbf{Z}}
\newcommand{\QQ}{\mathbf{Q}}
\newcommand{\cc}{\mathbf{C}}
\newcommand{\cJ}{\mathcal{J}}
\newcommand{\cC}{\mathcal{C}}
\newcommand{\cd}{\mathcal{D}}
\newcommand{\RR}{\mathbf{R}}
\newcommand{\Deltab}{\mathbf{\Delta}}

\newcommand{\Sp}{\mathrm{Sp}}

\newcommand{\Aff}{\mathrm{Aff}}

\newcommand{\Mfg}{\M_\mathrm{fg}}

\newcommand{\co}{\mathcal{O}}

\newcommand{\PP}{\mathbf{P}}

\newcommand{\GG}{\mathbf{G}}

\newcommand{\LMod}{\mathrm{LMod}}

\newcommand{\Mod}{\mathrm{Mod}}
\newcommand{\coMod}{\mathrm{coMod}}
\newcommand{\coLMod}{\mathrm{coLMod}}
\newcommand{\cL}{\mathcal{L}}

\newcommand{\CP}{{\mathbf{C}P}}
\newcommand{\HHP}{\mathbf{H}P}
\newcommand{\RP}{\mathbf{R}P}

\newcommand{\Map}{\mathrm{Map}}

\newcommand{\Lone}{L_{K(1)}}

\newcommand{\cA}{\mathcal{A}}

\newcommand{\cT}{\mathcal{T}}
\newcommand{\LL}{\mathbf{L}}

\newcommand{\Alg}{\mathrm{Alg}}
\newcommand{\Sq}{\mathrm{Sq}}

\newcommand{\N}{\mathrm{N}}

\newcommand{\Eoo}{{\mathbf{E}_\infty}}

\newcommand{\CAlg}{\mathrm{CAlg}}

\newcommand{\cI}{\mathcal{I}}
\newcommand{\cf}{\mathcal{F}}

\newcommand{\univ}{\mathrm{univ}}

\newcommand{\Def}{\mathrm{Def}}

\newcommand{\ad}{\mathrm{ad}}
\newcommand{\Lie}{\mathrm{Lie}}

\newcommand{\ev}{\mathrm{ev}}

\newcommand{\cH}{\mathcal{H}}

\newcommand{\cN}{\mathcal{N}}

\newcommand{\ul}[1]{\underline{#1}}
\newcommand{\ol}[1]{\overline{#1}}

\newcommand{\Perf}{\mathrm{Perf}}

\newcommand{\E}[1]{{\mathbf{E}_{{#1}}}}
\newcommand{\mmod}{/\!\!/}

\newcommand{\id}{\mathrm{id}}

\newcommand{\gl}{\mathfrak{gl}}
\newcommand{\SL}{\mathrm{SL}}
\newcommand{\GL}{\mathrm{GL}}
\newcommand{\BGL}{\mathrm{BGL}}

\newcommand{\MU}{\mathrm{MU}}
\newcommand{\BU}{\mathrm{BU}}

\newcommand{\SO}{\mathrm{SO}}

\newcommand{\KU}{\mathrm{KU}}
\newcommand{\KO}{\mathrm{KO}}
\newcommand{\ku}{\mathrm{ku}}
\newcommand{\ko}{\mathrm{ko}}

\newcommand{\tmf}{\mathrm{tmf}}

\newcommand{\TMF}{\mathrm{TMF}}

\newcommand{\U}{\mathrm{U}}
\newcommand{\SU}{\mathrm{SU}}

\renewcommand{\H}{\mathrm{H}}

\newcommand{\Spin}{\mathrm{Spin}}
\newcommand{\String}{\mathrm{String}}

\newcommand{\Res}{\mathrm{Res}}
\newcommand{\Ind}{\mathrm{Ind}}

\newcommand{\HH}{\mathrm{HH}}
\newcommand{\HC}{\mathrm{HC}}
\newcommand{\HP}{\mathrm{HP}}
\newcommand{\THH}{\mathrm{THH}}

\newcommand{\heart}{\heartsuit}

\newcommand{\Coh}{\mathrm{Coh}}

\newcommand{\op}{\mathrm{op}}

\newcommand{\nil}{\mathrm{nil}}

\newcommand{\cB}{\mathcal{B}}
\newcommand{\dR}{\mathrm{dR}}

\newcommand{\fr}[1]{\mathfrak{#1}}
\newcommand{\g}{\mathfrak{g}}
\newcommand{\G}{\mathrm{G}}

\newcommand{\cP}{\mathcal{P}}

\newcommand{\pr}{\mathrm{pr}}

\newcommand{\cU}{\mathcal{U}}

\newcommand{\Rep}{\mathrm{Rep}}

\newcommand{\Ran}{\mathrm{Ran}}

\newcommand{\gr}{\mathrm{gr}}

\newcommand{\shvcat}{\mathrm{ShvCat}}

\newcommand{\free}{\mathrm{free}}

\newcommand{\cR}{\mathcal{R}}

\newcommand{\coCAlg}{\mathrm{coCAlg}}

\newcommand{\bH}{\mathbf{H}}

\newcommand{\cV}{\mathcal{V}}

\newcommand{\mix}{\mathrm{mixed}}

\newcommand{\an}{\mathrm{an}}
\newcommand{\triv}{\mathrm{triv}}

\newcommand{\unit}{\mathbf{1}}

\newcommand{\cM}{\mathcal{M}}

\newcommand{\Shv}{\mathrm{Shv}}

\newcommand{\DMod}{\mathrm{DMod}}

\newcommand{\Bun}{\mathrm{Bun}}
\renewcommand{\ss}{\mathrm{ss}}

\newcommand{\IndCoh}{\mathrm{IndCoh}}

\newcommand{\Loc}{\mathrm{Loc}}

\newcommand{\Ad}{\mathrm{Ad}}

\newcommand{\M}{\mathcal{M}}

\newcommand{\Bl}{\mathrm{Bl}}

\newcommand{\PGL}{\mathrm{PGL}}

\renewcommand{\det}{\mathrm{det}}

\newcommand{\tate}{\mathrm{Tate}}

\newcommand{\DD}{\mathfrak{D}}

\newcommand{\bX}{\mathbb{X}}
\newcommand{\reg}{\mathrm{reg}}
\newcommand{\rss}{\mathrm{rss}}

\newcommand{\cS}{\mathcal{S}}

\renewcommand{\tr}{\mathrm{tr}}

\renewcommand{\S}{Section }

\newcommand{\PrL}{\mathrm{Pr}^\mathrm{L}}

\newcommand{\WW}{\mathbf{W}}

\newcommand{\sph}{\mathrm{sph}}

\newcommand{\punc}[1]{{\overset{\circ}{#1}}}

\newcommand{\Top}{{\mathcal{S}}}

\newcommand{\gen}{\mathrm{gen}}

\renewcommand{\sqcup}{\amalg}

\makeatletter
\providecommand{\leftsquigarrow}{%
  \mathrel{\mathpalette\reflect@squig\relax}%
}
\newcommand{\reflect@squig}[2]{%
  \reflectbox{$\m@th#1\rightsquigarrow$}%
}
\makeatother

\newcommand{\Cart}{\mathrm{Cart}}

\newcommand{\bull}{\bullet}

\newcommand{\bD}{\mathbf{D}}

\newcommand{\pw}[1]{[\![#1]\!]}
\newcommand{\ls}[1]{(\!(#1)\!)}

\newcommand{\act}{\circlearrowleft}

\newcommand{\Cp}{{\Z/p}}

\newcommand{\Efr}[1]{{\mathbf{E}_{{#1}}^\mathrm{fr}}}

\newcommand{\aff}{\mathrm{aff}}

\newcommand{\rank}{\mathrm{rank}}

\newcommand{\sh}{\mathrm{sh}}

\renewcommand{\sl}{\mathfrak{sl}}

\renewcommand{\tilde}{\widetilde}

\newcommand{\weak}{\mathrm{weak}}

\renewcommand{\min}{\mathrm{min}}
\newcommand{\modc}{{\text{-}\mathrm{mod}}}

\newcommand{\Fl}{\mathrm{Fl}}
\newcommand{\Gr}{\mathrm{Gr}}

\newcommand{\rot}{\mathrm{rot}}

\newcommand{\ld}[1]{{\check{{#1}}}}

\usepackage{slashed}
\newcommand{\fourd}{{4\mathrm{d}}}
\newcommand{\fived}{{5\mathrm{d}}}

\newcommand{\cLoc}{\mathcal{L}\mathrm{oc}}

\newcommand{\pdb}[1]{\langle{#1}\rangle}

\newcommand{\lcc}{\mathrm{lcc}}
\newcommand{\dreg}{\text{-}\mathrm{reg}}
\newcommand{\pgl}{\mathfrak{pgl}}
\newcommand{\elc}{\mathrm{ell}}
\newcommand{\RU}{\mathrm{RU}}
\newcommand{\cpl}{{\text{-}\mathrm{cplt}}}
\newcommand{\Perv}{\mathrm{Perv}}
\newcommand{\bV}{\mathbf{V}}
\newcommand{\std}{\mathrm{std}}

\newcommand{\Sat}{\mathrm{Sat}}
\newcommand{\faux}{\mathrm{faux}}
\newcommand{\OP}{\mathbf{O}P}
\renewcommand{\top}{\mathrm{top}}
\newcommand{\PSO}{\mathrm{PSO}}

\title[Chromatic aberrations of geometric Satake]{Chromatic aberrations of geometric Satake over the regular locus}
\author{S. K. Devalapurkar}
\address{1 Oxford St, Cambridge, MA 02139}
\email{sdevalapurkar@math.harvard.edu, \today}
\thanks{Part of this work was done when the author was supported by the PD Soros Fellowship and NSF DGE-2140743. The present article is a preliminary version, so any comments and suggestions for improving it are greatly appreciated! I'll post major updates to the arXiv, but I'll upload minor edits to my website at \url{https://sanathdevalapurkar.github.io/files/grG-regular.pdf}; so please see there for the most up-to-date version.}

\begin{document}

\epigraph{It is an established practice to take old theorems about ordinary homology, and generalise them so as to obtain theorems about generalised homology theories.}{J. F. Adams, \cite{adams-lectures-coh}}

\begin{abstract}
    Let $G$ be a connected, simply-laced, almost simple algebraic group over $\cc$, let $G_c$ be a maximal compact subgroup of $G(\cc)$, and let $T_c$ be a maximal torus therein. Let $\Gr_G$ denote the affine Grassmannian of $G$, and let $\ld{G}$ denote the Langlands dual group to $G$ with Lie algebra $\ld{\g}$. The derived geometric Satake equivalence of Bezrukavnikov-Finkelberg gives an equivalence between the $\infty$-category $\Loc_{G_c}(\Gr_G; \cc)$ of $G_c$-equivariant local systems of $\cc$-vector spaces on $\Gr_G$ and the $\infty$-category of quasicoherent sheaves on a large open substack of $\ld{\g}^\ast[2]/\ld{G}$. In this article, we study the analogous story when $\Loc_{G_c}(\Gr_G; \cc)$ is replaced by the $\infty$-category of $T_c$-equivariant local systems of $k$-modules over $\Gr_G(\cc)$, where $k$ is ($2$-periodic) rational cohomology, (complex) K-theory, or elliptic cohomology. Crucial to our work is the \textit{genuine} equivariant refinement of these cohomology theories. We show that, although there may not be an \textit{equivalence} as in derived geometric Satake, the $\infty$-category $\Loc_{T_c}(\Gr_G; k)$ admits a \textit{$1$-parameter degeneration} to an $\infty$-category of quasicoherent sheaves built out of the geometry of various Langlands-dual stacks associated to $k$ and the $1$-dimensional group scheme computing $S^1$-equivariant $k$-cohomology. For example, when $k$ is an elliptic cohomology theory with elliptic curve $E$, the $\infty$-category $\Loc_{T_c}(\Gr_G; k)$ degenerates to the $\infty$-category of quasicoherent sheaves on a large open locus in the moduli stack of $\ld{B}$-bundles of degree $0$ on $E$.
    We also study several applications of these equivalences, including: relating the generalized equivariant homology of $\Gr_T$ to generalizations of the Weyl algebra on a torus; proving multiplicative and elliptic versions of the Gelfand-Graev action on the affine closure of the cotangent bundle of the basic affine space; the interaction between power operations (like Steenrod and Adams operations) and derived geometric Satake; relations to work of Brylinski-Zhang and stable splittings of compact Lie groups; and the analogy to work of Hopkins-Kuhn-Ravenel.
\end{abstract}


\maketitle
\tableofcontents
\newpage

\section{Introduction}
\subsection{The derived geometric Satake equivalence}

Let $G$ be a semisimple algebraic group over $\cc$, and let $\Gr_G$ denote the affine Grassmannian, defined as the quotient $G(\cc\ls{t})/G(\cc\pw{t}) =: G\ls{t}/G\pw{t}$. The geometric Satake equivalence of Mirkovic-Vilonen states:
\begin{theorem}[{Mirkovic-Vilonen, \cite{mirkovic-vilonen}}]\label{thm: mirkovic-vilonen}
    Let $\ld{G}_\Z$ denote the smooth split reductive group scheme over $\Z$ whose root datum is dual to that of $G$.
    Then there is an equivalence of symmetric  monoidal abelian categories
    $$\mathrm{Perv}_{G\pw{t}}(\Gr_G; \Z) \simeq \Rep(\ld{G}_\Z).$$
\end{theorem}
The abelian category $\mathrm{Perv}_{G\pw{t}}(\Gr_G; \Z)$ arises as the heart of a $t$-structure on the stable $\infty$-category $\Shv^c_{G\pw{t}}(\Gr_G; \Z)$ of $G\pw{t}$-equivariant constructible sheaves on $\Gr_G$. However, \cref{thm: mirkovic-vilonen} does not lift to an equivalence of stable $\infty$-categories between $\Shv^c_{G\pw{t}}(\Gr_G; \Z)$ and the derived $\infty$-category of $\Rep(\ld{G}_\Z)$. Nevertheless, one has:
\begin{theorem}[{Bezrukavnikov-Finkelberg, \cite{bf-derived-satake}}]\label{thm: intro bf derived satake}
    Let $\ld{G} = \ld{G}_\cc$. There is a $\cc$-linear equivalence of monoidal stable $\infty$-categories
    $$\Shv^c_{G\pw{t}}(\Gr_G; \cc) \simeq \Perf(\ld{\g}^\ast[2]/\ld{G}).$$
    Here, $\ld{\g}^\ast[2]$ denotes the derived scheme given by $\ld{\g}^\ast$ placed in (homological) degree $2$. Using Koszul duality, the right-hand side can be rewritten as
    $$\Perf(\ld{\g}^\ast[2]/\ld{G}) \simeq \Coh((\{0\} \times_{\ld{\g}} \{0\})/\ld{G}).$$
\end{theorem}
In \cite[Theorem 6.6.1]{campbell-raskin-satake}, it was shown that the above equivalence can be refined to an equivalence of monoidal factorization stable $\infty$-categories. (We will not need this refinement here, and only mention it for the sake of completeness.) There are several variants of \cref{thm: intro bf derived satake} which have been proved in the literature; of relevance to us is a theorem from \cite{abg-iwahori-satake} concerning the Iwahori subgroup $I$ of $G\pw{t}$. Namely, fix a Borel subgroup $B\subseteq G$ with unipotent radical $N$, and let $I = G\pw{t} \times_G B$. Then:
\begin{theorem}[{Arkhipov-Bezrukavnikov-Ginzburg, \cite{abg-iwahori-satake}}]\label{thm: intro abg}
    There is an equivalence 
    $$\Shv_I^c(\Gr_G; \cc) \simeq \Coh((\tilde{\ld{\cN}} \times_{\ld{\g}} \{0\})/\ld{G}),$$
    where $\tilde{\ld{\cN}} = T^\ast(\ld{G}/\ld{B})$ is the Springer resolution. Using Koszul duality, the right-hand side can be rewritten as
    $$\Coh((\tilde{\ld{\cN}} \times_{\ld{\g}} \{0\})/\ld{G}) \simeq \Perf(\tilde{\ld{\g}}_\cc[2]/\ld{G}_\cc),$$
    where $\tilde{\ld{\g}}_\cc[2] = \ld{G} \times^{\ld{B}} \ld{\fr{n}}^\perp[2]$ is a shifted analogue of the Grothendieck-Springer resolution. 
\end{theorem}
The equivalences of \cref{thm: intro bf derived satake} and \cref{thm: intro abg} admit simpler analogues, where one considers the full subcategories of the ind-completions of $\Shv^c_{G\pw{t}}(\Gr_G; \cc)$ and $\Shv^c_{I}(\Gr_G; \cc)$ generated by the constant sheaf. Let us denote these $\infty$-categories by $\Loc_{G\pw{t}}(\Gr_G; \cc)$ and $\Loc_I(\Gr_G; \cc)$; they are categories of $G\pw{t}$- and $I$-equivariant local systems of complexes of $\cc$-vector spaces on $\Gr_G$. The notion of ``local system'' here must be interpreted in the derived sense, as (equivariant) representations of the fundamental $\infty$-groupoid of $\Gr_G$, and not of its fundamental groupoid (which would be trivial if $G$ is simply-connected). As we will discuss in the body of the article, the above results then restrict to equivalences
\begin{align}
    \Loc_{G\pw{t}}(\Gr_G; \cc) & \simeq \QCoh(\ld{\g}^{\ast,\reg}[2]/\ld{G}), \label{eq: intro reg locus satake} \\
    \Loc_I(\Gr_G; \cc) & \simeq \QCoh(\tilde{\ld{\g}}^{\reg}[2]/\ld{G}). \label{eq: intro reg locus abg}
\end{align}
Here, $\ld{\g}^{\ast,\reg}$ denotes the open subscheme of \textit{regular} elements in $\ld{\g}^\ast$, i.e., those elements whose stabilizer under the coadjoint action of $\ld{G}$ has dimension given by the rank of $\ld{G}$; and similarly for $\tilde{\ld{\g}}^\reg$. We will refer to these equivalences as the derived geometric Satake equivalence (resp. the ABG equivalence) over the regular locus.
In fact, one can can give a proof of \cref{thm: intro bf derived satake} using \cref{eq: intro reg locus satake} combined with the fact that $\ld{\g}^{\ast,\reg}$ has complement of large codimension in $\ld{\g}^\ast$ (see \cite[Section 3.2]{ku-rel-langlands}).

Our goal in this article is to study an analogue of the derived geometric Satake equivalence over the regular locus where $\cc$ is replaced by other ring (spectra) of coefficients. The observation allowing us to do this is that the $\infty$-categories $\Loc_{G\pw{t}}(\Gr_G; \cc)$ and $\Loc_I(\Gr_G; \cc)$ are homotopy-theoretic in nature. Indeed, they depend only on the $G\pw{t}$-equivariant (resp. $I$-equivariant) homotopy type of $\Gr_G$. This, in turn can be understood using a result of Quillen and Garland-Raghunathan (see \cite{garland-raghunathan, mitchell-buildings}), which gives a homotopy equivalence $\Gr_G \simeq \Omega G_c$ (and a homeomorphism onto the subspace of $\Omega G_c$ on those based loops with \textit{polynomial} Fourier expansion); so all computations reduce to homotopy-theoretic statements about $\Omega G_c$. Here, $G_c \subseteq G(\cc)$ is a maximal compact subgroup. Since $G\pw{t}$ is homotopy equivalent to $G_c$ (and $I$ is homotopy equivalent to a maximal torus $T_c \subseteq G_c$), the equivalences \cref{eq: intro reg locus satake} and \cref{eq: intro reg locus abg} can be restated as a pair of equivalences
\begin{align}
    \Loc_{G_c}(\Gr_G; \cc) & \simeq \QCoh(\ld{\g}^{\ast,\reg}[2]/\ld{G}), \label{eq: intro restated reg locus satake} \\
    \Loc_{T_c}(\Gr_G; \cc) & \simeq \QCoh(\tilde{\ld{\g}}^{\reg}[2]/\ld{G}). \label{eq: intro restated reg locus abg}
\end{align}
\subsection{A K-theoretic and elliptic variant}

As mentioned above, our goal in this article is to generalize the equivalence \cref{eq: intro reg locus abg} to the case of K-theoretic and elliptic cohomology coefficients. In order to motivate the discussion, we need to review some of the setup of equivariant generalized cohomology. (For the homotopy theorists: it is exactly the \textit{genuine} nature of equivariant generalized cohomology which lends substance to results like \cref{thm: intro omnibus} below.)

If $G_c$ is a compact Lie group, Atiyah and Segal defined $G_c$-equivariant complex K-theory $\KU_{G_c}$ in \cite{segal-equiv-KU, atiyah-segal-original}: this is a generalized cohomology theory, viewed as a spectrum in the sense of homotopy theory, which classifies $G_c$-equivariant vector bundles on finite $G_c$-spaces. Direct sum and tensor products of $G$-equivariant vector bundles equips $\KU_{G_c}$ with the structure of a \textit{ring} spectrum; in fact, it is an $\Eoo$-ring, meaning (for instance) that the multiplication on cohomology can be refined by Adams operations.
Despite its definition, the geometric interpretation of cocycles for equivariant K-theory as equivariant vector bundles will play \textit{no} role below. 

Two important examples are the following. When $G_c$ is the trivial group, $\KU_{G_c}$ is simply periodic complex K-theory $\KU$, and Bott periodicity gives a graded isomorphism $\pi_\ast \KU \cong \Z[\beta^{\pm 1}]$ with the Bott class $\beta$ in weight $2$. On the other hand, when $G_c$ is a connected compact Lie group with complex representation ring $\mathrm{RU}(G_c)$, the coefficient ring $\pi_\ast \KU_{G_c}$ is the tensor product $\mathrm{RU}(G_c) \otimes_\Z \Z[\beta^{\pm 1}]$. In particular, if $G_c$ is a compact torus $T_c$, then $\spec \pi_\ast \KU_{T_c}$ is the corresponding algebraic torus $T_{\Z[\beta^{\pm 1}]}$ over $\Z[\beta^{\pm 1}]$.

In \cref{sec: equiv coh}, we define a $\KU$-linear $\infty$-category $\Loc_{T_c}(\Gr_G; \KU)$ of $T_c$-equivariant local systems of $\KU$-modules on $\Gr_G$; this should be viewed as a $\KU$-theoretic analogue of the $\infty$-category $\Loc_I(\Gr_G; \KU)$. The ideal analogue of \cref{eq: intro reg locus abg} would identify $\Loc_{T_c}(\Gr_G; \KU)$ with the $\infty$-category of perfect complexes on some stack over $\KU$ defined in terms of the Langlands dual group. Unfortunately, we do not know how to define this putative stack over $\KU$; instead, we will study a particular \textit{degeneration} of this $\infty$-category, denoted $\Loc_{T_c}^\gr(\Gr_G; \KU)$. The reason for studying this degeneration is explained in the introduction to \cref{sec: degenerations}; we will also discuss its ``philosophical'' meaning momentarily. For the moment, let us just note the utility of this degeneration: while $\Loc_{T_c}(\Gr_G; \KU)$ is a $\KU$-linear $\infty$-category, $\Loc_{T_c}^\gr(\Gr_G; \KU)$ is instead an ordinary ($\Z$-linear) $\infty$-category; so we do not have to be concerned with questions such as the definition of the dual group over $\KU$.

Exactly the same setup works if $k$ is a complex oriented $2$-periodic $\Eoo$-ring which is an elliptic cohomology theory (in the sense of \cite{survey, t-equiv-tmf, gepner-meier}) with associated elliptic curve $E$ over $\pi_0(k)$. Namely, we construct a $k$-linear $\infty$-category $\Loc_{T_c}(\Gr_G; k)$, as well as a degeneration $\Loc_{T_c}^\gr(\Gr_G; k)$ which is instead an ordinary ($\pi_0(k)$-linear) $\infty$-category. In both this case and the case of $\KU$, an object $\cf \in \Loc_{T_c}(\Gr_G; k)$ defines a corresponding object $\cf^\gr \in \Loc_{T_c}^\gr(\Gr_G; k)$, and there is a spectral sequence
\begin{equation}\label{eq: intro sseq cohomology and Loc gr}
    \pi_\ast(k) \otimes_{\pi_0(k)} \pi_\ast \Map_{\Loc_{T_c}^\gr(\Gr_G; k)}(\ul{k}^\gr, \cf^\gr) \Rightarrow \pi_\ast \Map_{\Loc_{T_c}(\Gr_G; k)}(\ul{k}, \cf) = \pi_\ast \Gamma_{T_c}(\Gr_G; \cf).
\end{equation}
Here, $\ul{k}$ denotes the constant sheaf.

To state our main result, we need a small observation. Suppose $G$ is a connected, almost simple, and \textit{simply-laced} algebraic group over $\cc$, and let $\ld{G}_\cc$ denote the Langlands dual over $\cc$. Then $\ld{G}_\cc$ is centrally isogeneous to $G$. For instance, if $G$ is simply-connected, $\ld{G}_\cc$ is the quotient of $G$ by its center. In general (under the simply-laced hypothesis), the action of $G$ on itself by conjugation descends/ascends to an action of $\ld{G}_\cc$. In particular, the action of $T_\cc$ on $G$ by conjugation descends/ascends to an action of $\ld{T}_\cc$ on $G$. Therefore, if $T_c$ denote the maximal compact subgroup of $T_\cc$ (and similarly for $\ld{T}_c$), then the conjugation action of $T_c$ on $\Gr_G$ descends/ascends to an action of $\ld{T}_c$ on $\Gr_G$.

The main result of this article is the following.
\begin{theorem}[\cref{cor: reg locus ordinary ABG}, \cref{cor: ku reg locus ordinary ABG}, \cref{cor: ell reg locus ordinary ABG}]\label{thm: intro omnibus}
    Suppose $G$ is a connected, almost simple, and \textit{simply-laced} algebraic group over $\cc$. Let $T \subseteq G$ be a maximal torus, and let $\ld{T}_c$ denote a maximal compact subgroup of the Langlands dual torus over $\cc$. Let $k$ denote either $2$-periodified rational cohomology $\QQ[u^{\pm 1}]$, complex K-theory $\KU$, or elliptic cohomology with associated elliptic curve $E$, and let $F$ be an algebraically closed field over $\pi_0(k)$. Then there are equivalences
    \begin{align*}
        \Loc_{\ld{T}_c}^\gr(\Gr_G; \QQ[u^{\pm 1}]) \otimes_\QQ F & \simeq \QCoh(\tilde{\ld{\g}}^{',\reg}/\ld{G}), \\
        \Loc_{\ld{T}_c}^\gr(\Gr_G; \KU) \otimes_\Z F & \simeq \QCoh(\tilde{\ld{G}}^{\reg}/\ld{G}), \\
        \Loc_{\ld{T}_c}^\gr(\Gr_G; k) \otimes_{\pi_0(k)} F & \simeq \QCoh(\Bun_{\ld{B}}^0(E)^\reg).
    \end{align*}
    Here, the dual groups on the right-hand side are defined over $F$; the final equivalence is assuming that $k$ is an elliptic cohomology theory; $\tilde{\ld{\g}}^{'}$ denotes $\ld{G} \times^{\ld{B}} \ld{\fr{b}}$; $\tilde{\ld{G}}$ denotes $\ld{G} \times^{\ld{B}} \ld{B}$ for the conjugation action of $\ld{B}$ on itself\footnote{We warn the reader that the symbol $\tilde{\ld{G}}$ will mean $\ld{G} \times^{\ld{B}} \ld{B}$ \textit{only} in the introduction; it has a slightly different meaning in the body of the article, described in \cref{def: mult kostant slice}.}; $\Bun_{\ld{B}}^0(E)$ denotes the moduli stack of degree zero $\ld{B}$-bundles on $E$; and the adornment $\reg$ denotes an open ``regular'' locus whose complement has codimension $2$.
\end{theorem}
Since $T_c$ and $\ld{T}_c$ are isogeneous by a finite group, there is no difference between $\Loc_{T_c}^\gr(\Gr_G; \QQ[u^{\pm 1}])$ and $\Loc_{\ld{T}_c}^\gr(\Gr_G; \QQ[u^{\pm 1}])$. Therefore, the first part of \cref{thm: intro omnibus} is, of course, just \cref{eq: intro restated reg locus abg} (once one identifies $\tilde{\ld{\g}}^{'} \cong \tilde{\ld{\g}}$). However, we warn the reader that if $k$ is not $\QQ[u^{\pm 1}]$, the categories $\Loc_{T_c}^\gr(\Gr_G; k)$ and $\Loc_{\ld{T}_c}^\gr(\Gr_G; k)$ are generally genuinely different, and do not agree even upon base-change along $\pi_0(k) \to F$.

\begin{remark}\label{rmk: parabolic variant}
    The same argument used to prove \cref{thm: intro omnibus} shows that if $G$ is a connected, almost simple, and \textit{simply-laced} algebraic group over $\cc$ with torsion-free fundamental group, and $k$ denotes either $2$-periodified rational cohomology $\QQ[u^{\pm 1}]$, complex K-theory $\KU$, or elliptic cohomology with associated elliptic curve $E$, then there are equivalences
    \begin{align*}
        \Loc_{\ld{G}_c}^\gr(\Gr_G; \QQ[u^{\pm 1}]) \otimes_\QQ F & \simeq \QCoh(\ld{\g}^{\reg}/\ld{G}), \\
        \Loc_{\ld{G}_c}^\gr(\Gr_G; \KU) \otimes_\Z F & \simeq \QCoh(\ld{G}^{\reg}/\ld{G}), \\
        \Loc_{\ld{G}_c}^\gr(\Gr_G; k) \otimes_{\pi_0(k)} F & \simeq \QCoh(\Bun_{\ld{G}}^\ss(E)^\reg),
    \end{align*}
    where $\Bun_{\ld{G}}^\ss(E)^\reg \subseteq \Bun_{\ld{G}}^\ss(E)$ is a particular open substack of the moduli stack of semistable degree zero $\ld{G}$-bundles on $E$. More generally, our arguments are easily modified to show that if $L$ is the Levi quotient of a parabolic subgroup $P \subseteq G$ and $L$ has torsion-free fundamental group, then there are equivalences
    \begin{align*}
        \Loc_{\ld{L}_c}^\gr(\Gr_G; \QQ[u^{\pm 1}]) \otimes_\QQ F & \simeq \QCoh(\tilde{\ld{\g}}_{\ld{P}}^{',\reg}/\ld{P}), \\
        \Loc_{\ld{L}_c}^\gr(\Gr_G; \KU) \otimes_\Z F & \simeq \QCoh(\tilde{\ld{G}}_{\ld{P}}^{\reg}/\ld{P}), \\
        \Loc_{\ld{L}_c}^\gr(\Gr_G; k) \otimes_{\pi_0(k)} F & \simeq \QCoh(\Bun_{\ld{P}}^\ss(E)^\reg),
    \end{align*}
    where $\tilde{\ld{\g}}_{\ld{P}}^{'}$ denotes $\ld{G} \times^{\ld{P}} \ld{\fr{p}}$; $\tilde{\ld{G}}_{\ld{P}}$ denotes $\ld{G} \times^{\ld{P}} \ld{P}$ for the conjugation action of $\ld{P}$ on itself; and $\Bun_{\ld{P}}^\ss(E)$ denotes the moduli stack of degree zero semistable $\ld{P}$-bundles on $E$. The equivalence concerning $\Loc_{\ld{L}_c}^\gr(\Gr_G; \QQ[u^{\pm 1}])$ above is closely related to the parabolic variant of \cref{thm: intro abg} which can be deduced from \cite{chen-dhillon}.
\end{remark}
\cref{thm: intro omnibus} also gives an analogue of \cite[Theorem 4]{bf-derived-satake}. Namely, the stacks $\ld{\g}/\ld{G}$, $\ld{G}/\ld{G}$, and $\Bun_{\ld{G}}^\ss(E)$ each have a canonical map to $B\ld{G}$; let $\QCoh_\free(\ld{\g}/\ld{G})$, $\QCoh_\free(\ld{G}/\ld{G})$, and $\QCoh_\free(\Bun_{\ld{G}}^\ss(E))$ denote the essential images of the resulting pullback functors from $\Rep(\ld{G})$ to $\QCoh(\ld{\g}/\ld{G})$, $\QCoh(\ld{G}/\ld{G})$, and $\QCoh(\Bun_{\ld{G}}^\ss(E))$.
Then, we show:
\begin{corollary}[\cref{cor: ordinary minuscule equivalence}, \cref{cor: ku minuscule equivalence}, and \cref{cor: ell minuscule equivalence}]
    In the setting of \cref{thm: intro omnibus}, let $F$ be an algebraically closed field of characteristic zero over $\pi_0(k)$.
    Then there is a $t$-structure on $\Loc_{G_c}^\gr(\Gr_G; k)$ such that there are fully faithful embeddings
    \begin{align*}
        \QCoh_\free(\ld{\g}/\ld{G})^\heartsuit & \hookrightarrow \Loc_{\ld{G}_c}^\gr(\Gr_G; \QQ[u^{\pm 1}])^\heartsuit \otimes_{\QQ} F, \\
        \QCoh_\free(\ld{G}/\ld{G})^\heartsuit & \hookrightarrow \Loc_{\ld{G}_c}^\gr(\Gr_G; \KU)^\heartsuit \otimes_{\Z} F, \\
        \QCoh_\free(\Bun_{\ld{G}}^\ss(E))^\heartsuit & \hookrightarrow \Loc_{\ld{G}_c}^\gr(\Gr_G; k)^\heartsuit \otimes_{\pi_0(k)} F.
    \end{align*}
    When $\ld{G}$ is not of type $E_8$, we explicitly identify the essential image of these embeddings in terms of a topologically defined subcategory of $\Loc_{\ld{G}_c}^\gr(\Gr_G; k)^\heartsuit \otimes_{\pi_0(k)} F$.
\end{corollary}
In the case of K-theory, this is similar to the expectations of \cite{cautis-kamnitzer}. Again, just as in \cref{rmk: parabolic variant}, there are also parabolic versions of these embeddings. In the case of $2$-periodic rational cohomology, one obtains an upgrade (as in \cite[Theorem 2]{bf-derived-satake}): namely, if $\HC^{\hbar,\free}_{\ld{G}}$ denotes the category of Harish-Chandra bimodules of the form $U_\hbar(\ld{\g}) \otimes V$ for $V \in \Rep(\ld{G})$, then there is a fully faithful embedding 
$$(\HC^{\hbar,\free}_{\ld{G}})^\heartsuit \hookrightarrow \Loc_{G_c \times S^1_\rot}^\gr(\Gr_G; k)^\heartsuit.$$
See \cref{cor: reg locus quantized satake} and the discussion surrounding it.
At least in the case $G = \SL_2$, one can similarly relate a category of Harish-Chandra bimodules for the quantum group to $\Loc_{G_c \times S^1_\rot}^\gr(\Gr_G; \KU)^\heartsuit$ (see \cref{prop: rep q fully faithful}).
\begin{remark}
    We also study the effect of power operations (like Steenrod and Adams operations) on $k$ under the Langlands duality of \cref{thm: intro omnibus}. The reader is referred to \cref{thm: frobenius and langlands} for a precise statement, but let us just mention here that using the theory of degree $p$ isogenies on $\GG_a$, $\GG_m$, and the elliptic curve $E$, we show that power operations on $k$ correspond under \cref{thm: intro omnibus} to natural ``Artin-Schreier'' maps on $\tilde{\ld{\g}}'/\ld{G}$, $\tilde{\ld{G}}/\ld{G}$, and $\Bun_{\ld{B}}^0(E)$. (This is different, but related to, Lonergan's work in \cite{lonergan-frob}: we only consider power operations on $k$, while Lonergan considers power operations on the entire $\E{3}$-algebra $k^{G_c}[\Gr_G]$; we will address the latter situation in future work.)
\end{remark}

\begin{remark}
    Let us mention some previous work towards analogues of the geometric Satake equivalence with other coefficients. For instance, an early paper in the context of geometric representation theory is \cite{ginzburg-kapranov-vasserot}. A conjecture about complex K-theoretic geometric Satake was proposed in \cite{cautis-kamnitzer}; in a similar vein, a discussion of the K-theoretic case is the content of the talk \cite{lonergan-slides}. In \cite{yang-zhao-e-thy-quantum-group}, Yang and Zhao study a higher chromatic analogue of quantum groups, and it would be interesting to study the relationship between the present article and their work. After the first version of this paper was written, the preprint \cite{zhong-equiv-homology-of-Gr} was posted on the arXiv; it is concerned with ideas similar to the ones studied here. 
\end{remark}

The proofs of the equivalences in \cref{thm: intro omnibus} with $\QQ[u^{\pm 1}]$-coefficients follow from the work of Bezrukavnikov-Finkelberg-Mirkovic \cite{bfm} and Yun-Zhu \cite{homology-langlands}; and the proofs of the equivalences involving $\KU$ follow from the aforementioned work of Bezrukavnikov-Finkelberg-Mirkovic. However, the approaches taken in these references rely either on the geometric Satake equivalence (for which there is no existing analogue with coefficients in $\KU$ or elliptic cohomology), or on a geometric interpretation of cycles in the relevant cohomology theory (for which there is no known analogue in elliptic cohomology). We therefore reprove these results in the present article so as to provide a \textit{uniform} approach to all the equivalences of \cref{thm: intro omnibus} and of \cref{rmk: parabolic variant}.
\begin{remark}
    In \cref{sec: brylinski zhang}, we also study a degeneration $\Loc_{G_c}^\gr(G_c; k)$ of the category $\Loc_{G_c}(G_c; k)$ of \textit{conjugation-equivariant} locally constant sheaves on $G_c$. Namely, we show that (at least if $G_c$ has torsion-free fundamental group) $\Loc_{G_c}^\gr(G_c; k)$ is equivalent to the category of quasicoherent sheaves on the additive (resp. multiplicative; resp. elliptic) regular centralizer group scheme for $\ld{G}$ if $k = \QQ[u^{\pm 1}]$ (resp. $k = \KU$; resp. $k$ is an elliptic cohomology theory). Motivated by \cite{nadler-zaslow} and \cite[Theorem 1.1]{ganatra-pardon-shende}, one can heuristically interpret our discussion as describing a version of mirror symmetry for the wrapped Fukaya category of the symplectic stack $T^\ast(G_c/G_c)$, albeit with coefficients in the complex-oriented $2$-periodic $\Eoo$-ring $k$. Namely, the ``$k$-theoretic'' mirror to $T^\ast(G_c/G_c)$ is the appropriate variant of the regular centralizer group scheme for the Langlands dual group.
\end{remark}

Let us now discuss further the degeneration of the $\KU$-linear $\infty$-category $\Loc_{T_c}(\Gr_G; k)$ to $\Loc_{T_c}^\gr(\Gr_G; k)$. This can be explained in ``two'' ways\footnote{The quotes are to indicate that these two approaches are really the same; since it would be too digressive to do so here, we will explain the meaning of this (admittedly cryptic) statement in a sequel to this article.}:
\begin{enumerate}
    \item One important lesson from chromatic homotopy theory is that the stable homotopy groups of spheres are closely connected to the coherent cohomology of the moduli stack $\Mfg$ of $1$-dimensional formal groups. For instance, the Adams-Novikov spectral sequence can be restated as a spectral sequence
    $$E_2^{\ast,\ast} \cong \H^\ast(\Mfg; \omega^{\otimes \ast}) \Rightarrow \pi_\ast(S^0),$$
    where $S^0$ is the sphere spectrum and $\omega$ is the line bundle of invariant differentials on the universal $1$-dimensional formal group over $\Mfg$.
    This picture can in fact be categorified: the $\infty$-category $\Sp$ of spectra behaves like the $\infty$-category $\QCoh(\Mfg)$ in a very precise sense\footnote{This is not quite correct: namely, one has to instead replace $\QCoh(\Mfg)$ by the ind-completion of the thick subcategory of $\QCoh(\Mfg)$ generated by tensor powers of $\omega$. For brevity, we will simply denote this variant category by $\QCoh(\Mfg)$. In the homotopy theory literature, e.g., \cite{gregoric-synthetic} or \cite[Definition 5.14]{bhs-artin-tate}, this variant subcategory is often instead denoted by $\IndCoh(\Mfg)$. However, the use of the symbol $\IndCoh(\Mfg)$ does not agree with the more established notion of ind-coherent sheaves from \cite{gr-i}.}. Namely, the Adams-Novikov spectral sequence is categorified by a $1$-parameter degeneration of $\Sp$ into $\QCoh(\Mfg)$; see \cite{piotr-synthetic, gwx-special-fiber, gregoric-synthetic}. Moreover, this degeneration can be constructed using the stable motivic category over $\cc$. This gives a precise sense in which ``topology is approximated by algebra''. The degeneration of $\Loc_{T_c}(\Gr_G; k)$ into $\Loc_{T_c}^\gr(\Gr_G; k)$ is of exactly the same type. In fact, both degenerations ($\Sp \leadsto \QCoh(\Mfg)$ and $\Loc_{T_c}(\Gr_G; k) \leadsto \Loc_{T_c}^\gr(\Gr_G; k)$) can be constructed simultaneously using the even filtration of \cite{even-filtr, piotr-even-filtr}, and we will address this point in a future article. (That is to say, the Adams-Novikov spectral sequence and \cref{eq: intro sseq cohomology and Loc gr} should both be regarded as special cases of a more general construction.)
    \item In geometric representation theory, one often considers ``graded lifts'' of categories of ($\cc$-valued, say) sheaves on a scheme/stack $X$: this is generally defined as the category $\Shv^\mix(X; \cc)$ of \textit{mixed} sheaves on $X$. See \cite{bbdg} and the more recent \cite{ho-li-mixed}. It is generally expected, for instance, that there is a mixed variant of \cref{thm: intro bf derived satake}, stating that there is an equivalence $\Shv^\mix_{G\pw{t}}(\Gr_G; \cc) \simeq \Perf(\ld{\g}^\ast(2)/\ld{G})$, where $\ld{\g}^\ast(2)/\ld{G}$ is now a classical (not derived!) stack, except with a grading which places $\ld{\g}^\ast(2)$ in weight $2$. This grading can be ignored if we replace $\cc$ by its $2$-periodification. One can view $\Loc_{T_c}^\gr(\Gr_G; k)$ as a category of ``mixed'' $k$-valued local systems. (So where is the grading? It is not visible on the right-hand sides of the equivalences in \cref{thm: intro omnibus} is that the $\Eoo$-ring $k$ is assumed to be $2$-periodic; but the grading will reappear if we assume $k$ is ordinary (non-periodic) cohomology or connective K-theory as in \cite{ku-rel-langlands}.) A natural question, of course, is to define a full $\infty$-category $\Shv_I^\gr(\Gr_G; k)$ which specializes to the usual category of mixed sheaves when $k = \cc[u^{\pm 1}]$; we hope to address this in the future article referred to in the preceding bullet point.
\end{enumerate}

The perspective of the degeneration $\Loc_{T_c}(\Gr_G; k) \leadsto \Loc_{T_c}^\gr(\Gr_G; k)$ as being analogous to the degeneration $\Sp \leadsto \QCoh(\Mfg)$ -- and furthermore that both are related to the even filtration of \cite{even-filtr, piotr-even-filtr} -- is very helpful, because it gives us an indication of how to define $\Loc_{T_c}^\gr(\Gr_G; k)$ when $k$ is not necessarily complex-oriented and $2$-periodic. In particular, we also study the example of $k$ being \textit{real} K-theory $\KO$. This is an $\Eoo$-ring with somewhat complicated homotopy groups. Despite this, the $\Eoo$-ring $\KO$ is itself easy to describe using $\KU$: namely, there is an action of $\Z/2$ on $\KU$ by complex conjugation, and $\KO = \KU^{h\Z/2}$. Moreover, just like the standard Adams-Novikov spectral sequence for the homotopy of the sphere spectrum, there is a spectral sequence
$$E_2^{\ast,\ast} \cong \H^\ast(B\Z/2; \omega^\ast) \Rightarrow \pi_\ast(\KO),$$
where $\omega$ is the line bundle over $B\Z/2$ given by the sign action of $\Z/2$ on $\Z$; in fact, this can be identified with the Adams-Novikov spectral sequence for the homotopy of $\KO$. That is, the sphere spectrum is to $\Mfg$ as $\KO$ is to $B\Z/2$. There is also a good notion of $G_c$-equivariant real K-theory.

Instead of constructing a degeneration of the $\KO$-linear $\infty$-category $\Loc_{T_c}(\Gr_G; \KO)$ into a graded $\pi_\ast(\KO)$-linear $\infty$-category, one can construct a degeneration of $\Loc_{T_c}(\Gr_G; \KO)$ into a $\QCoh(B\Z/2)$-module $\infty$-category $\Loc_{T_c}^\gr(\Gr_G; \KO)$. The construction of $\Loc_{T_c}^\gr(\Gr_G; \KO)$ is straightforward: the $\Z/2$-action via complex conjugation on $\Loc_{T_c}(\Gr_G; \KU)$ defines a $\Z/2$-action on $\Loc_{T_c}^\gr(\Gr_G; \KU)$, and this defines the $\QCoh(B\Z/2)$-module category $\Loc_{T_c}^\gr(\Gr_G; \KO)$.
\begin{prop}[\cref{prop: cplx conj KU and B mod B^}]
    Let $\theta$ denote the involution on $\tilde{\ld{G}}/\ld{G}$ given on $\tilde{\ld{G}} = \ld{G} \times^{\ld{B}} \ld{B}$ by $(g,x)\mapsto (g, x^{-1})$, so that the quotient $(\tilde{\ld{G}}/\ld{G})/\pdb{\theta}$ defines a stack over $B\Z/2$.
    Then there is a $\QCoh(B\Z/2)$-linear equivalence
    $$\Loc_{\ld{T}_c}^\gr(\Gr_G; \KO) \otimes_\Z F \simeq \QCoh((\tilde{\ld{G}}^{\reg}/\ld{G})/\pdb{\theta}).$$
\end{prop}
For instance, if $G = \PGL_2$, so that $\tilde{\ld{G}}$ is the moduli of pairs $(g, \ell)$ with $g \in \SL_2$ and $\ell = [x:y] \in \PP^1$ is a line preserved by $g$, the involution $\theta$ simply inverts $g$. We also briefly discuss the case of coefficients in \textit{connective} real K-theory $\ko$.

Similarly, if one fixes a prime $p$ and sets $T_c[p^\infty]$ to be the $p$-power torsion subgroup of $T_c$, one can also define a $\QCoh(B\Z_p^\times)$-module category $\Loc_{T_c[p^\infty]}^\gr(\Gr_G; \Lone S^0)$. This $\infty$-category is a degeneration of the $\infty$-category $\Loc_{T_c[p^\infty]}(\Gr_G; \Lone S^0)$ of $T_c[p^\infty]$-equivariant local systems of $\Lone S^0$-modules on $\Gr_G$, where $\Lone S^0 = (\KU^\wedge_p)^{h\Z_p^\times}$ is the \textit{$K(1)$-local sphere} (also known as the ``image of J'' spectrum) \cite{adams-jx-iv, ravenel-loc}. In this case, let $\tilde{\ld{G}}_{p^\infty} \subseteq \tilde{\ld{G}}$ denote the locus of pairs $(g, \ld{B}') \in \tilde{\ld{G}}$ where $\ld{B}' \subseteq \ld{G}$ is a Borel subgroup containing $g$ such that the eigenvalues of $g$ are all $p$-power roots of unity. Then there is a $\Z_p^\times \times \ld{G}$-action on $\tilde{\ld{G}}_{p^\infty}$, where $n \in \Z_p^\times$ acts by $(g, \ld{B}')\mapsto (g^n, \ld{B}')$, and \cref{prop: imJ reg locus ABG} similarly states that a $\QCoh(B\Z_p^\times)$-linear equivalence
$$\Loc_{\ld{T}_c[p^\infty]}^\gr(\Gr_G; \Lone S^0) \otimes_{\Z_p} F \simeq \QCoh((\tilde{\ld{G}}_{p^\infty}^{\reg}/\ld{G})/\Z_p^\times).$$
The stack on the right-hand side can be thought of as an open substack in the moduli stack of $\ld{B}$-bundles on the $p$-adic solenoid, modulo the natural symmetries (by $\Z_p^\times$) of this solenoid. In \cref{rmk: morava e-theory}, we make some speculations about the analogous picture when $\Lone S^0$ is replaced by the $K(n)$-local sphere $L_{K(n)} S^0$: the primary modification is that one must now consider the moduli stack of $\ld{B}$-bundles on the $p$-adic $n$-dimensional solenoid, modulo an action of the units in the division algebra over $\QQ_p$ with Hasse invariant $1/n$.
\subsection{A unifying picture?}

In this informal section, we suggest a picture which attempts to unify the various calculations in this article. Throughout, we will again take $G$ to be a connected, almost simple, and \textit{simply-laced} algebraic group over $\cc$ with torsion-free fundamental group, and $k$ will denote either $2$-periodified rational cohomology $\QQ[u^{\pm 1}]$, complex K-theory $\KU$, or elliptic cohomology with associated elliptic curve $E$.  (As long as the discussion below is interpreted correctly, one could even take $k$ to be (connective) real K-theory or the $\Eoo$-ring of topological modular forms \cite{tmf}.)

First, just as there is a degeneration of the $k$-linear $\infty$-category $\Loc_{T_c}(\Gr_G; k)$ to an ordinary ($\pi_0(k)$-linear) $\infty$-category $\Loc_{T_c}^\gr(\Gr_G; k)$, we expect to show in future work that there is a well-behaved $k$-linear $\infty$-category $\Shv_I(\Gr_G; k)$ of Iwahori-equivariant sheaves on $\Gr_G$, along with a degeneration to an ordinary $\pi_0(k)$-linear $\infty$-category $\Shv_I^\gr(\Gr_G; k)$. (This is related to the ``faux'' definition from \cite[Construction 3.7.15]{ku-rel-langlands}.) Just as in \cref{thm: intro omnibus}, we also hope to show that there are equivalences
\begin{align*}
    \Shv_{\ld{I}}^\gr(\Gr_G; \QQ[u^{\pm 1}]) \otimes_\QQ F & \simeq \QCoh(\tilde{\ld{\g}}^{'}/\ld{G}), \\
    \Shv_{\ld{I}}^\gr(\Gr_G; \KU) \otimes_\Z F & \simeq \QCoh(\tilde{\ld{G}}/\ld{G}), \\
    \Shv_{\ld{I}}^\gr(\Gr_G; k) \otimes_{\pi_0(k)} F & \simeq \QCoh(\Bun_{\ld{B}}^0(E));
\end{align*}
there should also be similar equivalences when the Iwahori subgroup is replaced by standard parahorics in $L^+ \ld{G}$. In particular, when $\ld{I}$ is replaced by $L^+ \ld{G}$, we expect to prove equivalences
\begin{align*}
    \Shv_{L^+ \ld{G}}^\gr(\Gr_G; \QQ[u^{\pm 1}]) \otimes_\QQ F & \simeq \QCoh(\ld{\g}/\ld{G}), \\
    \Shv_{L^+ \ld{G}}^\gr(\Gr_G; \KU) \otimes_\Z F & \simeq \QCoh(\ld{G}/\ld{G}), \\
    \Shv_{L^+ \ld{G}}^\gr(\Gr_G; k) \otimes_{\pi_0(k)} F & \simeq \QCoh(\Bun_{\ld{G}}^\ss(E)).
\end{align*}
Let us write $\Bun_{G}^\ss(\GG_0^\vee)$ to denote either of the stacks $\g/G$, $G/G$, and $\Bun_{G}^\ss(E)$, depending on the choice of $k$. See \cref{rmk: 1-shifted cartier} for an explanation of this notation: roughly, $\GG_0^\vee$ is the stack of multiplicative line bundles on $\GG_0 = \GG_a$, $\GG_m$, or $E$, respectively. Note that there is a map $\Bun_G^\ss(\GG_0^\vee) \to BG$; in fact, it naturally lifts/descends to a stack $\Bun_{\ld{G}}^\ss(\GG_0^\vee)'$ defined over $B\ld{G}$. (For instance, if $k$ is $2$-periodic rational cohomology or $\KU$, then $\Bun_{\ld{G}}^\ss(\GG_0^\vee)'$ is isomorphic to $\g/\ld{G} \cong \ld{\g}^\ast/\ld{G}$ or $G/\ld{G}$, respectively.)

Suppose $P \subseteq G$ is a parabolic subgroup with Levi quotient $L$ (with torsion-free fundamental group), and let $I_P$ denote its preimage under the map $G\pw{t} \to G$. Set $I_P^0$ to be the kernel of the composite $I_P \to P \to L$, and $\ld{P} \subseteq \ld{G}$ to be the dual parabolic subgroup to $P$. One can then similarly define a stack $\Bun_{\ld{P}}^\ss(\GG_0^\vee)'$, and we expect to prove equivalences
$$\Shv_{I_P}^\gr(\Gr_G; k) \otimes_{\pi_0(k)} F \simeq \QCoh(\Bun_{\ld{P}}^\ss(\GG_0^\vee)').$$
If we replace $I_P$ on the left-hand side by $I_P^0$, then $\Bun_{\ld{P}}^\ss(\GG_0^\vee)'$ must be replaced by the stack $\Bun_{\ld{P}}^{\ss,\nil}(\GG_0^\vee)'$ defined as the fiber of the map $\Bun_{\ld{P}}^\ss(\GG_0^\vee)' \to \Bun_{\ld{P}}^\ss(\GG_0^\vee)^{',\mathrm{coarse}}$ over the basepoint of the trivial bundle. For instance, if $k$ is $2$-periodic rational cohomology or $\KU$, then $\Bun_{\ld{P}}^{\ss,\nil}(\GG_0^\vee)'$ is isomorphic to $\tilde{\cN}_P/\ld{G}$ or $\tilde{\cU}_P/\ld{G}$, respectively, where $\tilde{\cN}_P$ and $\tilde{\cU}_P$ are the partial resolutions (\`a la \cite{borho-macpherson}) of the nilpotent and unipotent cones in $G$, respectively.

The above discussion suggest the following informal picture, which I hope to make precise in future work. There should be a well-defined $(\infty,2)$-category of ``$\infty$-categories with an action of $\Shv(G\ls{t}; k)$'' (this seems to be rather subtle to make sense of in the genuine equivariant setting), and it should admit a graded analogue such that there is a faithful embedding
\begin{equation}\label{eq: pseudo local langlands}
    \LL: \left\{\begin{tabular}{l}
      Quasi-coherent sheaves of \\
      $\infty$-categories on $\Bun_{\ld{G}}^\ss(\GG_0^\vee)'$ \\
    \end{tabular}\right\} \hookrightarrow \left\{\begin{tabular}{l}
      $\infty$-categories with an \\
      action of $\Shv(G\ls{t}; k)$ \\
    \end{tabular}\right\}^\gr \otimes_{\pi_0(k)} F
\end{equation}
whose essential image consists of those $\Shv(G\ls{t}; k)$-module categories which are generated by their $G\pw{t}$-equivariant objects.
This is similar in spirit to (a small part of) the local geometric Langlands correspondence. In particular, if $\cC$ and $\cd$ are quasi-coherent sheaves of $\infty$-categories on $\Bun_{\ld{G}}^\ss(\GG_0^\vee)'$, then there should be an equivalence
$$\Map_{\mathrm{QCohCat}(\Bun_{\ld{G}}^\ss(\GG_0^\vee)')}(\cC, \cd) \xrightarrow{\sim} \Map_{\Shv(G\ls{t}; k)\modc^\gr}(\LL(\cC), \LL(\cd)) \otimes_{\pi_0(k)} F.$$
(One also expects a tamely-ramified variant of \cref{eq: pseudo local langlands}, where the left-hand side is replaced by the category of modules over $\QCoh(\Bun_{\ld{P}}^\ss(\GG_0^\vee)' \times_{\Bun_{\ld{G}}^\ss(\GG_0^\vee)'} \Bun_{\ld{P}}^\ss(\GG_0^\vee)')$ viewed as a monoidal category under convolution. The essential image of this functor would consist of those $\Shv(G\ls{t}; k)$-module categories which are generated by their $I_P$-equivariant objects.)
Some examples of the value of $\LL$ are given in \cref{table: pseudo local}. The calculations of \cite{ku-rel-langlands} also suggest additional lines in this table, at least when $k$ is connective complex K-theory.

The usual local geometric Langlands correspondence \cite{frenkel-gaitsgory-local-langlands} would relate the $(\infty,2)$-category of $\infty$-categories with an action of $\Shv(G\ls{t})$ to the $(\infty,2)$-category of ind-coherent sheaves of $\infty$-categories on the moduli stack $\Loc_{\ld{G}}(D^\times)$ of $\ld{G}$-local systems on the punctured formal disk $D^\times$. Let us try to informally explain the analogy between this and the generalized correspondence suggested above, where one must instead replace $\Loc_{\ld{G}}(D^\times)$ by $\Bun_{\ld{G}}^\ss(\GG_0^\vee)'$. In the de Rham setting, $\Loc_{\ld{G}}(D^\times)$ can informally be viewed as the stack of maps from the (ill-defined) de Rham stack $D^\times_\dR$ of $D^\times$ to $B\ld{G}$. We are then suggesting that one should view $\GG_0^\vee$ as analogous to $D^\times_\dR$ (or at least to a small part thereof).

In fact, this analogy is not new: if $\GG_0$ was a $1$-dimensional formal group (instead of an actual $1$-dimensional algebraic group), then $\GG_0^\vee$ has been studied in the homotopy theory literature \cite{sibilla-tomasini, toen-hkr, moulinos-loop} as the ``$\GG_0$-circle'' $S^1_{\GG_0}$, marketed as a $\GG_0$-analogue of the homotopy type of the usual circle. For instance, there should be an equivalence $R\Gamma(D^\times_\dR; \co) = R\Gamma_\dR(D^\times)$, so that its cohomology is an exterior algebra over $\cc$ on a single class in cohomological degree $1$. Similarly, the cohomology of $R\Gamma(\GG_0^\vee; \co)$ is also an exterior algebra over $\pi_0(k)$ on a single class in cohomological degree $1$ (but outside of characteristic zero, the derived algebra structures on $R\Gamma(\GG_0^\vee; \co)$ and $R\Gamma_\dR(D^\times)$ generally disagree). Thus, the affinizations of $\GG_0^\vee$ and $D^\times_\dR$, at least, behave similarly. This suggests that $\Bun_{\ld{G}}^\ss(\GG_0^\vee)'$ behaves like a slight variant of (the formal neighborhood of the trivial local system in) $\Loc_{\ld{G}}(D^\times)$.

Informally, \cref{eq: pseudo local langlands} is suggesting that the $(\infty,2)$-category of quasicoherent sheaves of $\infty$-categories on $\Bun_{\ld{G}}^\ss(S^1_{\GG_0})$ should be a full subcategory of a graded analogue of the $(\infty,2)$-category of $\infty$-categories with $\Shv(G\ls{t}; k)$-action. The latter is roughly the $(\infty,2)$-category of $\infty$-categories over the stack $\Bun_{G_k}(S^1_k)$ of $G_k$-bundles on the Betti stack (over $k$) of the topological circle\footnote{Here, if $X$ is a suitably nice topological space, we write $X_k$ to denote the stack such that $\QCoh(X_k) \simeq \Shv(X; k)$.}. Thus the Langlands duality of \cref{eq: pseudo local langlands} ``swaps'' the topological circle $S^1$ (manifested as its Betti stack $S^1_k$) with the $\GG_0$-circle $S^1_{\GG_0}$. The merit of this perspective, especially in the context of topological quantum field theories \cite{kapustin-witten, bzsv}, is completely unclear to me.
\subsection{Outline and other results}

We now give an overview of the content of this article. In \cref{sec: regular locus}, we briefly review the derived geometric Satake and the Arkhipov-Bezrukavnikov-Ginzburg equivalences, and show how to deduce the corresponding statements \cref{eq: intro reg locus satake} and \cref{eq: intro reg locus abg} over the regular loci. 

Motivated by the issue of \textit{decompleting} Borel-equivariant cohomology (which appears naturally in studying the derived geometric Satake and the Arkhipov-Bezrukavnikov-Ginzburg equivalences), we recall in \cref{sec: equiv coh} the setup of (genuine) equivariant generalized cohomology following \cite{survey, t-equiv-tmf, gepner-meier}. For a compact abelian group $T_c$ and a finite $T_c$-space $X$, we also introduce the category $\Loc_{T_c}(X; k)$ of equivariant local systems of $k$-modules on $X$ for an $\Eoo$-ring $k$ equipped with an ``oriented'' $1$-dimensional commutative group scheme $\GG$. This section also reviews/generalizes the all-important Atiyah-Bott localization theorem \cite{atiyah-bott-localization}, and the corresponding results of Goresky-Kottwitz-MacPherson \cite{gkm-original}.

In \cref{sec: degenerations}, we introduce the degeneration $\Loc_{T_c}^\gr(\Gr_G; k)$ of the $\infty$-category $\Loc_{T_c}(\Gr_G; k)$ that we discussed in the preceding subsection, and also discuss the definition of this $\infty$-category for the non-complex-oriented $\Eoo$-ring $\KO$. The definition of $\Loc_{T_c}^\gr(\Gr_G; k)$ depends only on the equivariant homology $k^{T_c}_\ast(\Gr_G)$ (denoted by $\pi_\ast \cf_{T_c}(\Gr_G)^\vee$ in the body of this article) as a coalgebra/Hopf algebroid over the equivariant cohomology $\pi_\ast k_{T_c}$ of a point.

The next \cref{sec: torus loop rot} is essentially purely combinatorial: it studies the case when $G$ is a torus $T$, and imposes the additional data of loop-rotation equivariance. In terms of the homotopy equivalence $\Gr_G \cong \Omega G_c$, this comes from the $S^1$-action on $\Omega G_c$ obtained by viewing it as $\Omega^2 (BG_c) = \Map_\ast(S^2, BG_c)$ and using the $S^1$-action by rotation on $S^2$. Namely, we show that if $k$ is an $\Eoo$-ring equipped with an ``oriented'' $1$-dimensional commutative group scheme $\GG$, then $k^{T_c \times S^1_\rot}_\ast(\Gr_T)$ (or really, its sheafification over $\Hom(\bX^\ast(T), \GG)$) can be identified with a \textit{$\GG$-analogue} of the Weyl algebra of the Langlands dual torus $\ld{T}$. For instance, if $k = \QQ[u^{\pm 1}]$ and $\GG = \GG_a$, then $k^{T_c \times S^1_\rot}_\ast(\Gr_T)$ is the rescaled Weyl algebra of $\ld{T}$; similarly, if $k = \KU$ and $\GG = \GG_m$, then $k^{T_c \times S^1_\rot}_\ast(\Gr_T)$ is the $q$-Weyl algebra of the dual torus $\ld{T}$. We also explain the relationship between this {$\GG$-analogue} of the Weyl algebra of $\ld{T}$ and the ``$F$-de Rham complex'' of \cite{generalized-n-series}.

In \cref{sec: review Q coeff}, we review the classical story concerning $\Loc_{T_c}^\gr(\Gr_G; k)$ when $k = \QQ[u^{\pm 1}]$. The purpose of this section is to reprove the results of \cite{bfm, homology-langlands} using only techniques amenable to generalization to other equivariant cohomology theories. In particular, in \cref{cor: reg locus ordinary ABG}, we reprove the equivalence between $\Loc_{T_c}^\gr(\Gr_G; k)$ and $\QCoh(\tilde{\ld{\g}}^\reg/\ld{G})$. The remainder of \cref{sec: review Q coeff} is concerned with the question of loop-rotation equivariance. Using results of \cite{ginzburg-kapranov-vasserot-residue-hecke} and \cite{ginzburg-whittaker}, we prove that $\Loc_{G_c \times S^1_\rot}^\gr(\Gr_G; k)$ can be identified with a certain localization of the Harish-Chandra category $\HC^\hbar_{\ld{G}} = U_\hbar(\ld{\g})\modc^{(\ld{G}, \weak)}$. This line of argument is, in some sense, precisely the opposite of that of \cite{lonergan-fourier}.

We finally turn to the K-theoretic story in \cref{sec: KU coeff}. \cref{cor: ku reg locus ordinary ABG} therein states that $\Loc_{T_c}^\gr(\Gr_G; k)$ is equivalent to $\QCoh(\tilde{\ld{G}}^\reg/\ld{G})$ when $G$ is connected, almost simple, and simply-laced; in this section, unlike in \cref{thm: intro omnibus}, the multiplicative Grothendieck-Springer resolution $\tilde{\ld{G}}$ is defined to be $\ld{G} \times^{\ld{B}} B$ (instead of $\ld{G} \times^{\ld{B}} \ld{B}$). This is to be understood as analogous to the usual Grothendieck-Springer resolution $\tilde{\ld{\g}}$, which is defined to be $\ld{G} \times^{\ld{B}} \ld{\fr{n}}^\perp$ (as opposed to $\ld{G} \times^{\ld{B}} \ld{\fr{b}}$). We also briefly study the question of loop-rotation equivariance in \cref{thm: ku loop-rot flag} using degenerate affine nil-Hecke algebras, and phrase some precise expectations about the relationship to the representation theory of quantum groups; but we do not yet know how to prove these statements. In \cref{prop: cplx conj KU and B mod B^}, we also study the effect under Langlands duality of complex conjugation on equivariant K-theory.

The elliptic story is studied in \cref{sec: ell coeff}, where we use the important results of \cite{davis-elliptic-springer} to show in \cref{cor: ell reg locus ordinary ABG} that $\Loc_{\ld{T}_c}^\gr(\Gr_G; k)$ is equivalent to a localization of $\QCoh(\Bun_{\ld{B}}^0(E))$ when $G$ is connected, almost simple, and simply-laced. We also briefly study the question of loop-rotation equivariance in \cref{thm: ell loop-rot flag}, but do not even know how to describe the expected Langlands dual story. It should, however, be related to the representation theory of elliptic quantum groups \cite{felder-elliptic-quantum}.

The remainder of this article is concerned with comparisons to (by now) classical constructions in equivariant algebraic topology. \cref{sec: power operations} studies the effect of ``power operations'' on $k$ under the Langlands duality of \cref{thm: intro omnibus}. These are additional symmetries of the $\Eoo$-ring $k$ which yield the theory of Steenrod operations in ordinary cohomology and Adams operations in K-theory. (This is closely related to, but distinct from, work \cite{lonergan-frob} of Lonergan.) We review how the theory of isogenies of $\GG_0$ controls power operations for $k$, which is used to show that power operations for $k$ identify with natural ``Artin-Schreier'' type maps on $\ld{\g}/\ld{G}$, $\ld{G}/\ld{G}$, and $\Bun_{\ld{G}}^\ss(E)$. As an amusing application, we reprove that the restricted Lie operation vanishes on the nilpotent cone in characteristics at least the Coxeter number (see \cref{prop: pth power zero on nilcone}).

In \cref{sec: brylinski zhang}, we explain how the degeneration of $\Loc_{T_c}(\Gr_G; k)$ to $\Loc_{T_c}^\gr(\Gr_G; k)$ should be viewed as analogous to the Hochschild-Kostant-Rosenberg degeneration of Hochschild homology to differential forms. (See \cite{raksit} for a modern take on this degeneration.) Using this perspective, we show how (when $G$ has torsion-free fundamental group) \cref{thm: intro omnibus} recovers results of \cite{brylinski-zhang} identifying the conjugation-equivariant cohomologies $\H^\ast_{G_c}(G_c; \QQ)$ (resp. $\KU^\ast_{G_c}(G_c)$) with the algebra of K\"ahler differentials on $\fr{t}\mmod W$ (resp. $T\mmod W$); the same argument also describes the equivariant elliptic cohomology of $G_c$ in terms of the algebra of K\"ahler differentials on the moduli \textit{space} of semistable degree zero $G$-bundles on the elliptic curve.

In \cref{sec: lifting SL2}, we show by explicit calculation that the group scheme $\SL_2$ (as well as other classical geometric objects, like Grassmannians which are not projective spaces) cannot admit a natural lifting from $\Z$ to the sphere spectrum or even to connective complex K-theory. The argument goes by showing that $\SL_2$ cannot admit a reasonable $\delta$-structure (even mod $p$). This supplements a comment made in the beginning of \cref{sec: degenerations}.

The results of this article were motivated by the work of Hopkins-Kuhn-Ravenel \cite{hkr} describing the generalized equivariant cohomology of \textit{finite} groups. In \cref{sec: comparison hkr}, we briefly review their results (and the corresponding categorifications, due to Lurie \cite{elliptic-iii}). Despite the case of finite groups being the diametric opposite to the case of connected compact Lie groups studied in this article, we give a heuristic argument showing that \cref{thm: intro omnibus} can be viewed as an analogue of (some of) the results of \cite{hkr, elliptic-iii}.

Another motivation for the results of this article came from physics. Namely, the equivariant homology of $\Gr_G$ describes the Coulomb branch of $3$d $\cN=4$ pure gauge theory \cite{bfn-ii}, and one expects that the generalized equivariant $\KU$-homology (resp. elliptic homology) of $\Gr_G$ is related to the Coulomb branch of $4$d $\cN=2$ (resp. $5$d $\cN = 1$) pure gauge theory. We briefly review this story in \cref{sec: coulomb}, and give explicit generators and relations for the Coulomb branches of $3$d $\cN = 4$ and $4$d $\cN = 2$ pure gauge theories with gauge groups $\SL_2$ and $\PGL_2$. The $4$-dimensional case is a $q$-analogue of the quantization of the Atiyah-Hitchin manifold \cite{atiyah-hitchin} from \cite[Equation 5.51]{bullimore-dimofte-gaiotto}.
\subsection{Notation and terminology}

In most of this article, $G$ will denote a connected, almost simple, and \textit{simply-laced} algebraic group over $\cc$, and $B$ will denote a Borel subgroup therein. If $H$ is a reductive algebraic group over $\cc$, we will write $H_c$ to denote the maximal compact subgroup of the complex Lie group $H(\cc)$, so $H_c$ is a compact Lie group. We will use the terminology ``finite $H_c$-space'' to mean a finite $H_c$-CW complex. I have tried to be careful to add the subscript $c$ where necessary, but some omissions have certainly inevitably crept in.

The symbol $k$ will denote an $\Eoo$-ring which will generally be fixed at the beginning of each section/theorem statement. The symbol $F$ will denote an algebraically closed field with a map $\pi_0(k) \to F$. If $H$ is a group scheme acting on a scheme $Y$, the symbol $Y/H$ will mean the stacky quotient, and $Y\mmod H$ will denote the invariant-theoretic quotient. The Langlands dual group $\ld{G}$ will generally be defined over $F$; if we wish to view it as defined over a commutative ring $R$, we will denote it by $\ld{G}_R$. If $T$ is a torus, we will write $\bX^\ast(T)$ and $\bX_\ast(T)$ to denote its lattice of characters and cocharacters. We will also write $\ld{\Lambda}$ to denote the root lattice of $G$ and $\Lambda$ to denote the coroot lattice of $G$. The symbol $\tilde{T}$ will denote the extended torus $T \times \GG_m^\rot$. If $\ld{N}\subseteq \ld{G}$ is the unipotent radical of a Borel subgroup of $\ld{G}$, we will write $\psi: \ld{\fr{n}} \to \GG_a$ to denote a nondegenerate character of $\ld{\fr{n}}$, i.e., an additive character which is nonzero on each simple root space. If $X$ is an $\ld{N}$-scheme, we will write $T^\ast(X/_\psi \ld{N})$ to denote the Hamiltonian reduction of $T^\ast X$ at $\psi$ (that is, if $\mu: T^\ast(X) \to \ld{\fr{n}}^\ast$ is the moment map, then $T^\ast(X/_\psi \ld{N}) \cong \mu^{-1}(\psi)/\ld{N}$).

Finally, we will write $\cS$ to denote the $\infty$-category of spaces (also known as ``anima'', but we will not use this terminology here). The symbols $\QCoh$, $\Mod$, etc. are all to be understood in the \textit{derived} sense, as are all fiber and tensor products; we will make it clear if any of these operations are to be understood in their classical sense. In particular, if $X$ is a scheme or stack, we will write $\QCoh(X)^\heartsuit$ to denote the abelian category of quasicoherent sheaves on $X$. If $A$ is an $\Eoo$-ring, $\cC$ is an $A$-linear $\infty$-category (that is, a $\Mod_A$-module in the $\infty$-category of presentable stable $\infty$-categories), and $A \to B$ is a map of $\Eoo$-rings, then $\cC \otimes_A B$ will denote the $B$-linear $\infty$-category $\cC \otimes_{\Mod_A} \Mod_B$; similarly if $A$ is a classical commutative ring and $\cC$ is an abelian $A$-linear category.

\subsection{Acknowledgements}

I am grateful to Lin Chen, Charles Fu, Tom Gannon, and Kevin Lin for helpful conversations and for entertaining my numerous silly questions. I am also grateful to David Ben-Zvi, Victor Ginzburg, Yiannis Sakellaridis, David Treumann, and Akshay Venkatesh for very enlightening discussions, and Pavel Safronov for a useful email. Thanks to Ben Gammage for discussions which helped shape my understanding of some of the topics in \cref{sec: coulomb}, and to Hiraku Nakajima for a very informative email exchange on the same topic. My interest in this area started after I took a class taught by Roman Bezrukavnikov when I was an undergraduate; I am very grateful to him for suggesting that I read \cite{chriss-ginzburg}, and also for introducing me to \cite{bfm}, which led me down the beautiful road to geometric representation theory. Since the first version of this article, my understanding of the subject has also been influenced by \cite{bzsv}. I would also like to thank some anonymous referees for helpful suggestions on improving this article. Last, but certainly far from least, the influence, support, advice, and encouragement of my advisors Dennis Gaitsgory and Mike Hopkins is evident throughout this project; I cannot thank them enough.
\begin{landscape}
\begin{table}
\centering
\vspace*{-1cm}
{
\begin{tabular}{ |c|c|c|c|c|c|c|c| } 
\hline
Stack over $\Bun_{\ld{G}}^\ss(\GG_0^\vee)'$ & Image under $\LL$ & Citation \\
\hline
$\Bun_{\ld{G}}^\ss(\GG_0^\vee)'$ & $\Shv(\Gr_G; k)$ & \cref{thm: intro omnibus} \\
$\Bun_{\ld{P}}^\ss(\GG_0^\vee)'$ & $\Shv(G\ls{t}/I_P; k)$ & \cref{rmk: parabolic variant} \\
$\Bun_{\ld{P}}^{\ss,\nil}(\GG_0^\vee)'$ & $\Shv(G\ls{t}/I_P^0; k)$ & \cref{rmk: parabolic variant} \\
$\Bun_{\ld{G}}^\ss(\GG_0^\vee)^{',\mathrm{coarse}}$ & $\Shv(G\ls{t}/G\ls{t}; k) = \Mod_k$ & Kostant/Steinberg/elliptic slices \\
$B\ld{G}$ & $\mathrm{Kir}(\Shv(G\ls{t}; k))$ & Conjectural in general; \cite{geometric-casselman-shalika-i, geometric-casselman-shalika-ii} when $k = \QQ[u^{\pm 1}]$ \\
$\ld{G}\backslash \ol{T^\ast_{\GG_0}(\ld{G}/\ld{N})}$ & $\Shv^\sph(G\ls{t}/T\ls{t}; k)$ & \cref{prop: ordinary gelfand-graev}, \cref{prop: ku gelfand-graev}, \cref{prop: ell gelfand-graev} \\
$\GL_n\backslash \cB_\beta(\AA^n, \AA^{n-1})/\GL_{n-1}$ & $\Shv^\sph(\GL\ls{t}_n; k)$ for $k = \QQ[u^{\pm 1}], \KU$ & \cite{mirabolic-satake}, \cite[Remark 4.3.4]{ku-rel-langlands} \\
$\gl_n/\GL_n$ & $\Shv^\sph(\GL\ls{t}_{2n}/\Sp\ls{t}_{2n}; k)$ for $k = \QQ[u^{\pm 1}], \KU$ & \cite{quat-satake}, \cref{ex: symplectic period} \\
\hline
\end{tabular}
}
\vspace{.5cm}
\caption{A (non-exhaustive) list of some stacks $\fr{X}$ over $\Bun_{\ld{G}}^\ss(\GG_0^\vee)'$, and the image of $\QCoh(\fr{X})$ under the conjectural functor \cref{eq: pseudo local langlands}. The citation in the table is possibly to a statement only concerning a ``regular locus'' in $\fr{X}$; but with the ``faux'' definition of graded sheaves from \cite[Construction 3.7.15]{ku-rel-langlands}, the statement can be extended to one about the entirety of $\QCoh(\fr{X})$. For most of these cases, natural symmetries on $k$ match with corresponding evident symmetries on the stack over $\Bun_{\ld{G}}^\ss(\GG_0^\vee)'$. For instance, power operations on $k$ match with the effect of isogenies on $\GG_0$; see \cref{sec: power operations}. However, the meaning of such symmetries on $k$ under this Langlands duality is sometimes less clear (like in the final line; see \cref{ex: symplectic period}). When $k = \QQ[u^{\pm 1}]$ or $\ku$, there are many other examples one could add to this table coming from relative Langlands duality \cite{bzsv}; see there, as well as some of the calculations in \cite{ku-rel-langlands}, for a more comprehensive list in these cases.
\newline
\newline
In the fifth row, $B\ld{G}$ is viewed as a stack over $\Bun_{\ld{G}}^\ss(\GG_0^\vee)'$ via pullback along the map $\GG_0^\vee \to \spec(F)$. The category $\mathrm{Kir}(\Shv(G\ls{t}; k))$ is the ``Kirillov model'' from \cite{gaitsgory-lysenko-kirillov}; the matching in this line should be a version of the geometric Casselman-Shalika formula \cite{geometric-casselman-shalika-i, geometric-casselman-shalika-ii}. See also \cite[Section 10]{lurie-icm} for a related discussion.
In the following row, the scheme $\ol{T^\ast_{\GG_0}(\ld{G}/\ld{N})}$ is introduced below in \cref{prop: ordinary gelfand-graev}, \cref{prop: ku gelfand-graev}, and \cref{prop: ell gelfand-graev}, and $\Shv^\mathrm{sph}(G\ls{t}/T\ls{t}; k)$ is the full subcategory of $\Shv(G\ls{t}/T\ls{t}; k)$ generated by the $G\pw{t}$-equivariant objects.
\newline
\newline
The next row has $G = \GL_n \times \GL_{n-1}$. In the left column, $\cB_\beta(\AA^n, \AA^{n-1})$ denotes the Hamiltonian $\GL_n \times \GL_{n-1}$-space $T^\ast \Hom(\AA^n, \AA^{n-1})$ if $k = \QQ[u^{\pm 1}]$, and denotes Van den Bergh's quasi-Hamiltonian $\GL_n \times \GL_{n-1}$-space \cite{van-den-bergh-double-poisson} of pairs $(u,v)\in \Hom(\AA^n, \AA^{n-1}) \times \Hom(\AA^{n-1}, \AA^n)$ such that $\id + uv$ is invertible if $k = \KU$. In the next column, $\GL_n$ is viewed as a $\GL_n \times \GL_{n-1}$-space by the left and right actions.
\newline
\newline
The next line concerns the quaternionic Satake equivalence, where $G = \GL_{2n}$. When $k = \QQ[u^{\pm 1}]$, the stack $\gl_n/\GL_n$ lives over $\Bun_{\GL_{2n}}^\ss(\GG_a^\vee) = \gl_{2n}/\GL_{2n}$ via the diagonal embedding $\GL_n \subseteq \GL_{2n}$ and the map $\gl_n \to \gl_{2n}$ sending $x\mapsto \begin{psmallmatrix}
    0 & \id_n \\
    x & 0
\end{psmallmatrix}$ (see \cite{quat-satake}). When $k = \KU$, the stack $\gl_n/\GL_n$ lives over $\Bun_{\GL_{2n}}^\ss(\GG_m^\vee) = \GL_{2n}/\GL_{2n}$ via the diagonal embedding $\GL_n \subseteq \GL_{2n}$ and the map $\gl_n \to \GL_{2n}$ sending $x\mapsto \begin{psmallmatrix}
    x+\id_n & \id_n \\
    x & \id_n
\end{psmallmatrix}$. As explained in \cref{rmk: ku symplectic}, one can interpolate between these two cases; but the situation is genuinely different when $k$ is elliptic cohomology (because the quaternionic affine Grassmannian is not $k$-orientable). 
}
\label{table: pseudo local}
\end{table}
\end{landscape}
\newpage

\section{The regular locus}\label{sec: regular locus}
In this section, we will quickly review the derived geometric Satake equivalence following \cite{bf-derived-satake} and \cite{arinkin-gaitsgory-singsupp}. Let $k$ denote a commutative $\QQ$-algebra; all Langlands dual objects will be assumed to live over $k$, and are base-changes of their ``split forms'' over $\QQ$.
\begin{setup}
    Let $G$ be a connected reductive group (over $\cc$, always), and let $\Gr_G = G\ls{t}/G\pw{t}$ denote the affine Grassmannian. There is a canonical left action of $G\ls{t}$ on $\Gr_G$, and hence an action of $G\pw{t} \subseteq G\ls{t}$. The affine Grassmannian is a union of $G\pw{t}$-invariant closed subschemes $X_\alpha$ of finite type, and one defines $\Shv_{G\pw{t}}(\Gr_G; k) = \colim_\alpha \Shv_{G\pw{t}}(X_\alpha; k)$.
    Inside $\Shv_{G\pw{t}}(\Gr_G; k)$ are two full subcategories: 
    \begin{itemize}
        \item $\Shv_{G\pw{t}}(\Gr_G; k)^\lcc$ is the full subcategory of objects whose image under the forgetful functor $\Shv_{G\pw{t}}(\Gr_G; k) \to \Shv(\Gr_G; k)$ is compact. Such objects are called ``locally compact''.
        \item $\Shv_{G\pw{t}}(\Gr_G; k)^\omega$ of compact objects in $\Shv_{G\pw{t}}(\Gr_G; k)$.
    \end{itemize}
    The $\infty$-category $\Shv_{G\pw{t}}(\Gr_G; k)$ admits a monoidal structure, which in fact restricts to a monoidal structure on each of the full subcategories above.
\end{setup}
\begin{setup}
    Let $(e,f,h)$ denote a principal $\sl_2$-triple in the Langlands dual Lie algebra $\ld{\g}$. The element $f$ defines a nondegenerate character $\psi: \ld{\fr{n}} \to \AA^1$. Let $\ld{\g}^{\ast,e}$ denote the orthogonal complement to the subspace $[e,\ld{\g}] \subseteq \ld{\g}$. This defines the \textit{Kostant slice} $\psi + \ld{\g}^{\ast,e} \subseteq \ld{\g}^\ast$; we will denote this inclusion by $\kappa$. Composing the invariant-theoretic quotient map $\chi: \ld{\g}^\ast \to \ld{\g}^\ast\mmod \ld{G}$ with the Kostant slice defines an isomorphism. In other words, the following composite is an isomorphism:
    $$\psi + \ld{\g}^{\ast,e} \xar{\kappa} \ld{\g}^\ast \xar{\chi} \ld{\g}^\ast\mmod \ld{G}.$$
    It will be convenient to identify $\psi + \ld{\g}^{\ast,e}$ with $\ld{\g}^\ast\mmod \ld{G}$ under this isomorphism. If the vector space $\ld{\g}^\ast$ is placed in weight $2$, the map $\kappa$ can be checked to give a \textit{graded} map
    $$\kappa: \ld{\g}^\ast(2)\mmod \ld{G} \to \ld{\g}^\ast(2).$$
    Shearing this graded map (in the sense of \cite[Section 2.1]{ku-rel-langlands}) defines a map $\ld{\g}^\ast[2]\mmod \ld{G} \to \ld{\g}^\ast[2]$, which we will also denote by $\kappa$.
\end{setup}
\begin{lemma}[Chevalley restriction]
    There is an isomorphism $\ld{\g}^\ast\mmod \ld{G} \cong \fr{t}\mmod W$, which refines to an isomorphism of graded schemes
    $$\ld{\g}^\ast(2)\mmod \ld{G} \cong \fr{t}(2)\mmod W \cong \spec \H^\ast_G(\ast; \cc).$$
\end{lemma}
The first part of the following result is \cite[Theorem 5]{bf-derived-satake}, and the second part is \cite[Theorem 12.5.3]{arinkin-gaitsgory-singsupp}.
\begin{theorem}[Bezrukavnikov-Finkelberg, Arinkin-Gaitsgory]\label{thm: derived satake}
    There is a monoidal equivalence
    $$\Shv_{G\pw{t}}(\Gr_G; k)^\lcc \simeq \Perf(\ld{\g}^\ast[2]/\ld{G}),$$
    which restricts to a monoidal equivalence
    $$\Shv_{G\pw{t}}(\Gr_G; k)^\omega \simeq \Perf_{\ld{\cN}/\ld{G}}(\ld{\g}^\ast[2]/\ld{G}),$$
    where the right-hand side is the full subcategory of those perfect complexes which are set-theoretically supported on the nilpotent cone of $\ld{\g}^\ast$.
    Furthermore, there is a commutative diagram
    $$\xymatrix{
    \Ind(\Shv_{G\pw{t}}(\Gr_G; k)^\lcc) \ar[r]^-\sim \ar[d]_{p_!} & \QCoh(\ld{\g}^\ast[2]/\ld{G}) \ar[d]^-{\kappa^\ast} \\
    \Shv_{G\pw{t}}(\ast; k) \ar[r]_-\sim & \QCoh(\ld{\g}^\ast[2]\mmod \ld{G}),
    }$$
    where $p: \Gr_G \to \ast$ is the canonical map to a point and $\kappa^\ast$ is pullback along the (shifted) Kostant slice.
\end{theorem}
We will refer to the first equivalence of \cref{thm: derived satake} as the \textit{derived geometric Satake equivalence}, or more colloquially as ``derived Satake''.
\begin{definition}
    A point $x\in \ld{\g}^\ast$ is called \textit{regular} if its centralizer $Z_{\ld{G}}(x)\subseteq \ld{G}$ has dimension given by the rank of $\ld{G}$. 
    Let $\ld{\g}^{\ast,\reg}$ denote the locus of regular elements; this is an open subscheme whose complement is of codimension $3$.
\end{definition}
\begin{theorem}[{Kostant, \cite{kostant-lie-group-reps}}]\label{thm: kostant reg locus}
    The $\ld{G}$-orbit of the Kostant slice $\kappa: \ld{\g}^\ast\mmod \ld{G} \to \ld{\g}^\ast$ identifies with the regular locus $\ld{\g}^{\ast,\reg}$.
\end{theorem}
\begin{corollary}\label{cor: reg locus satake}
    Let $\ul{k}_{\Gr_G} \in \Ind(\Shv_{G\pw{t}}(\Gr_G; k)^\lcc)$ denote the constant sheaf, and let $\Loc_{G\pw{t}}(\Gr_G; k)$ denote the full subcategory generated by $\ul{k}_{\Gr_G}$. Then there is an equivalence
    $$\Loc_{G\pw{t}}(\Gr_G; k) \simeq \QCoh(\ld{\g}^{\ast,\reg}[2]/\ld{G}).$$
\end{corollary}
\begin{proof}
    Observe that $\ul{k}_{\Gr_G}$ is the pullback $p^\ast \ul{k}$ of the (necessarily constant) sheaf $\ul{k} \in \Shv_{G\pw{t}}(\ast; k)$. Since $p^\ast$ is the right adjoint to $p_!$ (and $\kappa_\ast$ is the right adjoint to $\kappa^\ast$), the commutative diagram of \cref{thm: derived satake} says that $\Loc_{G\pw{t}}(\Gr_G; k)$ is equivalent to the full subcategory of $\QCoh(\ld{\g}^\ast[2]/\ld{G})$ generated by $\kappa_\ast \co_{\ld{\g}^\ast[2]\mmod \ld{G}}$. However, \cref{thm: kostant reg locus} implies that this full subcategory is equivalent to $\QCoh(\ld{\g}^{\ast,\reg}[2]/\ld{G})$, as desired.
\end{proof}
A parallel story holds for the Arkhipov-Bezrukavnikov-Ginzburg (called ``ABG'' in this article) equivalence from \cite{abg-iwahori-satake}. 
\begin{recall}
    Let $\tilde{\ld{\g}}$ denote the Grothendieck-Springer resolution, so that $\tilde{\ld{\g}} = T^\ast(\ld{G}/\ld{N})/\ld{T}$. The action of $\ld{G}$ on $T^\ast(\ld{G}/\ld{N})$ defines the moment map $\mu: \tilde{\ld{\g}} \to \ld{\g}^\ast$. Let $\tilde{\ld{\g}}^{\reg}$ denote the preimage of the regular locus $\ld{\g}^{\ast,\reg} \subseteq \ld{\g}^\ast$ under the moment map $\mu$.
\end{recall}
\begin{prop}\label{prop: psi + t}
    There is an isomorphism $\tilde{\ld{\g}} \cong \ld{G} \times^{\ld{B}} \ld{\fr{n}}^\perp$, as well as a map $\kappa: \psi + \ld{\fr{t}}^\ast \subseteq \ld{\fr{n}}^\perp$ which fits into a Cartesian square
    $$\xymatrix{
    \psi + \ld{\fr{t}}^\ast \ar[r] \ar[d] & \ld{\fr{n}}^\perp \ar[r] & \tilde{\ld{\g}} \ar[d]^-\mu \\
    \psi + \ld{\g}^{\ast,e} \ar[rr] & & \ld{\g}^\ast.
    }$$
\end{prop}
\begin{proof}
    Let $\ld{M}$ be a Hamiltonian $\ld{G}$-scheme with moment map $\mu: \ld{M} \to \ld{\g}^\ast$. Then the pullback $\ld{M} \times_{\ld{\g}^\ast} (\psi + \ld{\g}^{\ast,e})$ can be identified with the Whittaker reduction $\ld{M}/_\psi \ld{N}$. Indeed, a theorem of Kostant's from \cite{kostant-whittaker} identifies $\psi + \ld{\g}^{\ast,e}$ with $(\psi + \ld{\fr{n}}^{-,\perp})/\ld{N}^-$, so that there are isomorphisms
    \begin{align*}
        \ld{M} \times_{\ld{\g}^\ast} (\psi + \ld{\g}^{\ast,e}) & \cong \ld{M}/\ld{G} \times_{\ld{\g}^\ast/\ld{G}} (\psi + \ld{\g}^{\ast,e}) \\
        & \cong \ld{M}/\ld{G} \times_{\ld{\g}^\ast/\ld{G}} (\psi + \ld{\fr{n}}^{-,\perp})/\ld{N}^- \\
        & \cong (\ld{M} \times_{\ld{\fr{n}}^{-,\ast}} \{\psi\})/\ld{N}^- = \ld{M}/_\psi \ld{N}^-.
    \end{align*}
    Therefore, the fiber product in the statement of the proposition identifies with the Whittaker reduction $\tilde{\ld{\g}}/_\psi \ld{N}^-$. Since $\tilde{\ld{\g}} \cong T^\ast(\ld{G}/\ld{N})/\ld{T}$, we may identify $\tilde{\ld{\g}}/_\psi \ld{N}^-$ with the quotient by $\ld{T}$ of $T^\ast(\ld{N}^- {}_\psi\backslash \ld{G}/\ld{N})$. Since Whittaker functions are supported on the big cell, this twisted cotangent bundle is in turn isomorphic to $T^\ast(\ld{N}^- {}_\psi\backslash (\ld{N}^- \times \ld{T} \times \ld{N})/\ld{N}) \cong \ld{T} \times (\psi + \ld{\fr{t}}^\ast)$. The desired Cartesian square follows.
\end{proof}
Again, $\tilde{\ld{\g}}$ admits a $\GG_m$-action obtained by placing $\ld{\fr{n}}^\perp$ in weight $2$, and the map $\kappa: \ld{\fr{t}}^\ast \to \ld{\fr{n}}^\perp$ is equivariant if $\ld{\fr{t}}^\ast$ is also placed in weight $2$. Therefore, shearing (as in \cite[Section 2.1]{ku-rel-langlands}) defines a map
$$\ld{\fr{t}}^\ast[2] \xar{\kappa} \ld{\fr{n}}^\perp[2] \to \tilde{\ld{\g}}[2].$$
We will sometimes denote this composite also by $\kappa$.

The first part of the below equivalence was proved by Arkhipov-Bezrukavnikov-Ginzburg in \cite{abg-iwahori-satake}; the commutative diagram below follows from \cref{prop: psi + t} and \cref{thm: derived satake}.
\begin{theorem}\label{thm: abg}
    Let $B\subseteq G$ be a Borel subgroup, and let $I = G\pw{t} \times_G B$ denote the associated Iwahori subgroup. Then there is an equivalence
    $$\Shv_{I}(\Gr_G; k)^\lcc \simeq \Perf(\tilde{\ld{\g}}[2]/\ld{G}),$$
    which restricts to a monoidal equivalence
    $$\Shv_{I}(\Gr_G; k)^\omega \simeq \Perf_{{\ld{\cN}}/\ld{G}}(\tilde{\ld{\g}}[2]/\ld{G}).$$
    Furthermore, there is a commutative diagram
    $$\xymatrix{
    \Ind(\Shv_{I}(\Gr_G; k)^\lcc) \ar[r]^-\sim \ar[d]_{p_!} & \QCoh(\tilde{\ld{\g}}[2]/\ld{G}) \ar[d]^-{\kappa^\ast} \\
    \Shv_{I}(\ast; k) \ar[r]_-\sim & \QCoh(\ld{\fr{t}}^\ast[2]),
    }$$
    where $p: \Gr_G \to \ast$ is the canonical map to a point and $\kappa^\ast$ is pullback along the (shifted) Kostant slice.
\end{theorem}
As in \cref{cor: reg locus satake}, we find:
\begin{corollary}\label{cor: reg locus abg}
    Let $\ul{k}_{\Gr_G} \in \Shv_{I}(\Gr_G; k)$ denote the constant sheaf, and let $\Loc_{I}(\Gr_G; k)$ denote the full subcategory generated by $\ul{k}_{\Gr_G}$. Then there is an equivalence
    $$\Loc_{I}(\Gr_G; k) \simeq \QCoh(\tilde{\ld{\g}}^{\reg}[2]/\ld{G}).$$
\end{corollary}
The constant sheaf has singular support given by the zero section. In fact, the $\infty$-categories $\Loc_{G\pw{t}}(\Gr_G; k)$ and $\Loc_{I}(\Gr_G; k)$ are the subcategories of \textit{locally constant} (equivariant) sheaves on $\Gr_G$. As such, they depend only on the underlying homotopy types of $G\pw{t}$, $I$, and $\Gr_G$. 
\begin{notation}
    Let $G_c$ be the maximal compact subgroup of $G(\cc)$, and let $T_c$ be the maximal torus of $G_c$ corresponding to the Borel $B$. It is not difficult to see that there are homotopy equivalences
    \begin{align*}
        G\pw{t} & \simeq G(\cc) \simeq G_c \\
        I & \simeq B(\cc) \simeq T_c.
    \end{align*}
    The homotopy type of $\Gr_G$ follows from the next result, due to Quillen and Garland-Raghunathan:
\end{notation}
\begin{theorem}[Quillen, Garland-Raghunathan, \cite{garland-raghunathan, mitchell-buildings}]\label{thm: quillen}
    There is a homotopy equivalence $\Gr_G \simeq \Omega G_c$ (and a homeomorphism onto the subspace of $\Omega G_c$ on those based loops with \textit{polynomial} Fourier expansion) which is equivariant for the left-action of $G_c \subseteq G(\cc) \subseteq G(\cc\pw{t})$ on the left-hand side and the action of $G_c$ on the right-hand side by conjugation.
\end{theorem}
In our discussion below, we will mostly be concerned with the homology of $\Gr_G$, in which case we may replace $\Gr_G$ by $\Omega G_c$. To this extent, we will implicitly use \cref{thm: quillen} without further mention. We will describe analogues of the equivalences of \cref{cor: reg locus satake} and \cref{cor: reg locus abg} for equivariant K-theory and equivariant elliptic cohomology.
\newpage

\section{Equivariant cohomology and the case of tori}\label{sec: equiv coh}
In order to study and prove analogues of the equivalences of \cref{cor: reg locus satake} and \cref{cor: reg locus abg} for other cohomology theories, we need to review some foundational aspects of the theory of equivariant cohomology. I have reviewed some of the basics of equivariant K-theory in \cite[Section 2.2]{ku-rel-langlands}. The theory of equivariant elliptic cohomology is developed similarly, and we will now describe this story (in a somewhat leisurely fashion) following \cite{survey, gepner-meier, t-equiv-tmf}. At the end of this section, we describe the geometric Satake equivalence for tori.

The basic question we will address is giving a definition of the $\infty$-category $\Loc_{T_c}(X; k)$ for a $T_c$-space $X$ for a sufficiently general $\Eoo$-ring $k$. When $k$ is an $\Eoo$-$\QQ$-algebra, \cref{thm: abg} requires that there is an equivalence 
$$\Loc_{T_c}(\ast; k) \simeq \QCoh(\ld{\fr{t}}^\ast[2]).$$
One often defines the $\infty$-category of $k$-modules on a space $X$ as the $\infty$-category $\Fun(X, \Mod_k)$. However, when $X = BT_c$, the $\infty$-category $\Fun(BT_c, \Mod_k)$ does \textit{not} agree with $\QCoh(\ld{\fr{t}}^\ast[2])$; instead, it only agrees with a certain completion of this $\infty$-category, as we will now explain. 
\begin{lemma}\label{lem: Fun BT_c to Perf}
    Let $k$ be an $\Eoo$-algebra. Then there is an equivalence
    $$\Fun(BT_c, \Mod_k) \simeq \IndCoh(\{1\} \times_{\ld{T}} \{1\}).$$
    If, moreover, $k$ is an $\Eoo$-$\QQ$-algebra, this can be rewritten as an equivalence
    $$\Fun(BT_c, \Mod_k) \simeq \QCoh(\widehat{\ld{\fr{t}}}^\ast[2]),$$
    where $\widehat{\ld{\fr{t}}}^\ast$ denotes the completion of $\ld{\fr{t}}^\ast$ at the origin.
\end{lemma}
\begin{proof}
    If $X$ is a finite space, there is an equivalence $\Fun(X, \Mod_k) \simeq \IndCoh_{C_\ast(\Omega X; k)}$, where $C_\ast(\Omega X; k)$ is the $\E{1}$-$k$-algebra of $k$-chains on the based loop space $\Omega X$. When $X = BT_c$, we may identify $\Omega X = T_c$. Recall that $T_c$ is the classifying space of the lattice $\bX_\ast(T)$, so that there is an equivalence
    $$C_\ast(T_c; k) \cong k \otimes_{C_\ast(\bX_\ast(T); k)} k.$$
    Of course, we may identify $C_\ast(\bX_\ast(T); k) \cong k[\bX_\ast(T)]$ with the ring of functions on $\ld{T}$. Therefore, $\spec C_\ast(T_c; k) \cong \{1\} \times_{\ld{T}} \{1\}$, as desired.

    Koszul duality gives an equivalence $\IndCoh(C_\ast(T_c; k)) \to \QCoh(C^\ast(BT_c; k))$ given by $M\mapsto \Hom_{C_\ast(T_c; k)}(k, M)$. If $k$ is an $\Eoo$-$\QQ$-algebra, then $C^\ast(BT_c; k)$ is formal, and so it can be identified with the shearing of $\H^\ast(BT_c; k)$. But
    $$\spf \H^\ast(BT_c; k) \cong \widehat{\fr{t}}(2) \cong \widehat{\ld{\fr{t}}}^\ast(2),$$
    so $\IndCoh(C_\ast(T_c; k))$ is equivalent to $\QCoh(\widehat{\ld{\fr{t}}}^\ast[2])$, as desired.
\end{proof}
\begin{example}
    Suppose $T_c = S^1$. Then \cref{lem: Fun BT_c to Perf} tells us that $\Fun(BS^1, \Mod_k) \simeq \QCoh(\widehat{\AA^1}[2])$; the equivalence sends a functor $BS^1 \to \Mod_k$, regarded as a $k$-module $M$ with $S^1$-action, to its homotopy invariants $M^{hS^1}$. Let $t \in \pi_{-2}(k^{hS^1})$ denote a generator. Observe that if $a_\lambda: k \to k[2]$ denotes the boundary map in the cofiber sequence $k[1] \to C_\ast(S^1; k) \to k$, the homotopy invariants of $k[a_\lambda^{-1}]$ are simply $k^{hS^1}[t^{-1}]$ (i.e., the Tate construction). In particular, $\pi_\ast (k[a_\lambda^{-1}])^{hS^1} \cong \pi_\ast(k)\ls{t}$. However (even if $k$ is an $\Eoo$-$\QQ$-algebra), there is no (ind-)object in $\Fun(BS^1, \Mod_k)$ whose image in $\QCoh(\widehat{\AA^1}[2])$ has homotopy given by $\pi_\ast(k)[t^{\pm 1}]$: any object of $\QCoh(\widehat{\AA^1}[2])$ must have $t$ as a topologically nilpotent element in its homotopy.
\end{example}
We therefore need an alternative definition of $\Loc_{T_c}(\ast; k)$, so that it is equivalent to $\QCoh(\ld{\fr{t}}^\ast[2])$ when $k$ is an $\Eoo$-$\QQ$-algebra. Motivated by methods from equivariant homotopy theory, as well as \cite{survey, elliptic-i, elliptic-ii, elliptic-iii}, we will simply \textit{define} $\Loc_{T_c}(\ast; k)$ to be the category of quasicoherent sheaves on a (spectral) stack $\cM_T$. That this category has any relation to topology will come from the requirement that the category of quasicoherent sheaves on the \textit{completion} of $\cM_T$ at a certain basepoint is equivalent to the ind-completion of $\Fun(BT_c, \Mod_k)$.

For this, we review some constructions from \cite{survey} in a form suitable for our applications.
This review will necessarily be brief, since a detailed exposition may be found in \textit{loc. cit.}; there is also some discussion in the early sections of \cite{ginzburg-kapranov-vasserot} in the setting of ordinary (as opposed to spectral) algebraic geometry.
\begin{setup}
    Fix an $\Eoo$-ring $k$ and a commutative $k$-group $\GG$, so $\GG$ defines a functor $\CAlg_A \to \Mod_{\Z,\geq 0}$ which is representable by a \textit{flat} $k$-algebra; here, $\Mod_{\Z, \geq 0}$ denotes the category of connective $\Z$-module spectra. We will write $\GG_0$ to denote the resulting commutative group scheme over $\pi_0 k$. Note that taking zeroth spaces defines an equivalence between $\Mod_{\Z,\geq 0}$ and topological abelian groups.
\end{setup}
\begin{definition}
    A \textit{preorientation of $\GG$} is a pointed map $S^2 \to \Omega^\infty \GG(k)$ of spaces, i.e., a map $\Sigma^2 \Z \to \GG(k)$ of $\Z$-modules (by adjunction). This induces a map $\CP^\infty = \Omega^\infty \Sigma^2 \Z \to \Omega^\infty \GG(k)$ of topological abelian groups, and hence a map $\spf A^{\CP^\infty} \to \GG$ of $\Eoo$-$k$-group schemes. (Note that $\spf A^{\CP^\infty}$ need not admit the structure of a commutative $k$-group scheme: for instance, $A^{\CP^\infty}$ need not be flat over $k$.)
\end{definition}
\begin{definition}\label{def: orientation}
    Given a preorientation $S^2 \to \Omega^\infty \GG(k)$, we obtain a map $\co_\GG \to C^\ast(S^2; k)$ of $\Eoo$-$k$-algebras. On $\pi_0$, this induces a map $\pi_0 \co_\GG = \co_{\GG_0} \to \pi_0 C^\ast(S^2; k)$. However, the target can be identified with the trivial square-zero extension $\pi_0 k \oplus \pi_{-2} k$, so that the preorientation defines a derivation $\co_{\GG_0} \to \pi_{-2} k$. This defines a map $\beta: \omega = \Omega^1_{\GG_0/\pi_0 k} \to \pi_{-2} k$. The preorientation is called an \textit{orientation} if $\GG_0$ is smooth of relative dimension $1$ over $\pi_0 k$, and the composite
    $$\pi_n k \otimes_{\pi_0 k} \omega \to \pi_n k \otimes_{\pi_0 k} \pi_{-2} k \xar{\beta} \pi_{n-2} k$$
    is an isomorphism for each $n\in \Z$. This forces $k$ to be $2$-periodic (but does not force its homotopy to be concentrated in even degrees).
\end{definition}
\begin{warning}\label{warning: additive-orientation}
    As discussed in \cite[Section 3.2]{survey}, the universal $\Eoo$-$\Z$-algebra over which the additive group scheme $\GG_a$ admits an orientation is given by $\Z[\CP^\infty][\tfrac{1}{\beta}] = \QQ[\beta^{\pm 1}]$. Therefore, we are allowed to let $\GG = \GG_a$ in the story below only when $k$ is a $2$-periodic \textit{$\Eoo$-$\QQ$-algebra}. (If $k$ is not an $\Eoo$-$\Z$-algebra, one cannot in general define $\GG_a = \spec k[t]$ as a commutative $k$-group: the coproduct $k[t] \to k[x,y]$ will in general not be a map of $\Eoo$-$k$-algebras.)
\end{warning}
We can now review the definition of $T_c$-equivariant $k$-cohomology when $T_c$ is a compact torus. We will write $T$ to denote the corresponding  split torus over $\Z$.
\begin{construction}\label{cstr: def-equiv-coh}
    Fix an $\Eoo$-ring $k$ as above and a commutative $k$-group $\GG$. Given a compact abelian Lie group $T_c$, define a $k$-scheme $\cM_T$ by the mapping stack $\Hom(\bX^\ast(T), \GG)$. The underlying $\pi_0(k)$-schemes will be denoted by $\GG_0$ and $\cM_{T,0}$. If we wish to emphasize the dependence on $k$, we will add a superscript (e.g., $\cM_T^k$).
    
    We will be particularly interested in the case when $T_c$ is a torus. Let $\cT$ be the full subcategory of $\Top$ spanned by those spaces which are homotopy equivalent to $BT_c$ with $T_c$ being a compact abelian Lie group. By arguing as in \cite[Theorem 3.5.5]{elliptic-iii}, a preorientation of $\GG$ is equivalent to the data of a functor $\cM: \cT \to \Aff_k$ along with compatible equivalences $\cM(BT_c) \simeq \cM_T$. The $\Eoo$-$k$-algebra $\co_{\cM_T}$ is the $T_c$-equivariant $k$-cochains of a point, and will occasionally be denoted by $k_T$. 

    We can now sketch the construction of the $T_c$-equivariant $k$-cochains of more general $T_c$-spaces; see \cite[Theorem 3.2]{survey}. Let $T_c$ be a torus over $\cc$ for the remainder of this discussion, and let $\GG$ be an \textit{oriented} commutative $k$-group. Let $\Top(T_c)$ denote the $\infty$-category of finite $T_c$-spaces, i.e., the smallest subcategory of $\Fun(BT_c, \Top)$ which contains the quotients $T_c/T_c'$ for closed subgroups $T_c'\subseteq T_c$, and which is closed under finite colimits. There is a functor $\cf_T: \Top(T_c)^\op \to \QCoh(\cM_T)$ which is uniquely characterized by the requirement that it preserve finite limits and sends $T_c/T_c' \mapsto q_\ast \co_{\cM_{T_c'}}$. Here, $q: \cM_{T_c'} \to \cM_T$ is the canonical map induced by the inclusion $T_c'\subseteq T_c$. If $X\in \Top(T_c)$, then the \textit{$T_c$-equivariant $k$-cochains of $X$} is the global sections $\Gamma(\cM_T; \cf_T(X))$; we will denote it by $C^\ast_{T_c}(X; k)$. This can be extended to define $T_c$-equivariant $k$-cochains of filtered colimits of finite $T_c$-spaces. If we wish to emphasize the dependence on $k$, we will denote $\cf_T(X)$ by $\cf_T(X; k)$.
\end{construction}
\begin{remark}\label{rmk: oriented and completion}
    If $k$ is $2$-periodic and $\GG$ is a commutative $k$-group, then \cite[Proposition 4.3.23]{elliptic-ii} shows that the data of an orientation on $k$ (in the sense of \cref{def: orientation}) is equivalent to the formal completion of $\GG$ at the origin being isomorphic to $\spf C^\ast(BS^1; k)$. That is, when $\GG$ is oriented, the formal completion of $\cM_T$ at its basepoint is isomorphic to $\spf C^\ast(BT_c; k)$.
\end{remark}
We will denote the functor $\Gamma(\cM_T; \cf_T(-)): \Top(T_c)^\op \to \Mod(\Gamma(\cM_T; \co_{\cM_T}))$ by $C^\ast_{T_c}(-; k): \Top(T_c)^\op \to \Mod(k_T)$.
\begin{definition}\label{def: equiv-homology}
    If $X\in \Top(T_c)$, then the \textit{$T_c$-equivariant $k$-chains of $X$} is the quasicoherent sheaf on $\cM_T$ given by the $\co_{\cM_T}$-linear dual $\cf_T(X)^\vee$. We will denote its global sections by $C_\ast^{T_c}(X; k)$. Note that if $X$ admits an $\E{n}$-algebra structure (compatible with the $T_c$-action), then $\cf_T(X)^\vee$ admits the structure of an $\E{n}$-algebra\footnote{If $\cC$ is a symmetric monoidal $\infty$-category, \cite[Corollary 3.3.4]{elliptic-i} can be used to show that there is an equivalence $\coCAlg(\Alg_\E{n}(\cC)) \simeq \Alg_\E{n}(\coCAlg(\cC))$.} in $\coCAlg(\QCoh(\cM_T))$. 
    Note that $C_\ast^{T_c}(\ast; k) \simeq k_T$, which completes to the $k$-cochains (\textit{not} $k$-chains) of $BT_c$.
\end{definition}
If $X$ is a filtered colimit $\colim_\alpha X_\alpha$ of finite $T_c$-spaces, we will write $\cf_T(X)^\vee$ to denote $\colim_\alpha (\cf_T(X_\alpha)^\vee)$. Note that if we equip the presentation of $X$ as a filtered colimit $\colim_\alpha X_\alpha$ with the structure of a filtered $\E{n}$-algebra, then $\cf_T(X)^\vee$ acquires the structure of an $\E{n}$-algebra in $\coCAlg(\QCoh(\cM_T))$. 
\begin{notation}\label{notn: ideal-character}
    Let $\lambda: T \to \GG_m$ be a character, and let $T_\lambda = \ker(\lambda)$. Then the map $q: \cM_{T_\lambda} \to \cM_T$ is a closed immersion, and we will denote the ideal in $\co_{\cM_T}$ defined by this closed immersion by $\cI_\lambda$. Equivalently, let $V_\lambda$ denote the $T_c$-representation obtained by the projection $T \to T_\lambda$. Then $\cI_\lambda$ is given by the line bundle $\cf_T(S^{V_\lambda})$.
\end{notation}
It is trickier to extend the definition of equivariant cochains to nonabelian groups, but a construction is sketched in \cite[Section 3.5]{survey}, and a detailed construction is given in \cite{gepner-meier}. However, we will not recall this here, because we will only be concerned with torus-equivariance in the present article. 

We now take a moment to prove some foundational aspects of the theory of generalized equivariant cohomology.
\begin{lemma}[{Atiyah-Bott localization \cite{atiyah-bott-localization}}]\label{lem: atiyah localization}
    Let $X$ be a finite $T_c$-space, and let $\cU_X \subseteq \cM_T$ denote the complement of the union of the closed substacks $\cM_{T'}$ over all stabilizers $T'_c \subseteq T_c$ of points in $X$. Then the map $\cf_T(X) \to \cf_T(X^{T_c})$ is an isomorphism after pulling back to $\cU_X$.
\end{lemma}
\begin{proof}
    This follows from induction on the cell structure of $X$. Namely, the statement is true when the $T$-action on $X$ is trivial, which gives the base case. For the inductive step, note that if $X$ is the cofiber of a map $T/T' \to Y$, then there is a cofiber sequence $\cf_T(X) \to \cf_T(Y) \to \cf_T(T/T')$; but $\cf_T(T/T')$ is isomorphic to the pushforward of the structure sheaf along the map $\cM_{T'} \to \cM_T$, and so it vanishes upon pulling back to $\cU_X$. This implies that the map $\cf_T(X) \to \cf_T(Y)$ is an isomorphism upon pulling back to $\cU_X$, as desired.
\end{proof}
One consequence of \cref{lem: atiyah localization} which is worth restating is the following. Let $\punc{\cM}_T$ denote the complement of the union of the closed subschemes $\cM_{T'}$ ranging over all closed \textit{proper} subgroups $T' \subsetneq T$. Then the map $\cf_T(X) \to \cf_T(X^{T_c})$, and hence the map $\cf_T(X^{T_c})^\vee \to \cf_T(X)^\vee$, is an equivalence upon restriction to $\punc{\cM}_T$.

We will also need a version of the Goresky-Kottwitz-MacPherson approach \cite{gkm-original} to equivariant cohomology; in the setting of generalized equivariant cohomology, it has also been studied in \cite{generalized-gkm, gepner-meier}. As such, we will only give a sketch of the relevant argument.
\begin{definition}\label{def: GKM space}
    Let $X$ be a finite $T_c$-space equipped with a chosen presentation in terms of $T_c$-cells. Say that $X$ is a \textit{GKM space} if the following conditions are satisfied:
    \begin{enumerate}
        \item $\pi_0 \cf_T(X)$ is a vector bundle over $\cM_{T,0}$;
        \item if $X^{(1)}$ denotes the equivariant $1$-skeleton of $X$, then $X^{(1)}$ consists of a finite number of spheres $S^\lambda$ meeting only at the fixed points, where $\lambda$ ranges over characters of $T$.
    \end{enumerate}
    In this setup, let $V$ denote the set $X^{T_c}$ of fixed points, and let $E$ denote the set of characters $\lambda$ such that $S^\lambda \subseteq X^{(1)}$. There are two maps $E \rightrightarrows V$ sending $\lambda$ to the points $0,\infty\in S^\lambda \subseteq X^{(1)}$. The resulting graph with set of vertices $V$ and set of edges $E$ will be referred to as the \textit{GKM graph} of $X$.
\end{definition}
The utility of the first condition in the above definition is due to the following.
\begin{lemma}\label{lem: injective coh}
    Let $X$ be a finite $T_c$-space. If $\pi_0 \cf_T(X)$ is a vector bundle over $\cM_{T,0}$, the map $\pi_0 \cf_T(X) \to \pi_0 \cf_T(X^{T_c})$ is an injection.
\end{lemma}
\begin{proof}
    Since the map $\cf_T(X) \to \cf_T(X^{T_c}) \to \cf_T(X^{T_c})_{\punc{\cM}_T}$ factors as
    $\cf_T(X) \to \cf_T(X)|_{\punc{\cM}_T} \to \cf_T(X^{T_c})|_{\punc{\cM}_T}$,
    and the map $\cf_T(X)|_{\punc{\cM}_T} \to \cf_T(X^{T_c})|_{\punc{\cM}_T}$ is an equivalence by \cref{lem: atiyah localization}, it suffices to show that the map $\cf_T(X) \to \cf_T(X)|_{\punc{\cM}_T}$ induces an injection on $\pi_0$. But $\pi_0 \cf_T(X)$ was assumed to be a vector bundle over $\cM_{T,0}$, so one is reduced to the case $X = \ast$, i.e., to showing that the map $\co_{\cM_T} \to \co_{\cM_T}|_{\punc{\cM}_T}$ induces an injection on $\pi_0$. This, however, is clear, since the closed subscheme $\cM_{T',0} \hookrightarrow \cM_{T,0}$ defined by each closed subgroup $T'\subseteq T$ is cut out by a regular sequence.
\end{proof}
\begin{prop}[Goresky-Kottwitz-MacPherson]\label{prop: gkm}
    Let $X$ be a finite GKM $T_c$-space, and choose a presentation in terms of $T_c$-cells. For each character $\lambda: T \to S^1$, let $T_\lambda$ denote the kernel of $T$, let $q_\lambda: \cM_{T_\lambda} \to \cM_T$ denote the induced map, and let $S(\lambda)$ denote the unit representation sphere. Then there is an equalizer diagram
    $$\pi_0 \cf_T(X) \hookrightarrow \pi_0 \cf_T(X^{T_c}) \cong \Map(V, \co_{\cM_{T,0}}) \rightrightarrows \prod_{\lambda \in E} q_{\lambda, \ast} \co_{\cM_{T_\lambda,0}},$$ 
    where the two maps in the equalizer are defined in the evident manner.
\end{prop}
\begin{proof}[Proof sketch]
    First, we show that the maps $\pi_0 \cf_T(X) \to \pi_0 \cf_T(X^{T_c})$ and $\pi_0 \cf_T(X^{(1)}) \to \pi_0 \cf_T(X^{T_c})$ have the same image. There is an evident map from the image of $\pi_0 \cf_T(X) \to \pi_0 \cf_T(X^{T_c})$ to the image of $\pi_0 \cf_T(X^{(1)}) \to \pi_0 \cf_T(X^{T_c})$, which we will denote by $f$. The map $f$ is an injection by \cref{lem: injective coh}.
    Let $T'$ denote a proper closed subgroup of $T$ of codimension $1$, and let $U' \subseteq \cM_{T',0}$ denote the complement of the union of the closed varieties $\cM_{T'',0}$ ranging over the proper closed subgroups $T''\subseteq T'$.
    By \cref{lem: atiyah localization}, the map $f$ is an isomorphism upon restriction to $U'\subseteq \cM_{T',0} \subseteq \cM_{T,0}$ for each proper closed subgroup $T'\subseteq T$ of codimension $1$. 
    Therefore, the locus $Z \subseteq \cM_{T,0}$ over which $f$ fails to be an isomorphism is contained in the union of closed subvarieties $\cM_{T',0}$ for finitely many $T'\subseteq T$ of codimension at least $2$. However, the map $\pi_0 \cf_T(X) \to \pi_0 \cf_T(X)|_{\cM_{T,0} - Z}$ is an isomorphism (by Hartogs). Since the same is true of the map $\pi_0 \cf_T(X^{T_c}) \to \pi_0 \cf_T(X^{T_c})|_{\cM_{T,0} - Z}$, and the map $\pi_0 \cf_T(X) \to \pi_0 \cf_T(X^{T_c})$ factors through the map $\pi_0 \cf_T(X^{(1)}) \to \pi_0 \cf_T(X^{T_c})$, the desired result follows.

    For the equalizer diagram, an easy induction on the cell structure of $X$ reduces us to the case $X = S^\lambda$ for a character $\lambda: T \to S^1$. In this case, the isomorphism $T/T_\lambda \cong S^\lambda$ defines an isomorphism between $\pi_0 \cf_T(S(\lambda))$ and the pushforward of the structure sheaf along the map $\cM_{T_\lambda,0} \to \cM_{T,0}$. Since $S^\lambda \cong \Sigma S(\lambda)$, we obtain an equalizer diagram
    $$\pi_0 \cf_T(S^\lambda) \to \co_{\cM_{T,0}} \oplus \co_{\cM_{T,0}} \cong \Map(\{0,\infty\}, \co_{\cM_{T,0}}) \rightrightarrows q_{\lambda, \ast} \co_{\cM_{T_\lambda,0}}.$$
    This proves the desired claim.
\end{proof}
The same argument proves the following dual to \cref{prop: gkm} (see also \cite{brion-poincare-dual}).
\begin{prop}\label{prop: hmlgy gkm}
    Let $X$ be a finite GKM $T_c$-space, and choose a presentation in terms of $T_c$-cells. Then $\pi_0 \cf_T(X)^\vee$ is isomorphic to the subset of $\pi_0 \cf_T(X^{T_c})^\vee \cong \co_{\cM_{T,0}}[X^{T_c}]$ of those $\sum_{x \in X^{T_c}} f_x [x] \in \co_{\cM_{T,0}}[X^{T_c}]$ such that:
    \begin{itemize}
        \item For each fixed point $x \in X^{T_c}$, the poles of $f_x$ all have order $\leq 1$, and these are contained in the ideal sheaf of $\co_{\cM_{T_\lambda,0}}$ for each character $\lambda: T_c \to S^1$ such that the $T_c$-orbit $S^\lambda$ meets $x$.
        \item For each character $\lambda: T_c \to S^1$ such that the $T_c$-orbit $S^\lambda$ meets $x_0, x_\infty \in X^{T_c}$, we have
        $$\Res_{\cM_{T_\lambda,0}}(f_{x_0}) + \Res_{\cM_{T_\lambda,0}}(f_{x_\infty}) = 0.$$
    \end{itemize}
\end{prop}
These results can be extended without much trouble to ind-$T_c$-spaces $X$ with isolated fixed points satisfying the conditions of \cref{def: GKM space}. (The first condition therein should be replaced by the condition that $\pi_0 \cf_T(X)$ is an ind-vector bundle over $\cM_{T,0}$.)

The preceding discussion can be categorified, as we now explain. The following categorifies the $T_c$-equivariant $k$-cochains $C_{T_c}^\ast(X; k)$.
\begin{construction}\label{cstr: def-loc}
    Let $\Loc_{T_c}(\ast; k)$ denote the $\infty$-category $\QCoh(\cM_T)$. Let $T_c'\subseteq T_c$ be a closed subgroup, so that there is an associated morphism $q: \cM_{T'} \to \cM_T$. This defines a symmetric monoidal functor $\QCoh(\cM_T) \to \QCoh(\cM_{T'})$, which equips $\QCoh(\cM_{T'})$ with the structure of a $\QCoh(\cM_T)$-module.

    Let $\cLoc_{T_c}(-; k): \Top(T_c)^\op \to \CAlg(\shvcat(\cM_T))$ be the functor uniquely characterized by the requirement that it preserve finite limits and send $T/T' \mapsto \QCoh(\cM_{T'})$. If $X\in \Top(T_c)$, then the $\infty$-category $\Loc_{T_c}(X; k)$ of \emph{$T_c$-equivariant local systems of $k$-modules on $X$} is defined to be the global sections of the quasicoherent stack $\cLoc_{T_c}(X; k)$ on $\cM_T$.  If $X$ is a $T_c$-space which is presented as a filtered colimit of finite $T_c$-spaces $X_\alpha$, we will write $\Loc_{T_c}(X; k)$ to denote $\colim \Loc_{T_c}(X_\alpha; k)$.

    If $f: X \to Y$ is a map in $\Top(T_c)$, the associated symmetric monoidal functor $f^\ast: \Loc_{T_c}(Y; k) \to \Loc_{T_c}(X; k)$ (induced by taking global sections of the morphism $f^\ast: \cLoc_{T_c}(Y; k) \to \cLoc_{T_c}(X; k)$ of $\Eoo$-algebras in quasicoherent stacks over $\cM_T$) will be called the \textit{pullback}. One can show that $\Loc_{T_c}(X; k)$ is a presentable stable $\infty$-category, and that $f^\ast$ preserves small colimits (so it has a right adjoint $f_\ast$, which will be called \textit{pushforward}).
\end{construction}
For instance, if $T_c = \{1\}$, then $\Loc_{T_c}(X; k)$ is equivalent to the $\infty$-category $\Loc(X; k) := \Fun(X, \Mod_k)$ of local systems on $X$.
\begin{remark}
    Let $X$ be a finite $T_c$-space. The \textit{constant local system} $\ul{k}$ is defined to be the image of $\co_{\cM_T}$ under the symmetric monoidal functor $\Loc_{T_c}(\ast; k) \simeq \QCoh(\cM_T) \to \Loc_{T_c}(X; k)$ induced by pullback along $f: X \to \ast$. Observe that if $\ul{k}$ denotes the constant local system, then $\End_{\Loc_{T_c}(X; k)}(\ul{k}) \simeq C_{T_c}^\ast(X; k)$. Indeed, $\End_{\Loc_{T_c}(X; k)}(\ul{k}) \simeq \Gamma(\cM_T; f_\ast f^\ast \co_{\cM_T})$, but it is easy to see that $f_\ast f^\ast \co_{\cM_T} = \cf_T(X)\in \QCoh(\cM_T)$. The desired claim then follows from \cref{cstr: def-equiv-coh}.
\end{remark}
\begin{remark}
    If the complexification of $T_c$ were a \textit{finite} diagonalizable group scheme (such as $\mu_n$), the desired category $\Loc_{T_c}(X; k)$ is closely related to the $\infty$-category of \textit{$\GG$-tempered local systems} on the orbispace $X\mmod T$, as described in \cite{elliptic-iii}. Our understanding is that Lurie is planning to describe an extension of the work in \cite{elliptic-iii} and its connections to equivariant homotopy theory in a future article. We warn the reader that \cref{cstr: def-loc} is somewhat \textit{ad hoc}; so the resulting category of equivariant local systems may or may not agree with the output of forthcoming work of Lurie.
\end{remark}
\begin{remark}
    If $X$ is a finite $T_c$-space, a more straightforward definition of the category of $T_c$-equivariant local systems on $X$ is simply the category $\Fun(X/T_c, \Mod_k)$. Equivalently, it can be described as the functor $\Top(T_c)^\op \to \CAlg(\PrL)$ which is uniquely characterized by the requirement that it preserve finite limits and send $T_c/T_c' \mapsto \Fun(BT_c', \Mod_k)$. It follows from \cref{lem: Fun BT_c to Perf} that $\Fun(BT_c', \Mod_k)$ is equivalent to $\Mod(C^\ast(BT_c'; k))$. As discussed in \cref{rmk: oriented and completion}, if the group scheme $\GG$ is oriented, then this is in turn equivalent to $\QCoh(\widehat{\cM_T})$, where $\widehat{\cM_T}$ is the completion of $\cM_T$ at its basepoint. That is, $\Fun(BT_c', \Mod_k)$ can be viewed as a completion of $\QCoh(\cM_{T'})$. This implies that $\Fun(X/T_c, \Mod_k)$ can be viewed as a completion of the subcategory of compact objects of $\Loc_{T_c}(X; k)$. Motivated by this, we will write $\Loc_{T_c}^\wedge(X; k)$ to denote $\Fun(X/T_c, \Mod_k)$; we will use the same notation to denote the extension of the assignment $X\mapsto \Loc_{T_c}^\wedge(X; k)$ to filtered colimits of finite $T_c$-spaces.
\end{remark}
Using this discussion, let us now discuss geometric Satake with $k$-coefficients in the case of a torus. 
\begin{theorem}\label{thm: torus satake}
    Fix a complex-oriented $2$-periodic $\Eoo$-ring $k$ and an oriented commutative $k$-group scheme $\GG$. Let $\ld{T} = \spec k[\bX^\ast(\ld{T})]$ denote the dual torus over $k$. In the following statements, all actions of $\ld{T}$ are trivial.
    Then there are equivalences
    \begin{align*}
        \Loc_{T_c}^\wedge(\Gr_T; k) & \simeq \IndCoh((\{1\} \times_{\ld{T}} \{1\})/\ld{T}), \\
        \Loc_{T_c}(\Gr_T; k) & \simeq \QCoh(\cM_T/\ld{T}).
    \end{align*}
    Moreover, there is an isomorphism of spectral group $k$-schemes
    $$\spec \cf_T(\Gr_T)^\vee \cong \cM_T \times_{\spec(k)} \ld{T} \cong \cM_T \times_{\cM_T/\ld{T}} \cM_T.$$
\end{theorem}
\begin{proof}
    Since the underlying topological space of $\Gr_T$ is simply the lattice $\bX_\ast(T)$, it follows from \cref{lem: Fun BT_c to Perf} that 
    $$\Loc_{T_c}^\wedge(\Gr_T; k) \simeq \bigoplus_{\bX_\ast(T)} \Loc_{T_c}^\wedge(\ast; k) \simeq \QCoh(B\ld{T}) \otimes_{\Mod_k} \IndCoh(\{1\} \times_{\ld{T}} \{1\}).$$
    For the trivial action of $\ld{T}$ on $\{1\} \times_{\ld{T}} \{1\}$, this is precisely $\IndCoh((\{1\} \times_{\ld{T}} \{1\})/\ld{T})$. Exactly the same discussion proves the second equivalence:
    $$\Loc_{T_c}(\Gr_T; k) \simeq \bigoplus_{\bX_\ast(T)} \Loc_{T_c}(\ast; k) \simeq \QCoh(B\ld{T}) \otimes_{\Mod_k} \QCoh(\cM_T).$$
    The claim about $\cf_T(\Gr_T)^\vee$ can be proved similarly.
\end{proof}
\begin{remark}
    Note that in \cref{thm: torus satake}, the ``spectral''/algebro-geometric description of $\Loc_{T_c}^\wedge(\Gr_T; k)$ does not seem to depend on the choice of coefficient $k$ (in particular, not on $\GG$). This dependence, however, can be made more explicit by noting that $\IndCoh(\{1\} \times_{\ld{T}} \{1\})$ is equivalent to $\Mod(k^{hT_c}) \simeq \QCoh(\widehat{\cM_T})$. That is, there is an equivalence $\Loc_{T_c}^\wedge(\Gr_T; k) \simeq \QCoh(\widehat{\cM_T}/\ld{T})$.
\end{remark}
Our basic goal is to find a replacement of \cref{thm: torus satake} where $\Gr_T$ is replaced by $\Gr_G$ for a general connected reductive group $G$.
\newpage

\section{Degenerations}\label{sec: degenerations}
We begin this section by immediately amending the goal referred to at the end of the preceding section. Namely, instead of studying the $\infty$-category $\Loc_{T_c}(\Gr_G; k)$ for a connected reductive group $G$ and a maximal torus $T\subseteq G$, we will study a particular \textit{degeneration} of this $\infty$-category. Before discussing the construction of this degeneration, let us motivate \textit{why} it is useful (see also the introduction for some ``philosophy'' regarding this degeneration).

Suppose that there was an equivalence of the form $\Loc_{T_c}(\Gr_G; k) \simeq \QCoh(\fr{X}_k)$ for some spectral $k$-stack $\fr{X}_k$. In order for such an equivalence to be considered related to Langlands duality, the stack $\fr{X}_k$ must have some relationship to the dual group $\ld{G}$; for instance, one can wonder whether the underlying classical $\pi_0(k)$-stack of $\fr{X}_k$ lives over the classifying stack $B\ld{G}_{\pi_0(k)}$. Here, $\ld{G}_{\pi_0(k)}$ is the base-change of the Chevalley split form of $\ld{G}$ along the map $\Z \to \pi_0(k)$. (When $k$ is an $\Eoo$-$\QQ$-algebra, the stack $\fr{X}_k$ is $\tilde{\ld{\g}}[2]/\ld{G}$, which does indeed live over $B\ld{G}$.) The most satisfying description of $\fr{X}_k$ must therefore involve a lift of the dual group $\ld{G}$ to a (flat) spectral group scheme over $k$. Unfortunately, the existence of such a lift is far from clear: giving a flat lift of $\SL_2$ (even just as a \textit{scheme}!) to complex K-theory leads to very subtle questions; see \cref{sec: lifting SL2}.

Instead, let us return to the general situation of a finite $T_c$-space $X$. One can then view $\Loc_{T_c}(X; k)$ as a categorification of the cochains $\cf_T(X)\in \QCoh(\cM_T)$; so for the moment, let us just describe a degeneration of $\cf_T(X)$ and $\cM_T$. There is a natural filtered lift of $\cM_T = \spec k_T$ to a filtered $\tau_{\geq \star}(k)$-scheme, given by $\spec \tau_{\geq \star}(k_T)$. (This construction is, of course, closely related to the even filtration constructed in \cite{even-filtr, piotr-even-filtr}.) In particular, one obtains a corresponding graded $\pi_\ast(k)$-scheme $\spec \pi_\ast(k_T)$. Note that this is now a \textit{classical} scheme, with no spectral algebro-geometric nature.
If $k$ is even-periodic, i.e., is equipped with an isomorphism $\pi_\ast(k) \cong \pi_0(k)[u^{\pm 1}]$ with $u \in \pi_2(k)$, then this is equivalent to the data of the classical $\pi_0(k)$-scheme $\spec \pi_0(k_T)$. (Recall that this is the affinization of the scheme $\cM_{T,0}$; to get to the definition described below, one needs to replace $\spec \pi_0(k_T)$ in the below discussion by $\cM_{T,0}$.)

If the finite $T_c$-space $X$ has even cells, then one can construct a well-behaved filtered lift of $\cf_T(X)$ to a filtered quasicoherent sheaf over $\spec \tau_{\geq \star}(k_T)$, given by $\tau_{\geq \star} \cf_T(X)$. This defines a corresponding graded variant of $\cf_T(X)$, given simply by the quasicoherent sheaf $\pi_0 \cf_T(X)$ over $\spec \pi_0(k_T)$. Again, this is an object in the realm of \textit{classical} algebraic geometry; so when applied to the affine Grassmannian $\Gr_G$, it is something that could, in theory, be described in terms of the usual dual group $\ld{G}$ base-changed to $\pi_0(k)$. 

The idea for constructing the desired degeneration of $\Loc_{T_c}(X; k)$ is very similar; we now turn to its mechanics.  Let us begin with a simple observation. If $Y$ is a {connected} space, the $\infty$-category $\Loc(Y; k) = \Fun(Y, \Mod_k)$ of local systems on $Y$ is equivalent, by Koszul duality, to $\LMod_{C_\ast(\Omega Y; k)}$. This is very useful, since it allows one to reduce the study of local systems to the study of a particular (derived) algebra. A similar property is true for $\Loc_{T_c}(X; k)$:
\begin{prop}\label{prop: equivariant-koszul}
    Let $X$ be a connected finite $T_c$-space. Then there is an equivalence $\Loc_{T_c}(X; k) \simeq \LMod_{\cf_T(\Omega X)^\vee}(\QCoh(\cM_T))$.
\end{prop}
\begin{proof}
    Let $s: \ast \to X$ denote the inclusion of a point. We claim that $s^\ast: \Loc_{T_c}(X; k) \to \QCoh(\cM_T)$ admits a left adjoint $s_!$. Indeed, the statement for general $X$ follows formally from the case of $X = T/T'$ for some closed subgroup $T'\subseteq T$ (so $s$ is the inclusion of the trivial coset). In this case, $s^\ast$ is the functor $\QCoh(\cM_{T'}) \to \QCoh(\cM_T)$ given by pushforward along the associated morphism $q: \cM_{T'} \to \cM_T$, so it has a left adjoint $s_!$ given by $q^\ast$. Note that $s^\ast$ also has a right adjoint; in particular, it preserves small limits and colimits. Observe now that $s_! \co_{\cM_T}$ is a compact generator of $\Loc_{T_c}(X; k)$: indeed, suppose $\cf\in \Loc_{T_c}(X; k)$ such that $\Map_{\Loc_{T_c}(X; k)}(s_! \co_{\cM_T}, \cf) \simeq 0$ as an object of $\QCoh(\cM_T)$. Because $s^\ast \cf \simeq \Map_{\Loc_{T_c}(X; k)}(s_! \co_{\cM_T}, \cf)$ in $\QCoh(\cM_T)$, we see that $s^\ast \cf \simeq 0$. Using the connectivity of $X$, we see that $\cf$ itself must be zero, which implies that $s_! \co_{\cM_T}$ is a compact generator of $\Loc_{T_c}(X; k)$. It follows from the Barr-Beck-Lurie theorem \cite[Theorem 4.7.3.5]{HA} that $\Loc_{T_c}(X; k)$ is equivalent to the $\infty$-category of left $\End_{\Loc_{T_c}(X; k)}(s_! \co_{\cM_T})$-modules in $\QCoh(\cM_T)$. But $\End_{\Loc_{T_c}(X; k)}(s_! \co_{\cM_T}) \simeq s^\ast s_! \co_{\cM_T}$, which identifies with $\cf_T(\Omega X)^\vee$.
\end{proof}
\begin{remark}\label{rmk: loc and comod}
    Modifying the preceding argument shows that if $X$ is a connected finite $T_c$-space, there is an equivalence 
    \begin{equation}\label{eq: loc and comod}
        \Loc_{T_c}(X; k) \simeq \coLMod_{\cf_T(X)^\vee}(\QCoh(\cM_T)).
    \end{equation}
    In particular, if $X$ admits an $\E{n}$-algebra structure (compatible with the $T_c$-action), then $\cf_T(X)^\vee$ admits the structure of an $\E{n}$-algebra\footnote{If $\cC$ is a symmetric monoidal $\infty$-category, \cite[Corollary 3.3.4]{elliptic-i} can be used to show that there is an equivalence $\coCAlg(\Alg_\E{n}(\cC)) \simeq \Alg_\E{n}(\coCAlg(\cC))$.} in $\coCAlg(\QCoh(\cM_T))$, and the equivalence \cref{eq: loc and comod} is $\E{n}$-monoidal for the convolution tensor product on both sides. 
\end{remark}
\cref{prop: equivariant-koszul} and \cref{rmk: loc and comod} continue to hold even when $X$ is a filtered colimit of finite $T_c$-spaces. In order for the claim in \cref{rmk: loc and comod} about $\E{n}$-algebra structures to hold, we need the filtered diagram $\{X_\lambda\}$ presenting $X$ to admit the structure of an $\E{n}$-algebra in filtered $T_c$-spaces. We will need to apply this in the case when $X$ is the affine Grassmannian, in which case we can apply the following observation. 
\begin{lemma}\label{filtered E2}
    The $\bX_\ast(T)^+$-indexed Schubert filtration $\{\Gr_G^{\leq \lambda}(\cc)\}$ naturally admits the structure of an $\E{2}$-algebra in $\Fun(\bX_\ast(T)^+, \Top(T_c))$.
\end{lemma}
\begin{proof}
    This can be proved in essentially the same way as \cite[Theorem 3.10]{hahn-yuan}; let us sketch the argument. We will utilize \cite[Proposition 5.4.5.15]{HA}, which states that if $\cC$ is a symmetric monoidal $\infty$-category, then a nonunital $\E{2}$-algebra object in $\cC$ is equivalent to the datum of a locally constant $\mathrm{N}(\mathrm{Disk}(\cc))_\mathrm{nu}$-algebra object in $\cC$. Concretely, this amounts to specifying an object $A(D)\in \cC$ for every disk $D\subseteq \cc$ and coherent maps $\bigotimes_{i=1}^n A(D_i)\to A(D)$ for every inclusion $\coprod_{i=1}^n D_i\to D$ of disks, such that for every embedding $D\subseteq D'$ of disks, the induced map $A(D)\to A(D')$ is an equivalence.

    In this case, $\cC = \Fun(\bX_\ast(T)^+, \Top(T_c))$, and the object $A(D)\in \Fun(\bX_\ast(T)^+, \Top(T_c))$ assigned to a disk $D\subseteq \cc$ may be defined via the Beilinson-Drinfeld Grassmannian $\Gr_{G,\Ran}$. Namely, the Beilinson-Drinfeld Grassmannian admits (by construction) a morphism $\Gr_{G, \Ran} \to \Ran_{\AA^1}$; upon taking complex points, we obtain a map $\Gr_{G, \Ran}(\cc) \to \Ran(\cc)$. If $S\subseteq \cc$ is a subset, then the preimage of $\Ran(S)\subseteq \Ran(\cc)$ defines a subspace $\Gr_{G, \Ran}(S\subseteq \cc)\subseteq \Gr_{G, \Ran}(\cc)$. The filtration of $\Gr_G$ via the Bruhat decomposition extends to a filtration $\Gr_{G, \Ran, \leq \mu}$ of $\Gr_{G, \Ran}$ by dominant coweights $\mu\in \bX_\ast(T)^+$; see \cite[3.1.11]{zhu-grass}. Finally, the object $A(D)\in \Fun(\bX_\ast(T)^+, \Top(T_c))$ associated to a disk $D\subseteq \cc$ is the functor $\bX_\ast(T)^+\to \Top(T_c)$ sending $\mu\in \bX_\ast(T)^+$ to $\Gr_{G, \Ran, \leq \mu}(D\subseteq \cc)$.

    Suppose $\coprod_{i=1}^n D_i\to D$ is an inclusion of disks. The induced map $\bigotimes_{i=1}^n A(D_i)\to A(D)$ is defined as follows. Let $\mu\in \bX_\ast(T)^+$; for every $n$-tuple $(\mu_1, \cdots, \mu_n)$ with $\sum_{i=1}^n \mu_i\leq \mu$, we need to exhibit maps $\bigotimes_{i=1}^n A(D_i)(\mu_i)\to A(D)(\mu)$ satisfying the obvious coherences. But
    $$\bigotimes_{i=1}^n A(D_i)(\mu_i) = \prod_{i=1}^n \Gr_{G, \Ran, \leq \mu_i}(D_i\subseteq \cc),$$
    so it suffices to show that if $\mu_1 + \mu_2 \leq \mu$, then there are maps $\Gr_{G, \Ran, \leq \mu_1}(D_1\subseteq \cc) \times \Gr_{G, \Ran, \leq \mu_2}(D_2\subseteq \cc)\to \Gr_{G, \Ran, \leq \mu}(D\subseteq \cc)$. The argument for this is exactly as in \cite[Construction 3.15]{hahn-yuan}.

    We next need to show that the $\mathrm{N}(\mathrm{Disk}(\cc))_\mathrm{nu}$-algebra $A$ defined above is locally constant, i.e., that if $D\subseteq D'$ is an embedding of disks, then $A(D)\to A(D')$ is an equivalence of functors $\bX_\ast(T)^+\to \Top(T_c)$. This follows from \cite[Proposition 3.17]{hahn-yuan}. To conclude, it suffices (by \cite[Theorem 5.4.4.5]{HA}) to establish the existence of a quasi-unit for the functor $A:\bX_\ast(T)^+\to \Top(T_c)$, i.e., a map $1_{\Fun(\bX_\ast(T)^+, \Top(T_c))}\to A$ which is both a left and right unit up to homotopy. Since the unit in $\Fun(\bX_\ast(T)^+, \Top(T_c))$ is the functor sending $\mu\in \bX_\ast(T)^+$ to the point $\ast$, a quasi-unit is the datum of a map $\ast \to \Gr_{G, \leq \mu}(\cc)$ for each $\mu\in \bX_\ast(T)^+$. As in the proof of \cite[Theorem 3.10]{hahn-yuan}, this can be taken to be the inclusion of the point corresponding to the trivial $G$-bundle over $\AA^1$ with the canonical trivialization away from the origin.
\end{proof}

Suppose, now, that $A$ is an $\E{1}$-ring with even homotopy. Any left $A$-module $M$ then defines a filtered left $\tau_{\geq 2\star}(A)$-module $\tau_{\geq 2\star}(M)$; we will denote the corresponding associated graded left $\pi_{2\ast}(A)$-module by $\gr_\ev(M)$. If $M,N\in \LMod_A$, there is then a canonical (complete and exhaustive) filtration on the $A$-module $\Map_A(M,N)$ whose associated graded is given by the shearing of $\Map_{\pi_{2\ast}(A)}(\gr_\ev(M), \gr_\ev(N))$. Informally, this means that there is a $1$-parameter degeneration (constructed using the double-speed Postnikov filtration) from $\LMod_A$ to the category $\LMod_{\pi_{2\ast}(A)}^\gr$, given by the filtered category $\LMod_{\tau_{\geq 2\star} A}$.
Motivated by the preceding discussion, we can now define our desired degeneration of $\Loc_{T_c}(X; k)$.
\begin{definition}\label{def: graded Loc}
    Suppose that $X$ is a (ind-)finite $T_c$-space with even cells (such as $\Gr_G$). The $\infty$-category $\Loc_{T_c}^\gr(X; k)$ is defined as
    $$\Loc_{T_c}^\gr(X; k) = \coLMod_{\pi_0(\cf_T(X)^\vee)}(\QCoh(\cM_{T,0})).$$
    The ``constant sheaf'' $\ul{k}^\gr$ in this category is the comodule $\pi_0(\cf_T(X)^\vee)$ itself.
    Similarly, suppose $Y$ is a finite $T_c$-space such that $\Omega Y$ has even cells (such as $G_c$). The $\infty$-category $\Loc_{T_c}^\gr(Y; k)$ is defined as
    $$\Loc_{T_c}^\gr(Y; k) = \LMod_{\pi_0(\cf_T(\Omega Y)^\vee)}(\QCoh(\cM_{T,0})).$$
    The ``constant sheaf'' $\ul{k}^\gr$ in this category is the structure sheaf $\co_{\cM_{T,0}}$ viewed as a $\pi_0(\cf_T(\Omega Y)^\vee)$-module via the augmentation.
\end{definition}
These should be viewed as ``mixed'' (in the sense of \cite{bbdg}) variants of the full $\infty$-categories $\Loc_{T_c}(X; k)$ and $\Loc_{T_c}(Y; k)$.
\begin{remark}\label{rmk: loc gr convolution tensor}
    If $X$ admits an $\E{n}$-algebra structure (compatible with the $T_c$-action), then the $\E{n}$-algebra structure on $\cf_T(X)^\vee$ equips $\pi_0(\cf_T(X)^\vee)$ with the structure of a commutative algebra object in $\QCoh(\cM_{T,0})$. In particular, $\Loc_{T_c}^\gr(X; k)$ acquires a symmetric monoidal structure, which we will refer to as the ``convolution tensor structure'' and denote by $\star$.
\end{remark}
\begin{remark}
    There is an apparent asymmetry in \cref{def: graded Loc}: why could we not have defined $\Loc_{T_c}^\gr(Y; k)$ to be $\coLMod_{\pi_0(\cf_T(Y)^\vee)}(\QCoh(\cM_{T,0}))$? The issue is that since $Y$ contains odd-dimensional cells, taking $\pi_0$ of $\cf_T(Y)^\vee$ is a very destructive process. More generally, as in the discussion at the beginning of this section, $\pi_0 \cf_T(X)^\vee$ for a finite $T_c$-space $X$ should only be regarded as a well-behaved reflection of $\cf_T(X)^\vee$ itself when $X$ has even cells.

    However, if $Y$ was the total space of an iterated fibration of odd-dimensional spheres (which happens when $Y = \U(n)$ or $\Sp(n)$), then one could alternatively consider the category of comodules in $\QCoh(\cM_{T,0})$ over the truncation $\tau_{[-1,0]}(\cf_T(Y)^\vee)$. If the cobar construction on $\tau_{[-1,0]}(\cf_T(Y)^\vee)$ is given by $\pi_0(\cf_T(\Omega Y)^\vee)$, it then follows from Koszul duality that (up to finiteness questions) this new category would be equivalent to the definition of $\Loc_{T_c}^\gr(Y; k)$ from \cref{def: graded Loc}.
\end{remark}
\begin{remark}
    If $k  = \QQ[u^{\pm 1}]$ with $u$ in degree $2$, then (using the results of \cite{abg-iwahori-satake}) $\Loc_{T_c}(\Gr_G; k)$ is equivalent to the shearing of the $2$-periodification of the category $\Loc_{T_c}^\gr(\Gr_G; k)$. This can be understood as a statement about formality. If $k$ is a more general $\Eoo$-ring (like complex K-theory $\KU$), then formality is generally impossible: for instance, a $\KU$-module $M$ is generally not equivalent (even as a spectrum!) to the shearing of $\pi_\ast(M)$, unless $M$ is also a $\QQ$-module.
\end{remark}
\begin{remark}\label{def: graded G equiv loc sys}
    We will not discuss $G_c$-equivariant cohomology much in this article, except for the end of \cref{sec: review Q coeff}. There, we will only consider the case $k = \QQ[u^{\pm 1}]$ with $u$ in degree $2$. In this case, the equivariant cohomology $\H^\ast_{G_c}(\ast; \QQ)$ is concentrated in even weights; in fact, we may identify $\spec \H^0_{G_c}(\ast; k) \cong \fr{t}\mmod W$. It is still reasonable to define $\Loc_{G_c}^\gr(\Gr_G; k)$ to be
    $$\Loc_{G_c}^\gr(\Gr_G; k) = \coLMod_{\H_0^{G_c}(\Gr_G; k)}(\QCoh(\fr{t}\mmod W)).$$
    Similarly, the $\infty$-category $\Loc_{G_c}^\gr(G_c; k)$ can be defined as
    $$\Loc_{G_c}^\gr(G_c; k) = \LMod_{\H_0^{G_c}(\Gr_G; k)}(\QCoh(\fr{t}\mmod W)).$$
\end{remark}
\begin{example}\label{ex: graded torus satake}
    If $G = T$ is a maximal torus, it follows from \cref{thm: torus satake} that there are equivalences of $\pi_0(k)$-linear $\infty$-categories
    \begin{align*}
        \Loc_{T_c}^\gr(\Gr_T; k) & \simeq \QCoh(\cM_{T,0}/\ld{T}), \\
        \Loc_{T_c}^\gr(T_c; k) & \simeq \QCoh(\cM_{T,0} \times_{\spec \pi_0(k)} \ld{T}).
    \end{align*}
\end{example}

Suppose $X$ is a (ind-)finite $T_c$-space with even cells. Since $\Loc_{T_c}^\gr(X; k)$ is a degeneration of $\Loc_{T_c}(X; k)$, one should expect a spectral sequence computing the cohomology $\Gamma_{T_c}(X; \cf)$ for $\cf \in \Loc_{T_c}(X; k)$ from corresponding objects $\cf^\gr\in \Loc_{T_c}^\gr(X; k)$. Similarly, if $Y$ is a finite $T_c$-space such that $\Omega Y$ has even cells, one should expect a spectral sequence computing the cohomology $\Gamma_{T_c}(Y; \cf)$ for $\cf \in \Loc_{T_c}(Y; k)$ from corresponding objects $\cf^\gr\in \Loc_{T_c}^\gr(Y; k)$. This is a special case of the following general setup.
\begin{construction}\label{cstr: sseq degeneration cohomology}
    Recall that if $\fr{X}$ is a spectral stack and $\cf \in \QCoh(\fr{X})$, the truncation $\ul{\tau}_{\geq n}(\cf)$ is the quasicoherent $\co_{\fr{X}}$-module given on an affine open $U$ by $\tau_{\geq n}(\cf(U))$; similarly for $\ul{\tau}_{\leq n}$ and $\ul{\tau}_{[n,m]}$ with $m\geq n$.
    There is a functor $\QCoh(\cM_T) \to \QCoh(\cM_{T,0})$ given by sending a quasicoherent sheaf $\cf$ on $\cM_T$ to the quasicoherent sheaf $\ul{\tau}_{[0,1]}(\cf)$ over $\cM_{T,0}$. This functor can be expressed as the composite of two functors: the first sends the $\co_{\cM_T}$-module $\cf$ to the filtered  $\ul{\tau}_{\geq 2\star} \co_{\cM_T}$-module $\ul{\tau}_{\geq 2\star}(\cf)$; and the second is given by taking associated graded. Note that since the structure sheaf $\co_{\cM_T}$ is $2$-periodic, the data of the graded $\ul{\pi}_{2\ast} \co_{\cM_T}$-module $\gr(\ul{\tau}_{\geq 2\star}(\cf))$ is equivalent to the data of the (ungraded) $\co_{\cM_{T,0}}$-module $\ul{\tau}_{[0,1]}(\cf)$.
    
    Let $\cA$ be an $\Eoo$-coalgebra in $\QCoh(\cM_T)$ whose homotopy sheaves are concentrated in even degrees (such as $\cf_T(X)^\vee$).  If $\cf \in \coMod_\cA(\QCoh(\cM_T))$, the comodule map $\cf \to \cf \otimes_{\co_{\cM_T}} \cA$ induces a comodule map 
    $$\tau_{\leq 2\star} \cf \to \tau_{\leq 2\star} (\cf \otimes_{\co_{\cM_T}} \cA) \to \tau_{\leq 2\star} (\cf) \otimes_{\tau_{\leq 2\star}(\co_{\cM_T})} \tau_{\leq 2\star}(\cA)$$
    due to the oplax symmetric monoidality of the truncation functor. Taking associated graded and using the $2$-periodicity of $\co_{\cM_T}$, we obtain a $\pi_0(\cA)$-comodule structure on the $\co_{\cM_{T,0}}$-module $\ul{\tau}_{[0,1]}(\cf)$. This defines a functor $\coMod_\cA(\QCoh(\cM_T)) \to \coMod_{\pi_0(\cA)}(\QCoh(\cM_{T,0}))$, which we will denote by $\cf \mapsto \cf^\gr$. For instance, if $\cA = \cf_T(X)^\vee$ and $\cf \in \Loc_{T_c}(X; k) = \coMod_\cA(\QCoh(\cM_T))$, then there is a spectral sequence
    \begin{equation}\label{eq: sseq for coh of sheaf from loc gr}
        \pi_\ast(k) \otimes_{\pi_0(k)} \pi_\ast \Map_{\Loc_{T_c}^\gr(X; k)}(\ul{k}^\gr, \cf^\gr) \Rightarrow \pi_\ast \Map_{\Loc_{T_c}(X; k)}(\ul{k}, \cf) = \pi_\ast \Gamma_{T_c}(X; \cf).
    \end{equation}

    Similarly, let $\cB$ be an $\E{1}$-algebra in $\QCoh(\cM_T)$ whose homotopy sheaves are concentrated in even degrees (such as $\cf_T(\Omega Y)^\vee$).  If $\cf \in \LMod_\cB(\QCoh(\cM_T))$, the module map $\cB \otimes_{\co_{\cM_T}} \cf \to \cf$ induces a comodule map 
    $$\tau_{\geq 2\star} (\cB) \otimes_{\tau_{\geq 2\star}(\co_{\cM_T})} \tau_{\geq 2\star}(\cf) \simeq \tau_{\geq 2\star} (\cB \otimes_{\co_{\cM_T}} \cf) \to \tau_{\geq 2\star}(\cf)$$
    due to the lax symmetric monoidality of the cotruncation functor. Taking associated graded and using the $2$-periodicity of $\co_{\cM_T}$, we obtain a left $\pi_0(\cB)$-module structure on the $\co_{\cM_{T,0}}$-module $\ul{\tau}_{[0,1]}(\cf)$. This defines a functor $\LMod_\cB(\QCoh(\cM_T)) \to \LMod_{\pi_0(\cB)}(\QCoh(\cM_{T,0}))$, which we will denote by $\cf \mapsto \cf^\gr$. For instance, if $\cB = \cf_T(\Omega Y)^\vee$ and $\cf \in \Loc_{T_c}(Y; k) = \LMod_\cB(\QCoh(\cM_T))$, then there is a spectral sequence
    $$\pi_\ast(k) \otimes_{\pi_0(k)} \pi_\ast \Map_{\Loc_{T_c}^\gr(Y; k)}(\ul{k}^\gr, \cf^\gr) \Rightarrow \pi_\ast \Map_{\Loc_{T_c}(Y; k)}(\ul{k}, \cf) = \pi_\ast \Gamma_{T_c}(Y; \cf).$$
\end{construction}

Let us now discuss how one might define analogous degenerations if $k$ is not necessarily an even and $2$-periodic $\Eoo$-ring. Although this discussion can be generalized to some other $\Eoo$-rings (such as $\TMF$), we will focus only on the case when $k$ is the $\Eoo$-ring $\KO$ of \textit{real K-theory}. Here is a brief summary of its relevant properties: $\KO$ can be defined from $\KU$ using the $\Z/2$-action on $\KU$ via complex conjugation. Namely, $\KO = \KU^{h\Z/2}$; in fact, as proved in \cite{rognes}, the map $\KO \to \KU$ is a $\Z/2$-Galois extension, meaning that the base-change of any $\KO$-module to $\KU$ acquires the structure of a $\Z/2$-equivariant $\KU$-module. In the discussion below, we will not need to know much about $\KO$, other than the following facts: the generator of $\Z/2$ sends the Bott class $\beta \in \pi_2(\KU)$ to $-\beta$; and the homotopy groups of $\KO$ are \textit{not} even, nor are they $2$-periodic\footnote{In fact, there is an isomorphism
$$\pi_\ast(\KO) \cong \Z[\eta, 2\beta^2, \beta^{\pm 4}]/(2\eta, \eta^3, \eta \cdot (2\beta^2), (2\beta)^2 - 4\beta^4),$$
where $\eta$ is in degree $1$, $2\beta^2$ is in degree $4$, and $\beta^4$ is in degree $8$. The map $\pi_\ast(\KO) \to \pi_\ast(\KU) \cong \Z[\beta^{\pm 1}]$ kills $\eta$, and sends the other classes to their eponyms.}. Therefore, $\KO$ does not quite fit into the setup of \cref{sec: equiv coh} and \cref{sec: degenerations}. Nevertheless, the fact that $\KO$ is the homotopy fixed points $\KU^{h\Z/2}$ does admit a spectral algebro-geometric description: the global sections of the spectral stack $\spec(\KU)/(\Z/2)$ can be identified with $\KO$. Moreover, any $\KO$-module $N$ defines a quasicoherent sheaf over this spectral stack given by the $\Z/2$-action on $\KU \otimes_\KO N$.

Therefore, a more reasonable analogue of the degeneration from a $\KU$-module $M$ to $\pi_\ast(M)$ for a $\KO$-module $N$ is given by considering the graded $\Z/2$-equivariant $\pi_\ast(\KU)$-module $\pi_\ast(\KU \otimes_\KO N)$. If $\KU \otimes_\KO N$ is even, then (since $\pi_\ast(\KU)$ is isomorphic to $\Z[\beta^{\pm 1}]$ with $\beta$ in weight $2$), we may simply view this as the data of the $\Z/2$-equivariant abelian group $\pi_0(\KU \otimes_\KO N)$. That is, studying (spectral) algebraic geometry over $\KO$ amounts simply to keeping track of $\Z/2$-equivariance for (spectral) algebraic geometry over $\KU$. Moreover, the analogue of the degeneration of the spectral scheme $\spec \KU$ to $\spec(\pi_\ast(\KU))/\GG_m \cong \spec(\Z)$ should be understood as a degeneration of the spectral scheme $\spec \KO$ to the $\GG_m$-quotient of $\spec(\pi_\ast(\KU))/(\Z/2) \cong \GG_m/(\Z/2)$, i.e., to the classifying stack $B\Z/2$. Note that if we identify $\Z/2$ with $\spec \Map(\Z/2, \Z) = \spec \Z[a]/(a^2-a)$, where $a$ is the delta function at the non-identity element of $\Z/2$, then the action of $\Z/2$ on $\pi_\ast(\KU)$ is given by the coaction
\begin{equation}\label{eq: Z/2 coaction on KU}
    \Z[\beta^{\pm 1}] \to \Z[\beta^{\pm 1}, a]/(a^2 - a), \ \beta \mapsto (1 - 2a) \beta.
\end{equation}

Motivated by this discussion, we may define $\KO_{T_c}$ for a compact torus $T_c$ as the homotopy $\Z/2$-fixed points of $\KU_{T_c}$ for a particular $\Z/2$-action extending the action of $\Z/2$ on $\KU^{hT_c}$ by complex conjugation. To do so, we need the following simple observation.
\begin{lemma}\label{lem: cplx conj and inversion}
    Under the isomorphism $\pi_0(\KU^{hS^1}) \cong \Z\pw{q-1}$, the action of $\Z/2$ by complex conjugation sends $q\mapsto q^{-1}$. In other words, the action of $\Z/2$ on $\pi_0(\KU^{hS^1})$ is given by the coaction
    $$\Z\pw{q-1} \to \Z\pw{q-1}[a]/(a^2-a), \ q \mapsto q^{1-2a}.$$
\end{lemma}
Motivated by \cref{lem: cplx conj and inversion}, we make the following:
\begin{construction}
    There is an action of $\Z/2$ on the multiplicative group $(\GG_m)_\KU$ over $\KU$ given by inversion. If $T_c$ is a compact torus, this extends to an action of $\Z/2$ on $\cM_T^\KU = T_\KU$. Define $\cM_T^\KO$ to be the spectral stack over $\spec(\KU)/(\Z/2)$ given by $\cM_T^\KU/(\Z/2)$. Observe that the underlying stack of $\cM_T^\KO$ is given by $\cM_{T,0}/(\Z/2)$ over $B\Z/2$ (again, $\Z/2$ acts on $\cM_{T,0} \cong T$ by inversion).
    
    It is clear from \cref{cstr: def-equiv-coh} that the functor $\cf_T(-; \KU): \Top(T_c)^\op \to \QCoh(\cM_T^\KU)$ factors through a functor $\Top(T_c)^\op \to \QCoh(\cM_T^\KO)$. We will denote this new functor by $\cf_T(-; \KO)$. In exactly the same way as in \cref{cstr: def-loc}, one can define a $\QCoh(\cM_T^\KO)$-linear $\infty$-category $\Loc_{T_c}(X; \KO)$ for a finite $T_c$-space $X$. As in \cref{rmk: loc and comod}, there will be an equivalence
    $$\Loc_{T_c}(X; \KO) \simeq \coMod_{\cf_T(X; \KO)^\vee}(\QCoh(\cM_T^\KO));$$
    furthermore, the latter category is equivalent to the $\infty$-category of $\Z/2$-equivariant objects in $\Loc_{T_c}(X; \KU)$.
\end{construction}
Thus, following \cref{def: graded Loc}, we are led to the following.
\begin{definition}\label{def: KO graded Loc}
    Suppose that $X$ is a (ind-)finite $T_c$-space with even cells (such as $\Gr_G$).
    The $\infty$-category $\Loc_{T_c}^\gr(X; \KO)$ is defined as
    $$\Loc_{T_c}^\gr(X; \KO) = \coLMod_{\pi_0(\cf_T(X; \KU)^\vee)}(\QCoh(\cM_{T,0}/(\Z/2))).$$
    Similarly, suppose $Y$ is a finite $T_c$-space such that $\Omega Y$ has even cells (such as $G_c$). The $\infty$-category $\Loc_{T_c}^\gr(Y; \KO)$ is defined as
    $$\Loc_{T_c}^\gr(Y; \KO) = \LMod_{\pi_0(\cf_T(\Omega Y; \KU)^\vee)}(\QCoh(\cM_{T,0}/(\Z/2))).$$
    These categories admit an interesting grading (unlike the analogues with $\KU$-coefficients): the stack $\cM_{T,0}/(\Z/2) = T/(\Z/2)$ lives over $B\GG_m$ via the composite $T/(\Z/2) \to B\Z/2 \to B\GG_m$ where the final map classifies the sign representation of $\Z/2$. We will denote the resulting line bundle over $T/(\Z/2)$ by $\omega$.
\end{definition}
Just as in \cref{eq: sseq for coh of sheaf from loc gr}, if $\cf \in \Loc_{T_c}(X; \KO)$, there is a spectral sequence
\begin{equation}\label{eq: sseq for coh of sheaf from loc gr KO}
    E_2^{\ast,\ast} \cong \pi_\ast \Map_{\Loc_{T_c}^\gr(X; \KO)}(\ul{\KO}^\gr, \cf^\gr \otimes \omega^{\otimes \ast}) \Rightarrow \pi_\ast \Map_{\Loc_{T_c}(X; \KO)}(\ul{k}, \cf) = \pi_\ast \Gamma_{T_c}(X; \cf).
\end{equation}
There is an isomorphism
$$E_2^{\ast,\ast} \cong \H^\ast(B\Z/2, \pi_\ast \Map_{\Loc_{T_c}^\gr(X; \KU)}(\ul{\KU}^\gr, \cf^\gr)[\beta^{\pm 1}]),$$
where $\Z/2$ acts on $\beta$ by negation.
\begin{example}
    It follows from \cref{ex: graded torus satake} that there are equivalences of $\QCoh(B\Z/2)$-linear $\infty$-categories
    \begin{align*}
        \Loc_{T_c}^\gr(\Gr_T; \KO) & \simeq \QCoh(T/(\Z/2) \times B\ld{T}), \\
        \Loc_{T_c}^\gr(T_c; \KO) & \simeq \QCoh(T/(\Z/2) \times \ld{T}).
    \end{align*}
\end{example}

Before proceeding to describing an analogue of the above picture with $\KO$ replaced by the $K(1)$-local sphere, we will describe $\KO_T = \Gamma(\cM_T^\KO; \co)$ for the sake of completeness. There is a spectral sequence 
\begin{equation}\label{eq: sseq KO T}
    E_2^{s,\ast} \cong \H^s(\Z/2; \co_T[\beta^{\pm 1}]) \cong \H^s(\bV(\omega^{-1})^\times; \co) \Rightarrow \pi_{\ast-s}(\KO_T),
\end{equation}
where $\ast$ denotes the grading on $\co_T[\beta^{\pm 1}]$ (so $\beta$ is in weight $2$).
Here, $\bV(\omega^{-1})^\times$ is the complement of the zero section in the total space of the line bundle $\omega^{-1}$ over $T/(\Z/2)$. The action of $\Z/2$ on $\co_T[\beta^{\pm 1}]$ is given by inversion on $T$, and sends $\beta \mapsto -\beta$. One can view \cref{eq: sseq KO T} as the spectral sequence \cref{eq: sseq for coh of sheaf from loc gr KO} computing the cohomology of the constant sheaf on a point. As we will explain below, this spectral sequence has nontrivial differentials, so it does not immediately collapse.

For simplicity, we will focus on the case $T = S^1$, so $\co_T = \Z[x^{\pm 1}]$. Then an elementary calculation in group cohomology shows that the $E_2$-page of \cref{eq: sseq KO T} is given by
$$E_2^{\ast,\ast} \cong \Z[\eta, \beta^{\pm 2}, x + x^{-1}, \tfrac{x^n - x^{-n}}{\beta}]_{n\geq 1}/2\eta,$$
where all classes except for $\eta$ lie in $E_2^{0,\ast}$, and $\eta \in E_2^{1,2}$. A standard calculation in homotopy theory (coming from the analysis of the Adams-Novikov spectral sequence) says that there is a differential $d_3(\beta^2) = \eta^3$. There are no further differentials past this point, and propagating this differential shows:
\begin{prop}\label{prop: htpy KOS1}
    There is an isomorphism
    $$\pi_\ast(\KO_{S^1}) \cong \Z[\eta, 2\beta^2, \beta^{\pm 4}, x + x^{-1}, \tfrac{x^n - x^{-n}}{\beta}]_{n\geq 1}/(2\eta, \eta^3, \eta \cdot (2\beta^2), (2\beta^2)^2 = 4\beta^4),$$
    where the terms $\eta, 2\beta^2, \beta^{\pm 4}$ simply contribute a copy of $\pi_\ast(\KO)$, the term $x + x^{-1}$ contributes a class to $\pi_0(\KO_{S^1})$, and the terms $\tfrac{x^n - x^{-n}}{\beta}$ contribute infinitely many classes to $\pi_2(\KO_{S^1})$.
\end{prop}
It is hard to extract concrete implications\footnote{This is not to say that computing $\KO$-(co)homology groups is a worthless endeavor: in \cite{adams-vector-fields}, Adams famously computed the $\KO$-cohomology of real projective spaces to solve the question of counting linearly independent vector fields on spheres.} for Langlands duality from the structure of $\pi_\ast(\KO_{S^1})$; so we will not compute the homotopy groups of $\pi_0(\cf_T(\Gr_G; \KU)^\vee)$ below, and content ourselves with just describing the $\Z/2$-action on $\pi_0(\cf_T(\Gr_G; \KU)^\vee)$.

\begin{remark}\label{rmk: connective ko def}
    The above story can be extended to include the case of \textit{connective} real K-theory $\ko = \tau_{\geq 0}(\KO)$, too. Since we will only return to this picture occasionally in this article, we will be scant on details. The article \cite{ku-rel-langlands} studied a variant of Langlands duality with coefficients in connective complex K-theory $\ku$, which is an $\Eoo$-ring such that $\pi_\ast(\ku) = \Z[\beta]$ with $\beta$ in degree $2$ (so that $\ku/\beta = \Z$ and $\ku[\beta^{-1}] = \KU$). Its $S^1$-equivariant version $\ku_{S^1}$ has homotopy groups given by $\pi_\ast(\ku_{S^1}) \cong \Z[\beta, x, \tfrac{1}{1+\beta x}]$ with $x$ in weight $-2$. Let $\GG_\beta = \spec \pi_\ast(\ku_{S^1})/\GG_m$, where the group law is given by $x + y + \beta xy$. If $T$ is a torus, let $T_\beta = \Hom(\bX^\ast(T), \GG_\beta)$.

    Since $\ku$ is the connective cover $\tau_{\geq 0}(\KU)$ of $\KU$, the action of $\Z/2$ on $\KU$ by complex conjugation lifts to an action of $\Z/2$ on $\ku$. While there is a map $\ko \to \ku^{h\Z/2}$, this map is \textit{not} an equivalence; rather, it exhibits $\ko$ as the connective cover $\tau_{\geq 0}(\ku^{h\Z/2})$. In particular, while the $\Eoo$-ring $\ku \otimes_\ko \ku$ is not equivalent to $\Map(\Z/2, \ku)$, it is still a finite free $\ku$-algebra with even homotopy.
    Therefore, the appropriate degeneration of the spectral scheme $\spec(\ko)$ is no longer the algebraic stack $(\spec(\pi_\ast(\ku))/\GG_m)/(\Z/2)$, but is rather given by the stack\footnote{The reason for the notation ``$\spev$'' will be explained in a future article, and uses the even filtration of \cite{even-filtr, piotr-even-filtr}.} $\spev(\ko)$ defined as the quotient by $\GG_m$ of the geometric realization of the simplicial stack
    $$\begin{tikzcd}
    	\cdots & {\spec(\pi_\ast(\ku^{\otimes_\ko 3}))} & {\spec(\pi_\ast(\ku^{\otimes_\ko 2}))} & {\spec(\pi_\ast(\ku))}.
    	\arrow[shift left, from=1-1, to=1-2]
    	\arrow[shift left=3, from=1-1, to=1-2]
    	\arrow[shift right=3, from=1-1, to=1-2]
    	\arrow[shift right, from=1-1, to=1-2]
    	\arrow[from=1-2, to=1-3]
    	\arrow[shift right=2, from=1-2, to=1-3]
    	\arrow[shift left=2, from=1-2, to=1-3]
    	\arrow[shift left, from=1-3, to=1-4]
    	\arrow[shift right, from=1-3, to=1-4]
    \end{tikzcd}$$
    A standard calculation says that $\pi_\ast(\ku \otimes_\ko \ku) \cong \Z[\beta, r]/(r^2 - r\beta)$ with $r$ in weight $2$, and that the two maps $\eta_L, \eta_R: \ku \rightrightarrows \ku \otimes_\ko \ku$ send $\eta_L: \beta \mapsto \beta$ and $\eta_R: \beta \mapsto \beta - 2r$. Upon inverting $\beta$, we may identify $\pi_\ast(\ku \otimes_\ko \ku)[\beta^{-1}]$ with $\Z[\beta^{\pm 1}, a]/(a^2 - a)$ where $a = r\beta^{-1}$, and then $\eta_R$ is precisely the coaction from \cref{eq: Z/2 coaction on KU}. As described in \cite[Section 9]{tmf}, $\spev(\ko)$ classifies isomorphism classes of possibly singular quadratic curves (which are locally of the form $y = x^2 + \beta x$).
    
    Note that 
    $$\eta_R(\beta^n) = \beta^n + ((-1)^n - 1) r\beta^{n-1},$$
    so $\beta^{2n}$ is a well-defined function on $\spev(\ko)$ for any $n\geq 0$; the complement of its vanishing locus is precisely $B\Z/2$.
    Just as $B\Z/2$ is an open substack in $\spev(\ko)$, the stack $\cM^\KO_T$ is also open in a certain stack $\cM_T^\ko$, which can be defined as the stack of homomorphisms from $\bX^\ast(T)$ to the quotient of $\GG_\beta$ (viewed as a scheme over $\spec(\Z[\beta])/\GG_m$) by the coaction of $\pi_\ast(\ku \otimes_\ko \ku)$ given by
    \begin{equation}\label{eq: coaction connective ku on Gbeta}
        \Z[\beta, x, \tfrac{1}{1+\beta x}] \to \Z[\beta, x, \tfrac{1}{1+\beta x}, r]/(r^2 - \beta r), \ x\mapsto x - \tfrac{rx^2}{1+\beta x}.
    \end{equation}
    This might look a bit strange, but it is a pleasant exercise to verify (using the binomial formula) that upon inverting $\beta$, it identifies with the map
    $$\Z[\beta^{\pm 1}, x, \tfrac{1}{1+\beta x}] \to \Z[\beta^{\pm 1}, x, \tfrac{1}{1+\beta x}, a]/(a^2 - a), \ (1+\beta x)\mapsto (1+\beta x)^{1-2a}$$
    as forced by \cref{lem: cplx conj and inversion}.
    In any case, given the stack $\cM_T^\ko$, one can define $\QCoh(\cM_T^\ko)$-linear $\infty$-categories $\Loc_{T_c}^\gr(X; \ko)$ exactly as in \cref{def: KO graded Loc}. We will return to this below in \cref{sec: KU coeff}.
\end{remark}

Finally, we turn to the $K(1)$-local sphere. To motivate it, note that the action of complex conjugation on $\KU$ is given simply by the action of the Adams operation $\psi^{-1}$. It is therefore natural to wonder about the action of other Adams operations. To this end, we will fix a prime $p$ and contemplate a parallel story with $\KO$ replaced by the ``image of J''/$K(1)$-local sphere spectrum $L_{K(1)} S^0 = (\KU^\wedge_p)^{h\Z_p^\times}$, where $\Z_p^\times$ acts continuously on $\KU^\wedge_p$ by Adams operations: there is a map $\Z_p^\times \to \Aut_\Eoo(\KU^\wedge_p)$ sending $n \in \Z_p^\times$ to the Adams operation $\psi^n: \KU^\wedge_p \to \KU^\wedge_p$. (In fact, this map is an equivalence!)

The homotopy groups of $L_{K(1)} S^0$ are somewhat complicated\footnote{Explicitly, if $p>2$, then $\pi_i L_{K(1)} S^0$ is isomorphic to $\Z_p$ when $i=0,-1$, and is isomorphic to $\Z/p^{v_p(j)+1}$ for $i = 2(p-1)j - 1$. The order of the latter subgroup is precisely the $p$-part of the denominator of $B_{2(i+1)}/(i+1)$, where $B_{2j}$ is the $2j$th Bernoulli number.}, but just as with $\KO$, studying (spectral) algebraic geometry over $L_{K(1)} S^0$ amounts simply to keeping track of $\Z_p^\times$-equivariance for (spectral) algebraic geometry over $\KU^\wedge_p$. That is to say, $L_{K(1)} S^0$ is the global sections of the structure sheaf on the spectral stack $\spf(\KU^\wedge_p)/\Z_p^\times$. Moreover, the analogue of the degeneration of the spectral scheme $\spf \KU^\wedge_p$ to $\spf(\pi_\ast(\KU^\wedge_p))/\GG_m \cong \spf(\Z_p)$ should be understood as a degeneration of the spectral scheme $\spf L_{K(1)} S^0$ to the $\GG_m$-quotient of $\spf(\pi_\ast(\KU^\wedge_p))/\Z_p^\times$, i.e., to the classifying stack $B\Z_p^\times$.

To define an analogue of \cref{def: KO graded Loc} for $L_{K(1)} S^0$, we need to upgrade the $\Z_p^\times$-action on $\KU$ to an action on equivariant K-theory. Recall that if $\cT$ denotes the full subcategory of $\Top$ spanned by those spaces which are homotopy equivalent to $BT_c$ with $T_c$ being a compact abelian Lie group, the data of a preorientation of $\GG = \GG_m$ is equivalent to the data of a functor $\cM: \cT \to \Aff_\KU$ along with compatible equivalences $\cM(BT_c) \simeq \cM_T$. This can be composed with the functor $\Aff_\KU \to \Aff_{\KU^\wedge_p}^{p\cpl}$ of $p$-completion. 

Unfortunately, even at the level of classical algebra, there is no natural action of $\Z_p^\times$ on $\GG_m = \spf \Z_p[x^{\pm 1}]$ where $n\in \Z_p^\times$ sends $x\mapsto x^n$: the power series $x^n = \sum_{i\geq 0} \binom{n}{i} (x-1)^n$ need not converge without a further completion. Nevertheless, such an action of $\Z_p^\times$ does exist if we restrict to the subgroups $\mu_{p^n} = \spf \Z_p[\Z/p^n] \subseteq \GG_m$; in fact, the action factors through the surjection $\Z_p^\times \twoheadrightarrow (\Z/p^n)^\times$. The subgroups $\mu_{p^n}$ naturally lift to $\KU$ (by $\spec \KU[\Z/p^n]$), and each admits a natural $\Z_p^\times$-action. Of course, these $\Z_p^\times$-actions exist even before $p$-completion; but to get a well-behaved operation on $\Z/p^n$-equivariant $\KU$-cohomology, we need the $\Z_p^\times$-action to preserve the preorientation on $\mu_{p^n}$, and this in turn happens once $\KU$ is $p$-completed. 

Suppose, therefore, that we restrict to the full subcategory $\cT_p \subseteq \cT$ spanned by those spaces which are homotopy equivalent to $BA$ with $A$ being a $p$-power torsion compact abelian Lie group. Then the preceding paragraph implies that $\cM|_{\cT_p}: \cT_p \to \Aff_\KU$ refines to a functor $\cT_p \to (\Aff_{\KU^\wedge_p}^{p\cpl})^{h\Z_p^\times}$. Following \cref{cstr: def-equiv-coh} verbatim defines an action of $\Z_p^\times$ on $\cM_A$, and furthermore equips the quasicoherent sheaf $\cf_A(X) \in \QCoh(\cM_A)$ associated to a finite $A$-space $X$ with a $\Z_p^\times$-equivariant structure. We will write $\cM_A^{L_{K(1)} S^0} = \cM_A/\Z_p^\times$, and let $\cf_A(-; L_{K(1)} S^0)$ denote the corresponding functor $\Top(A)^\op \to \QCoh(\cM_A^{L_{K(1)} S^0})$. Again, following \cref{def: graded Loc}, we are led to\footnote{Just as with connective $\ko$, one can also define a variant of \cref{def: J graded Loc} for the \textit{connective} image of J spectrum $j$. We leave this to the interested reader.}:
\begin{definition}\label{def: J graded Loc}
    Suppose that $A$ is a $p$-power torsion abelian group, and $X$ is a (ind-)finite $A$-space with even cells (such as $\Gr_G$).
    The $\infty$-category $\Loc_{A}^\gr(X; L_{K(1)} S^0)$ is defined as
    $$\Loc_{A}^\gr(X; L_{K(1)} S^0) = \coLMod_{\pi_0(\cf_A(X; \KU)^\vee)}(\QCoh(\cM_{A,0}/\Z_p^\times)).$$
    Similarly, suppose $Y$ is a finite $A$-space such that $\Omega Y$ has even cells (such as $G_c$). The $\infty$-category $\Loc_{A}^\gr(Y; L_{K(1)} S^0)$ is defined as
    $$\Loc_{A}^\gr(Y; L_{K(1)} S^0) = \LMod_{\pi_0(\cf_A(\Omega Y; \KU)^\vee)}(\QCoh(\cM_{A,0}/\Z_p^\times)).$$
    These categories admit an interesting grading (just like the analogue with $\KO$-coefficients): the stack $\cM_{A,0}/(\Z/2) = A/\Z_p^\times$ lives over $B\GG_m$ via the composite $A/\Z_p^\times \to B\Z_p^\times \to B\GG_m$ where the final map classifies the standard (cyclotomic) representation of $\Z_p^\times$ on $\Z_p$. We will denote the resulting line bundle over $A/\Z_p^\times$ by $\omega$.
\end{definition}
For the sake of completness (and partly because it is a pleasant calculation), let us describe $(L_{K(1)} S^0)_{T[p^\infty]} = \Gamma(\cM_{T[p^\infty]}^{L_{K(1)} S^0}; \co)$ when $p$ is odd. Since this is built as a limit of the spectra $(L_{K(1)} S^0)_{T[p^n]}$, we will just compute each of these individually. There is a spectral sequence 
\begin{equation}\label{eq: sseq K1-local mod pn}
    E_2^{s,\ast} \cong \H^s_\mathrm{cts}(\Z_p^\times; \co_{T[p^n]}[\beta^{\pm 1}]) \cong \H^s(\bV(\omega^{-1})^\times; \co) \Rightarrow \pi_{\ast-s}((L_{K(1)} S^0)_{T[p^n]}),
\end{equation}
where $\ast$ denotes the grading on $\co_{T[p^n]}[\beta^{\pm 1}]$ (so $\beta$ is in weight $2$).
Here, $\bV(\omega^{-1})^\times$ is the complement of the zero section in the total space of the line bundle $\omega^{-1}$ over $A/\Z_p^\times$. Fix a topological generator $g \in \Z_p^\times$ such that $g^{p-1} = 1 + p$, so that its action (denoted $\psi^g$) on $\co_{T[p^n]}[\beta^{\pm 1}]$ is given by exponentiation on $T[p^n]$, and sends $\beta \mapsto g\beta$. One can view \cref{eq: sseq K1-local mod pn} as the spectral sequence \cref{eq: sseq for coh of sheaf from loc gr} computing the cohomology of the constant sheaf on a point. This spectral sequence has no nontrivial differentials, so it collapses; this, however, is no longer true if $p=2$.

For simplicity, we will focus on the case $T = S^1$, so $\co_{T[p^n]} = \co_{\mu_{p^n}} = \Z_p[x^{\pm 1}]/(x^{p^n}-1)$ (recall that we have $p$-completed!). We then have:
\begin{prop}\label{prop: htpy K1-local S0}
    If $p>2$, there are isomorphisms
    $$\pi_j((L_{K(1)} S^0)_{\mu_{p^n}}) \cong \begin{cases}
        \pi_j(L_{K(1)} S^0) \oplus \Z_p^{\oplus n} & j=0,-1 \\
        \pi_j(L_{K(1)} S^0) \oplus \bigoplus_{i=0}^{n-1} \Z_p/kp^{n-i} & j=2k-1, k \in \Z \\
        0 & \text{else}.
    \end{cases}$$
\end{prop}
\begin{proof}
    The $E_2$-page of \cref{eq: sseq K1-local mod pn} is given by the group cohomology of $\Z_p^\times$ acting on $\Z_p[x^{\pm 1}, \beta^{\pm 1}]/(x^{p^n}-1)$, so $E_2^{\ast,2k}$ is given by the cohomology of the two-term complex
    \begin{align*}
        \Z_p[x^{\pm 1}]/(x^{p^n}-1) & \xrightarrow{\psi^g - 1} \Z_p[x^{\pm 1}]/(x^{p^n}-1), \\
        f(x) & \mapsto g^k f(x^g) - f(x).
    \end{align*}
    Let us sketch the calculation of the cohomology of this complex, which we will denote by $C^\bull$ below. Write $\Z_p[x^{\pm 1}]/(x^{p^n} - 1) = \Z_p[\Z/p^n]$, so it is a free $\Z_p$-module on the classes $\{1,x,\cdots,x^{p^n - 1}\}$. The action of $\Z_p^\times$ on $\Z/p^n$ (which factors through the quotient map $\Z_p^\times \twoheadrightarrow (\Z/p^n)^\times$) has $n+1$ orbits, with representatives given by $\{p^i\}_{0\leq i \leq n-1} \cup \{0\}$. The orbit of $0$ is a singleton, and the orbit of $p^i$ has size $p^{n-i-1}(p-1)$. It follows that the $p^n \times p^n$-matrix $\psi^g - 1$ can be written as the block sum $(g^k - 1) \oplus \bigoplus_{i=0}^{n-1} A_i$, where $A_i$ is an $p^{n-i-1}(p-1) \times p^{n-i-1}(p-1)$-matrix. For consistency, we will write $A_{-1}$ to denote the scalar $g^k - 1$.

    Let $0\leq i \leq n-1$. Then the matrix $A_i$ acts on the submodule $\Z_p^{\oplus p^{n-i-1}(p-1)} = \Z_p\{x^{p^i}, x^{p^i g}, \cdots, x^{p^i g^{p^{n-i-1}(p-1) - 1}}\}$, and each row and column of $A_i$ has exactly two entries (namely, $-1$ on the diagonal entry, and $g^k$ elsewhere). Computing the Smith normal form of this matrix shows that $A_i$ has no kernel unless $k=0$, in which case its kernel is free of rank $1$. If $k=0$, then the cokernel of $A_i$ is also free of rank $1$, and if $k\neq 0$, then the cokernel of $A_i$ is $\Z_p/(g^{k p^{n-i-1}(p-1)} - 1)$. Since $g \in \Z_p^\times$ was chosen to satisfy $g^{p-1} = 1+p$, it follows that $\Z_p/(g^{kp^{n-i-1}(p-1)} - 1) \cong \Z_p/kp^{n-i}$.
    
    We only need to take care of the block $A_{-1}$. If $k = 0$, then $A_{-1}$ is the zero matrix; but if $k$ is nonzero, then $A_{-1}$ has no kernel, and has cokernel given by $\Z_p/(g^k - 1)$. It follows that if $k=0$, then
    $$\H^s(C^\bull) \cong \Z_p^{\oplus n+1} \text{ for }s=0,1.$$
    If $k\neq 0$, then
    $$\H^s(C^\bull) \cong \begin{cases}
        0 & s=0\\
        \Z_p/(g^k - 1) \oplus \bigoplus_{i=0}^{n-1} \Z_p/kp^{n-i} & s=1.
    \end{cases}$$
    The groups $E_2^{s,\ast}$ vanish for $s>1$, so there cannot be any differentials in the spectral sequence \cref{eq: sseq K1-local mod pn}. Using the calculation of the homotopy groups of $K(1)$-local sphere, we obtain the desired answer for $\pi_\ast((L_{K(1)} S^0)_{\mu_{p^n}})$.
\end{proof}
Just as with $\KO_{S^1}$, the groups $\pi_\ast((L_{K(1)} S^0)_{\mu_{p^n}})$ are interesting but form a rather unpleasant ring to do algebraic geometry with; so we will content ourselves with just understanding the category $\Loc_{T_c[p^\infty]}^\gr(\Gr_G; L_{K(1)} S^0)$ below (and not calculate actual homotopy groups).
\newpage

\section{Loop rotation equivariance}\label{sec: torus loop rot}
In this section, we describe an extension of \cref{thm: torus satake} (or rather, of \cref{ex: graded torus satake}) which includes loop-rotation equivariance. Recall that \cref{thm: torus satake} gives an isomorphism $\cf_{T_c}(\Gr_T)^\vee \cong \co(\ld{T}_k \times_{\spec(k)} \cM_T)$. The action of $T$ on $\Gr_T$ refines to an action of $\tilde{T} = T \times \GG_m^\rot$, where $\GG_m^\rot$ acts by loop rotation;  we may therefore consider the \textit{loop-rotation equivariant} homology $\cf_{\tilde{T}_c}(\Gr_T)^\vee$. There is an equivalence $\cM_{\tilde{T}} \simeq \cM_T \times \GG$, where the second factor is identified as $\cM_{\GG_m^\rot}$. Therefore, $\cf_{\tilde{T}_c}(\Gr_T)^\vee$ is a quasicoherent sheaf over $\cM_T \times \GG$ whose fiber over the zero section of $\GG$ recovers $\cf_{{T}_c}(\Gr_T)^\vee$.
\begin{definition}\label{def: G-diff ops}
    Let $\bH$ be a smooth $1$-dimensional group scheme over a base commutative ring $A$, let $T_c$ be a compact torus, and let $\bH_T = \Hom(\bX^\ast(T), \bH)$. (When $\GG$ is an oriented commutative $k$-group scheme, and $\bH = \GG_0$ is its underlying group scheme over $A = \pi_0(k)$, then $\bH_T$ is precisely $\cM_{T,0}$.)
    Let $\lambda$ be a cocharacter of $T_c$, so that $\lambda$ defines a homomorphism $\bX^\ast(T) \to \Z$, and hence a homomorphism $\lambda^\ast: \bH \to \bH_T$. In turn, this defines a map
    $$f^\lambda: \bH_{\tilde{T}} \simeq \bH_T \times \bH \xar{\pr \times \lambda^\ast} \bH_T.$$
    If $y$ is a local section of $\co_{\bH_T}$, we will write $\lambda^\ast(y)$ to denote the resulting local section of $\co_{\bH_{\tilde{T}}}$.
    
    Let $\cd_{\ld{T}}^{\bH}$ denote the quotient of the associative $\co_{\bH}$-algebra $\co_{\bH_{\tilde{T}}}\pdb{x_\lambda | \lambda\in \bX_\ast(T)}$ by the relations given locally by
    $$x_\lambda \cdot x_\mu = x_{\lambda+\mu}, \  y \cdot x_\lambda = x_\lambda \cdot \lambda^\ast(y).$$
    Here, $\lambda,\mu\in \bX_\ast(T)$, and $y$ is a local section of $\co_{\bH_T}$. We will call $\cd_{\ld{T}}^{\bH}$ the \textit{algebra of $\bH$-differential operators} on $\ld{T}$.
\end{definition}
\begin{remark}\label{rmk: G-mellin}
    The algebra $\cd_{\ld{T}}^{\bH}$ satisfies a Mellin transform: namely, it follows from unwinding the definition that there is an equivalence
    $$\cd_{\ld{T}}^{\bH}\modc \simeq \IndCoh(\bH_{\tilde{T}}/\bX^\ast(\ld{T})),$$
    where $\lambda\in \bX^\ast(\ld{T}) \cong \bX_\ast(T)$ acts on $\bH_{\tilde{T}}$ via $y\mapsto \lambda^\ast y$.
\end{remark}
\begin{notation}
    If $k$ is a complex-oriented $2$-periodic $\Eoo$-ring and $\GG_0$ is the $\pi_0(k)$-group underlying a oriented commutative $A$-group $\GG$, we will write $\cd_{\ld{T}}^\GG$ to denote $\cd_{\ld{T}}^{\GG_0}$, and refer to it as the \textit{algebra of $\GG$-differential operators} on $\ld{T}$. We hope this does not cause any confusion.
\end{notation}
\begin{prop}\label{prop: T homology and quantized diffop}
    There is an isomorphism 
    $$\pi_0 \cf_{\tilde{T}}(\Gr_T)^\vee \cong \cd_{\ld{T}}^\GG$$
    of $\co_{\GG_0}$-algebras. In particular, there is an equivalence
    $$\Loc_{\tilde{T}_c}(\Gr_T; k) \simeq \cd_{\ld{T}}^\GG\modc^{(\ld{T} \times \ld{T}, \weak)},$$
    where the right-hand side denotes the category of left $\cd_{\ld{T}}^\GG$-modules whose underlying quasicoherent sheaf over $\ld{T}$ is equivariant for $\ld{T} \times \ld{T}$-action on $\ld{T}$ given by left and right multiplication.
\end{prop}
\begin{proof}
    Since $\Gr_T \cong \bX_\ast(T)$, it is easy to see that $\pi_0 \cf_{\tilde{T}}(\Gr_T)^\vee \cong \bigoplus_{\lambda\in \bX_\ast(T)} \pi_0 \co_{\cM_{\tilde{T}}}$; let $x_\lambda$ be a generator of the summand indexed by $\lambda\in \bX_\ast(T)$. If $\lambda\in \bX_\ast(T) = \Hom(\bX^\ast(T), \Z)$, the map $\Omega T_c \to \Omega T_c$ given by multiplication-by-$\lambda$ is $T_c\times S^1_\rot$-equivariant for the homomorphism $T_c\times S^1_\rot \to T_c\times S^1_\rot$ given by 
    $$(t, \theta) \mapsto (t \lambda(\theta), \theta),$$
    where $\lambda$ is viewed as a homomorphism $S^1 \to T$. On weight lattices, this homomorphism induces the map $\bX^\ast(T) \times \Z \to \bX^\ast(T) \times \Z$ which sends $(\mu, n) \mapsto (\mu, n+\bX_\ast(T)(\mu))$. In particular, the composite 
    $$\bX^\ast(T) \to \bX^\ast(T) \times \Z \to \bX^\ast(T) \times \Z$$
    sends $\mu \mapsto (\mu, \bX_\ast(T)(\mu))$. Applying $\Hom(-, \GG)$ to this composite precisely produces the map $f^\lambda: \cM_{\tilde{T}} \to \cM_T$ from \cref{def: G-diff ops}. This implies the desired identification of $\pi_0 \cf_{\tilde{T}}(\Gr_T)^\vee$.
\end{proof}
\begin{example}\label{ex: ordinary quantized diffop}
    Let $T \cong S^1$ be a torus of rank $1$ (for simplicity).
    Suppose $k = \QQ[u^{\pm 1}]$ with $u$ in degree $2$, so $\GG = \GG_a$ and $\co_{\GG_0} \cong \QQ[\hbar]$. Then the algebra of \cref{def: G-diff ops} is the quotient of the $\QQ[\hbar]$-algebra $\QQ[\hbar]\pdb{y, x^{\pm 1}}$ by the relation $yx = x(y+\hbar)$. In other words, $y$ acts as the operator $\hbar x\partial_x$, so we simply have that 
    $$\H^{\tilde{T}}_0(\Gr_T; \QQ[u^{\pm 1}]) \cong \H^{\tilde{T}}_\ast(\Gr_T; \QQ) \cong \QQ[\hbar]\pdb{\hbar x\partial_x, x^{\pm 1}}.$$
    This has been stated previously as \cite[Proposition 5.19(2)]{bfn-ii}.
    In particular, the localization $\H^{\tilde{T}}_0(\Gr_T; \QQ[u^{\pm 1}])[\hbar^{-1}]$ is isomorphic to the rescaled Weyl algebra $\cd_{\ld{T}}^\hbar$; this is the motivation behind the terminology in \cref{def: G-diff ops}.
    Note that for a general torus, \cref{rmk: G-mellin} simply reduces to the standard Mellin transform, which gives an equivalence between $\DMod_\hbar(\ld{T})$ and $\IndCoh(\fr{t}_{\QQ[\hbar]}/\bX^\ast(\ld{T}))$; here, $\lambda \in \bX^\ast(\ld{T})$ acts on $\fr{t}_{\QQ[\hbar]}$ by $x \mapsto x + (d\lambda)(\hbar)$.
\end{example}
\begin{example}\label{ex: q quantized diffop}
    Again, let $T \cong S^1$ be a torus of rank $1$ (for simplicity).
    Suppose $k = \KU$, so $\GG = \GG_m$ and $\co_{\GG_0} \cong \Z[q^{\pm 1}]$. Then the algebra of \cref{def: G-diff ops} is the quotient of the $\Z[q^{\pm 1}]$-algebra $\Z[q^{\pm 1}]\pdb{y^{\pm 1}, x^{\pm 1}}$ by the relation $yx = qxy$. (This is also known as the ``quantum torus''.) In other words, $y$ acts as the operator $q^{x\partial_x}$ sending $f(x) \mapsto f(qx)$, so we simply have that 
    $$\KU^{\tilde{T}}_0(\Gr_T) \cong \Z[q^{\pm 1}]\pdb{q^{x\partial_x}, x^{\pm 1}}.$$
    This is closely related to the $q$-Weyl algebra $\cd_q = \Z[q^{\pm 1}]\pdb{\Theta, x^{\pm 1}}/(\Theta x = x(q\Theta+1))$ for $\ld{T} = \GG_m$: indeed, since the logarithmic $q$-derivative $\Theta = x\nabla_q$ is given by the fraction $\frac{q^{x\partial_x}-1}{q-1}$, the pullback of $\cd_{\ld{T}}^\GG$ along $\GG_m-\{1\} \hookrightarrow \GG_m$ is isomorphic to the algebra $\cd_q[\frac{1}{q-1}]$.
    Note that \cref{rmk: G-mellin} gives a ``$q$-Mellin transform'', i.e., an equivalence between $\LMod_{\KU^{\tilde{T}}_0(\Gr_T)}$ and $\IndCoh(T_{\Z[q^{\pm 1}]}/\bX^\ast(\ld{T}))$, where $\lambda \in \bX^\ast(\ld{T}) = \bX_\ast(T)$ acts on $T_{\Z[q^{\pm 1}]}$ by sending $y\mapsto \lambda(q) y$.
\end{example}

Let us briefly outline the relationship between the algebra $\cd_{\ld{T}}^{\bH}$ of \cref{def: G-diff ops} and the $F$-de Rham complex of \cite{generalized-n-series}.
\begin{notation}
    For the purpose of this discussion, we will assume that $T \cong S^1$ is a torus of rank $1$, so that $\ld{T} \cong \GG_m$. We will also fix an invariant differential form on the formal completion $\hat{\bH}$ of $\bH$ at the zero section, so that there is an isomorphism $\hat{\bH} \cong \spf A\pw{t}$ of formal $A$-schemes. Let $F(x,y)$ denote the resulting formal group law over $A$, and define the $n$-series of $F$ by
    $$[n]_F := \overbrace{F(t, F(t, F(t, \cdots F(t, t) \cdots )))}^n.$$
    We will often write $x+_Fy = x+_\GG y$ to denote $F(x,y)$.
    Let $\hat{\cd}_{\ld{T}}^{\bH}$ denote the completion of $\cd_{\ld{T}}^{\bH}$ at the zero section of $\bH$.
\end{notation}
\begin{lemma}[Cartier duality]\label{cartier-duality}
    Let $\hat{\bH}$ be a $1$-dimensional formal group over a commutative ring $A$, and let $\Cart(\hat{\bH})$ denote its Cartier dual (see \cite[Section 3.3]{drinfeld-formal-group} for more on Cartier duals of formal groups). Then there is an equivalence of categories $\QCoh(\hat{\bH}) \simeq \QCoh(B\Cart(\hat{\bH}))$ sending the convolution tensor product on the left-hand side to the usual tensor product on the right-hand side. Under this equivalence, the functor $\QCoh(\hat{\bH}) \to \Mod_A$ given by restriction to the zero section is identified with the functor $\QCoh(B\Cart(\hat{\bH})) \to \Mod_A$ given by pullback along the map $\spec(A) \to B\Cart(\hat{\bH})$.
\end{lemma}
\begin{prop}\label{prop: endomorphism F-de Rham}
    There is a canonical action of $\hat{\cd}_{\ld{T}}^{\bH}$ on $(\GG_m)_{A\pw{t}} = \spf A\pw{t}[x^{\pm 1}]$ such that $A\pw{t}[x^{\pm 1}] \otimes_{\hat{\cd}_{\ld{T}}^{\bH}} A\pw{t}[x^{\pm 1}]$ is isomorphic to the two-term complex
    $$C^\bull = (A\pw{t}[x^{\pm 1}] \to A\pw{t}[x^{\pm 1}]dx), \ x^n \mapsto [n]_F x^n dx$$
    from \cite[Remark 4.3.8]{generalized-n-series}.
\end{prop}
\begin{proof}[Proof sketch]
    Since $T$ is of rank $1$, there is an isomorphism $\bH_T \cong \bH$, and hence an isomorphism $\hat{\bH}_T \cong \hat{\AA}^1$ of formal $A$-schemes, where $\hat{\bH}_T$ denotes the completion of $\bH_T$ at the zero section. Let $y$ be a local coordinate on $\bH_T$.
    Then, $\hat{\cd}_{\ld{T}}^{\bH}$ is isomorphic to the quotient of the associative $\hat{\co}_{\bH}$-algebra $\hat{\co}_{\bH \times \bH_T}\pdb{x^{\pm 1}}$ subject to the relation $yx = x(y +_\GG t)$.
    The $t$-adic filtration on $\hat{\cd}_{\ld{T}}^{\bH}$ therefore has associated graded $\gr(\hat{\cd}_{\ld{T}}^{\bH}) \cong \hat{\co}_{\bH_T}[x^{\pm 1}]\pw{\ol{t}}$, where $\ol{t}$ lives in weight $1$. View $A$ as a $\co_{\bH_T}$-algebra via the zero section, i.e., the augmentation $\co_{\bH_T} \to A$. Then, the action of $\gr(\hat{\cd}_{\ld{T}}^{\bH})$ on $A[x^{\pm 1}]\pw{\ol{t}}$ is induced by the augmentation $\hat{\co}_{\bH_T} \to A$. The isomorphism $\hat{\bH}_T \cong \hat{\AA}^1$ of formal $A$-schemes then implies an isomorphism $A \otimes_{\co_{\bH_T}} A \cong A[\epsilon]/\epsilon^2$ with $\epsilon$ in homological degree $1$.
    It follows that
    $$A\pw{\ol{t}}[x^{\pm 1}] \otimes_{\gr(\hat{\cd}_{\ld{T}}^{\bH})} A\pw{\ol{t}}[x^{\pm 1}] \simeq A\pw{\ol{t}}[x^{\pm 1}][\epsilon]/\epsilon^2,$$
    where $\ol{t}$ is in weight $1$ and degree $0$, and $\epsilon$ is in weight $0$ and degree $1$.
    
    By \cref{cartier-duality}, the $t$-adic filtration on $\hat{\cd}_{\ld{T}}^{\bH}$ is equivalent to the data of a $\Cart(\hat{\bH})$-action on $A\pw{\ol{t}}[x^{\pm 1}] \otimes_{\gr(\hat{\cd}_{\ld{T}}^{\bH})} A\pw{\ol{t}}[x^{\pm 1}] \simeq A\pw{\ol{t}}[x^{\pm 1}][\epsilon]/\epsilon^2$. This in turn is equivalent to the data of a differential 
    $$\nabla: A\pw{\ol{t}}[x^{\pm 1}] \to A\pw{\ol{t}}[x^{\pm 1}]\cdot \epsilon$$
    satisfying an $\hat{\bH}$-analogue of the Leibniz rule: if\footnote{Note that $\nabla$ has to be homogeneous in the degree of the monomial in $x$, as can be seen by keeping track of the $x$-weight.} $\nabla(x^n) = f(n) x^n \epsilon$ for some $f(n)\in A\pw{t}$, then $f(n+m) = f(n) +_\GG f(m)$.
    It therefore suffices to determine $\nabla(x)$; but the relation $yx = x(y +_\GG t)$ forces $\nabla(x) = tx\epsilon$. This implies that 
    $$\nabla(x^n) = (\overbrace{t +_\GG \cdots +_\GG t}^n) x^n \epsilon = [n]_F x^n \epsilon,$$
    as desired.
\end{proof}
\begin{example}
    When $\bH = {\GG}_a$ over\footnote{Of course, one can work over $\Z$ too; we just chose $\QQ$ to continue with \cref{ex: ordinary quantized diffop}.} $\QQ$, the complex $C^\bull$ is
    $$C^\bull = (\QQ\pw{\hbar}[x^{\pm 1}] \to \QQ\pw{\hbar}[x^{\pm 1}]dx), \ x^n \mapsto n\hbar x^n dx.$$
    Indeed, since $yx = x(y+\hbar)$, we have $yx^n = x^n(y+n\hbar)$; since $t = \hbar$ in this case, we have $x^n\mapsto n\hbar x^n \epsilon$. This is evidently a $\hbar$-rescaling of the classical de Rham complex of $\GG_m$.

    When $\bH = \GG_m$ over $\Z$, the complex $C^\bull$ is
    $$C^\bull = (\Z\pw{q-1}[x^{\pm 1}] \to \Z\pw{q-1}[x^{\pm 1}]dx), \ x^n \mapsto (q^n-1) x^n dx.$$
    Indeed, since $yx = x(qy)$, we have $yx^n = x^n (q^n y)$, and hence 
    $$(y-1)x^n = x^n(q^n y - 1) = x^n((y-1) +_F (q^n-1)),$$
    where $F(z,w) = z + w + zw$ is the multiplicative formal group law; since $t = q-1$ in this case, we have $x^n \mapsto (q^n-1) x^n \epsilon$. The complex $C^\bull$ is a $(q-1)$-rescaling of the $q$-de Rham complex of $\GG_m$ from \cite{scholze-q-def}.
\end{example}
\begin{remark}
    The complex of \cref{prop: endomorphism F-de Rham} is not quite the $F$-de Rham complex of \cite[Definition 4.3.6]{generalized-n-series}; rather, if $\eta_t$ denotes the d\'ecalage functor of \cite{berthelot-ogus} with respect to the ideal $(t)\subseteq A\pw{t}$, the $F$-de Rham complex is given by the d\'ecalage $\eta_t C^\bull$. In particular, the complex of \cref{prop: endomorphism F-de Rham} is isomorphic to the $F$-de Rham complex after inverting $t$. One can modify the algebra $\cd_{\ld{T}}^{\bH}$ of \cref{def: G-diff ops} (by performing a noncommutative analogue of an affine blowup/deformation to the normal cone\footnote{For instance, in the case of \cref{ex: ordinary quantized diffop}, this procedure simply adjoins the fraction $\frac{y}{\hbar}$; in the case of \cref{ex: q quantized diffop}, this procedure simply adjoins the fraction $\frac{y-1}{q-1}$.}) such that the relative tensor product as in \cref{prop: endomorphism F-de Rham} is the $F$-de Rham complex itself. Since it is not needed for this article, we will not describe this modification here.
\end{remark}
\begin{remark}\label{rmk: koszul duality LT}
    Suppose $k$ is a complex-oriented $2$-periodic $\Eoo$-ring equipped with an oriented commutative $k$-group scheme $\GG$.
    \cref{prop: endomorphism F-de Rham} says that $\hat{\cd}_{\ld{T}}^{\GG_0}$ is Koszul dual to the complex $C^\bull$. Forthcoming work of Arpon Raksit shows that the d\'ecalage $\eta_t C^\bull$ can be recovered from the ``even filtration'' (in the sense of \cite{even-filtr}) on the periodic cyclic homology $\HP(\tau_{\geq 0} k[x^{\pm 1}]/\tau_{\geq 0} k)$. See also the discussion in \cite[Section 3.3]{thh-xn}.
    Using similar techniques, one can show that $C^\bull$ can be recovered from the even filtration on the negative cyclic homology $\HC^-(k[x^{\pm 1}]/k) = \HH(k[x^{\pm 1}]/k)^{hS^1}$.

    Recalling that $T = S^1$, this $\Eoo$-$k$-algebra is simply $\HC^-(k[\Omega T]/k)$. The Hochschild homology $\HH(k[\Omega T]/k) \simeq k \otimes \THH(S[\Omega T])$ is $S^1$-equivariantly equivalent to the $k$-chains $C_\ast(\cL T; k)$ on the free loop space of $T$. (For a reference, see \cite[Corollary IV.3.3]{nikolaus-scholze}.) The $k$-chains $k[\cL T]$ is $S^1$-equivariantly Koszul dual\footnote{This Koszul duality essentially stems from the (non-$S^1_\rot$-equivariant) decomposition $\cL T \simeq T \times \Omega T$.} to $k[\Omega T]^{hT}$; this can be identified as a completion of $\cf_T(\Omega T)^\vee$ at the zero section of $\cM_T$. In other words, $\HC^-(k[\Omega T]/k)$ is Koszul dual to the completion of $\cf_{T\times S^1_\rot}(\Omega T)^\vee$ at the zero section of $\cM_T \times \GG$. This is the topological source of the Koszul duality of \cref{prop: endomorphism F-de Rham}.
\end{remark}
\newpage

\section{Review of the classical case}\label{sec: review Q coeff}
To prepare ourselves for the calculations of $\pi_0(\cf_T(\Gr_G)^\vee)$ for $k$ being complex K-theory or elliptic cohomology, we begin with the simpler case of $k$ being $\QQ[u^{\pm 1}]$ with $u$ in degree $2$; recall that $\cM_{T,0}$ is then isomorphic to $\fr{t}$. In this case, the discussion in the present section follows from the work of Bezrukavnikov, Finkelberg, and Mirkovic in \cite{bfm}, as well as the work of Yun and Zhu in \cite{homology-langlands}. We will nevertheless go through this calculation (and discuss several applications) since the argument is different from that of the papers mentioned above, and also because it will serve as a useful template later. Our goal is specifically to \textit{not} appeal to the derived geometric Satake equivalence of \cite{bf-derived-satake}, but rather do the calculation in such a way that proof technique generalizes to the K-theoretic or elliptic setting, so as to apply it to prove an analogue of \cite{bf-derived-satake}.

\textit{In the remainder of this article, we will assume the group $G$ is connected, almost simple, and simply-laced.} The assumption that $G$ is simply-laced provides many simplifications; in particular, it implies that the Chevalley split forms of the groups $G$ and $\ld{G}$ are centrally isogenous (so that the adjoint action of $G$ on $\g$ descends to an action of $\ld{G}$ on $\g$), and that there is a $\ld{G}$-equivariant isomorphism $\g \cong \ld{\g}^\ast$ (even over $\Z$). However, we will \textit{never} use an $\ld{G}$-equivariant isomorphism $\ld{\g} \cong \ld{\g}^\ast$! The latter fails over $\Z$ (e.g., $\sl_2 \not\cong \pgl_2$ over $\Z$), and the effect of reliance on such failures becomes amplified in the settings of K-theory and elliptic cohomology.

In the following discussion, all dual groups are to be understood as defined over $\QQ$ (although some of our discussion will work even over $\Z$, perhaps with some small primes inverted). 

\begin{definition}[(Additive) Kostant slice]\label{def: additive kostant slice}
    Fix a nondegenerate character $\psi \in \ld{\fr{n}}^\ast$; under the isomorphism $\ld{\g}^\ast \cong \g$, there is an isomorphism $\ld{\fr{n}}^\ast \cong \fr{n}$, and $\psi$ corresponds to a principal nilpotent element $f\in \fr{n}$. Let $(e,f,h)$ be the associated $\sl_2$-triple in $\g$, and let $\psi_-: \ld{\fr{n}}_- \to \AA^1$ denote the element corresponding to $e$. Let $\ld{\g}^{\ast,\psi_-} \cong \g^e$ denote the centralizer (so $\g = \g^e \oplus [e,\g]$), and let $\cS := f + \g^e\subseteq \g^\reg$ be the Kostant slice. Note that $\cS \cong \psi + \ld{\g}^{\ast,\psi_-} \subseteq \ld{\g}^{\ast,\reg}$. The composite $f + \g^e \to \g \to \g\mmod G \cong \fr{t}\mmod W$ is an isomorphism, by \cite{kostant-lie-group-reps}.
    
    Recall that the Grothendieck-Springer resolution is defined as
    $$\tilde{\ld{\g}} = \ld{\fr{n}}^\perp \times^\ld{B} \ld{G} \cong \fr{b} \times^{\ld{B}} \ld{G},$$
    so that $\tilde{\ld{\g}}/\ld{G} \simeq \fr{b}/\ld{B}$. A point of $\tilde{\ld{\g}}$ can be regarded as a pair $(\ld{\fr{b}}', x\in (\ld{\fr{n}}')^\perp)$; here, $\ld{\fr{b}}'$ denotes a Borel subalgebra of $\ld{\g}$, and $\ld{\fr{n}}'$ denotes its nilpotent radical.
    There is a map $\tilde{\chi}: \tilde{\ld{\g}} \to \fr{t}$ which sends a pair $(\ld{\fr{b}}', x)$ to the image of $x$ modulo $(\ld{\fr{b}}')^\perp$. Let $\tilde{\cS}$ denote the fiber product $\cS\times_{\ld{\g}^\ast} \tilde{\ld{\g}}$, so that 
    $$\tilde{\cS} \subseteq \tilde{\ld{\g}}^\reg = \ld{\g}^{\ast, \reg} \times_{\ld{\g}^\ast} \tilde{\ld{\g}}.$$
    Then, Kostant's result on the Kostant slice implies formally that the composite 
    $$\tilde{\cS} \to \tilde{\ld{\g}} \xar{\tilde{\chi}} \fr{t}$$
    is an isomorphism. We will often abusively write the inclusion of $\tilde{\cS}$ as a map $\kappa: \fr{t} \to \tilde{\ld{\g}}$.

    In fact, we will only care about the composite $\fr{t} \to \tilde{\ld{\g}} \to \tilde{\ld{\g}}/\ld{G}$ below, so we will also denote it by $\kappa$. If we identify $\tilde{\ld{\g}}/\ld{G} \cong \fr{b}/\ld{B}$, then the map $\kappa$ admits a simple description: it is the composite $f + \fr{t} \to \fr{b} \to \fr{b}/\ld{B}$. (See \cref{prop: psi + t}.) In our discussion below, we will often identify $f + \fr{t}$ with $\psi + \ld{\fr{t}}^\ast$. 
\end{definition}
\begin{definition}
    The stabilizer (inside $\ld{G}$) of the Kostant slice $\cS \subseteq \g^\reg$ is a closed subgroup scheme of the constant group scheme $\ld{G} \times \cS$, and will be denoted by $\ld{J}$. It will be called the \textit{regular centralizer group scheme}; if we wish to emphasize the dependence on $G$, we will denote it by $\ld{J}(G)$.  Note that since the composite $\cS \to \g^\reg \to \g\mmod \ld{G}$ is an isomorphism, we may identify
    $$\ld{J} \cong \cS \times_{\g/\ld{G}} \cS.$$
    Similarly, the stabilizer (inside $\ld{G}$) of the Kostant slice $\tilde{\cS} \subseteq \tilde{\ld{\g}}^\reg$ is a closed subgroup scheme of the constant group scheme $\ld{G} \times \tilde{\cS}$, and will be denoted by $\tilde{\ld{J}}$. Since $\tilde{\cS} \cong \cS\times_{\ld{\g}^\ast} \tilde{\ld{\g}}$, we may identify
    $$\tilde{\ld{J}} \cong \ld{J} \times_{\cS} \tilde{\cS} \cong (f + \fr{t}) \times_{\fr{b}/\ld{B}} (f + \fr{t}).$$
\end{definition}
\begin{theorem}\label{thm: ordinary hmlgy reg centr}
    There is an isomorphism of group schemes over $f + \fr{t} \cong \fr{t} \cong \cM_{T,0}$:
    $$\spec \pi_0 \cf_T(\Gr_G)^\vee \cong (f + \fr{t}) \times_{\fr{b}/\ld{B}} (f + \fr{t}).$$
\end{theorem}
\cref{thm: ordinary hmlgy reg centr} can be proved directly using \cref{prop: hmlgy gkm}, but I find the discussion below more enlightening (of course, it is essentially an elaboration of the application of \cref{prop: hmlgy gkm}).
We first need a few lemmas.
\begin{lemma}\label{lem: kappa for borel is flat}
    The projection map $\tilde{\ld{J}} \to \psi + \ld{\fr{t}}^\ast$ (onto either factor) is flat.
\end{lemma}
\begin{proof}
    For this, we will follow \cite[Step II]{homology-langlands}. 
    Consider the morphism $\ld{B} \times \ld{\fr{t}}^\ast \to \ld{\fr{n}}^{\perp}$ sending $(g, x)\mapsto \Ad_g(\psi + x) - \psi$. Unwinding definitions shows that there is a Cartesian square
    $$\xymatrix{
    \tilde{\ld{J}} \ar[r] \ar[d] & \ld{\fr{t}}^\ast \ar[d] \\
    \ld{B} \times \ld{\fr{t}}^\ast \ar[r] & \ld{\fr{n}}^{\perp},
    }$$
    so $\tilde{\ld{J}}$ is a closed subscheme of $\ld{B} \times \ld{\fr{t}}^\ast$ of codimension $\dim(\ld{\fr{b}}^{\perp}) = \dim(\ld{N})$. This means that the fibers of the map $\tilde{\ld{J}} \to \ld{\fr{t}}^\ast$ are have dimension at least $\dim(\ld{B}) - \dim(\ld{N}) = \rank(\ld{G})$. If all fibers had dimension \textit{exactly} $\rank(\ld{G})$, then miracle flatness would imply that the map $\tilde{\ld{J}} \to \ld{\fr{t}}^\ast$ is flat. To show that all fibers have dimension $\rank(\ld{G})$, observe that there is a contracting $\GG_m$-action on the vector space $\ld{\fr{t}}^\ast$ which pushes everything down to the origin; so it suffices to show that the fiber over $0 \in \ld{\fr{t}}^\ast$ is of the correct dimension.
    
    That is, we need to see that the scheme 
    $$Y := \{(g,x)\in \ld{B} \times \ld{\fr{t}}^\ast | \Ad_g(\psi) = \psi + x\}$$
    is $\rank(\ld{G})$-dimensional. First, observe that if $\Ad_g(\psi) = \psi + x \in \ld{\fr{n}}^{\perp}$ with $x\in \ld{\fr{t}}^\ast$, then actually $x = 0$. This is because the image of $x$ under the map 
    $$\ld{\fr{n}}^{\perp} \to (\ld{\fr{n}} \oplus \ld{\fr{n}}^-)^\perp \cong \ld{\fr{t}}^\ast$$
    is the same as the image of $\psi + x$, which is the same as the image of $\Ad_g(\psi)$. But the above map $\ld{\fr{n}}^{\perp} \to \ld{\fr{t}}^\ast$ is $\Ad$-invariant, and so the image of $\Ad_g(\psi)$ is equal to the image of $\psi$, which is zero. This means that the image of $x$ is also zero. But the inclusion $\ld{\fr{t}}^\ast \subseteq \ld{\fr{n}}^{\perp}$ splits the map $\ld{\fr{n}}^{\perp} \to \ld{\fr{t}}^\ast$, and so we see that $x = 0$. Therefore,
    $$Y \cong \{g\in \ld{B} | \Ad_g(\psi) = \psi\} = Z_{\ld{B}}(\psi).$$
    The centralizer of $\psi$ is contained entirely in $\ld{B}$, so $Z_{\ld{B}}(\psi) \cong Z_{\ld{G}}(\psi)$. This, in turn, has dimension given by $\rank(\ld{G})$ since $\psi$ (corresponding to $e \in \g$) is a regular nilpotent.
\end{proof}
Note that
$$\tilde{\ld{J}} \cong \{(x,y,g) \in \ld{\fr{t}}^\ast \times \ld{\fr{t}}^\ast \times \ld{B} | \Ad_g(\psi + x) = \psi + y\}.$$
The argument at the end of \cref{lem: kappa for borel is flat} allows us to identify $x = y\in \ld{\fr{t}}^\ast$, and so
$$\tilde{\ld{J}} \cong \{(x,g) \in \ld{\fr{t}}^\ast \times \ld{B} | \Ad_g(\psi + x) = \psi + x\}.$$

\begin{notation}
    If $\alpha$ is a root of $\ld{G}$, let $\{e_\alpha, h_\alpha\}$ denote a pinning of $\ld{G}$. Say that a point $x \in \ld{\fr{t}}^\ast$ is \textit{$\alpha$-generic} if $x(h_\beta) \neq 0$ for all roots $\beta\neq \alpha$. This implies that the centralizer $Z_{\ld{G}}(x)$ has semisimple rank at most $1$. Let $\ld{\fr{t}}^\ast_{\alpha\dreg}$ denote the $\alpha$-regular locus. Observe that $\ld{\fr{t}}^\ast_\reg = \bigcup_{\alpha\in \Phi} \ld{\fr{t}}^\ast_{\alpha\dreg} \subseteq \ld{\fr{t}}^\ast$ is open, with complement of codimension $2$.
\end{notation}
\begin{lemma}\label{lem: localization of ordinary J}
    There is an isomorphism
    \begin{equation}\label{eq: J of centralizer}
        \tilde{\ld{J}}(\ld{G})|_{\ld{\fr{t}}^\ast_{\alpha\dreg}} \xar{\sim} \tilde{\ld{J}}(Z_{\ld{G}}(x)^\circ)|_{\ld{\fr{t}}^\ast_{\alpha\dreg}},
    \end{equation}
    where $Z_{\ld{G}}(x)$ is the centralizer of some $x\in \ld{\fr{t}}^\ast_{\alpha\dreg}$ which lies on the $\alpha$-hyperplane, and $Z_{\ld{G}}(x)^\circ$ denotes the connected component of the identity. 
\end{lemma}
\begin{proof}[Proof sketch]
    Let us, for simplicity, write $\ld{H}$ to denote $Z_{\ld{G}}(x)^\circ$.
    There is a map from the left-hand side to the right-hand side, which sends 
    $$\ld{\fr{t}}^\ast \times \ld{B} \ni (x, g)\mapsto (x, g)\in \ld{\fr{t}}^\ast \times (\ld{B} \cap \ld{H}).$$
    Note that $\ld{B} \cap \ld{H}$ is a Borel subgroup of $\ld{H}$.
    To see that the above map gives an isomorphism, observe that if $y\in \ld{\fr{t}}^\ast$, we may identify the centralizer in $\ld{G}$ of $\psi + y$ with the centralizer in $Z_{\ld{G}}(y)^\circ$ of $\psi$.  That \cref{eq: J of centralizer} is an isomorphism is now a consequence of the observation that if $y\in \ld{\fr{t}}^\ast_{\alpha\dreg}$, then this centralizer $Z_{\ld{G}}(y)^\circ$ is contained in $\ld{H}$. That is, if $(x,g)\in \tilde{\ld{J}}(\ld{G})|_{\ld{\fr{t}}^\ast_{\alpha\dreg}}$, then $g$ is already contained in $H$, and so $(x,g)\in \tilde{\ld{J}}(\ld{H})|_{\ld{\fr{t}}^\ast_{\alpha\dreg}}$.
\end{proof}
\begin{proof}[Proof of \cref{thm: ordinary hmlgy reg centr}]
    We begin by noting that $\Gr_G$ only has even cells; so $\pi_0 \cf_T(\Gr_G)^\vee = \pi_0 C_\ast^T(\Gr_G; \QQ[u^{\pm 1}])$ can be identified with $\H_\ast^T(\Gr_G; \QQ)$, regarded now as an ungraded $\QQ$-algebra. Similarly, $\pi_0(k_T) \cong \H^\ast_T(\ast; \QQ)$, again regarded as an ungraded $\QQ$-algebra.
    The equivariant formality of $\Gr_G$ implies that $\H_\ast^T(\Gr_G; \QQ)$ is flat over $\H^\ast_T(\ast; \QQ)$. To prove \cref{thm: ordinary hmlgy reg centr}, it therefore suffices to prove an isomorphism 
    $$\tilde{\ld{J}}|_{\ld{\fr{t}}^\ast_{\alpha\dreg}} \cong \spec \H^{T_c}_\ast(\Omega G; \QQ)|_{\ld{\fr{t}}^\ast_{\alpha\dreg}}$$
    for each root $\alpha$. By Atiyah-Bott localization, the right-hand side can be identified with 
    \begin{equation}\label{eq: atiyah bott for gr}
        \spec \H^{T_c}_\ast(\Omega G; \QQ)|_{\ld{\fr{t}}^\ast_{\alpha\dreg}} \cong \spec \H^{T_c}_\ast(\Omega Z_G(x); \QQ)|_{\ld{\fr{t}}^\ast_{\alpha\dreg}},
    \end{equation}
    where $Z_G(x)$ is the centralizer of some $x\in \fr{t}_{\alpha\dreg}$ which lies on the $\alpha$-hyperplane. 
    Note that the right-hand side depends only on the connected component $Z_G(x)^\circ$ of the identity in $Z_G(x)$; so we might as well replace $Z_G(x)$ by $Z_G(x)^\circ$. Using \cref{lem: localization of ordinary J}, we are therefore reduced to showing that there is an isomorphism
    $$\H^{T_c}_\ast(\Omega Z_G(x)^\circ; \QQ)|_{\ld{\fr{t}}^\ast_{\alpha\dreg}} \cong \tilde{\ld{J}}(Z_{\ld{G}}(x)^\circ)|_{\ld{\fr{t}}^\ast_{\alpha\dreg}}.$$
    Since $Z_G(x)^\circ$ has semisimple rank $1$, we are reduced to checking that \cref{thm: ordinary hmlgy reg centr} holds in this case.

    That is, we may assume $G$ is the product of a torus with $\GL_2$, $\SL_2$, or $\PGL_2$. It is easy to match up the contribution from the toral factors, so we will assume that $G$ is $\GL_2$, $\SL_2$, or $\PGL_2$.
    \begin{itemize}
        \item For $\GL_2$, we may identify $\gl_2^\ast \cong \gl_2$. Then, $\tilde{\ld{J}}$ is the centralizer (in $\ld{B}$) of $\begin{psmallmatrix}
            x & 0 \\
            1 & y
        \end{psmallmatrix}$. It is not hard to compute directly that $\begin{psmallmatrix}
            a & 0\\
            c & d
        \end{psmallmatrix}$ stabilizes $\begin{psmallmatrix}
            x & 0 \\
            1 & y
        \end{psmallmatrix}$ if and only if $c = \tfrac{a-d}{x-y}$, meaning that
        $$\tilde{\ld{J}} \cong \spec \QQ[x,y, a^{\pm 1}, d^{\pm 1}, \tfrac{a-d}{x-y}].$$
        The coproduct sends $a\mapsto a \otimes a$ and $d\mapsto d \otimes d$.
        
        Let us now calculate $\H^{T^2}_\ast(\Omega \GL_2; \QQ)$ as an algebra over $\H^\ast_{T^2}(\ast; \QQ) \cong \QQ[x, y]$. There is a simple $T^2$-equivariant cell decomposition of $\Omega \GL_2$ with $\bX_\ast(T^2) = \Z^2$ many $0$-cells, and where there is a $T^2$-equivariant $1$-cell connecting $\mu_1$ to $\mu_2$ if and only if $\mu_1 - \mu_2$ is a multiple of a root of $\GL_2$. (There are higher equivariant cells, but they will not matter.) This implies, by Atiyah-Bott localization, that the fixed points of the $T^2$-action on $\Omega \GL_2$ are simply $\Omega T^2 = \Z^2$, and so
        $$\H^{T^2}_\ast(\Omega \GL_2; \QQ)[\tfrac{1}{x-y}] \cong \H^{T^2}_\ast(\Omega T^2; \QQ)[\tfrac{1}{x-y}] \cong \QQ[x, y, \tfrac{1}{x-y}, a^{\pm 1}, d^{\pm 1}].$$
        On the other hand, the \textit{completion} $\H^{T^2}_\ast(\Omega \GL_2; \QQ)^\wedge_{(x-y)}$ can be determined directly. After completing at $(x-y, y) = (x,y)$, the equivariant homology $\H^{T^2}_\ast(\Omega \GL_2; \QQ)$ simply becomes the \textit{Borel-equivariant} homology, and this can be computed directly via a spectral sequence
        $$E_2 = \H^\ast(BT^2; \QQ) \otimes_k \H_\ast(\Omega \GL_2; \QQ) \Rightarrow \H^{T^2}_\ast(\Omega \GL_2; \QQ)^\wedge_{(x,y)}.$$
        Since $\H_\ast(\Omega \GL_2; \QQ) = \QQ[A^{\pm 1}, b]$ with $A$ in weight $0$ and $b$ in weight $2$, the $E_2$-page of this spectral sequence is concentrated entirely in even degrees, and hence collapses. This means that 
        $$\H^{T^2}_\ast(\Omega \GL_2; \QQ)^\wedge_{(x,y)} \cong k\pw{x,y}[A^{\pm 1}, b].$$
        This in fact comes from an isomorphism
        $$\H^{T^2}_\ast(\Omega \GL_2; \QQ)^\wedge_{(x-y)} \cong \QQ[x, y, A^{\pm 1}, b]^\wedge_{(x-y)}.$$
        We may therefore recover $\H^{T^2}_\ast(\Omega \GL_2; \QQ)$ via the gluing square
        $$\xymatrix{
        \H^{T^2}_\ast(\Omega \GL_2; \QQ) \ar[r] \ar[d] & \H^{T^2}_\ast(\Omega \GL_2; \QQ)[\tfrac{1}{x-y}] \ar[d] \\
        \H^{T^2}_\ast(\Omega \GL_2; \QQ)^\wedge_{(x-y)} \ar[r] & \H^{T^2}_\ast(\Omega \GL_2; \QQ)^\wedge_{(x-y)}[\tfrac{1}{x-y}].
        }$$
        Explicitly:
        $$\xymatrix{
        \H^{T^2}_\ast(\Omega \GL_2; \QQ) \ar[r] \ar[d] & \QQ[x, y, \tfrac{1}{x-y}, a^{\pm 1}, d^{\pm 1}]\ar[d] \\
        \QQ[x, y, A^{\pm 1}, b]^\wedge_{(x-y)} \ar[r] & \QQ[x, y, A^{\pm 1}, b]^\wedge_{(x-y)}[\tfrac{1}{x-y}].
        }$$
        The right vertical map sends $a-d \mapsto b(x-y)$; and $d \mapsto A$. Note that $b(x-y)$ is topologically nilpotent, so $A + b(x-y)$ is a unit, and this is what $a$ maps to. This discussion implies that the fiber product above identifies with
        $$\H^{T^2}_\ast(\Omega \GL_2; \QQ) \cong \QQ[x, y, a^{\pm 1}, d^{\pm 1}, \tfrac{a-d}{x-y}].$$
        We need to determine the coproduct. Since this ring is flat over $\QQ[x,y]$, it suffices to determine the coproduct after inverting $x-y$. As we have seen, $\H^{T^2}_\ast(\Omega \GL_2; \QQ)[\tfrac{1}{x-y}] \cong \H^{T^2}_\ast(\Omega T^2; \QQ)[\tfrac{1}{x-y}]$, and $\Omega T^2 = \Z^2$. The coproduct here simply comes from the \textit{diagonal} on $\Z^2$, which obviously sends $a\mapsto a \otimes a$ and $d\mapsto d \otimes d$. It follows that 
        $$\spec \H^{T^2}_\ast(\Omega \GL_2; \QQ) \cong \tilde{\ld{J}}$$
        as (graded) group schemes over $\QQ[x,y]$, as desired.
        \item For $G = \SL_2$, one can similarly calculate that
        $$\H^{S^1}_\ast(\Gr_{\SL_2}; \QQ) \cong \QQ[x, a^{\pm 1}, b]/(a = 1 + 2xb) \cong \QQ[x, a^{\pm 1}, \tfrac{a-1}{2x}].$$
        The coproduct is determined by the formula $a\mapsto a \otimes a$, so that
        $$b\mapsto b \otimes 1 + 1\otimes b + 2x b \otimes b.$$
        For completeness, let us quickly summarize the argument. The fixed points of $S^1$ acting on $\Omega \SU(2)$ is $\Omega S^1 = \Z$, and the action of $S^1$ on $\SU(2) \cong S^3$ exhibits it as the one-point compactification of $\RR\oplus \cc$, where $\RR$ is the trivial representation and $\cc$ is the \textit{weight $2$} representation. Therefore, inverting the Chern class $2x$ of the weight $2$ representation lets us identify
        $$\H^{S^1}_\ast(\Gr_{\SL_2}; \QQ)[\tfrac{1}{2x}] \cong \H^{S^1}_\ast(\Omega S^1; \QQ)[\tfrac{1}{2x}] \cong \QQ[x^{\pm 1}, a^{\pm 1}].$$
        On the other hand, the completion of $\H^{S^1}_\ast(\Gr_{\SL_2}; \QQ)$ at the class $2x$ is, via the same spectral sequence argument as in the preceding bullet point, given by
        $$\H^{S^1}_\ast(\Gr_{\SL_2}; \QQ)^\wedge_{(2x)} \cong \QQ\pw{x}[b],$$
        with $b$ in weight $2$. The ring $\H^{S^1}_\ast(\Gr_{\SL_2}; \QQ)$ can be recovered via the gluing square
        $$\xymatrix{
        \H^{S^1}_\ast(\Gr_{\SL_2}; \QQ) \ar[r] \ar[d] & \H^{S^1}_\ast(\Gr_{\SL_2}; \QQ)[\tfrac{1}{2x}] \ar[d] \\
        \H^{S^1}_\ast(\Gr_{\SL_2}; \QQ)^\wedge_{(2x)} \ar[r] & \H^{S^1}_\ast(\Gr_{\SL_2}; \QQ)^\wedge_{(2x)}[\tfrac{1}{2x}].
        }$$
        The right vertical map sends $a-1 \mapsto b \cdot 2x$, and so the above Cartesian square gives an isomorphism
        $$\H^{S^1}_\ast(\Gr_{\SL_2}; \QQ) \cong \QQ[x, a^{\pm 1}, b]/(a = 1 + 2xb),$$
        as desired.
        
        On the other hand, $\tilde{\ld{J}}$ is the centralizer in $\ld{B}\subseteq \PGL_2$ of $\begin{psmallmatrix}
            x & 0\\
            1 & -x
        \end{psmallmatrix} \subseteq \fr{sl}_2 \cong \ld{\g}^\ast$. It is easy to compute directly that $\begin{psmallmatrix}
            a & 0 \\
            c & 1
        \end{psmallmatrix} \in \ld{B}$ (where we only care about this as an element of $\PGL_2$!) stabilizes $\begin{psmallmatrix}
            x & 0\\
            1 & -x
        \end{psmallmatrix}$ if and only if $2xc = a-1$. Therefore, 
        $$\tilde{\ld{J}} \cong \spec \QQ[x, a^{\pm 1}, c]/(a = 1 + 2xc),$$
        and again the group law is determined by the formulae
        $$a\mapsto a \otimes a, \ c \mapsto c \otimes 1 + 1\otimes c + 2x c \otimes c.$$
        Therefore, 
        $$\spec \H^{S^1}_\ast(\Gr_{\SL_2}; \QQ) \cong \tilde{\ld{J}}$$
        as (graded) group schemes over $\QQ[x]$, as desired.
        \item In exactly the same way, for $G = \PGL_2$, one can similarly calculate that
        $$\H^{S^1}_\ast(\Omega \PGL_2; \QQ) \cong \QQ[x, a^{\pm 1}, b]/(a^2 = 1 + xb) \cong \QQ[x, a^{\pm 1}, \tfrac{a^2-1}{x}].$$
        This is because the fixed points of $S^1$ acting on $\Omega \PGL_2 \simeq \Z/2 \times \Omega S^3$ is $\Z$, and the action of $S^1$ on $\PGL_2$, which is homotopy equivalent to $\RP^3$, exhibits it as the $\Z/2$-quotient of the one-point compactification of $\RR\oplus \cc$, where $\RR$ is the trivial representation and $\cc$ is the \textit{weight $1$} representation. The coproduct is determined by the formula $a\mapsto a \otimes a$, so that
        $$b\mapsto b \otimes 1 + 1\otimes b + x b \otimes b.$$

        On the other hand, $\tilde{\ld{J}}$ is the centralizer in $\ld{B}\subseteq \SL_2$ of the equivalence class of $\begin{psmallmatrix}
            x & 0\\
            1 & 0
        \end{psmallmatrix}$ in $\pgl_2 \cong \ld{\g}^\ast$. It is easy to compute directly that $\begin{psmallmatrix}
            a & 0 \\
            c & a^{-1}
        \end{psmallmatrix} \in \ld{B}$ stabilizes $\begin{psmallmatrix}
            x & 0\\
            1 & 0
        \end{psmallmatrix}$ if and only if $xc = a-a^{-1}$. Therefore, 
        $$\tilde{\ld{J}} \cong \spec \QQ[x, a^{\pm 1}, c]/(a = a^{-1} + xc) \cong \spec \QQ[x, a^{\pm 1}, \tfrac{a - a^{-1}}{x}].$$
        Replacing $c$ by $b := ca^{-1}$, we see that the group law is determined by the formulae
        $$a\mapsto a \otimes a, \ b\mapsto b \otimes 1 + 1\otimes b + x b \otimes b.$$
        Therefore, 
        $$\spec \H^{S^1}_\ast(\Omega \PGL_2; \QQ) \cong \tilde{\ld{J}}$$
        as (graded) group schemes over $\QQ[x]$, as desired.\qedhere
    \end{itemize}
\end{proof}
\begin{remark}
    Just for posterity, let us record a more canonical variant of the calculation above for $\ld{G} = \SL_2$, which does not require picking a Borel subgroup (i.e., which does not involve identifying $\tilde{\ld{\g}}/\ld{G} \cong \fr{b}/\ld{B}$). For simplicity, we will use the fact that $2$ is invertible in $\QQ$ to identify $\fr{sl}_2 \cong \fr{pgl}_2$. In this case, the Grothendieck-Springer resolution $\tilde{\ld{\g}} = T^\ast(\AA^2-\{0\})/\GG_m$ is the total space of $\co(-1) \oplus \co(-1)$ over $\PP^1$; we will think of a point in $\tilde{\g}$ as a pair $(x\in \sl_2, \ell\subseteq \cc^2)$ such that $x$ preserves $\ell$. The Kostant slice $\kappa:\fr{t} \cong \AA^1 \to \tilde{\ld{\g}}$ is the map sending $\lambda \in \AA^1$ to the pair $(x, \ell)$ with $x = \begin{psmallmatrix}
    0 & \lambda^2 \\
    1 & 0
    \end{psmallmatrix}$ and $\ell = [\lambda: 1]$. Indeed, this is essentially immediate from the requirement that the following diagram commutes:
    $$\xymatrix{
    \AA^1 \cong \fr{t} \ar[r]^-\kappa \ar[d]_-{\lambda \mapsto \lambda^2} & \tilde{\sl}_2 \ar[d]\\
    \AA^1 \cong \fr{t}\mmod W \ar[r]^-\kappa_-{\lambda\mapsto \begin{psmallmatrix}
    0 & \lambda \\
    1 & 0
    \end{psmallmatrix}} & \sl_2.
    }$$
    Moreover, the $\SL_2$-action on $\tilde{\ld{\g}}$ sends $g\in \SL_2$ and $(x,\ell)$ to $(\Ad_g(x), g\ell)$. If $g = \begin{psmallmatrix}
    a & b \\
    c & d
    \end{psmallmatrix}$, we compute that
    $$\Ad_g\begin{pmatrix}
    0 & \lambda^2 \\
    1 & 0
    \end{pmatrix} = \begin{pmatrix}
    bd-ac\lambda^2 & (a\lambda)^2 - b^2 \\
    d^2 - (c\lambda)^2 & ac\lambda^2 - bd
    \end{pmatrix}, \ g\cdot [\lambda: 1] = [a\lambda + b: c\lambda + d].$$
    From this, we see that $\Ad_g(x) = x$ if and only if $a = d$ and $b = c\lambda^2$, in which case $g$ also fixes $[\lambda: 1]$. In other words, $g = \begin{psmallmatrix}
    a & c\lambda^2 \\
    c & a
    \end{psmallmatrix}$ with $a,c\in k$; in order for $\det(g) = 1$, we need $a^2-c^2\lambda^2=1$. When $\lambda \neq 0$, both $x$ and $g$ are diagonalized by the matrix $\tfrac{1}{2}\begin{psmallmatrix}
    1 & -1 \\
    -\lambda^{-1} & -\lambda^{-1}
    \end{psmallmatrix}\in \SL_2$: the diagonalization of $x$ is $\begin{psmallmatrix}
    \lambda & 0 \\
    0 & \lambda^{-1}
    \end{psmallmatrix}$, and the diagonalization of $g$ is $\begin{psmallmatrix}
    t & 0 \\
    0 & w
    \end{psmallmatrix}$ where $2a = t+w$ and $2\lambda c = t-w$. Since we have $\det(g) = a^2 - (c\lambda)^2 = 1$, we have $w = t^{-1}$. This shows that if $k$ is not of characteristic $2$, then $\fr{t} \times_{\tilde{\sl}_2/\SL_2} \fr{t} \cong \spec \QQ[\lambda, t^{\pm 1}, \tfrac{t-t^{-1}}{\lambda}]$.
\end{remark}
\begin{corollary}\label{cor: reg locus ordinary ABG}
    There is an equivalence
    $$\Loc_{T_c}^\gr(\Gr_G; k) \simeq \QCoh(\tilde{\ld{\g}}^\reg/\ld{G}).$$
    Furthermore, the pushforward functor $\Loc_{T_c}^\gr(\Gr_G; k) \to \Loc_{T_c}^\gr(\ast; k)$ identifies with the pullback functor $\kappa^\ast: \QCoh(\tilde{\ld{\g}}^\reg/\ld{G}) \to \QCoh(\fr{t})$.
\end{corollary}
\begin{proof}
    By definition, $\Loc_{T_c}^\gr(\Gr_G;k)$ is equivalent to the category of comodules over $\pi_0 \cf_T(\Gr_G)^\vee = \H_\ast^T(\Gr_G; \QQ)$ in the category of $\pi_0 k_T \cong \H^\ast_T(\ast; \QQ)$-modules. By \cref{thm: ordinary hmlgy reg centr}, it can be identified the category of quasicoherent sheaves on the quotient stack $(f + \fr{t})/\tilde{\ld{J}}$. As discussed after \cref{lem: kappa for borel is flat}, we may view $\tilde{\ld{J}}$ as a closed subgroup scheme of the constant group scheme $\ld{B} \times (f + \fr{t})$. This gives an isomorphism
    $$(f + \fr{t})/\tilde{\ld{J}} \cong \ld{B} \backslash (\ld{B} \times (f + \fr{t}))/\tilde{\ld{J}}.$$
    It follows from Kostant's work in \cite{kostant-lie-group-reps} that the $\ld{B}$-orbit of $f + \fr{t}$ inside $\fr{b}$ is precisely the regular locus $\fr{b}^\reg$. Since $\tilde{\ld{J}}$ is definitionally the stabilizer of $f + \fr{t} \subseteq \fr{b}$, the quotient $(\ld{B} \times (f + \fr{t}))/\tilde{\ld{J}}$ is isomorphic to $\fr{b}^\reg$; so there is an isomorphism $(f + \fr{t})/\tilde{\ld{J}} \cong \fr{b}^\reg/\ld{B}$.
    To finish, note that $\tilde{\ld{\g}}^\reg/\ld{G} \cong \fr{b}^\reg/\ld{B}$.
\end{proof}
The equivalence of \cref{cor: reg locus ordinary ABG} is in fact symmetric monoidal for the convolution tensor structure on $\Loc_{T_c}^\gr(\Gr_G; k)$ (described in \cref{rmk: loc gr convolution tensor}) and the standard tensor product on $\QCoh(\tilde{\ld{\g}}^\reg/\ld{G})$.

\begin{remark}
    Note that the definition of the Kostant slice $f + \fr{t} \subseteq \fr{b}$ involved the choice of a regular nilpotent element $f \in \g$. However, this choice does not materialize in \cref{cor: reg locus ordinary ABG}. This is because two such slices obtained by choosing two different regular nilpotent elements in $\g$ are \textit{conjugate} to each other (by $\ld{B}$). That is, while the specific inclusion $f + \fr{t} \subseteq \fr{b}$ depends on the choice of $f$, the composite $f + \fr{t} \subseteq \fr{b} \to \fr{b}/\ld{B}$ is independent of said choice. 
\end{remark}
\begin{example}
    Suppose $G = \SL_n$. In this case, $\H_\ast(\Gr_{\SL_n}; \QQ)$ is simply isomorphic to a polynomial algebra $\QQ[b_1, \cdots, b_{n-1}]$ on $n-1$ generators. The coproduct is given by $b_j \mapsto \sum_i b_i \otimes b_{j-i}$, where $b_0$ is understood to be $1$. This result  is classical, and can be found, for instance, in \cite{bott-space-of-loops}. The proof there amounts to the following observation. Consider the map $\CP^{n-1} \to \Gr_{\SL_n}$ given by sending $\ell \in \CP^{n-1}$ to (an appropriate rescaling of) the loop sending $\theta \in S^1$ to rotation by angle $\theta$ about the line $\ell$. Then the image of the induced map $\H_\ast(\CP^{n-1}; \QQ) \to \H_\ast(\Gr_{\SL_n}; \QQ)$ generates $\H_\ast(\Gr_{\SL_n}; \QQ)$; that is, $\CP^{n-1}$ is a generating complex for $\Gr_{\SL_n}$. The formula for the coproduct comes from the coproduct on $\H_\ast(\CP^{n-1}; \QQ)$, which is determined easily by the cup product on $\H^\ast(\CP^{n-1}; \QQ)$. The above description of $\H_\ast(\Gr_{\SL_n}; \QQ)$ implies that $\spec \H_\ast(\Gr_{\SL_n}; \QQ)$ is isomorphic to the group scheme $\WW_{n-1}$ of big Witt vectors of length $n-1$.

    On the other hand, \cref{thm: ordinary hmlgy reg centr} implies that $\spec \H_\ast(\Gr_{\SL_n}; \QQ)$ is isomorphic to the centralizer inside $\PGL_n$ of of the regular nilpotent $f \in \sl_n$. Indeed, if $R$ is a $\QQ$-algebra, then an element $g \in \GL_n(R)$ commutes with $f$ if and only if $g$ is an invertible polynomial in $e$. By the Cayley-Hamilton theorem, such a polynomial is divisible by the minimal polynomial $t^n$ of $e$; that is, $g \in (R[t]/t^n)^\times$. For this to live in $\PGL_n(R)$, we need to quotient out by the scalars $R^\times$. The assignment $R \mapsto (1 + tR[t]/t^n)^\times$ is precisely the functor of points of $\WW_{n-1}$. One can therefore understand the isomorphism between $\spec \H_\ast(\Gr_{\SL_n}; \QQ)$ and $Z_{\PGL_n}(e)$ as being a way to identify the two descriptions of the Witt vector group scheme (either via its functor of points, or via the explicit Witt addition law).
\end{example}
\begin{example}\label{ex: equiv homology Omega SUn and Witt}
    Continuing the preceding example (so $G = \SL_n$), it is not hard to add in torus-equivariance (so $T_c = (S^1)^{n-1}$). In this case, we will identify $\H^\ast_{T_c}(\ast; \QQ) \cong \QQ[x_1, \cdots, x_n]$. One can write down an explicit $T_c$-equivariant cell structure on $\Omega \SU(n)$ to find that $\spec \H_\ast^{T_c}(\Gr_{\SL_n}; \QQ)$ is isomorphic to the deformation of $\WW_{n-1}$ over $\spec \H^\ast_{T_c}(\ast; \QQ) \cong \AA^n$ which sends a $\QQ[x_1, \cdots, x_n]$-algebra $R$ to the group of units $(1 + tR[t]/(t-x_1)\cdots(t-x_n))^\times$. On the other hand, by the same argument using Cayley-Hamilton, the centralizer inside $\GL_n(R)$ of $f+x \in \sl_n$ is isomorphic to the group $(R[t]/(t-x_1)\cdots(t-x_n))^\times$, since the characteristic polynomial of $f+x$ is precisely $(t-x_1)\cdots(t-x_n)$. Quotienting by the scalars $R^\times$, we obtain an isomorphism between $\spec \H_\ast^{T_c}(\Gr_{\SL_n}; \QQ)$ and $Z_{\PGL_n}(f + x)$.
\end{example}
\begin{remark}
    For a general reductive group $G$, Kostant proved (in \cite{kostant-whittaker}) an isomorphism $(f + \fr{b})/\ld{N} \cong \fr{t}\mmod W$. In fact, the natural map $(f + \fr{b})/\ld{N} \to \g/\ld{G}$ identifies with the map $\cS \to \g/\ld{G}$ given by the Kostant slice. Since $(f + \fr{b})/\ld{N}$ is isomorphic to the quotient $\ld{G} \backslash T^\ast(\ld{G}/_\psi \ld{N})$ of the Whittaker reduction of $T^\ast(\ld{G})$, it follows that there are isomorphisms
    \begin{align*}
        \ld{J} & \cong (f + \fr{b})/\ld{N} \times_{\g/\ld{G}} (f + \fr{b})/\ld{N} \\
        & \cong \ld{G} \backslash T^\ast(\ld{G}/_\psi \ld{N}) \times_{\ld{\g}^\ast/\ld{G}} T^\ast(\ld{N} {}_\psi \backslash \ld{G})/\ld{G} \\
        & \cong T^\ast(\ld{N} {}_\psi \backslash \ld{G} /_\psi \ld{N}).
    \end{align*}
    That is, $\ld{J}$ can be identified with the \textit{bi-Whittaker reduction} of the cotangent bundle $T^\ast(\ld{G})$. In particular, \cref{thm: ordinary hmlgy reg centr} gives an isomorphism
    $$\spec \H^{T_c}_\ast(\Gr_G; \QQ) \cong \ld{\fr{t}}^\ast \times_{\ld{\fr{t}}^\ast\mmod W} T^\ast(\ld{N} {}_\psi \backslash \ld{G} /_\psi \ld{N}).$$
    In fact, this isomorphism can be checked to be $W$-equivariant (for the action of $W$ on $\H^{T_c}_\ast(\Gr_G; \QQ)$ via the action on $T$, and for the action on the right-hand side coming from the cover $\ld{\fr{t}}^\ast \to \ld{\fr{t}}^\ast\mmod W$). This implies that there is an isomorphism
    $$\spec \H^G_\ast(\Gr_G; \QQ) \cong T^\ast(\ld{N} {}_\psi \backslash \ld{G} /_\psi \ld{N}).$$
    This isomorphism has been exploited heavily in \cite{teleman-icm}, among others.
\end{remark}

The map $\tilde{\ld{\g}}^\reg/\ld{G} \to B\ld{G}$ defines a functor
\begin{equation}\label{eq: Rep G^ to ordinary loc}
    \Rep(\ld{G}) \to \QCoh(\tilde{\ld{\g}}^\reg/\ld{G}) \simeq \Loc_{T_c}^\gr(\Gr_G; k).
\end{equation}
More generally, the map $\tilde{\ld{\g}}^\reg/\ld{G} \to B\ld{T} \times B\ld{G}$ defines a functor
\begin{equation}\label{eq: Rep T^ x G^ to ordinary loc}
    \Rep(\ld{T} \times \ld{G}) \to \QCoh(\tilde{\ld{\g}}^\reg/\ld{G}) \simeq \Loc_{T_c}^\gr(\Gr_G; k).
\end{equation}
If $V \in \Rep(\ld{G})$, let $\cS_k(V)$ denote the corresponding object of $\Loc_{T_c}^\gr(\Gr_G; k)$. It is natural to ask whether $\cS_k(V) \in \Loc_{T_c}^\gr(\Gr_G; k)$ is given by $\cf^\gr$ for some $\cf \in \Loc_{T_c}(\Gr_G; k)$. Of course, \cref{cor: reg locus abg} says that the answer is yes; but it is not clear how to answer this question in a manner that will generalize to other $\Eoo$-rings $k$. However, it is possible to give a positive (and generalizable) answer to this question in the case when $V$ is a direct sum of tensor products of irreducible representations with \textit{minuscule} highest weights.
\begin{prop}\label{prop: ordinary realizing minuscule reps}
    Let $\lambda_\bull = (\lambda_1, \cdots, \lambda_n)$ be a tuple of dominant minuscule weights of $\ld{G}$, let $|\lambda_\bull| = \sum_i \lambda_i$, and let $\ol{\Gr_G^{\lambda_\bull}}$ denote the corresponding \textit{convolution variety} \cite{mirkovic-vilonen, ngo-polo}. Let $\cf_{\lambda_\bull}$ denote the pushforward of the constant sheaf along the canonical map $q: \ol{\Gr_G^{\lambda_\bull}} \to \ol{\Gr_G^{|\lambda|}} \subseteq \Gr_G$. If $V_{\lambda_i}$ denotes the irreducible representation of $\ld{G}$ with highest weight $\lambda_i$, then there is an isomorphism $\cS_k(\bigotimes_i V_{\lambda_i}) \cong \cf_{\lambda_\bull}^\gr$.
\end{prop}
\begin{proof}
    First, suppose that $\lambda_\bull = \lambda$ consists of single element. Let $P_\lambda \subseteq G$ denote the corresponding maximal parabolic subgroup, so that $\ol{\Gr_G^\lambda} \cong G/P_\lambda$, and let $\cf_\lambda \in \Loc_{T_c}(\Gr_G; k)$ denote the pushforward of the constant sheaf along the inclusion $G/P_\lambda \hookrightarrow \Gr_G$. We then need to show that there is an isomorphism $\cS_k(V_\lambda) \cong \cf_\lambda^\gr$.
    
    Since $V_\lambda$ is an $\ld{G}$-representation, the tensor product $V_\lambda \otimes_\QQ \co_{\fr{t}}$ is a comodule over $\co_{\ld{G} \times \fr{t}}$. In particular, it is a comodule over $\co_{\tilde{\ld{J}}}$ via the closed immersion $\tilde{\ld{J}} \hookrightarrow \ld{G} \times \fr{t}$. It follows from \cref{cor: reg locus ordinary ABG} that we need to show that $V_\lambda \otimes_\QQ \co_{\fr{t}}$ is isomorphic to $\pi_0 \cf_T(G/P_\lambda)$ as $\pi_0 \cf_T(\Gr_G)^\vee \cong \co_{\tilde{\ld{J}}}$-comodules.

    Let $\fr{t}^\gen$ denote the complement of $\bigcup_{\alpha} \fr{t}_\alpha$ as $\alpha$ ranges over the roots of $\ld{G}$, and $\fr{t}_\alpha$ denotes the hyperplane cut out by $\alpha$. Since $V_\lambda \otimes_\QQ \co_{\fr{t}}$, $\pi_0 \cf_T(G/P_\lambda)$, and $\pi_0 \cf_T(\Gr_G)^\vee$ are all flat over $\fr{t}$, it suffices to prove that there is an isomorphism $V_\lambda \otimes_\QQ \co_{\fr{t}} \cong \pi_0 \cf_T(G/P_\lambda)$ of quasicoherent sheaves over $\fr{t}$, and further show that they are isomorphic as $\pi_0 \cf_T(\Gr_G)^\vee \cong \co_{\tilde{\ld{J}}}$-comodules when restricted to $\fr{t}^\gen$.
    
    Let $W_\lambda$ denote the Weyl group of $P_\lambda$, so that $W_\lambda$ is the stabilizer of the weight $\lambda$. 
    Since $G/P_\lambda$ has even cells, $\pi_0 \cf_T(G/P_\lambda)$ is a vector bundle over $\co_{\fr{t}}$, and its rank can be determined by its restriction to $\fr{t}^\gen$. By Atiyah-Bott localization, $\pi_0 \cf_T(G/P_\lambda)|_{\fr{t}^\gen} \cong \pi_0 \cf_T((G/P_\lambda)^T)|_{\fr{t}^\gen}$; but $(G/P_\lambda)^T = W/W_\lambda$, so we conclude that
    $\pi_0 \cf_T(G/P_\lambda)$ is a free $\co_{\fr{t}}$-module of rank $|W/W_\lambda|$. Since $\lambda$ is minuscule, there is an isomorphism $V_\lambda \cong \Map(W/W_\lambda, \QQ)$ (see, e.g., \cite[Proposition 5.1]{gross-minuscule}). We therefore conclude that there is an isomorphism $V_\lambda \otimes_\QQ \co_{\fr{t}} \cong \pi_0 \cf_T(G/P_\lambda)$ of quasicoherent sheaves over $\fr{t}$.
    
    To see that they are isomorphic as $\pi_0 \cf_T(\Gr_G)^\vee \cong \co_{\tilde{\ld{J}}}$-comodules when restricted to $\fr{t}^\gen$, note that $\pi_0 \cf_T(\Gr_G)^\vee|_{\fr{t}^\gen} \cong \co_{\ld{T} \times \fr{t}^\gen}$. We therefore need to check that the weights of $\ld{T}$ acting on $V_\lambda$ and $\pi_0 \cf_T(G/P_\lambda)|_{\fr{t}^\gen}$ agree. The $T$-fixed points of the map $G/P_\lambda \to \Gr_G$ is given by the map $W/W_\lambda \to \Gr_T \cong \bX_\ast(T)$ which is the inclusion of the $W$-orbit of $\lambda$; these are the weights of $\ld{T}$ acting on $\pi_0 \cf_T(G/P_\lambda)|_{\fr{t}^\gen}$. The desired isomorphism of $\ld{T}$-representations between $V_\lambda$ and $\pi_0 \cf_T(G/P_\lambda)|_{\fr{t}^\gen}$ now follows from the fact that the weights of $\ld{T}$ on $V_\lambda$ are also precisely the elements in the $W$-orbit of $\lambda$. 

    Suppose that $\lambda_\bull$ has more than one element. Since the equivalence of \cref{cor: reg locus ordinary ABG} is symmetric monoidal, we find that $\cS_k(\bigotimes_i V_{\lambda_i})$ is equivalent to the convolution tensor product $\cf_{\lambda_1}^\gr \star \cdots \star \cf_{\lambda_n}^\gr$. We therefore need to show that there is an isomorphism $\cf_{\lambda_1}^\gr \star \cdots \star \cf_{\lambda_n}^\gr \cong \cf_{\lambda_\bull}^\gr$. For this, note that the following diagram of homotopy types commutes:
    $$\xymatrix{
    \ol{\Gr_G^{\lambda_\bull}} \ar[r] \ar[d] \ar[dr]^-q & \ol{\Gr_G^{|\lambda|}} \ar[d] \\
    \Gr_G^{\times n} \ar[r] & \Gr_G;
    }$$
    here, the bottom horizontal map is the $\E{2}$-multiplication on $\Gr_G \simeq \Omega G_c$. This implies that $\cf_{\lambda_\bull}^\gr$ is isomorphic to the pushforward of $(\cf_{\lambda_1} \boxtimes \cdots \boxtimes \cf_{\lambda_n})^\gr \cong \cf_{\lambda_1}^\gr \boxtimes \cdots \boxtimes \cf_{\lambda_n}^\gr$ along the map $\Gr_G^{\times n} \to \Gr_G$; but this is precisely the definition of $\cf_{\lambda_1}^\gr \star \cdots \star \cf_{\lambda_n}^\gr$, as desired.
\end{proof}

To convince homotopy theorists that the flag varieties appearing in \cref{prop: ordinary realizing minuscule reps} are in fact (relatively) familiar objects, we have recorded the list of such $G/P_\lambda$ for dominant minuscule $\lambda$ (even in the non-simply-laced cases) in \cref{table: minuscule varieties}.\footnote{Those homotopy theorists who have reached this far in the article may not need this table to be convinced!}
\begin{table}[]
\hspace*{-.5cm}
\begin{tabular}{l|l|l|l|l}
$G$ & {$G/P_\lambda$} & {$\ld{G} \act V_\lambda$} & {$\dim_\cc(G/P_\lambda)$} & {$|W/W_\lambda|$} \\ \hline
$A_n$ & $\Gr_j(\cc^{n+1})$, $1\leq j \leq n$ & $\wedge^j \std_{n+1}$ & $j(n+1-j)$ & $\binom{n+1}{j}$ \\
$B_n$ & Smooth quadric in $\CP^n$ & $\std_{2n}$ & $2n-1$ & $2n$ \\
$C_n$ & Lagrangian Grassmannian $\mathrm{LGr}_n(\cc^{2n})$ & Spin & $\binom{n+1}{2}$ & $2^n$ \\
$D_n$ & Smooth quadric in $\CP^{n-1}$ & $\std_{2n}$ & $2n-2$ & $2n$ \\
$D_n$ & Orthogonal Grassmannian $\mathrm{OGr}_n(\cc^{2n})$ & Half-spin (both) & $\binom{n}{2}$ & $2^{n-1}$ \\
$E_6$ & $\mathrm{EIII} \simeq (E_6)_c/\Spin(10) \cdot \U(1)$ & $\std_{27}, \std_{27}^\ast$ & $16$ & $27$ \\
$E_7$ & $\mathrm{EVII} \simeq (E_7)_c/(E_6)_c \cdot \U(1)$ & $\std_{56}$ & $27$ & $56$
\end{tabular}
\vspace{1cm}
\caption{Minuscule homogeneous varieties for $G$ of adjoint type. Here, $(E_6)_c$ and $(E_7)_c$ denote the compact forms of $E_6$ and $E_7$, respectively. The labelings $\mathrm{EIII}$ and $\mathrm{EVII}$ denote the labelings of these symmetric spaces in \'E. Cartan's classification. In the example of $D_n$ acting on the orthogonal Grassmannian, there are two realizations as a homogeneous variety, which correspond to the two half-spin representations: namely, $\SO_{2n}/P_{\alpha_{n-1}} \cong \SO_{2n}/P_{\alpha_n}$. Note, also, that $|W/W_\lambda|$ is equal to the dimension of $V_\lambda$ and also to the number of cells in a minimal ($T$-equivariant) cell structure on $G/P_\lambda$, while $\dim_\cc(G/P_\lambda)$ is the highest weight of the restriction of $\ld{G} \to \GL(V_\lambda)$ to the principal $\SL_2$ inside $\ld{G}$.}
\label{table: minuscule varieties}
\end{table}
\begin{remark}\label{rmk: ordinary action on minuscule}
    Let $\lambda$ be a dominant minuscule weight of $\ld{G}$. The coaction of $\pi_0 \cf_T(\Gr_G)^\vee \cong \H^T_\ast(\Gr_G; \QQ)$ on $\pi_0 \cf_T(G/P_\lambda) \cong \H^\ast_T(G/P_\lambda; \QQ)$ defines a homomorphism 
    \begin{equation}\label{eq: map from tilde J to cohomology of minuscule flag}
        \spec \pi_0 \cf_T(\Gr_G)^\vee \to \GL(\H^\ast_T(G/P_\lambda; \QQ))
    \end{equation}
    of group schemes over $\fr{t}$, where $\GL(\H^\ast_T(G/P_\lambda; \QQ))$ denotes the group scheme of $\co_{\fr{t}}$-linear automorphisms of the vector bundle $\H^\ast_T(G/P_\lambda; \QQ)$. By \cref{thm: ordinary hmlgy reg centr} and \cref{prop: ordinary realizing minuscule reps}, this homomorphism factors as the composite
    \begin{equation}\label{eq: factorization action on minuscule}
        \tilde{\ld{J}} \to \ld{G} \times \fr{t} \to \GL(V_\lambda) \times \fr{t},
    \end{equation}
    where the second map describes the $\ld{G}$-action on $V_\lambda$. Similarly, the coaction of $\pi_0 \cf_G(\Gr_G)^\vee \cong \H^G_\ast(\Gr_G; \QQ)$ on $\pi_0 \cf_G(G/P_\lambda) \cong \H^\ast_{P_\lambda}(\ast; \QQ)$ defines a homomorphism 
    \begin{equation}\label{eq: map from J to G-equiv cohomology of minuscule flag}
        \spec \pi_0 \cf_G(\Gr_G)^\vee \to \GL(\H^\ast_{P_\lambda}(\ast; \QQ))
    \end{equation}
    of group schemes over $\spec \H^\ast_G(\ast; \QQ) \cong \fr{t}\mmod W$. As an $\co_{\fr{t}\mmod W}$-module, $\H^\ast_{P_\lambda}(\ast; \QQ)$ is isomorphic to $\co_{\fr{t}\mmod W} \otimes V_\lambda$, and \cref{eq: map from J to G-equiv cohomology of minuscule flag} factors as the composite
    \begin{equation}\label{eq: factorization G-action on minuscule}
        {\ld{J}} \to \ld{G} \times \fr{t}\mmod W \to \GL(V_\lambda) \times \fr{t}\mmod W.
    \end{equation}
    In fact, all of these maps already exist \textit{integrally} (i.e., using cohomology with integral coefficients, and working with group schemes over $\Z$). 

    For instance, suppose $G = \SO_{2n}$ is of type $D_n$, and let us take coefficients in $\Z' = \Z[1/2]$; otherwise, the cohomology of $BG$ is not isomorphic to $\co_{\fr{t}\mmod W}$. (See \cite[Example 3.2.14]{ku-rel-langlands} for other classical types.) Then the Levi quotient of $P_\lambda$ is $\SO_2 \times \SO_{2n-2}$, so that $\H^\ast_{P_\lambda}(\ast; \Z') \cong \Z'[x, p_1', \cdots, p_{n-2}', c'_{n-1}]$ with $x$ in weight $2$, $p_i'$ in weight $-4i$, and $c'_i$ in weight $-2i$. A simple calculation with symmetric polynomials shows that as an algebra over $\H^\ast_G(\ast; \Z') \cong \Z'[p_1, \cdots, p_{n-1}, c_n]$, there is an isomorphism
    $$\H^\ast_{P_\lambda}(\ast; \Z') \cong \H^\ast_G(\ast; \Z')[x, c'_{n-1}]/(xc'_{n-1} = c_n, \ x^{2n-2} - p_1 x^{2n-4} - \cdots - p_{n-1} + {c'_{n-1}}^2).$$
    As an $\H^\ast_G(\ast; \Z')$-module, this is indeed isomorphic to $\H^\ast_G(\ast; \Z') \otimes \std_{2n}$. 
    Building on \cref{ex: equiv homology Omega SUn and Witt} shows that as a group scheme over $\H^\ast_G(\ast; \Z')$, the functor of points of $\ld{J}$ sends an $\H^\ast_G(\ast; \Z')$-algebra $R$ to the subgroup of those units $f(x, c'_{n-1}) \in (\H^\ast_{P_\lambda}(\ast; \Z') \otimes_{\H^\ast_G(\ast; \Z')} R)^\times$ such that $f(x, c'_{n-1})^{-1} = f(-x,-c'_{n-1})$. The action of $\ld{J}$ on $\H^\ast_{P_\lambda}(\ast; \Z')$ preserves the symmetric bilinear form given by
    \begin{align*}
        \H^\ast_{P_\lambda}(\ast; \Z') \otimes_{\H^\ast_G(\ast; \Z')} \H^\ast_{P_\lambda}(\ast; \Z') & \xrightarrow{\pdb{-,-}} \H^\ast_G(\ast; \Z') \\
        f,g & \mapsto \text{coefficient of }x^{2n-2}\text{ in }f(x, c'_{n-1}) g(-x, -c'_{n-1}).
    \end{align*}
    Geometrically, this bilinear form comes from $G$-equivariant Poincar\'e duality on $G/P_\lambda$, twisted by the natural action of $\Z/2$ on $G/P_\lambda$. (This $\Z/2$ acts on $\H^\ast_{P_\lambda}(\ast; \Z')$ by sending $x\mapsto -x$ and $c'_{n-1} \mapsto -c'_{n-1}$.)
    The bilinear form $\pdb{-,-}$ on $\H^\ast_{P_\lambda}(\ast; \Z')$ gives the desired factorization \cref{eq: factorization G-action on minuscule} of the map ${\ld{J}} \to \GL_{2n} \times \fr{t}\mmod W$ through the inclusion $\SO_{2n} \times \fr{t}\mmod W \hookrightarrow \GL_{2n} \times \fr{t}\mmod W$.
\end{remark}

As we have seen, the calculation of \cref{thm: ordinary hmlgy reg centr} is quite powerful. Here is another simple application, motivated by \cite{ginzburg-kazhdan} and \cite{ginzburg-riche}; see also \cite[Example 3.6.13]{ku-rel-langlands}, where the same example is presented.
\begin{prop}[Gelfand-Graev action]\label{prop: ordinary gelfand-graev}
    The natural action of $\ld{G} \times \ld{T}$ on the affine closure $\ol{T^\ast(\ld{G}/\ld{N})}$ extends to an action of $\ld{G} \times (W \rtimes \ld{T})$, where $W$ is the Weyl group.
\end{prop}
\begin{proof}
    Let $T^\ast(\ld{G}/\ld{N})_\reg = \ld{G} \times^{\ld{N}} \ld{\fr{n}}^\perp_\reg$ denote the regular locus in $T^\ast(\ld{G}/\ld{N})$; then $T^\ast(\ld{G}/\ld{N})_\reg \subseteq T^\ast(\ld{G}/\ld{N})$ is open, with complement of codimension $2$, so that $\ol{T^\ast(\ld{G}/\ld{N})} \cong \ol{T^\ast(\ld{G}/\ld{N})_\reg}$. Note that there is an isomorphism
    $$\ld{G} \backslash T^\ast(\ld{G}/\ld{N})_\reg/\ld{T} \cong \ld{\fr{n}}^\perp_\reg/\ld{B},$$
    so (the proof of) \cref{cor: reg locus ordinary ABG} gives isomorphisms
    \begin{equation}\label{eq: reg locus in T*G/N}
        T^\ast(\ld{G}/\ld{N})_\reg \cong (\ld{G} \times \ld{T}) \times^{\ld{B}} \ld{\fr{n}}^\perp_\reg \cong (\ld{G} \times \ld{T} \times (\psi + \ld{\fr{t}}^\ast))/\tilde{\ld{J}}.
    \end{equation}
    There is a canonical $W$-action on $\ld{G} \times \ld{T} \times (\psi + \ld{\fr{t}}^\ast)$, given by the natural $W$-actions on $\ld{T}$ and on $\psi + \ld{\fr{t}}^\ast \cong \ld{\fr{t}}^\ast$. Similarly, $\tilde{\ld{J}}$ also admits a natural $W$-action; it is given via \cref{thm: ordinary hmlgy reg centr} by the natural $W$-action on $\H^{T_c}_\ast(\Gr_G; \QQ)$. Moreover, the closed immersion
    $$\tilde{\ld{J}} \to \ld{G} \times \ld{T} \times (\psi + \ld{\fr{t}}^\ast)$$
    is $W$-equivariant (indeed, the map $\tilde{\ld{J}} \to \ld{T} \times (\psi + \ld{\fr{t}}^\ast)$ is induced by the inclusion $\H^{T_c}_\ast(\Gr_T; \QQ) \to \H^{T_c}_\ast(\Gr_G; \QQ)$ on equivariant homology). This implies that the quotient of \cref{eq: reg locus in T*G/N} admits a $W$-action, which defines a $W$-action on the affine closure of $T^\ast(\ld{G}/\ld{N})_\reg$ as desired.
\end{proof}
Note that we assumed in \cref{cor: reg locus ordinary ABG} that $G$ is simply-laced; but this is not necessary, because we know (by the discussion in \cref{sec: regular locus}) that the main result of \cite{abg-iwahori-satake} implies \cref{cor: reg locus ordinary ABG} is true for any connected reductive $G$. Alternatively, one can observe that the proof of \cref{cor: reg locus ordinary ABG} itself never seriously appeals to $G$ being simply-laced. 
\begin{remark}
    The proof of \cref{prop: ordinary gelfand-graev} generalizes to show that if $\ld{P} \subseteq \ld{G}$ is a parabolic subgroup with Levi quotient $\ld{L}$ and unipotent radical $U_{\ld{P}}$, then the natural action of $\ld{G} \times \ld{L}$ on the affine closure $\ol{T^\ast(\ld{G}/U_{\ld{P}})}$ extends to an action of $\ld{G} \times (W_L \rtimes \ld{L})$, where $W_L = N_{\ld{G}}(\ld{L})/\ld{L}$ is the Weyl group. (Also see \cite{affine-closure-G-UP}.)
\end{remark}

The $W$-action of \cref{prop: ordinary gelfand-graev} is known as the (semiclassical) \textit{Gelfand-Graev action}. The moment map $\ol{T^\ast(\ld{G}/\ld{N})} \to \ld{\g}^\ast$ is $W$-equivariant for the trivial action on the target. There is a commutative diagram
$$\xymatrix{
\tilde{\ld{\g}} \ar@{^(->}[r] \ar[dr] & \ol{T^\ast(\ld{G}/\ld{N})}/\ld{T} \ar[d] \\
& \ld{\g}^\ast
}$$
which relates $\ol{T^\ast(\ld{G}/\ld{N})}$ to the Grothendieck-Springer resolution; and via this diagram, the Gelfand-Graev action is closely related to the Weyl action in Springer theory.
\begin{example}\label{rmk: Z/2 symplectic fourier}
    When $\ld{G} = \SL_2$, the affine closure $\ol{T^\ast(\ld{G}/\ld{N})}$ is simply $T^\ast(\AA^2)$, and the $W = \Z/2$-action on it is given by the symplectic Fourier transform. To see this, let $\ld{J}_X$ denote the kernel of the homomorphism $\tilde{\ld{J}} \to\ld{T} \times (\psi + \ld{\fr{t}}^\ast)$ of group schemes over $\psi + \ld{\fr{t}}^\ast$. (This follows the notation from \cite{ku-rel-langlands}.) Then \cref{eq: reg locus in T*G/N} gives an isomorphism
    $$(\ld{G} \times (\psi + \ld{\fr{t}}^\ast))/\ld{J}_X \xrightarrow{\cong} T^\ast(\ld{G}/\ld{N})_\reg.$$
    In the case at hand, $\psi + \ld{\fr{t}}^\ast \cong \AA^1$ with coordinate $x$, and the group scheme $\ld{J}_X$ is just $\spec \Z[x,b]/bx$ (where the group law sends $b\mapsto b \otimes 1 + 1 \otimes b$). The above isomorphism defines a map
    $$q: \SL_2 \times \AA^1 \to \ol{T^\ast(\ld{G}/\ld{N})} = T^\ast(\AA^2),$$
    and the affine closure of the image is all of $T^\ast(\AA^2)$. The map $q$ can be explicitly described as follows. View a point of $T^\ast(\AA^2)$ as a pair $\left(\begin{psmallmatrix}
        u_1 \\
        u_2
    \end{psmallmatrix}, (v_1, v_2)\right)$ of a vector and a covector. Then $q$ is the natural extension to $\SL_2 \times \AA^1$ of the map $\kappa: \AA^1 \to T^\ast(\AA^2)$ which sends $x\mapsto \left(\begin{psmallmatrix}
        1 \\
        0
    \end{psmallmatrix}, (x,0)\right)$. In other words, $q$ sends
    $$(g, x) = \left(\begin{psmallmatrix}
        a & b \\
        c & d
    \end{psmallmatrix}, x\right) \mapsto \left(g \begin{psmallmatrix}
        1 \\
        0
    \end{psmallmatrix}, (g^T)^{-1} (x,0)\right) = \left(\begin{psmallmatrix}
        a \\
        c
    \end{psmallmatrix}, (dx, -bx)\right).$$ 
    Of course, one could also swap the roles of $\AA^2$ and $(\AA^2)^\ast$ in $T^\ast(\AA^2)$; the map $\kappa$ would then send $x\mapsto \left(\begin{psmallmatrix}
        0 \\
        x
    \end{psmallmatrix}, (0,1)\right)$, and $q$ would send
    $$(g, x) = \left(\begin{psmallmatrix}
        a & b \\
        c & d
    \end{psmallmatrix}, x\right) \mapsto \left(\begin{psmallmatrix}
        0 \\
        x
    \end{psmallmatrix} \cdot g^T, (0,1) \cdot g^{-1}\right) = \left(\begin{psmallmatrix}
        bx \\
        dx
    \end{psmallmatrix}, (-c, a)\right).$$ 
    If we compose with the involution sending $x\mapsto -x$, the resulting involution
    $$\left(\begin{psmallmatrix}
        a \\
        c
    \end{psmallmatrix}, (dx, -bx)\right) \mapsto \left(\begin{psmallmatrix}
        -bx \\
        -dx
    \end{psmallmatrix}, (-c, a)\right).$$
    This, of course, is precisely the symplectic Fourier transform, which sends
    $$\left(\begin{psmallmatrix}
        u_1 \\
        u_2
    \end{psmallmatrix}, (v_1, v_2)\right) \mapsto \left(\begin{psmallmatrix}
        v_2 \\
        -v_1
    \end{psmallmatrix}, (-u_2, u_1)\right).$$
\end{example}
We also have the following, which is an obvious consequence of \cref{cor: reg locus satake} and \cref{cor: reg locus abg} (but can be reproved using \cref{cor: reg locus ordinary ABG}).
\begin{prop}\label{prop: ordinary full faithful on gr loc}
    Let $\Loc_{T_c}^\gr(\Gr_G; k)^\heart$ denote the heart of the $t$-structure on $\Loc_{T_c}^\gr(\Gr_G; k) = \coMod_{\pi_0(\cf_T(\Gr_G))^\vee}(\QCoh(\fr{t}))$ coming from the standard (homological truncation) $t$-structure on $\QCoh(\fr{t})$. 
    Then, the composite functor
    $$\Loc_{T_c}^\gr(\Gr_G; k) \simeq \QCoh(\tilde{\ld{\g}}^\reg/\ld{G}) \to \QCoh(\ld{G}\backslash \ol{T^\ast(\ld{G}/\ld{N})}/\ld{T})$$
    is $t$-exact, and on hearts, it restricts to a fully faithful functor on the essential image of the functor \cref{eq: Rep T^ x G^ to ordinary loc}. Furthermore, this functor is $W$-equivariant for the natural action of $W = \N_{G_c}(T_c)/T_c$ on the left-hand side and the Gelfand-Graev action of \cref{prop: ordinary gelfand-graev} on the right-hand side.

    Similarly, let $\Loc_{G_c}^\gr(\Gr_G; k)^\heart$ denote the heart of the $t$-structure on $\Loc_{G_c}^\gr(\Gr_G; k) = \coMod_{\pi_0(\cf_G(\Gr_G))^\vee}(\QCoh(\fr{t}\mmod W))$ coming from the standard (homological truncation) $t$-structure on $\QCoh(\fr{t}\mmod W)$. 
    Then, the composite functor
    $$\Loc_{G_c}^\gr(\Gr_G; k) \simeq \QCoh(\ld{\g}^{\ast,\reg}/\ld{G}) \to \QCoh(\ld{\g}^\ast/\ld{G})$$
    is $t$-exact, and on hearts, it restricts to a fully faithful functor on the essential image of the functor $\Rep(\ld{G}) \to \Loc_{G_c}^\gr(\Gr_G; k)$ (analogous to \cref{eq: Rep G^ to ordinary loc}).
\end{prop}
\begin{proof}
    If $V \in \Rep(\ld{G})$, the object $\cS_k(V)$ lies in $\Loc_{T_c}^\gr(\Gr_G; k)^\heart$, and we need to show that there is an isomorphism
    \begin{multline*}
        \Map_{\QCoh(\ld{G}\backslash \ol{T^\ast(\ld{G}/\ld{N})}/\ld{T})^\heart}(V_1 \otimes_\QQ \co_{\ol{T^\ast(\ld{G}/\ld{N})}}, V_2 \otimes_\QQ \co_{\ol{T^\ast(\ld{G}/\ld{N})}}) \\
        \xrightarrow{\cong} \Map_{\Loc_{T_c}^\gr(\Gr_G; k)^\heart}(\cS_k(V_1), \cS_k(W_1))
    \end{multline*}
    of $\QQ$-vector spaces for any two representations $V_1,V_2 \in \Rep(\ld{G})$. In other words, 
    By \cref{cor: reg locus ordinary ABG}, there is an isomorphism
    $$\Map_{\Loc_{T_c}^\gr(\Gr_G; k)^\heart}(\cS_k(V), \cS_k(W)) \cong \Map_{\QCoh(\tilde{\ld{\g}}^\reg/\ld{G})^\heart}(V \otimes_\QQ \co_{\tilde{\ld{\g}}^\reg}, W \otimes_\QQ \co_{\tilde{\ld{\g}}^\reg}).$$
    Since $\tilde{\ld{\g}}^\reg/\ld{G} \hookrightarrow \ld{G} \backslash \ol{T^\ast(\ld{G}/\ld{N})}/\ld{T}$ has complement of codimension $2$ and $\ol{T^\ast(\ld{G}/\ld{N})}$ is affine and normal, the algebraic Hartogs lemma implies that the restriction map
    \begin{multline*}
        \Map_{\QCoh(\ld{G} \backslash \ol{T^\ast(\ld{G}/\ld{N})}/\ld{T})^\heart}(V \otimes_\QQ \co_{\ol{T^\ast(\ld{G}/\ld{N})}}, W \otimes_\QQ \co_{\ol{T^\ast(\ld{G}/\ld{N})}}) \\
        \to \Map_{\QCoh(\tilde{\ld{\g}}^\reg/\ld{G})^\heart}(V \otimes_\QQ \co_{\tilde{\ld{\g}}^\reg}, W \otimes_\QQ \co_{\tilde{\ld{\g}}^\reg})
    \end{multline*}
    is an isomorphism, as desired. (This is where it is crucial that we work at the level of abelian categories.) The same argument works for $\Loc_{G_c}^\gr(\Gr_G; k)$; in this case, $\ld{\g}^{\ast,\reg} \hookrightarrow \ld{\g}^\ast$ even has complement of codimension $3$.
\end{proof}
\cref{prop: ordinary full faithful on gr loc} gives an analogue of \cite[Theorem 4]{bf-derived-satake}: namely, if $\QCoh_\free(\ld{\g}^\ast/\ld{G})$ denotes the essential image of the pullback functor $\Rep(\ld{G}) \to \QCoh(\ld{\g}^\ast/\ld{G})$, then there is a fully faithful embedding 
$$\QCoh_\free(\ld{\g}^\ast/\ld{G})^\heartsuit \hookrightarrow \Loc_{G_c}^\gr(\Gr_G; k)^\heartsuit.$$
Similarly, if $\QCoh_\free(\ld{G} \backslash \ol{T^\ast(\ld{G}/\ld{N})}/\ld{T})$ denotes the essential image of the pullback functor $\Rep(\ld{G} \times \ld{T}) \to \QCoh(\ld{G} \backslash \ol{T^\ast(\ld{G}/\ld{N})}/\ld{T})$, then there is a fully faithful embedding 
$$\QCoh_\free(\ld{G} \backslash \ol{T^\ast(\ld{G}/\ld{N})}/\ld{T})^\heartsuit \hookrightarrow \Loc_{T_c}^\gr(\Gr_G; k)^\heartsuit.$$

Let $\Rep_\min(\ld{G})$ denote the idempotent completion of the subcategory of $\Rep(\ld{G})$ spanned by tensor products of irreducible $\ld{G}$-representations with minuscule highest weights. In general, if $\ld{G}$ is simple (but not necessarily simply-laced) and not of types $G_2$, $F_4$, or $E_8$, then any representation is a summand of a tensor product of irreducible $\ld{G}$-representations with minuscule highest weights, so that $\Rep_\min(\ld{G}) \simeq \Rep(\ld{G})$.\footnote{When $\ld{G}$ is of types $G_2$, $F_4$, or $E_8$, there are no minuscule weights at all. In general, $\Rep(\ld{G})$ is the idempotent completion of its full subcategory spanned by tensor products of irreducible $\ld{G}$-representations with minuscule and {quasi}-minuscule highest weights.}
\begin{corollary}\label{cor: ordinary minuscule equivalence}
    Let $\QCoh_{\free}(\ld{\g}^\ast/\ld{G})^{\min,\heartsuit}$ denote the essential image of $\Rep_\min(\ld{G})$ under the pullback functor $\Rep(\ld{G})^\heartsuit \to \QCoh(\ld{\g}^\ast/\ld{G})^\heartsuit$ (so it is the entirety of $\QCoh(\ld{\g}^\ast/\ld{G})^\heartsuit$ if $\ld{G}$ is not of type $E_8$). Similarly, let $\Loc_{G_c}^\gr(\Gr_G; k)^{\min,\heartsuit}$ denote the idempotent completion of the subcategory of $\Loc_{G_c}^\gr(\Gr_G; k)^\heartsuit$ spanned by $\cf_{\lambda_\bull}^\gr$ ranging over sequences $\lambda_\bull$ of minuscule highest weights. Then there is an equivalence
    $$\QCoh_\free(\ld{\g}^\ast/\ld{G})^{\min,\heartsuit} \simeq \Loc_{G_c}^\gr(\Gr_G; k)^{\min,\heartsuit}.$$
\end{corollary}
There is a similar equivalence 
$$\Loc_{T_c}^\gr(\Gr_G; k)^{\min,\heartsuit} \simeq \QCoh_\free(\ld{G} \backslash \ol{T^\ast(\ld{G}/\ld{N})}/\ld{T})^{\min,\heartsuit},$$
where these categories are defined analogously by idempotent completion. 

Note that the category $\Loc_{G_c}^\gr(\Gr_G; k)^{\min,\heartsuit}$ is the heart of a degeneration, in the sense of \cref{sec: degenerations}, of the similarly-defined category $\Loc_{G_c}(\Gr_G; k)^\min$. (In particular, \cref{cor: ordinary minuscule equivalence} gives an equivalence between the purely algebraically defined category $\QCoh_\free(\ld{\g}^\ast/\ld{G})^{\min,\heartsuit}$ and a degeneration of the purely topologically defined category $\Loc_{G_c}(\Gr_G; k)^\min$.) If $\lambda_\bull$ and $\mu_\bull$ are two sequences of dominant minuscule weights of $\ld{G}$, there is an equivalence of $k$-modules
\begin{align*}
    \Map_{\Loc_{G_c}(\Gr_G; k)^\min}(\cf_{\lambda_\bull}, \cf_{\mu_\bull}) & \simeq \cf_G(\ol{\Gr_G^{\lambda_\bull}}) \otimes_{\coMod_{\cf_{G}(\Gr_G)^\vee}(\QCoh(\cM_G))} \cf_G(\ol{\Gr_G^{\mu_\bull}}) \\
    & \simeq \cf_{G}(\ol{\Gr_G^{\lambda_\bull}} \times_{\Gr_G} \ol{\Gr_G^{\mu_\bull}}),
\end{align*}
where the final equivalence uses the K\"unneth formula at the level of $k$-cochains. The category $\Loc_{G_c}(\Gr_G; k)^\min$ therefore compares to (the $k$-analogue of) the category from \cite[Section 3.4]{cautis-kamnitzer}.

At first glance, the existence of the $t$-structure on $\Loc_{T_c}^\gr(\Gr_G; k)$ from \cref{prop: ordinary full faithful on gr loc} may perhaps be a bit surprising, since $k$ is a \textit{$2$-periodic} $\Eoo$-ring. In fact, this periodicity prohibits $\Loc_{T_c}(\Gr_G; k)$ itself from having a $t$-structure. However, the $\infty$-category $\Loc_{T_c}^\gr(\Gr_G; k)$ ``flattens'' out the homological periodicity in $\Loc_{T_c}(\Gr_G; k)$ to a weight periodicity, but it is itself also a stable $\infty$-category. In particular, it has a \textit{homological} shift operation, which is distinct from the operation of shifting \textit{weights} (just as with the usual category of mixed sheaves). (The $2$-periodicity of $k$ implies that the weight-shifting operation on $\Loc_{T_c}^\gr(\Gr_G; k)$ is an equivalence, which is why we do not see gradings/weights when discussing $\Loc_{T_c}^\gr(\Gr_G; k)$; but the weight-shifting operation will be nontrivial on, say, $\Loc_{T_c}^\gr(\Gr_G; \QQ)$.) The resulting homological shift on $\Loc_{T_c}^\gr(\Gr_G; k)$ is no longer periodic, and it is therefore reasonable to equip this $\infty$-category with a $t$-structure. (This $t$-structure is unrelated to the perverse $t$-structure on $\Shv_I(\Gr_G; \QQ)$.)

We will now discuss a \textit{deformation quantization} of \cref{cor: reg locus ordinary ABG} by adding loop-rotation equivariance. Write $\tilde{T} = T \times \GG_m^\rot$ to denote the corresponding affine torus. In the case when $G$ is a torus, we have already discussed this in \cref{sec: torus loop rot}. For more general $G$, this turns out to be a bit tricky: while $\H^{T_c}_\ast(\Gr_G; \QQ)$ is a bicommutative Hopf algebra\footnote{To be more precise, the $\E{2}$-space structure on $\Gr_G$ equips $C^{T_c}_\ast(\Gr_G; \QQ)$ with the structure of an $\E{2}$-algebra in $\Eoo$-coalgebras over $C_{T_c}^\ast(\ast; \QQ)$.}, the loop-rotation equivariant homology $\H^{\tilde{T}_c}_\ast(\Gr_G; \QQ)$ is only a cocommutative coalgebroid over $\H^\ast_{\tilde{T}_c}(\ast; \QQ)$. That is, it does not admit an algebra structure. While this is not a mathematical issue, it does make the task of explicitly understanding $\H^{\tilde{T}_c}_\ast(\Gr_G; \QQ)$ in a satisfactory way more complicated. 
Instead, it turns out to be easier to describe $\H^{\tilde{T}_c}_\ast(\Fl_G; \QQ)$, where $\Fl_G$ is the \textit{affine flag variety}, defined as the quotient $G\ls{t}/I$ for the Iwahori subgroup $I \subseteq G\pw{t}$ associated to a Borel subgroup $B \subseteq G$.
To state the result, we need a definition from \cite{ginzburg-kapranov-vasserot-residue-hecke}.
\begin{definition}\label{def: nil-hecke}
    Let $(\Lambda, \Phi, \ld{\Lambda}, \ld{\Phi})$ be a root datum with associated Weyl group $W$ and torus $T = \Hom(\Lambda, \GG_m)$. Let $\Delta$ be a base of simple roots, let $\Phi^+$ denote the corresponding set of positive roots, and let $\Phi'$ denote the subset $W \cdot \Delta \subseteq \Phi$. Let $\bH$ be a $1$-dimensional group scheme (over a commutative ring $R$). As in \cref{def: G-diff ops}, let $\bH_T = \Hom(\Lambda, T)$, and for each character $\lambda \in \Lambda$, let $\bH_{T_\lambda} \hookrightarrow \bH_T$ denote the subgroup corresponding to the subtorus $T_\lambda = \ker(\lambda) \subseteq T$.
    Let $Q(\co_{\bH_T})$ denote the sheaf of functions on the generic point of $\co_{\bH_T}$. The twisted group algebra $Q(\co_{\bH_T})[W]$ is the algebra which is additively given by the tensor product $Q(\co_{\bH_T}) \otimes_F F[W]$, and whose multiplication law is given by 
    $$(f_1 \otimes w_1) \cdot (f_2 \otimes w_2) = (f \cdot w_1 g) \otimes (w_1 w_2).$$
    The algebra $\cH(\bH, T, W)$ is defined to be the subset of $Q(\co_{\bH_T})[W]$ of those elements $\sum_{w \in W} f_w [w]$ such that:
    \begin{itemize}
        \item The poles of $f_x$ all have order $\leq 1$, and these are contained in the divisors $\bH_{T_\alpha}$ for each $\alpha \in \Phi'$.
        \item For each $w \in W$ and $\alpha \in \Phi^+ \cap \Phi'$, we have
        $$\Res_{\bH_{T_\alpha}}(f_w) + \Res_{\bH_{T_\alpha}}(f_{s_\alpha w}) = 0.$$
    \end{itemize}
    In \cite{ginzburg-kapranov-vasserot-residue-hecke}, this algebra is denoted $\tilde{\bH}$.
    It is proved in \cite[Theorem 1.4]{ginzburg-kapranov-vasserot-residue-hecke} that $\cH(\bH, T, W)$ is a \textit{subalgebra} of $Q(\co_{\bH_T})[W]$.
\end{definition}
\begin{remark}
    The pair $(Q(\co_{\bH_T}), Q(\co_{\bH_T})[W])$ admits the structure of a (cocommutative) Hopf algebroid; we will abusively say that $Q(\co_{\bH_T})[W]$ admits the structure of a Hopf $Q(\co_{\bH_T})$-algebroid. The coproduct comes from the diagonal on $W$; the left unit comes from the inclusion $Q(\co_{\bH_T}) \subseteq Q(\co_{\bH_T})[W]$; and the right unit comes from the action of $W$ on $\bH_T$ (which defines a coaction of $W$ on $\co_{\bH_T}$ that extends to a coaction on $Q(\co_{\bH_T})$). The resulting Hopf $\co_{\bH_T}$-algebroid structure on $Q(\co_{\bH_T})[W]$ restricts to $\cH(\bH, T, W)$, so that $\cH(\bH, T, W)$ admits the structure of a (cocommutative) Hopf $\co_{\bH_T}$-algebroid. (See \cite[Theorem 4.11]{hopf-algebroid-nil-hecke} for the case $\bH = \GG_a$.)
    
    When $W$ is finite, \cite[Proposition 2.3]{ginzburg-kapranov-vasserot-residue-hecke} states that {upon rationalization}, the action of $\cH(\bH, T, W)$ on $\co_{\bH_T}$ gives an isomorphism between $\cH(\bH, T, W)$ and $\End_{\co_{\bH_T}^W}(\co_{\bH_T})$. This gives a Morita equivalence between the category of $\cH(\bH, T, W)$-modules and the category of $\co_{\bH_T}^W$-modules. Under this equivalence, the symmetric monoidal structure on the category of $\cH(\bH, T, W)$-modules from the cocommutative Hopf algebroid structure on $\cH(\bH, T, W)$ identifies with the standard symmetric monoidal structure on the category of $\co_{\bH_T}^W$-modules.
\end{remark}
If $\Lambda$ denotes the \textit{co}root lattice of ${G}$, let $W^\aff = \Lambda \rtimes W$ denote the corresponding affine Weyl group, and let $\tilde{W} = \bX_\ast({T}) \rtimes W$ denote the extended affine Weyl group.\footnote{The affine Weyl group $W^\aff$ introduced above is very slightly different from the affine Weyl group studied in \cite[Section 7.2]{ginzburg-whittaker} or \cite{gannon-tmmodw}; the affine Weyl group there is the semidirect product $W'^\aff = \ld{\Lambda} \rtimes W$, where $\ld{\Lambda}$ is the \textit{root} lattice of $G$. When $G$ is simply-laced, these are, of course, isomorphic; but they differ otherwise.} 
For clarity, note that the action of $\tilde{W}$ on $\bX^\ast(T)$ (and hence on $\bH_T$) is given as follows: if $\alpha$ is a coweight of $T$ and $n \in \Z$, the generator $s_{\alpha,n}$ of $W^\aff$ acts on $\bX^\ast(T)$ by reflection along the affine hyperplane $\{x\in \bX^\ast(T) | \pdb{x, \alpha} = n\}$.
The \textit{degenerate nil-Hecke algebra} $\cH(\bH, \tilde{T}, \tilde{W})$ is defined to be $\bX_\ast(T) \ltimes_\Lambda \cH(\bH, \tilde{T}, W^\aff)$. In the following discussion, we will simply write $Q(\co_{\bH_{\tilde{T}}})[\tilde{W}]$ to denote $\bX_\ast({T}) \ltimes_\Lambda Q(\co_{\bH_{\tilde{T}}})[W^\aff]$, so that $\cH(\bH, \tilde{T}, \tilde{W})$ is contained in $Q(\co_{\bH_{\tilde{T}}})[\tilde{W}]$.

There is a natural inclusion $\cd_{\ld{T}}^\bH[W] \hookrightarrow \cH(\bH, \tilde{T}, \tilde{W})$ of (sheaves of) algebras. The following result can be proved exactly as in \cite[Proposition 7.2.4]{ginzburg-whittaker}; one only has to use \cite[Proposition 2.3]{ginzburg-kapranov-vasserot-residue-hecke} in place of \cite[Lemma 7.1.5]{ginzburg-whittaker}, and also observe that the arguments of \cite{lonergan-descent} generalize to the setting of descent along the map $\bH_T/W \to \bH_T \mmod W$.
\begin{prop}\label{prop: ascending to nil-hecke module}
    Let $\cf$ be $\cd_{\ld{T}}^\bH[W]$-module\footnote{Here, we mean a module in the usual, underived, sense of the word; but it is easy to generalize the statement to the setting of perfect $\cd_{\ld{T}}^\bH[W]$-modules by induction on the length of the bounded complex.}. Then the action of $\cd_{\ld{T}}^\bH[W]$ extends (necessarily uniquely) along the inclusion $\cd_{\ld{T}}^\bH[W] \hookrightarrow \cH(\bH, \tilde{T}, \tilde{W})$ if and only if the natural map $\co_{\bH_T} \otimes_{\co_{\bH_T}^W} \cf^W \to \cf$ is an isomorphism.
\end{prop}
\begin{remark}\label{rmk: relationship to t mmod Waff}
    Let us remark on a relationship to \cite{gannon-tmmodw}. Following \textit{loc. cit.}, let $\Gamma_{W^\aff}$ denote the ind-scheme given by the union of graphs of the affine Weyl group $W^\aff$ acting on $\bH_{\tilde{T}}$, and let $\Gamma_{\tilde{W}}$ denote $\tilde{W} \times^{W^\aff} \bH_{\tilde{T}}$. Then there are two projections $\Gamma_{\tilde{W}} \rightrightarrows \bH_{\tilde{T}}$. This can be extended to a simplicial diagram $\Gamma_\bull$ of ind-schemes. Define the stack $\bH_{\tilde{T}}\mmod \tilde{W}$ to be the geometric realization of $\Gamma_\bull$. (For instance, if $W$ is trivial, this is the quotient $\bH_{\tilde{T}}/\bX_\ast(T)$. Similarly, if $\bH = \GG_a$, so that $\bH_{\tilde{T}} \cong \fr{t} \oplus \AA^1_\hbar$, then the specialization of the quotient $\bH_{\tilde{T}}\mmod \tilde{W}$ to $\hbar=1$ agrees with the quotient $\fr{t}\mmod \tilde{W}$ from \cite{gannon-tmmodw}.) In general, there is a map of stacks $\phi: \bH_{\tilde{T}}/\tilde{W} \to \bH_{\tilde{T}} \mmod \tilde{W}$. By arguing exactly as in \cite[Theorem 4.23]{gannon-tmmodw}, one can show that the pullback functor $\phi^!$ is fully faithful; and furthermore, an object of $\IndCoh(\bH_{\tilde{T}}/\tilde{W})$ descends\footnote{That is, it lies in the essential image of the left adjoint $\phi^!$ to $\phi_\ast^\IndCoh$.} along $\phi$ if and only if the corresponding object of $\IndCoh(\bH_T/W)$ descends to $\bH_T\mmod W$. Since \cref{rmk: G-mellin} gives an equivalence between $\IndCoh(\bH_{\tilde{T}}/\tilde{W})$ and $\cd_{\ld{T}}^\bH[W]\modc$, \cref{prop: ascending to nil-hecke module} can be used to obtain an equivalence between $\cH(\bH, \tilde{T}, \tilde{W})\modc$ and $\IndCoh(\bH_{\tilde{T}} \mmod \tilde{W})$.
\end{remark}

\cref{prop: hmlgy gkm} yields the following result due to Kostant and Kumar \cite{kostant-kumar, kostant-kumar-2, kumar-kac-moody}, which (as we will explain momentarily) could also be seen as a consequence of results from \cite{ginzburg-whittaker, lonergan-fourier, bf-derived-satake}.
In the discussion below, $\bH = \GG_a$. Note that $\H^\ast_{\tilde{T}_c}(\ast; \QQ)$ is isomorphic to $\co_{\tilde{\fr{t}}} \cong \co_{\bH_{\tilde{T}}}$. 
\begin{theorem}\label{thm: ordinary loop-rot flag}
    There is an isomorphism of associative $\QQ[\hbar]$-algebras
    \begin{equation}\label{eq: comparison to nil hecke}
        \H^{\tilde{T}_c}_\ast(\Fl_G; \QQ) \cong \cH(\GG_a, \tilde{T}, \tilde{W}).
    \end{equation}
    Here, $\H^{\tilde{T}_c}_\ast(\Fl_G; \QQ)$ is equipped with the associative algebra structure coming from convolution. Moreover, the above isomorphism is also one of (cocommutative) Hopf $\H^\ast_{\tilde{T}_c}(\ast; \QQ) \cong \co_{\bH_{\tilde{T}}}$-algebroids.
\end{theorem}
\begin{proof}
    The affine flag variety $\Fl_G$ is an ind-finite GKM space in the sense of \cref{def: GKM space}, and so we may use \cref{prop: hmlgy gkm} to describe $\H_\ast^{\tilde{T}_c}(\Fl_G; \QQ)$. The GKM graph of $\Fl_G$ has set of vertices given by $\Fl_G^{\tilde{T}_c} = \bX_\ast(T) \rtimes W \cong \tilde{W}$, and an edge $w \to s_{\alpha,n} w$ for each affine reflection $s_{\alpha,n} \in \tilde{W}$. In particular, if $\punc{\tilde{\fr{t}}}$ denotes the complement of the union of affine hyperplanes in $\tilde{\fr{t}}$, then $\H^{\tilde{T}_c}_\ast(\Fl_G; \QQ)$ is a subalgebra of $\H^{\tilde{T}_c}_\ast(\Fl_G^{\tilde{T}_c}; \QQ)|_{\punc{\tilde{\fr{t}}}}$. The latter is isomorphic to $\H^{\tilde{T}_c}_\ast(\tilde{W}; \QQ)|_{\punc{\tilde{\fr{t}}}}$, which in turn can be identified (using \cref{prop: T homology and quantized diffop}, for instance) with a localization of $\cd^\hbar_{\ld{T}}[W]$. This localization of $\cd^\hbar_{\ld{T}}[W]$ is  isomorphic to $Q(\co_{\bH_{\tilde{T}}})[\tilde{W}]$, so $\H^{\tilde{T}_c}_\ast(\Fl_G; \QQ)$ is a subalgebra of $Q(\co_{\bH_{\tilde{T}}})[\tilde{W}]$. 
    
    \cref{prop: hmlgy gkm} now gives an isomorphism between the two subsets
    $$\H^{\tilde{T}_c}_\ast(\Fl_G; \QQ) \subseteq Q(\co_{\bH_{\tilde{T}}})[\tilde{W}] \supseteq \cH(\GG_a, \tilde{T}, \tilde{W}).$$
    To see that this is an isomorphism of sub\textit{algebras}, simply observe that both $\H^{\tilde{T}_c}_\ast(\Fl_G; \QQ)$ and $\cH(\GG_a, \tilde{T}, \tilde{W})$ inherit their multiplicative structure from $Q(\co_{\bH_{\tilde{T}}})[\tilde{W}]$. That this is an isomorphism of Hopf $\co_{\bH_{\tilde{T}}}$-algebroids is also elementary: for instance, the coproduct on both $\H^{\tilde{T}_c}_\ast(\Fl_G; \QQ)$ and $\cH(\GG_a, \tilde{T}, \tilde{W})$ are inherited from the $\co_{\bH_{\tilde{T}}}$-linear coproduct on $Q(\co_{\bH_{\tilde{T}}})[\tilde{W}]$ coming from the diagonal on $\Fl_G^{\tilde{T}_c} = \tilde{W}$.
\end{proof}
\begin{remark}
    The left-hand side of \cref{thm: ordinary loop-rot flag} admits an obvious grading; on the right-hand side, the resulting grading on $\cH(\GG_a, \tilde{T}, \tilde{W})$ can be identified with that inherited from $Q(\co_{(\GG_a)_{\tilde{T}}})[\tilde{W}]$, where the coordinates of $(\GG_a)_{\tilde{T}}$ are placed in weight $2$.

    Moreover, \cref{thm: ordinary loop-rot flag} holds even if $\QQ$ is replaced by $\Z$ (as long as, on the right-hand side, $\GG_a$ is viewed as defined over $\Z$).
\end{remark}

\begin{remark}
    Suppose $W$ is finite. Then \cite[Proposition 2.3]{ginzburg-kapranov-vasserot-residue-hecke} states that \textit{upon rationalization}, the action of $\cH(\GG_a, T, W)$ on $\co_{(\GG_a)_T} = \co_{\fr{t}}$ gives an isomorphism between $\cH(\GG_a, T, W)$ and $\End_{\co_{\fr{t}}^W}(\co_{\fr{t}})$. (This result is false without rationalization, or at least without inverting enough primes.) Its $\co_{\fr{t}}$-linear dual is therefore $\co_{\fr{t}} \otimes_{\co_{\fr{t}}^W} \co_{\fr{t}} \cong \co_{\fr{t} \times_{\fr{t}\mmod W} \fr{t}}$. Note that this naturally admits the structure a cocommutative Hopf $\co_{\fr{t}}$-algebroid. The analogue of \cref{thm: ordinary loop-rot flag} states that there is an isomorphism $\H^{T_c}_\ast(G_c/T_c; \QQ) \cong \cH(\GG_a, T, W)$ of (cocommutative) Hopf $\H^\ast_{T_c}(\ast; \QQ) \cong \co_{\fr{t}}$-algebroids.
\end{remark}

Let $\mathbf{e} = \frac{1}{|W|} \sum_{w \in W} [w]$ denote the symmetrizer, viewed as an element of $\QQ[W]$. The spherical subalgebra $\cH(\GG_a, \tilde{T}, \tilde{W})^\sph$ is defined to be $\mathbf{e} \cH(\GG_a, \tilde{T}, \tilde{W}) \mathbf{e}$. The following result is now an easy consequence of \cref{thm: ordinary loop-rot flag}.
\begin{corollary}\label{cor: ordinary loop-rot Gr}
    There is an isomorphism of associative $\QQ[\hbar]$-algebras
    $$\H^{G_c \times S^1_\rot}_\ast(\Gr_G; \QQ) \cong \cH(\GG_a, \tilde{T}, \tilde{W})^\sph.$$
    Here, $\H^{G_c \times S^1_\rot}_\ast(\Gr_G; \QQ)$ is equipped with the associative algebra structure coming from convolution. Moreover, the above isomorphism is also one of (cocommutative) Hopf $\H^\ast_{G \times S^1_\rot}(\ast; \QQ) \cong \co_{\fr{t}\mmod W \times \AA^1_\hbar}$-algebroids.
\end{corollary}
Let $\cd^\hbar_{\ld{G}}$ denote the algebra of (rescaled) differential operators on $\ld{G}$, and let $\ld{N} {}_\psi \backslash \cd^\hbar_{\ld{G}}/_\psi \ld{N}$ denote its bi-Whittaker reduction (that is, its two-sided Hamiltonian reduction by the left and right actions of $\ld{N}$ with respect to a nondegenerate character $\psi: \ld{\fr{n}} \to \GG_a$). \cref{cor: ordinary loop-rot Gr} and \cite[Theorem 1.2.1]{ginzburg-whittaker} yield:
\begin{corollary}[{\cite[Theorem 3]{bf-derived-satake}}]\label{cor: loop-rot Gr and biWhit}
    There is an isomorphism of associative $\QQ[\hbar]$-algebras
    $$\H^{G_c \times S^1_\rot}_\ast(\Gr_G; \QQ) \cong \ld{N} {}_\psi \backslash \cd^\hbar_{\ld{G}}/_\psi \ld{N}.$$
\end{corollary}
Note that the diagonal on $\ld{G}$ equips $\cd^\hbar_{\ld{G}}$ with the structure of a coalgebra in the category of $U_\hbar(\ld{\g})$-bimodules. This in turn equips the bi-Whittaker reduction $\ld{N} {}_\psi \backslash \cd^\hbar_{\ld{G}}/_\psi \ld{N}$ with the structure of a (cocommutative) Hopf algebroid over $U_\hbar(\ld{\g})/_\psi \ld{N}$; by \cite{kostant-whittaker}, the latter is isomorphic to $Z(U_\hbar(\ld{\g})) \cong \Sym(\fr{t}^\ast)^W[\hbar]$.
Again, one can verify (by reduction to the case of a torus) that the isomorphism of \cref{cor: loop-rot Gr and biWhit} is one of cocommutative Hopf coalgebroids over $\H_{G_c \times S^1_\rot}^\ast(\ast; \QQ) \cong \Sym(\fr{t}^\ast)^W[\hbar]$.
\begin{remark}
    Since $\H^{G_c \times S^1_\rot}_\ast(\Gr_G; \QQ)$ is Morita equivalent to $\H^{T_c \times S^1_\rot}_\ast(\Fl_G; \QQ)$, \cref{rmk: relationship to t mmod Waff}, \cref{thm: ordinary loop-rot flag}, \cref{cor: ordinary loop-rot Gr}, and \cref{cor: loop-rot Gr and biWhit} together tell us that there are equivalences of categories
    \begin{multline*}
        \H^{T_c \times S^1_\rot}_\ast(\Fl_G; \QQ)\modc \simeq \H^{G_c \times S^1_\rot}_\ast(\Gr_G; \QQ)\modc \\
        \simeq \ld{N} {}_\psi \backslash \cd^\hbar_{\ld{G}}/_\psi \ld{N}\modc
        \simeq \cH(\GG_a, \tilde{T}, \tilde{W})\modc \simeq \IndCoh(\tilde{\fr{t}}\mmod \tilde{W}).
    \end{multline*}
\end{remark}

\begin{definition}\label{def: kappa-hbar}
    Denote by $\HC^\hbar_{\ld{G}}$ the $\infty$-category $\cd^\hbar_{\ld{G}}\modc^{\ld{G} \times \ld{G}, \weak} \simeq U_\hbar(\ld{\g})\modc^{\ld{G}, \weak}$ of Harish-Chandra bimodules. Let $\kappa_\hbar: \HC^\hbar_{\ld{G}} \to U_\hbar(\ld{\g})\modc^{(\ld{N}, \psi)}$ denote the Kostant functor of \cite[Section 2.3]{bf-derived-satake}, so that it is given by the composite
    $$\HC^\hbar_{\ld{G}} \xrightarrow{\mathrm{forget}} U_\hbar(\ld{\g})\modc \xrightarrow{\mathrm{Av}_{\ld{N},\psi}} U_\hbar(\ld{\g})\modc^{(\ld{N}, \psi)}.$$
    Note that by Skryabin's theorem (see the appendix of \cite{premet}), there is an equivalence $U_\hbar(\ld{\g})\modc^{(\ld{N}, \psi)} \simeq \QCoh(\fr{t}\mmod W \times \AA^1_\hbar)$.  Define $(\HC^\hbar_{\ld{G}})_\reg$ to denote the localizing subcategory of $\HC^\hbar_{\ld{G}}$ on which $\kappa_\hbar$ is conservative.
\end{definition}
One can check that upon ``setting $\hbar = 1$'', the category $(\HC^\hbar_{\ld{G}})_\reg$ identifies with the category $\mathcal{HC}_\mathrm{nondeg}$ from \cite[Remark 4.22]{gannon-thesis}.\footnote{Let me note here my aversion to the phrase ``setting $\hbar = 1$''. As we have seen above, $\hbar$ arises naturally as a generator of $\H^2_{S^1_\rot}(\ast; \QQ)$, and as such, it lives in \textit{nonzero grading}. It is therefore not sensible to set $\hbar$ to be equal to a nonzero number. A better -- and in some sense equivalent -- way to ``set $\hbar = 1$'' in a graded $\QQ[\hbar]$-module/category $M_\hbar$ is to extract the weight zero piece of the localization $M_\hbar[\hbar^{-1}]$. Doing this procedure to $(\HC^\hbar_{\ld{G}})_\reg$ will product $\mathcal{HC}_\mathrm{nondeg}$.} 
Before proceeding, we need a category-theoretic result, which follows from \cite[Corollary 4.7.5.3]{HA}.
\begin{prop}\label{prop: cosimplicial full faithful}
    Let $\cC^\bull$ be an augmented cosimplicial presentable stable $\infty$-category. Suppose that:
    \begin{enumerate}
        \item For every $[n]\in \Deltab^+$, the face map $d^0: \cC^i \to \cC^{i+1}$ admits a left adjoint $(d^0)^L$.
        \item The ``Beck-Chevalley conditions'' hold. That is, for every morphism $\alpha: [m] \to [n]$ in $\Deltab^+$, the following diagram commutes:
        $$\xymatrix{
        \cC^{m+1} \ar[d]_{(d^0)^L} \ar[r]^-{([0] \star \alpha)^\ast} & \cC^{n+1} \ar[d]^{(d^0)^L} \\
        \cC^m \ar[r]^-{\alpha^\ast} & \cC^n.
        }$$
    \end{enumerate}
    Then the functor $\cC^{-1} \to \Tot(\cC^\bull|_{N(\Deltab)})$ admits a fully faithful right adjoint; moreover, the essential image of this functor identifies with the full subcategory of $\cC^{-1}$ on which the functor $\cC^{-1} \to \cC^0$ is conservative.
\end{prop}
It is my understanding that the following result is closely related to recent work of Gannon and Ginzburg \cite{gannon-ginzburg} (but I have not made a comparison).
\begin{corollary}\label{cor: reg locus quantized satake}
    Recall the category $\Loc_{G_c \times S^1_\rot}^\gr(\Gr_G; k)$ from \cref{def: graded G equiv loc sys}. There is an equivalence
    $$\Loc_{G_c \times S^1_\rot}^\gr(\Gr_G; k) \simeq (\HC^\hbar_{\ld{G}})_\reg.$$
    Furthermore, the pushforward functor $\Loc_{G_c \times S^1_\rot}^\gr(\Gr_G; k) \to \Loc_{G_c \times S^1_\rot}^\gr(\ast; k)$ identifies with the functor $\kappa_\hbar: (\HC^\hbar_{\ld{G}})_\reg \to \QCoh(\fr{t}\mmod W \times \AA^1_\hbar)$.
\end{corollary}
\begin{proof}
    By definition, $\Loc_{G_c \times S^1_\rot}^\gr(\Gr_G; k)$ is the $\infty$-category of left comodules over $\H^{G_c \times S^1_\rot}_\ast(\Gr_G; k)$ in $\Loc_{G_c \times S^1_\rot}(\ast; k)$. The latter category can be identified with 
    $$\Loc_{G_c \times S^1_\rot}(\ast; k) \simeq \QCoh(\fr{t}\mmod W \times \AA^1_\hbar) \simeq U_\hbar(\ld{\g})\modc^{(\ld{N}, \psi)}.$$
    Let us denote this category by $\cC^0$. Just as in Skryabin's theorem, there is an equivalence 
    $$\ld{N} {}_\psi \backslash \cd^\hbar_{\ld{G}}/_\psi \ld{N}\modc \simeq \cd^\hbar_{\ld{G}}\modc^{(\ld{N} \times \ld{N}, \psi \times \psi)}.$$
    The Hopf algebroid structure on the pair $(U_\hbar(\ld{\g})/_\psi \ld{N}, \ld{N} {}_\psi \backslash \cd^\hbar_{\ld{G}}/_\psi \ld{N})$ defines a cosimplicial diagram
    $$\begin{tikzcd}
	{\cC^0} & {\cC^1} & {\cC^1 \otimes_{\cC^0} \cC^1} & \cdots
	\arrow[shift right, from=1-1, to=1-2]
	\arrow[shift left, from=1-1, to=1-2]
	\arrow[shift right=2, from=1-2, to=1-3]
	\arrow[shift left=2, from=1-2, to=1-3]
	\arrow[from=1-2, to=1-3]
	\arrow[shift right=3, from=1-3, to=1-4]
	\arrow[shift left=3, from=1-3, to=1-4]
	\arrow[shift right, from=1-3, to=1-4]
	\arrow[shift left, from=1-3, to=1-4]
    \end{tikzcd}$$
    The preceding discussion implies that its totalization computes the $\infty$-category of comodules over the cocommutative Hopf algebroid $(U_\hbar(\ld{\g})/_\psi \ld{N}, \ld{N} {}_\psi \backslash \cd^\hbar_{\ld{G}}/_\psi \ld{N})$. \cref{cor: loop-rot Gr and biWhit} gives an isomorphism $\H^{G_c \times S^1_\rot}_\ast(\Gr_G; k) \cong \ld{N} {}_\psi \backslash \cd^\hbar_{\ld{G}}/_\psi \ld{N}$ of cocommutative Hopf algebroids over $U_\hbar(\ld{\g})/_\psi \ld{N} \cong \H^\ast_{G_c \times S^1_\rot}(\ast; k)$, and so the totalization of the above cosimplicial diagram is equivalent to $\Loc_{G_c \times S^1_\rot}^\gr(\Gr_G; k)$.

    There are equivalences
    \begin{align*}
        \cC^0 & = U_\hbar(\ld{\g})\modc^{(\ld{N}, \psi)} \simeq \cd^\hbar_\ld{G}\modc^{(\ld{G}, \weak), (\ld{N}, \psi)}, \\
        \cC^1 & = \cd^\hbar_{\ld{G}}\modc^{(\ld{N} \times \ld{N}, \psi \times \psi)} \simeq \cC^0 \otimes_{\HC^\hbar_{\ld{G}}} \cC^0 \simeq \End_{\HC^\hbar_{\ld{G}}}(\cC^0),
    \end{align*}
    which refine to give an equivalence of cosimplicial $\infty$-categories
    $$\cC^\bull \simeq (\cC^0)^{\otimes_{\HC^\hbar_{\ld{G}}} \bull+1}.$$
    Observe that $\cC^\bull$ extends to an {augmented} cosimplicial $\infty$-category $\tilde{\cC^\bull}$ by setting $\cC^{-1} = \HC^\hbar_{\ld{G}}$, where the functor $\cC^{-1} \to \cC^0$ induced by the unique morphism $[-1] \to [0]$ in $\Deltab^+$ is given by the Kostant functor $\kappa_\hbar$.
    It is straightforward to check that both conditions in \cref{prop: cosimplicial full faithful} hold for $\tilde{\cC^\bull}$, so we find that $\Tot(\cC^\bull)$ is equivalent the localizing subcategory $(\HC^\hbar_{\ld{G}})_\reg$ of $\cC^{-1} = \HC^\hbar_{\ld{G}}$ spanned those objects on which the Kostant functor is conservative.
\end{proof}
\begin{remark}
    One can also deduce \cref{cor: reg locus quantized satake} from \cite{bf-derived-satake}, as discussed in \cite{lonergan-fourier}. This, combined with \cite[Theorem 1.2.1]{ginzburg-whittaker}, gives an alternative proof of \cref{thm: ordinary loop-rot flag} assuming the results of \cite{bf-derived-satake}. However, as mentioned in the introduction to this section, we specifically do \textit{not} want to appeal to \cite{bf-derived-satake}, since it does not have analogues in the K-theoretic or elliptic settings.
\end{remark}
\begin{remark}\label{rmk: S1-rot ordinary full faithful on gr loc}
    Just as with \cref{prop: ordinary full faithful on gr loc}, if $\HC^{\hbar,\free}_{\ld{G}}$ denotes the essential image of the pullback functor $\Rep(\ld{G}) \to \HC^\hbar_{\ld{G}}$, then there is a fully faithful embedding 
    $$(\HC^{\hbar,\free}_{\ld{G}})^\heartsuit \hookrightarrow \Loc_{G_c \times S^1_\rot}^\gr(\Gr_G; k)^\heartsuit.$$
    This can be understood as an analogue of \cite[Theorem 2]{bf-derived-satake}.
\end{remark}
\begin{remark}
    There is a Kostant functor 
    $$\kappa_\hbar: \DMod_\hbar(\ld{G} / \ld{N})^{(\ld{G} \times \ld{T}, \weak)} \to U_\hbar(\ld{\fr{t}})\modc \simeq \QCoh(\fr{t} \times \AA^1_\hbar)$$
    given by the composite
    \begin{align*}
        \DMod_\hbar(\ld{G} / \ld{N})^{(\ld{G} \times \ld{T}, \weak)} & \xrightarrow{\mathrm{forget}} \DMod_\hbar(\ld{G} / \ld{N})^{(\ld{T}, \weak)}\\
        & \xrightarrow{\mathrm{Av}_{\ld{N},\psi}} \DMod_\hbar(\ld{G} / \ld{N})^{(\ld{T}, \weak), (\ld{N}, \psi)} \\
        & \simeq \DMod_\hbar(\ld{T})^{(\ld{T}, \weak)} \simeq U_\hbar(\ld{\fr{t}})\modc.
    \end{align*}
    Using $\kappa_\hbar$, one can define an $\infty$-category $\DMod_\hbar(\ld{G} / \ld{N})^{(\ld{G} \times \ld{T}, \weak)}_\reg$.
    Just as in \cref{cor: reg locus quantized satake}, there is an equivalence
    \begin{equation}\label{eq: gen quantized ABG}
        \Loc_{T_c \times S^1_\rot}^\gr(\Gr_G; k) \simeq \DMod_\hbar(\ld{G} / \ld{N})^{(\ld{G} \times \ld{T}, \weak)}_\reg.
    \end{equation}
    Furthermore, the pushforward functor $\Loc_{T_c \times S^1_\rot}^\gr(\Gr_G; k) \to \Loc_{T_c \times S^1_\rot}^\gr(\ast; k)$ identifies with the Kostant functor $\DMod_\hbar(\ld{G} / \ld{N})^{(\ld{G} \times \ld{T}, \weak)}_\reg \to \QCoh(\fr{t} \times \AA^1_\hbar)$. The arguments in this case are slightly more subtle, though: the equivariant homology $\H^{\tilde{T}_c}_\ast(\Gr_G; \QQ)$ no longer admits an algebra structure, but it still does admit the structure of a cocommutative coalgebra over $\H^\ast_{\tilde{T}_c}(\ast; \QQ)$. In fact, $\H^{\tilde{T}_c}_\ast(\Gr_G; \QQ)$ is isomorphic as a $(\H^{G_c \times S^1_\rot}_\ast(\Gr_G; \QQ), \H^{\tilde{T}_c}_\ast(\Fl_G; \QQ))$-bimodule to the $(\cH(\GG_a, \tilde{T}, \tilde{W}), \mathbf{e} \cH(\GG_a, \tilde{T}, \tilde{W}) \mathbf{e})$-bimodule 
    $$\cH(\GG_a, \tilde{T}, \tilde{W}) \mathbf{e} \cong \ld{N} {}_\psi \backslash \cd^\hbar_{\ld{G}}/_\psi \ld{N} \otimes_{Z(U_\hbar(\ld{\g}))} \Sym(\ld{\fr{t}})[\hbar].$$
    This bimodule is denoted $\mathbb{M}_\hbar$ in \cite[Theorem 8.1.2]{ginzburg-whittaker}.
\end{remark}
\begin{remark}
    Just as \cref{cor: reg locus abg} can be viewed as a ``generic'' version of the Arkhipov-Bezrukavnikov-Ginzburg \cite{abg-iwahori-satake} equivalence
    $$\Shv^c_I(\Gr_G; k) \simeq \QCoh(\tilde{\ld{\g}}/\ld{G}),$$
    the equivalence \cref{cor: reg locus quantized satake} can be viewed as a ``generic'' version of the Bezrukavnikov-Finkelberg \cite{bf-derived-satake} equivalence
    $$\Shv^c_{G(\co) \rtimes \GG_m^\rot}(\Gr_G; k) \simeq \HC^\hbar_{\ld{G}}.$$
    Similarly, the equivalence of \cref{eq: gen quantized ABG} can be viewed as a  ``generic'' version of the quantized Arkhipov-Bezrukavnikov-Ginzburg equivalence
    $$\Shv^c_{I \rtimes \GG_m^\rot}(\Gr_G; k) \simeq \DMod_\hbar(\ld{G} / \ld{N})^{(\ld{G} \times \ld{T}, \weak)}.$$
    Unfortunately, I am not aware of a reference for this final statement, but it can be deduced from the work of Ginzburg-Riche in \cite{ginzburg-riche}.
\end{remark}
\newpage

\section{The K-theoretic story}\label{sec: KU coeff}
Our goal in this section is to prove an analogue of \cref{cor: reg locus ordinary ABG}, albeit with coefficients in $k = \KU$. Note that in this case, $\cM_{T,0} \cong T$. To do so, we need an analogue of \cref{def: additive kostant slice} and constructions surrounding it. Recall that the group $G$ (over $\cc$) is connected, almost simple, and simply-laced. We will also fix an algebraically closed field $F$,
over which the Langlands dual group $\ld{G}$ will live. When dealing with the algebraic geometry (as opposed to the topology) of $G$, we will also view it as living over $F$; since $G$ is simply-laced, it is isogenous to $\ld{G}$.
\begin{definition}\label{def: mult kostant slice}
    Let $G^\mathrm{sc}$ denote the simply-connected cover of $G$, and let $f\in G^\mathrm{sc}$ be a principal nilpotent element as defined in \cite[Theorem 4.6]{steinberg-slice}. We will denote its image under the map $G^\mathrm{sc} \to G$ also by $f$. Then the map $\GG_a \to G$ corresponding to $f$ factors through the map $\GG_a = B \to \SL_2$; we will denote the image of the standard generator $\begin{psmallmatrix}
    1 & 0\\
    1 & 1
    \end{psmallmatrix}$ under the map $\SL_2 \to G$ by $e\in G$. Let $Z_G(e)^\circ$ be the connected component of the identity in the centralizer of $e$ in $G$. Define the \textit{multiplicative Kostant slice} $\cS_\mu$ by $f\cdot Z_G(e)^\circ \subseteq G$. Since $G$ is assumed to be simply-connected, the composite 
    $$\cS_\mu \to G \to G\mmod G \cong G\mmod \ld{G} \cong T\mmod W$$
    is an isomorphism. We will often denote the inclusion of the Kostant slice by $\kappa: T\mmod W \to G$.

    The \textit{multiplicative Grothendieck-Springer resolution} $\tilde{\ld{G}}$ is defined as
    $$\tilde{\ld{G}} = B \times^{\ld{B}} \ld{G},$$
    where $\ld{B}$ acts on $B$ by conjugation. (This makes sense thanks to the assumption that $G$ is simply-laced.) There is a natural map $\tilde{\ld{G}} \to G$, given by the conjugation action of $\ld{G}$ on $B$.
    Let $\tilde{\cS}_\mu$ denote the fiber product $\tilde{\cS}_\mu \times_G \tilde{\ld{G}}$, so that the composite 
    $$\tilde{\cS}_\mu \to \tilde{\ld{G}} \to T$$
    is an isomorphism; we will denote the inclusion of $\tilde{\cS}_\mu$ as a map $\kappa: \tilde{\cS}_\mu \cong T \to \tilde{\ld{G}}$.

    As with the additive Kostant slice, we will only care about the composite $T \to \tilde{\ld{G}} \to \tilde{\ld{G}}/\ld{G}$ below, so we will also denote it by $\kappa$. If we identify $\tilde{\ld{G}}/\ld{G} \cong B/\ld{B}$, then the map $\kappa$ admits a simple description: it is the composite $f\cdot T \to B \to B/\ld{B}$. 
\end{definition}
\begin{definition}
    The stabilizer (inside $\ld{G}$) of the multiplicative Kostant slice $\cS_\mu \subseteq G^\reg$ is a closed subgroup scheme of the constant group scheme $\ld{G} \times \cS_\mu$, and will be denoted by $\ld{J}_\mu$. It will be called the \textit{multiplicative regular centralizer group scheme}; if we wish to emphasize the dependence on $G$, we will denote it by $\ld{J}_\mu(G)$.  Note that since the composite $\cS_\mu \to G^\reg \to G\mmod \ld{G}$ is an isomorphism, we may identify
    $$\ld{J}_\mu \cong \cS_\mu \times_{G/\ld{G}} \cS_\mu.$$
    Similarly, the stabilizer (inside $\ld{G}$) of the multiplicative Kostant slice $\tilde{\cS}_\mu \subseteq \tilde{\ld{G}}^\reg$ is a closed subgroup scheme of the constant group scheme $\ld{G} \times \tilde{\cS}_\mu$, and will be denoted by $\tilde{\ld{J}}_\mu$. Since $\tilde{\cS}_\mu \cong \cS_\mu \times_{G} \tilde{\ld{G}}$, we may identify
    $$\tilde{\ld{J}}_\mu \cong \ld{J}_\mu \times_{\cS_\mu} \tilde{\cS}_\mu \cong (f \cdot T) \times_{B/\ld{B}} (f\cdot T).$$
\end{definition}
The following calculation also appears in \cite{bfm}, albeit using different techniques.
\begin{theorem}\label{thm: ku hmlgy reg centr}
    There is an isomorphism of group schemes over $f \cdot T \cong T \cong \cM_{T,0}$:
    $$\spec (\pi_0 \cf_T(\Gr_G)^\vee \otimes_\Z F) \cong (f \cdot T) \times_{B/\ld{B}} (f\cdot T).$$
\end{theorem}
Just as in \cref{thm: ordinary hmlgy reg centr}, the proof of \cref{thm: ku hmlgy reg centr} will rely on two lemmas.
\begin{lemma}\label{lem: ku borel is flat}
    The projection map $\tilde{\ld{J}}_\mu \to f\cdot T$ (onto either factor) is flat.
\end{lemma}
\begin{proof}
    Like in the proof of \cref{lem: kappa for borel is flat}, it suffices, by miracle flatness, to show that the fibers of the map $\tilde{\ld{J}}_\mu \to f\cdot T$ have dimension exactly $\rank(\ld{G})$. The fiber of this map over $f \cdot x\in f \cdot T$ is the scheme 
    $$Y = \{(g,y) \in \ld{B} \times T | \Ad_g(fy) = fx\}.$$
    Observe that the image of $\Ad_g(fy)$ and $fx$ (viewed as elements of $B$) under the map $B \to T$ are $y$ and $x$; so $y=x$ in $T$, which means that $Y$ is isomorphic to the centralizer $Z_{\ld{B}}(fx)$. The dimension estimate is equivalent to the claim that $fx$ is a regular element of $G$, since this means that its centralizer has minimal dimension (namely, the rank of $G$, which is also the rank of $\ld{G}$). The desired regularity of $fx$ follows from the discussion in \cite[Remark 4.7]{steinberg-slice}. (Note that, as mentioned in \textit{loc. cit.}, the specific choice of the regular unipotent $f$ is crucial for the regularity of $fx$.)
\end{proof}
\begin{notation}
    Let $\alpha$ be a root of $\ld{G}$. Say that a point $x \in T$ is \textit{$\alpha$-generic} if $x(h_\beta) \neq 1$ for all roots $\beta\neq \alpha$. This implies that the centralizer $Z_{\ld{G}}(x)$ has semisimple rank at most $1$. Let $T_{\alpha\dreg}$ denote the $\alpha$-regular locus. Observe that $T_\reg = \bigcup_{\alpha\in \Phi} T_{\alpha\dreg} \subseteq T$ is open, with complement of codimension $2$.
\end{notation}
The proof of the next result is exactly as in \cref{lem: localization of ordinary J}.
\begin{lemma}\label{lem: ku localization of ordinary J}
    There is an isomorphism
    \begin{equation}\label{eq: ku J of centralizer}
        \tilde{\ld{J}}_\mu(\ld{G})|_{T_{\alpha\dreg}} \xar{\sim} \tilde{\ld{J}}_\mu(Z_{\ld{G}}(x)^\circ)|_{T_{\alpha\dreg}},
    \end{equation}
    where $Z_{\ld{G}}(x)$ is the centralizer of some $x\in T_{\alpha\dreg}$ which lies on the $\alpha$-hyperplane, and $Z_{\ld{G}}(x)^\circ$ denotes the connected component of the identity. 
\end{lemma}
\begin{proof}[Proof of \cref{thm: ku hmlgy reg centr}]
    The argument of \cref{thm: ordinary hmlgy reg centr} reduces us to checking that the isomorphism of \cref{thm: ku hmlgy reg centr} holds if $G$ has semisimple rank $1$, i.e., is the product of a torus with one of $\GL_2$, $\SL_2$, or $\PGL_2$. Again, it is easy to match up the contributions from the toral factors, so we will assume that $G$ is either $\GL_2$, $\SL_2$, or $\PGL_2$. In this case, we can even replace $F$ by $\Z$. 
    \begin{itemize}
        \item When $G = \GL_2$, we may identify $\tilde{\ld{J}}_\mu$ with the centralizer (in $\ld{B}$) of $\begin{psmallmatrix}
            x & 0\\
            x & y
        \end{psmallmatrix}$. It is easy to compute that $\begin{psmallmatrix}
            a & 0 \\
            c & d
        \end{psmallmatrix}$ stabilizes $\begin{psmallmatrix}
            x & 0\\
            x & y
        \end{psmallmatrix}$ if and only if $c = \frac{a-d}{x-y} \cdot x$, meaning that 
        $$\tilde{\ld{J}}_\mu \cong \spec \Z[x^{\pm 1}, y^{\pm 1}, a^{\pm 1}, d^{\pm 1}, \tfrac{a-d}{x-y}].$$
        The coproduct sends $a\mapsto a \otimes a$ and $d\mapsto d \otimes d$.
        The same argument as in \cref{thm: ordinary hmlgy reg centr} implies that
        $$\KU^{T_c}_\ast(\Omega \GL_2) \cong \Z[u^{\pm 1}, x^{\pm 1}, y^{\pm 1}, a^{\pm 1}, d^{\pm 1}, \tfrac{a-d}{x-y}].$$
        The map induced on $T$-equivariant $\KU$-homology by the inclusion $T^2 \to \GL_2$ is simply given by the inclusion of the subalgebra $\Z[u^{\pm 1}, x^{\pm 1}, y^{\pm 1}, a^{\pm 1}, d^{\pm 1}]$. The coproduct on this subalgebra (and hence, on $\KU^{T_c}_\ast(\Omega \GL_2)$) is determined by the formulas $a\mapsto a \otimes a$ and $d \mapsto d \otimes d$. It follows that $\spec \KU^{T_c}_0(\Omega \GL_2)$ is isomorphic to $\tilde{\ld{J}}_\mu$ as group schemes over $\spec \pi_0 \KU_{T_c} \cong \spec \Z[x^{\pm 1}, y^{\pm 1}]$, as desired.
        \item When $G = \SL_2$, we may identify $\tilde{\ld{J}}_\mu$ with the centralizer (in $\ld{B} \subseteq \PGL_2$) of $\begin{psmallmatrix}
            x & 0\\
            x & x^{-1}
        \end{psmallmatrix}$. An element $\begin{psmallmatrix}
            a & 0 \\
            c & 1
        \end{psmallmatrix} \in \ld{B} \subseteq \PGL_2$ stabilizes $\begin{psmallmatrix}
            x & 0\\
            x & x^{-1}
        \end{psmallmatrix}$ if and only if $c = \frac{a-1}{x-x^{-1}} \cdot x$. Therefore, 
        $$\tilde{\ld{J}}_\mu \cong \spec \Z[x^{\pm 1}, a^{\pm 1}, \tfrac{a-1}{x-x^{-1}}];$$
        the coproduct sends $a\mapsto a \otimes a$.
        
        Next, there is an isomorphism
        $$\KU^{S^1}_\ast(\Omega \SL_2) \cong \Z[u^{\pm 1}, x^{\pm 1}, a^{\pm 1}, \tfrac{a-1}{x^2-1}].$$
        This is proved exactly as in \cref{thm: ordinary hmlgy reg centr}; the role of the class $2x$ is now played by the Chern class $x^2 - 1 \in \pi_0 \KU_{S^1}$ of the weight $2$ representation of $S^1$. (Recall that the action of $S^1$ on $G_c \cong \SU(2) \cong S^3$ exhibits it as the one-point compactification of the trivial $1$-dimensional representation summed with the weight $2$ representation of $S^1$ on $\cc$.)
        The map induced on $T$-equivariant $\KU$-homology by the inclusion $S^1 \to \SU(2)$ of the maximal torus is simply given by the inclusion of the subalgebra $\Z[u^{\pm 1}, x^{\pm 1}, a^{\pm 1}]$. The coproduct on this subalgebra (and hence, on $\KU^{S^1}_\ast(\Omega \SL_2)$) is determined by the formula $a\mapsto a \otimes a$. It follows that $\spec \KU^{S^1}_0(\Omega \SL_2)$ is isomorphic to $\tilde{\ld{J}}_\mu$ as group schemes over $\spec \pi_0 \KU_{S^1} \cong \spec \Z[x^{\pm 1}]$, as desired.
        \item When $G = \PGL_2$, we may identify $\tilde{\ld{J}}_\mu$ with the centralizer (in $\ld{B} \subseteq \SL_2$) of $\begin{psmallmatrix}
            x & 0\\
            x & 1
        \end{psmallmatrix}$. An element $\begin{psmallmatrix}
            a & 0 \\
            c & a^{-1}
        \end{psmallmatrix} \in \ld{B} \subseteq \SL_2$ stabilizes $\begin{psmallmatrix}
            x & 0\\
            x & 1
        \end{psmallmatrix}$ if and only if $c = \frac{a-a^{-1}}{x-1}\cdot x$. Therefore, 
        $$\tilde{\ld{J}}_\mu \cong \spec \Z[x^{\pm 1}, a^{\pm 1}, \tfrac{a-a^{-1}}{x-1}];$$
        the coproduct sends $a\mapsto a \otimes a$.
        Again, as in the preceding cases, there is an isomorphism
        $$\KU^{S^1}_\ast(\Omega \PGL_2) \cong \Z[u^{\pm 1}, x^{\pm 1}, a^{\pm 1}, \tfrac{a-a^{-1}}{x-1}],$$
        where the coproduct sends $a\mapsto a \otimes a$. It follows that $\spec \KU^{S^1}_0(\Omega \PGL_2)$ is isomorphic to $\tilde{\ld{J}}_\mu$ as group schemes over $\spec \pi_0 \KU_{S^1} \cong \spec \Z[x^{\pm 1}]$, as desired.\qedhere
    \end{itemize}
\end{proof}

\begin{remark}\label{rmk: mult GS alternative}
    Just for posterity, let us record a more canonical variant of the calculation above for $\ld{G} = \SL_2$, which does not require picking a Borel subgroup (i.e., which does not involve identifying $\tilde{\ld{G}}/\ld{G} \cong B/\ld{B}$). If $\lambda\in \GG_m$, we denote $\lambda + \lambda^{-1}\in \AA^1$ by $f(\lambda)$. The Kostant slice $\kappa:\ld{T} \cong \GG_m \to \tilde{\SL}_2$ is the map sending $\lambda \in \GG_m$ to the pair $(x, \ell)$ with
    $$x = \begin{pmatrix}
    f(\lambda)-1 & f(\lambda)-2 \\
    1 & 1
    \end{pmatrix}, \ \ell = \left[\lambda-1: 1\right].$$
    Note that this indeed a well-defined point in $\tilde{\SL}_2$, since one can check that $x$ preserves $\ell$: the key point is the conic relation
    $$2\lambda = f(\lambda)-\sqrt{f(\lambda)^2-4}.$$
    Indeed, this calculation of $\kappa(\lambda)$ is essentially immediate from the requirement that the following diagram commutes:
    $$\xymatrix{
    \GG_m \cong \ld{T} \ar[r]^-\kappa \ar[d]_-{\lambda \mapsto f(\lambda)} & \tilde{\SL}_2 \ar[d]\\
    \AA^1 \cong \ld{T}\mmod W \ar[r]^-\kappa_-{\lambda\mapsto \begin{psmallmatrix}
    \lambda-1 & \lambda-2 \\
    1 & 1
    \end{psmallmatrix}} & \SL_2.
    }$$
    Moreover, the $\SL_2$-action on $\tilde{\SL}_2$ sends $g\in \SL_2$ and $(x,\ell)$ to $(\Ad_g(x), g\ell)$. If $g = \begin{psmallmatrix}
    a & b \\
    c & d
    \end{psmallmatrix}$, we directly compute that $\Ad_g(x) = x$ if and only if $b = c(f(\lambda) - 2)$ and $a-d = (f(\lambda) - 2)c$, in which case $g$ also preserves $\ell$. Therefore, $g = \begin{psmallmatrix}
    (f(\lambda) - 2)c + d & (f(\lambda)-2)c \\
    c & d
    \end{psmallmatrix}$ for $c,d\in k$. In order for $\det(g) = 1$, we need 
    $$d^2 + c(f(\lambda)-2)(d-c) = 1.$$
    Both $x$ and $g$ can be simultaneously diagonalized (if $f(\lambda) \neq \pm 2$); note that $\lambda+\lambda^{-1}$ is an eigenvalue of $x$. If $t$ is an eigenvalue of $g$, then we have $c = \tfrac{t-t^{-1}}{\lambda - \lambda^{-1}}$ and $d = \tfrac{t^2\lambda + 1}{t(\lambda+1)}$.
    When $k$ is not of characteristic $2$, this shows that 
    $$\GG_m \times_{\tilde{\SL}_2/\SL_2} \GG_m \cong k[\lambda^{\pm 1}, t^{\pm 1}, \tfrac{t-t^{-1}}{\lambda - \lambda^{-1}}].$$
    This in turn implies that
    $$\GG_m \times_{\tilde{\SL}_2/\PGL_2} \GG_m \cong k[\lambda^{\pm 1}, t^{\pm 2}, \tfrac{t^2-1}{\lambda - \lambda^{-1}}],$$
    as desired.
\end{remark}
There is \textit{another} choice of slice when $G$ is simply-connected; the calculation of \cref{thm: ku hmlgy reg centr} continues to hold for it, too, as we now illustrate in the example of $\SL_2$.
\begin{definition}[Steinberg slice]
    Let $G$ be a simply-connected semisimple algebraic group. Given $w\in W$, let $N_w = N \cap w^{-1} N^- w$, so that $N_w = \prod_{\alpha\in \Phi_w} U_\alpha$, where $\Phi_w$ is the set of roots made negative by $w$. Let $w = \prod_{\alpha\in \Delta} s_\alpha\in W$ be a Coxeter element, and let $\dot{w}$ be a lift of $w$ to $N_G(T)$. Define the Steinberg slice $\Sigma = \dot{w} N_w\subseteq G$. Then \cite{steinberg-slice} proved/stated that the composite $\Sigma \to G \to G\mmod G \cong T\mmod W$ is an isomorphism. 
    Let $\tilde{\Sigma}$ denote the fiber product $\Sigma\times_G \tilde{\ld{G}}$, so that the composite $\tilde{\Sigma} \to \tilde{\ld{G}} \to T$ is an isomorphism. We will denote the inclusion of $\tilde{\Sigma}$ by $\sigma: T \to \tilde{\ld{G}}$.
\end{definition}
\begin{observe}
    We will illustrate the calculation of $T \times_{\tilde{\ld{G}}/\ld{G}} T$ (with $T$ mapping to $\tilde{\ld{G}}$ by $\sigma$) when ${G} = \SL_2$.
    View a point in $\tilde{\ld{G}}$ as a pair $(x\in \SL_2, \ell\subseteq \cc^2)$ such that $x$ preserves $\ell$. The Steinberg slice $\sigma:\GG_m \to \tilde{\SL}_2$ is the map sending $\lambda \in \GG_m$ to the pair $(x, \ell)$ with 
    $$x = \begin{pmatrix}
    \lambda + \lambda^{-1} & -1 \\
    1 & 0
    \end{pmatrix}, \ \ell = \left[\lambda: 1\right].$$
    Note that this indeed a well-defined point in $\tilde{\SL}_2$, since one can check that $x$ preserves $\ell$. This calculation of $\sigma(\lambda)$ is essentially immediate from the requirement that the following diagram commutes:
    $$\xymatrix{
    \GG_m \cong \ld{T} \ar[r]^-\sigma \ar[d]_-{\lambda \mapsto \lambda + \lambda^{-1}} & \tilde{\SL}_2 \ar[d]\\
    \AA^1 \cong \ld{T}\mmod W \ar[r]^-\sigma_-{\lambda\mapsto \begin{psmallmatrix}
    \lambda & -1 \\
    1 & 0
    \end{psmallmatrix}} & \SL_2.
    }$$
    Moreover, the $\SL_2$-action on $\tilde{\SL}_2$ sends $g\in \SL_2$ and $(x,\ell)$ to $(\Ad_g(x), g\ell)$. If $g = \begin{psmallmatrix}
    a & b \\
    c & d
    \end{psmallmatrix}$, one can directly compute that $g$ commutes with $\begin{psmallmatrix}
    \lambda + \lambda^{-1} & -1 \\
    1 & 0
    \end{psmallmatrix}$ if and only if $a = c(\lambda + \lambda^{-1}) + d$ and $b=-c$. Therefore, $g = \begin{psmallmatrix}
    c(\lambda + \lambda^{-1}) + d & -c \\
    c & d
    \end{psmallmatrix}$ for $c,d\in k$. In order for $\det(g) = 1$, we need 
    $$c^2 + d^2 + cd(\lambda + \lambda^{-1}) = 1.$$
    As long as $\lambda\neq \pm 1$, both $x$ and $g$ can be simultaneously diagonalized by $\begin{psmallmatrix}
    \lambda & \lambda^{-1} \\
    1 & 1
    \end{psmallmatrix}$: the diagonalization of $x$ is $\begin{psmallmatrix}
    \lambda & 0 \\
    0 & \lambda^{-1}
    \end{psmallmatrix}$, and the diagonalization of $g$ is $\begin{psmallmatrix}
    c\lambda + d & 0 \\
    0 & c\lambda^{-1} + d
    \end{psmallmatrix}$. If $t = c\lambda+d$, then $c\lambda^{-1}+d = t^{-1}$ by the above determinant relation. We also have that $a = t - \tfrac{\lambda(t-t^{-1})}{\lambda - \lambda^{-1}}$ and $c = \tfrac{t-t^{-1}}{\lambda - \lambda^{-1}}$. This shows that 
    $$\GG_m \times_{\tilde{\SL}_2/\SL_2} \GG_m \cong \spec k[\lambda^{\pm 1}, t^{\pm 1}, \tfrac{t-t^{-1}}{\lambda - \lambda^{-1}}],$$
    and hence that
    $$\GG_m \times_{\tilde{\SL}_2/\PGL_2} \GG_m \cong \spec k[\lambda^{\pm 1}, t^{\pm 2}, \tfrac{t^2-1}{\lambda - \lambda^{-1}}],$$
    as desired.
\end{observe}

\begin{corollary}\label{cor: ku reg locus ordinary ABG}
    There is an $F$-linear equivalence
    $$\Loc_{T_c}^\gr(\Gr_G; \KU) \otimes_\Z F \simeq \QCoh(\tilde{\ld{G}}^\reg/\ld{G}).$$
    Furthermore, the pushforward functor $\Loc_{T_c}^\gr(\Gr_G; \KU) \to \Loc_{T_c}^\gr(\ast; \KU)$ identifies with the pullback functor $\kappa^\ast: \QCoh(\tilde{\ld{G}}^\reg/\ld{G}) \to \QCoh(T)$.
\end{corollary}
\begin{proof}
    By definition, $\Loc_{T_c}^\gr(\Gr_G; \KU)$ is equivalent to the category of comodules over $\pi_0 \cf_T(\Gr_G)^\vee = \KU_0^T(\Gr_G)$ in the category of $\pi_0 \KU_{T_c}$-modules. By \cref{thm: ku hmlgy reg centr}, it can be identified the category of quasicoherent sheaves on the quotient stack $(f \cdot T)/\tilde{\ld{J}}_\mu$. We may view $\tilde{\ld{J}}_\mu$ as a closed subgroup scheme of the constant group scheme $\ld{B} \times (f \cdot T)$. This gives an isomorphism
    $$(f \cdot T)/\tilde{\ld{J}}_\mu \cong \ld{B} \backslash (\ld{B} \times (f \cdot T))/\tilde{\ld{J}}_\mu.$$
    It follows from Steinberg's work in \cite{steinberg-slice} that the $\ld{B}$-orbit of $f \cdot T$ inside $B$ is precisely the regular locus $B^\reg$. Since $\tilde{\ld{J}}_\mu$ is definitionally the stabilizer of $f \cdot T \subseteq B$, the quotient $(\ld{B} \times (f \cdot T))/\tilde{\ld{J}}_\mu$ is isomorphic to $B^\reg$; so there is an isomorphism $(f \cdot T)/\tilde{\ld{J}}_\mu \cong B^\reg/\ld{B}$.
    To finish, note that $\tilde{\ld{G}}^\reg/\ld{G} \cong B^\reg/\ld{B}$.
\end{proof}
Similarly, there is an $F$-linear equivalence
$$\Loc_{\ld{T}_c}^\gr(\Gr_G; \KU) \otimes_\Z F \simeq \QCoh(\tilde{\ld{G}}^{',\reg}/\ld{G}),$$
where $\tilde{\ld{G}}^{'}$ is $\ld{G} \times^{\ld{B}} \ld{B}$, with $\ld{B}$ acting on itself by conjugation. Note that $\tilde{\ld{G}}^{',\reg}/\ld{G} \cong \ld{B}^\reg/\ld{B}$ is an open substack of the stack $\ld{B}/\ld{B} \cong \Map(B\Z, B\ld{B})$ of $\ld{B}$-bundles on the circle $S^1 = B\Z$.

The equivalence of \cref{cor: ku reg locus ordinary ABG} is in fact symmetric monoidal for the convolution tensor structure on $\Loc_{T_c}^\gr(\Gr_G; \KU)$ (described in \cref{rmk: loc gr convolution tensor}) and the standard tensor product on $\QCoh(\tilde{\ld{G}}^\reg/\ld{G})$.

\begin{remark}\label{rmk: Loc G for KU}
    It can be shown that if $G$ has torsion-free fundamental group, there is an $F$-linear equivalence
    $$\Loc_{G_c}^\gr(\Gr_G; \KU) \otimes_\Z F \simeq \QCoh(G^\reg/\ld{G}).$$
    Just as in \cref{sec: degenerations}, the left-hand side is defined as
    $$\Loc_{G_c}^\gr(\Gr_G; \KU) = \coLMod_{\pi_0(\cf_G(\Gr_G)^\vee)}(\QCoh(T\mmod W)).$$
    Note that this is a sensible definition since $\pi_\ast \cf_G(\Gr_G)^\vee$ is concentrated in even degrees.
    Furthermore, the pushforward functor $\Loc_{G_c}^\gr(\Gr_G; \KU) \to \Loc_{G_c}^\gr(\ast; \KU)$ identifies with the pullback functor $\kappa^\ast: \QCoh(G^\reg/\ld{G}) \to \QCoh(T\mmod W)$. The proof of the displayed equivalence is quite similar to that of \cref{cor: ku reg locus ordinary ABG}, and in fact can be deduced from it using the observation that $\pi_0(\cf_G(\Gr_G)^\vee) = \pi_0(\cf_T(\Gr_G)^\vee)^W$ and that the natural map $\tilde{\ld{G}}^\reg \to G^\reg$ is a (ramified) $W$-cover. The first statement uses that $G$ has torsion-free fundamental group, and the second is a multiplicative version of Grothendieck-Springer theory.
\end{remark}
\begin{remark}
    In \cite[Section 3.7]{ku-rel-langlands}, we study a variant of \cref{cor: ku reg locus ordinary ABG}, where $\KU$ is replaced by \textit{connective} complex K-theory $\ku$; that is, $\Loc_{T_c}^\gr(\Gr_G; \KU)$ is replaced by $\Loc_{T_c}^\gr(\Gr_G; \ku)$. On the Langlands dual side, this has the effect of replacing $\tilde{\ld{G}}^\reg/\ld{G}$ by the $1$-parameter family over $\spec(\pi_\ast(\ku))/\GG_m \cong \AA^1/\GG_m$ whose generic fiber is $\tilde{\ld{G}}^\reg/\ld{G}$, and whose special fiber is $\tilde{\ld{\g}}^\reg/\ld{G}$.
\end{remark}
\begin{remark}\label{rmk: non-simply-laced ku}
    There is a variant of \cref{cor: ku reg locus ordinary ABG} if $G$ is not simply-laced, but it is more complicated to state. Let us just give the analogue of \cref{thm: ku hmlgy reg centr}. Suppose $G$ is not simply-laced, and let $T$ be a maximal torus of $G$; then $\ld{\g}$ is the fixed point subalgebra $\ld{\fr{h}}^\tau$ of an finite-order outer automorphism $\tau$ of a simply-laced Lie algebra $\ld{\fr{h}}$. Let $H$ denote the simply-connected simply-laced group corresponding to the Langlands dual $\fr{h}$, and let $T_H$ denote its maximal torus.  Then we may identify the fixed subset $\bX^\ast(T')^\tau$ with $\bX^\ast(T)$. If $n$ denotes the order of $\tau$, there is an action of $\Z/n$ on $T\pw{t}$, $G\pw{t}$, and $G\ls{t}$, given by $\tau$ on $T$ and $G$, and $t\mapsto \zeta_n \tau$ for a primitive $n$th root of unity $\zeta_n$. The appropriate replacement of $\pi_0 \cf_T(\Gr_G)^\vee$ in this case is $\pi_0 \cf_{T\pw{t}^{\Z/n}}(G_\ad\ls{t}^{\Z/n}/G_\ad\pw{t}^{\Z/n})^\vee$. The analogue of \cref{thm: ku hmlgy reg centr} (see \cite[Theorem 3.9]{finkelberg-tsymbaliuk}) states that this algebra is isomorphic to the stabilizer $\cS_\mu \times_{\ld{G}/\ld{G}} \cS_\mu$. 
\end{remark}
The map $\tilde{\ld{G}}^\reg/\ld{G} \to B\ld{G}$ defines a functor
\begin{equation}\label{eq: Rep G^ to KU loc}
    \Rep(\ld{G}) \to \QCoh(\tilde{\ld{G}}^\reg/\ld{G}) \simeq \Loc_{T_c}^\gr(\Gr_G; \KU) \otimes_\Z F.
\end{equation}
More generally, the map $\tilde{\ld{G}}^\reg/\ld{G} \to B\ld{G} \times B\ld{T}$ defines a functor
\begin{equation}\label{eq: Rep T^ x G^ to KU loc}
    \Rep(\ld{G} \times \ld{T}) \to \QCoh(\tilde{\ld{G}}^\reg/\ld{G}) \simeq \Loc_{T_c}^\gr(\Gr_G; \KU) \otimes_\Z F.
\end{equation}
If $V \in \Rep(\ld{G})$, let $\cS_\KU(V)$ denote the corresponding object of $\Loc_{T_c}^\gr(\Gr_G; \KU) \otimes_\Z F$. The same argument as in \cref{prop: ordinary realizing minuscule reps} shows the following, which says that $\cS_\KU(V) \in \Loc_{T_c}^\gr(\Gr_G; \KU)$ is the associated graded of a particular object $\cf_\lambda \in \Loc_{T_c}(\Gr_G; \KU)$ if $V$ is a minuscule $\ld{G}$-representation. 
\begin{prop}\label{prop: KU realizing minuscule reps}
    Let $\lambda_\bull = (\lambda_1, \cdots, \lambda_n)$ be a tuple of dominant minuscule weights of $\ld{G}$, let $|\lambda_\bull| = \sum_i \lambda_i$, and let $\ol{\Gr_G^{\lambda_\bull}}$ denote the corresponding \textit{convolution variety}. Let $\cf_{\lambda_\bull}$ denote the pushforward of the constant sheaf along the canonical map $q: \ol{\Gr_G^{\lambda_\bull}} \to \ol{\Gr_G^{|\lambda|}} \subseteq \Gr_G$. If $V_{\lambda_i}$ denotes the irreducible representation of $\ld{G}$ with highest weight $\lambda_i$, then there is an isomorphism $\cS_\KU(\bigotimes_i V_{\lambda_i}) \cong \cf_{\lambda_\bull}^\gr$.
\end{prop}
It would be very interesting to understand whether \cref{prop: KU realizing minuscule reps} can be extended to other non-minuscule irreducible representations. As in \cref{rmk: ordinary action on minuscule}, if $\lambda$ is a dominant minuscule weight of $\ld{G}$, then the coaction of $\pi_0 \cf_T(\Gr_G)^\vee$ on $\pi_0 \cf_T(G/P_\lambda)$ defines a homomorphism 
\begin{equation}
    \spec \pi_0 \cf_T(\Gr_G)^\vee \to \GL(\pi_0 \cf_T(G/P_\lambda))
\end{equation}
of group schemes over $T$, where $\GL(\pi_0 \cf_T(G/P_\lambda))$ denotes the group scheme of $\co_T$-linear automorphisms of the vector bundle $\pi_0 \cf_T(G/P_\lambda)$. Under the isomorphisms of \cref{thm: ku hmlgy reg centr} and \cref{prop: KU realizing minuscule reps}, this homomorphism factors as the composite
\begin{equation}\label{eq: KU factorization action on minuscule}
    \tilde{\ld{J}}_\mu \to \ld{G} \times T \to \GL(V_\lambda) \times T,
\end{equation}
where the second map describes the $\ld{G}$-action on $V_\lambda$.
Similar statements hold with $\tilde{\ld{J}}_\mu$ replaced by $\ld{J}_\mu$ and $\pi_0 \cf_T(G/P_\lambda)$ replaced by $\pi_0 \cf_G(G/P_\lambda) \cong \pi_0 \KU_{L_\lambda}$ (where $L_\lambda$ is the Levi quotient of $P_\lambda$).

\cref{thm: ku hmlgy reg centr} has several applications. Here is one, following the same proof as in \cref{prop: ordinary gelfand-graev}; it gives a \textit{multiplicative} version of the Gelfand-Graev action on the affine closure $\ol{T^\ast(\ld{G}/\ld{N})}$:
\begin{prop}[Multiplicative Gelfand-Graev action]\label{prop: ku gelfand-graev}
    The natural action of $\ld{G} \times \ld{T}$ on the affine closure $\ol{\ld{G} \times^{\ld{N}}B}$ extends to an action of $\ld{G} \times (W \rtimes \ld{T})$, where $W$ is the Weyl group.
\end{prop}
In the following, we will write $\ol{T^\ast_{\GG_m}(\ld{G}/\ld{N})}$ to denote the affine closure of the ``multiplicative'' cotangent bundle ${\ld{G} \times^{\ld{N}}B}$.
Unlike with \cref{prop: ordinary gelfand-graev}, \cref{prop: ku gelfand-graev} does require $G$ to be simply-laced; otherwise $\ol{T^\ast_{\GG_m}(\ld{G}/\ld{N})}$ would not even be well-defined. The moment map $\ol{T^\ast_{\GG_m}(\ld{G}/\ld{N})} \to G$ is $W$-equivariant for the trivial action on the target. There is a commutative diagram
$$\xymatrix{
\tilde{\ld{G}} \ar@{^(->}[r] \ar[dr] & \ol{T^\ast_{\GG_m}(\ld{G}/\ld{N})}/\ld{T} \ar[d] \\
& G
}$$
which relates $\ol{T^\ast_{\GG_m}(\ld{G}/\ld{N})}$ to the multiplicative Grothendieck-Springer resolution; and via this diagram, the multiplicative Gelfand-Graev action is closely related to the Weyl action in trigonometric/multiplicative Springer theory.
\begin{remark}
    The proof of \cref{prop: ku gelfand-graev} generalizes to show that if $\ld{P} \subseteq \ld{G}$ is a parabolic subgroup with Levi quotient $\ld{L}$ and unipotent radical $U_{\ld{P}}$, then the natural action of $\ld{G} \times \ld{L}$ on the affine closure $\ol{\ld{G} \times^{U_{\ld{P}}} P}$ extends to an action of $\ld{G} \times (W_L \rtimes \ld{L})$, where $W_L = N_{\ld{G}}(\ld{L})/\ld{L}$ is the Weyl group.
\end{remark}
\begin{example}\label{ex: Z/2 multiplicative symplectic fourier}
    Let us make the above action explicit in the example of $\ld{G} = \SL_2$ (so $W = \Z/2$). The group $B$ in this case is contained in $\PGL_2$, and can be chosen to be represented by matrices of the form $\begin{psmallmatrix}
        x & y \\
        0 & 1
    \end{psmallmatrix}$. The action of $\begin{psmallmatrix}
        1 & n \\
        0 & 1
    \end{psmallmatrix} \in \ld{N}$ on $\ld{G} \times B$ sends
    $$\begin{psmallmatrix}
        a & b \\
        c & d
    \end{psmallmatrix} \mapsto \begin{psmallmatrix}
        a & an + b \\
        c & cn + d
    \end{psmallmatrix}, \ \begin{psmallmatrix}
        x & y \\
        0 & 1
    \end{psmallmatrix} \mapsto \begin{psmallmatrix}
        x & y-n(x-1) \\
        0 & 1
    \end{psmallmatrix}.$$
    As explained in \cite[Remark 5.1.19]{ku-rel-langlands}, this means that the $\GG_a$-action fixes $a,c,x$, $B := ay + (x-1) b$, and $D = cy + (x-1) d$. There is a single relation between these classes, given by
    $$aD - cB = x-1.$$
    Let us relabel these variables so that $u = \begin{psmallmatrix}
        u_1 \\
        u_2
    \end{psmallmatrix} = \begin{psmallmatrix}
        a \\
        c
    \end{psmallmatrix}$ and $v = (v_1, v_2) = (D, -B)$. Since $x$ must be invertible, it follows that the affine closure $\ol{\SL_2 \times^{\GG_a} B}$ is given by the complement of the hypersurface $1 + \pdb{u,v}$ in $T^\ast(\AA^2)$. This is Van den Bergh's multiplicative quiver variety $\cB(U, V)$ from \cite{van-den-bergh-double-poisson}, specialized to the case when the vector spaces $U, V$ are $\AA^2, \AA^1$. An elementary analysis as in \cref{rmk: Z/2 symplectic fourier} shows that the $\Z/2$-action of \cref{prop: ku gelfand-graev} is given on $\ol{\SL_2 \times^{\GG_a} B} \subseteq T^\ast(\AA^2)$ by the formula
    $$\left(\begin{psmallmatrix}
        u_1\\
        u_2
    \end{psmallmatrix}, (v_1, v_2)\right) \mapsto \left(\tfrac{1}{1 + \pdb{u,v}}\begin{psmallmatrix}
        -v_2 \\
        v_1
    \end{psmallmatrix}, (u_2, -u_1)\right).$$
    In particular, it can be viewed as a multiplicative version of the symplectic Fourier transform.
\end{example}
\begin{remark}
    The multiplicative symplectic Fourier transform of \cref{ex: Z/2 multiplicative symplectic fourier} is related to another, more geometric, Fourier-type transform, as we now describe. Let $\ell$ be a (complex) line. Recall from \cite{beilinson-glue-perverse} that the ($1$-)category $\Perv(\ell)$ of perverse sheaves on $\ell$ with respect to the stratification by $0\in \ell$ and its complement is equivalent to the category of diagrams of the form
    \begin{equation}\label{eq: perverse on disk}   
    \begin{tikzcd}
	    X & Y
	\arrow["u"', shift right, from=1-1, to=1-2]
	\arrow["v"', shift right, from=1-2, to=1-1]
    \end{tikzcd}
    \end{equation}
    with $X$ and $Y$ being vector spaces, such that $\id_Y+uv$ (and therefore $\id_X + vu$) is invertible. This equivalence sends $\cf \in \Perv(\ell)$ to its spaces of nearby and vanishing cycles at $0\in \ell$ (and the maps $u,v$ arise via monodromy). The Fourier-Sato transform (see \cite[Definition 3.7.8]{kashiwara-schapira}) gives an equivalence $\Perv(\ell) \to \Perv(\ell^\ast)$, and one can check that it sends an object \cref{eq: perverse on disk} to the object
    $$\begin{tikzcd}
	    Y & X.
	\arrow["-v"', shift right, from=1-1, to=1-2]
	\arrow["u(\id+vu)^{-1}"', shift right, from=1-2, to=1-1]
    \end{tikzcd}$$
    \cref{ex: Z/2 multiplicative symplectic fourier} defines a morphism from $\ol{\SL_2 \times^{\GG_a} B}$ to the moduli of isomorphism classes of objects of $\Perv(\ell)$ (where $X = \AA^2$ and $Y = \AA^1$); this morphism intertwines the multiplicative symplectic Fourier transform with the Fourier-Sato transform.
\end{remark}
We also have the following analogue of \cref{prop: ordinary full faithful on gr loc}, whose proof is exactly the same (one only needs to note that $\tilde{\ld{G}}^\reg \hookrightarrow \tilde{\ld{G}}$ has complement of codimension $2$, and similarly for $G^\reg \hookrightarrow G$).
\begin{prop}\label{prop: KU full faithful on gr loc}
    Let $\Loc_{T_c}^\gr(\Gr_G; \KU)^\heart$ denote the heart of the $t$-structure on $\Loc_{T_c}^\gr(\Gr_G; \KU) = \coMod_{\pi_0(\cf_T(\Gr_G))^\vee}(\QCoh(T))$ coming from the standard (homological truncation) $t$-structure on $\QCoh(T)$. 
    Then, the composite functor
    $$\Loc_{T_c}^\gr(\Gr_G; \KU) \otimes_\Z F \simeq \QCoh(\tilde{\ld{G}}^\reg/\ld{G}) \to \QCoh(\ld{G}\backslash \ol{T^\ast_{\GG_m}(\ld{G}/\ld{N})}/\ld{T})$$
    is $t$-exact, and on hearts, it restricts to a fully faithful functor on the essential image of \cref{eq: Rep T^ x G^ to KU loc}. Furthermore, this functor is $W$-equivariant for the natural action of $W = \N_{G_c}(T_c)/T_c$ on the left-hand side and the Gelfand-Graev action of \cref{prop: ku gelfand-graev} on the right-hand side.

    Similarly, suppose $G$ has torsion-free fundamental group, and let $\Loc_{G_c}^\gr(\Gr_G; \KU)^\heart$ denote the heart of the $t$-structure on $\Loc_{G_c}^\gr(\Gr_G; \KU) = \coMod_{\pi_0(\cf_G(\Gr_G))^\vee}(\QCoh(T\mmod W))$ coming from the standard (homological truncation) $t$-structure on $\QCoh(T\mmod W)$. 
    Then, the composite functor
    $$\Loc_{G_c}^\gr(\Gr_G; \KU) \otimes_\Z F \simeq \QCoh(G^\reg/\ld{G}) \to \QCoh(G/\ld{G})$$
    is $t$-exact, and on hearts, it restricts to a fully faithful functor on the essential image of the functor $\Rep(\ld{G}) \to \Loc_{G_c}^\gr(\Gr_G; \KU) \otimes_\Z F$ (analogous to \cref{eq: Rep G^ to KU loc}).
\end{prop}
\cref{prop: KU full faithful on gr loc} gives an analogue of \cite[Theorem 4]{bf-derived-satake}: namely, if $\QCoh_\free(G/\ld{G})$ denotes the essential image of the pullback functor $\Rep(\ld{G}) \to \QCoh(G/\ld{G})$, then there is a fully faithful embedding 
$$\QCoh_\free(G/\ld{G})^\heartsuit \hookrightarrow \Loc_{G_c}^\gr(\Gr_G; \KU)^\heartsuit \otimes_\Z F.$$
Similarly, if $\QCoh_\free(\ld{G}\backslash \ol{T^\ast_{\GG_m}(\ld{G}/\ld{N})}/\ld{T})$ denotes the essential image of the pullback functor $\Rep(\ld{G} \times \ld{T}) \to \QCoh(\ld{G}\backslash \ol{T^\ast_{\GG_m}(\ld{G}/\ld{N})}/\ld{T})$, then there is a fully faithful embedding 
$$\QCoh_\free(\ld{G}\backslash \ol{T^\ast_{\GG_m}(\ld{G}/\ld{N})}/\ld{T})^\heartsuit \hookrightarrow \Loc_{T_c}^\gr(\Gr_G; \KU)^\heartsuit \otimes_\Z F.$$
This implies the following result.
\begin{corollary}\label{cor: ku minuscule equivalence}
    Let $\QCoh_{\free}(G/\ld{G})^{\min,\heartsuit}$ denote the essential image of $\Rep_\min(\ld{G})$ under the pullback functor $\Rep(\ld{G})^\heartsuit \to \QCoh(G/\ld{G})^\heartsuit$ (so it is the entirety of $\QCoh(G/\ld{G})^\heartsuit$ if $F$ has characteristic zero and $\ld{G}$ is not of type $E_8$). Similarly, let $(\Loc_{G_c}^\gr(\Gr_G; \KU)^{\heartsuit} \otimes_\Z F)^\min$ denote the idempotent completion of the subcategory of $\Loc_{G_c}^\gr(\Gr_G; \KU)^\heartsuit \otimes_\Z F$ spanned by $\cf_{\lambda_\bull}^\gr$ ranging over sequences $\lambda_\bull$ of minuscule highest weights. Then there is an equivalence
    $$\QCoh_\free(G/\ld{G})^{\min,\heartsuit} \simeq (\Loc_{G_c}^\gr(\Gr_G; \KU)^{\heartsuit} \otimes_\Z F)^\min.$$
\end{corollary}
There is a similar equivalence 
$$(\Loc_{T_c}^\gr(\Gr_G; \KU)^{\heartsuit} \otimes_\Z F)^\min \simeq \QCoh_\free(\ld{G}\backslash \ol{T^\ast_{\GG_m}(\ld{G}/\ld{N})}/\ld{T})^{\min,\heartsuit},$$
where these categories are defined analogously by idempotent completion. 

Note, again, that the category $(\Loc_{G_c}^\gr(\Gr_G; \KU)^{\heartsuit} \otimes_\Z F)^\min$ is the heart of a degeneration, in the sense of \cref{sec: degenerations}, of the similarly-defined category $(\Loc_{G_c}(\Gr_G; \KU) \otimes_\KU F[u^{\pm 1}])^\min$. (In particular, \cref{cor: ku minuscule equivalence} gives an equivalence between the purely algebraically defined category $\QCoh_\free(G/\ld{G})^{\min,\heartsuit}$ and a degeneration of the purely topologically defined category $(\Loc_{G_c}(\Gr_G; \KU) \otimes_\KU F[u^{\pm 1}])^\min$.) If $\lambda_\bull$ and $\mu_\bull$ are two sequences of dominant minuscule weights of $\ld{G}$, there is an equivalence of $\KU$-modules
$$\Map_{(\Loc_{G_c}(\Gr_G; \KU) \otimes_\KU F[u^{\pm 1}])^\min}(\cf_{\lambda_\bull}, \cf_{\mu_\bull}) \simeq \cf_{G_c}(\ol{\Gr_G^{\lambda_\bull}} \times_{\Gr_G} \ol{\Gr_G^{\mu_\bull}}),$$
so that the category $(\Loc_{G_c}(\Gr_G; \KU) \otimes_\KU F[u^{\pm 1}])^\min$ compares to the category from \cite[Section 3.5]{cautis-kamnitzer}.

As with \cref{prop: ordinary full faithful on gr loc}, the existence of the $t$-structure on $\Loc_{T_c}^\gr(\Gr_G; \KU)$ from \cref{prop: KU full faithful on gr loc} may at first glance perhaps be a bit surprising, since $\KU$ is a \textit{$2$-periodic} $\Eoo$-ring. Again, this periodicity prohibits $\Loc_{T_c}(\Gr_G; \KU)$ itself from having a $t$-structure; but the $\infty$-category $\Loc_{T_c}^\gr(\Gr_G; k)$ itself has both a \textit{homological} shift operation and a (periodic) \textit{weight} shifting operation. The homological shift on $\Loc_{T_c}^\gr(\Gr_G; \KU)$ is no longer periodic, and it is therefore reasonable to equip this $\infty$-category with a $t$-structure.

We now turn to the question of the analogue of \cref{cor: ku reg locus ordinary ABG} if $\KU$ is replaced by \textit{real} K-theory $\KO$. (Recall the definition of $\Loc_{T_c}^\gr(\Gr_G; \KO)$ from \cref{def: KO graded Loc}.) We begin by constructing a $\Z/2$-action on $\tilde{\ld{G}}/\ld{G} \cong B/\ld{B}$.
\begin{lemma}\label{lem: gamma from T to B^}
    There is a map $\gamma: T \to \ld{B}$ such that if $x\in T$, then $\Ad_{\gamma(x)}$ sends $(fx)^{-1}$ to $fx^{-1}$; moreover, $\gamma(x)$ squares to the identity.
\end{lemma}
\begin{proof}
    This follows from the fact that $(fx)^{-1}$ and $fx^{-1}$ in $B$ both have image $x^{-1}$ under the map $B \to B \mmod \ld{B} \cong T$.
\end{proof}
\begin{definition}\label{def: cplx conj on B mod B}
    Denote by $\chi$ the map $B \to B \mmod \ld{B} \cong T$. There is an involution $\theta$ of $B$ sending $x\mapsto \Ad_{\gamma(\chi(x))}(x^{-1})$, and similarly an involution $\theta$ of the constant group scheme $\ld{B} \times T$ over $T$ sending $(g, y) \mapsto (\Ad_{\gamma(y)}(g), y^{-1})$. This defines an involution $\theta$ of $B/\ld{B}$, and hence a $\Z/2$-action on it.
\end{definition}
\begin{example}
    Suppose $G = \GL_2$ or $\SL_2$. Then one can take for $\gamma$ the constant map $T \cong \GG_m^2 \to \ld{B}$ sending $(x,y)\mapsto \begin{psmallmatrix}
        1 & 0\\
        1 & -1
    \end{psmallmatrix}$. If $G = \PGL_2$, one can simply multiply $\gamma$ by a primitive fourth root of unity to get an element of $\ld{G} = \SL_2$. If $G = \GL_3$, then one can take for $\gamma$ the map $T \cong \GG_m^3 \to \ld{B}$ sending $(x,y,z)\mapsto \begin{psmallmatrix}
        1 & 0 & 0\\
        1 & -1 & 0 \\
        0 & 0 & zy^{-1}
    \end{psmallmatrix}$.
\end{example}
It is easy to show:
\begin{lemma}
    The involution $\theta: B/\ld{B} \to B/\ld{B}$ is isomorphic to the map induced by inversion on $B$.
\end{lemma}
\begin{prop}\label{prop: cplx conj KU and B mod B^}
    The $\Z/2$-action by $\theta$ on $\tilde{\ld{G}}/\ld{G} \cong B/\ld{B}$ restricts to an action on $\tilde{\ld{G}}^\reg/\ld{G}$, and under the equivalence of \cref{cor: ku reg locus ordinary ABG}, it identifies with the $\Z/2$-action via complex conjugation on equivariant $\KU$. In particular, there is an equivalence
    $$\Loc_{T_c}^\gr(\Gr_G; \KO) \otimes_\Z F \simeq \QCoh((\tilde{\ld{G}}^\reg/\ld{G})/\pdb{\theta}).$$
\end{prop}
\begin{proof}
    It follows from \cref{def: cplx conj on B mod B} that there is a commutative diagram
    $$\xymatrix{
    T \ar[r]^-{x \mapsto x^{-1}} \ar[d]_-\kappa & T \ar[d]^-\kappa \\
    B \ar[r]_-\theta & B.
    }$$
    Therefore, $\theta$ induces an automorphism of $T \times_{B^\reg/\ld{B}} T$, and it suffices (by the proof of \cref{cor: ku reg locus ordinary ABG}) to show that under the isomorphism 
    \begin{equation}\label{eq: for proof hmlgy cplx conj}
        \spec \pi_0 \cf_T(\Gr_G)^\vee \cong T \times_{B^\reg/\ld{B}} T
    \end{equation}
    of \cref{thm: ku hmlgy reg centr}, the action of $\theta$ corresponds to the action of complex conjugation on equivariant K-theory. Let $T^\gen \subseteq T$ denote the complement of the union of all hypertori cut out by the coroots of $G$. Since both sides of \cref{eq: for proof hmlgy cplx conj} are flat and affine over $T$, their rings of functions inject into the corresponding localizations along the map $T^\gen \to T$. Furthermore, these localizations are $\Z/2$-equivariant (for complex conjugation and $\theta$, respectively), and so it suffices to show that these localizations are $\Z/2$-equivariantly isomorphic.
    
    By \cref{lem: atiyah localization}, there is an isomorphism 
    $$\pi_0 \cf_T(\Gr_G)^\vee|_{T^\gen} \cong \pi_0 \cf_T(\Gr_T)^\vee|_{T^\gen} \cong \co_{T^\gen}[\bX_\ast(T)].$$
    Under this isomorphism, the action via complex conjugation on $\KU$ is given simply by inversion on $T^\gen$, and acts trivially on $\bX_\ast(T)$.
    Similarly, since $fx \in B$ is regular \textit{semisimple} if $x\in T^\gen$, and the centralizers of regular semisimple elements are tori, there is an isomorphism
    $$(T \times_{B^\reg/\ld{B}} T) \times_T T^\gen \cong T^\gen \times \ld{T}.$$
    Under this isomorphism, the action of $\theta$ is given simply by inversion on $T^\gen$, and acts trivially on $\ld{T}$. This clearly matches with the action on $\pi_0 \cf_T(\Gr_G)^\vee|_{T^\gen}$ via complex conjugation on $\KU$, as desired.
\end{proof}
\cref{prop: cplx conj KU and B mod B^} says that, up to replacing $B/\ld{B}$ by $\ld{B}/\ld{B}$ (that is, replacing $\Loc_{T_c}^\gr(\Gr_G; \KU)$ by $\Loc_{\ld{T}_c}^\gr(\Gr_G; \KU)$), the $\Z/2$-action via complex conjugation on equivariant $\KU$ identifies under \cref{cor: ku reg locus ordinary ABG} with the $\Z/2$-action on $\ld{B}/\ld{B} = \Map(B\Z, B\ld{B})$ coming from inversion on $\Z$.
\begin{remark}\label{rmk: cplx conj on G equiv KU}
    Assume $G$ has torsion-free fundamental group. One can similarly compute the effect of complex conjugation for $G_c$-equivariant local systems. Namely, as in \cref{lem: gamma from T to B^}, there is a map $\delta: T\mmod W \to \ld{B}$ such that if $x\in T\mmod W$, then $\Ad_{\delta(x)}$ sends $(fx)^{-1}$ to $fx$. Just as in \cref{def: cplx conj on B mod B}, we obtain an involution $\Theta$ on $G/\ld{G}$ which can be identified with the effect of inversion on $G$, and the resulting $\Z/2$-action on $\QCoh(G^\reg/\ld{G})$ identifies, under the equivalence of \cref{rmk: Loc G for KU}, with the $\Z/2$-action on $\Loc_{G_c}^\gr(\Gr_G; \KU)$ coming from complex conjugation on equivariant $\KU$. This gives an equivalence
    $$\Loc_{G_c}^\gr(\Gr_G; \KO) \otimes_\Z F \simeq \QCoh((G^\reg/\ld{G})/\pdb{\Theta}).$$
    Applied to the constant sheaf, the spectral sequence \cref{eq: sseq for coh of sheaf from loc gr KO} becomes
    $$E_2^{\ast,\ast} \cong \H^\ast(\Z/2; \co_{T\mmod W \times_{G/\ld{G}} T\mmod W}[u^{\pm 1}]) \Rightarrow \KO^{G_c}_\ast(\Gr_G) \otimes_\Z F.$$
\end{remark}
Let us now make a brief comment about the case of \textit{connective} real K-theory $\ko$, discussed in \cref{rmk: connective ko def}. For this, recall from \cite[Section 3.7]{ku-rel-langlands} that if $\GG_\beta$ denotes the group scheme over $\spec(\Z[\beta])/\GG_m$ given by $\spec \Z[\beta, x, \tfrac{1}{1+\beta x}]$ with group law $x + y + \beta xy$, $\DD(\GG_\beta)$ denotes its Cartier dual, and $H_\beta$ denotes $\Hom(\DD(\GG_\beta), H)$ for any group scheme $H$, then there is an equivalence
$$\Loc_{T_c}^\gr(\Gr_G; \ku) \otimes_\Z F \simeq \QCoh(B_\beta^\reg/\ld{B}),$$
where $B_\beta^\reg$ is the regular locus in $B_\beta$. Similarly, if $G$ has torsion-free fundamental group, there is an equivalence
$$\Loc_{G_c}^\gr(\Gr_G; \ku) \otimes_\Z F \simeq \QCoh(G_\beta^\reg/\ld{G}),$$
where $G_\beta^\reg$ is the regular locus in $G_\beta$. To descend these equivalences to $\ko$-coefficients, we need to describe an action of $\spec(\pi_\ast(\ku \otimes_\ko \ku))/\GG_m$ on $B_\beta$ and $G_\beta$. This is provided by \cref{eq: coaction connective ku on Gbeta}: if we write $\pi_\ast(\ku \otimes_\ko \ku) \cong \Z[\beta, r]/(r^2 - \beta r)$, then the action is given by the map
$$\spec(\Z[\beta, r]/(r^2 - \beta r)) \times_{\spec(\Z[\beta])} G_\beta \to G_\beta, (r,x) \mapsto x - \tfrac{rx^2}{1+\beta x}.$$
When $G = \SL_n$, this can be expressed in terms of $g = \id + \beta x$:
$$(r,g) \mapsto 1 + \tfrac{(\beta - 2 r)(g-1)}{\beta} + r \tfrac{g-g^{-1}}{\beta}.$$
The action of $\spec(\pi_\ast(\ku \otimes_\ko \ku))/\GG_m$ on $B_\beta$ and $G_\beta$ defines stacks $B_\beta^\ko$ and $G_\beta^\ko$ over $\spev(\ko)$. For any closed point $\spec(F) \to \spev(\ko)$, we obtain $F$-linear equivalences
\begin{align*}
    \Loc_{T_c}^\gr(\Gr_G; \ko) \otimes_{\spev(\ko)} F & \simeq \QCoh(B_\beta^{\ko,\reg}/\ld{B}), \\
    \Loc_{G_c}^\gr(\Gr_G; \ko) \otimes_{\spev(\ko)} F & \simeq \QCoh(G_\beta^{\ko,\reg}/\ld{G});
\end{align*}
upon inverting $\beta$, these are the equivalences of \cref{prop: cplx conj KU and B mod B^} and \cref{rmk: cplx conj on G equiv KU}.

Let us note that the calculation in \cref{prop: htpy KOS1} tells us that the actual homotopy groups of $\KO^{G_c}_\ast(\Gr_G)$ could differ from a calculation at the level of $\QCoh((G^\reg/\ld{G})/\pdb{\Theta})$, and furthermore that the resulting answer could be somewhat complicated. In general, the groups $\KO^{G_c}_\ast(\Gr_G)$ will not necessarily be concentrated in even degrees, and the differentials in the preceding spectral sequence will capture some of the (equivariant) attaching maps of the (equivariant) cells in $\Gr_G$. Let us now illustrate this by describing $\KO_\ast(\Gr_G)$ for $G = \SL_3$.
\begin{example}\label{ex: KO of Gr SLn}
    If $G = \SL_2$, then the James splitting says that stable homotopy type of $\Gr_G$ splits as the direct sum $\bigoplus_{n\geq 0} S^{2n}$. This implies that $\KO_\ast(\Gr_G) \simeq \bigoplus_{n\geq 0} \KO_{\ast-2n}$, and in fact there is a ring isomorphism $\KO_\ast(\Gr_G) \cong \KO_\ast[a]$ with $a$ in weight $2$. However, already in the case $G = \SL_3$, the analogous ring isomorphism $\KO_\ast(\Gr_G) \cong \KO_\ast[a,b]$ (with $a$ in weight $2$ and $b$ in weight $4$) fails. Let us indicate the topological reason for this failure: there is a map $\CP^2 \to \Gr_{\SL_3}$ which exhibits $\CP^2$ as a generating complex, meaning that the $2$- and $4$-cells of $\CP^2$ hit the classes $a$ and $b$, respectively. The ring $\KO_\ast(\Gr_{\SL_3})$ is therefore controlled by the $\KO_\ast$-module $\KO_\ast(\CP^2)$. The key point is that a classical theorem of Wood \cite{wood-banach} (which we reprove below) gives a $\KO$-module equivalence $\KO[\CP^2] \simeq \KO \oplus \KU$. In particular, the $\KO$-module $\KO[\CP^2]$ is not equivalent to $\KO \oplus \Sigma^2 \KO$ (unlike $\KU[\CP^2]$, which is equivalent to $\KU \oplus \Sigma^2 \KU$). This implies that $\KO_\ast(\Gr_{\SL_3})$ cannot be isomorphic to $\KO_\ast[a,b]$. 
    
    In fact, this can be generalized: there are equivalences $\KO[\CP^{2n}] \simeq \KO \oplus \KU^{\oplus n}$ and $\KO[\CP^{2n+1}] \simeq \KO \oplus \KU^{\oplus n} \oplus \Sigma^{2n+2} \KO$. This implies that $\KO_\ast(\Gr_{\SL_n})$ is not isomorphic to $\KO_\ast[a_1,\cdots, a_{n-1}]$ for $n>2$ (and in fact the behavior of $\KO_\ast(\Gr_{\SL_n})$ will depend on the parity of $n$, in stark contrast to the way one usually thinks about special linear groups!). Geometrically, this arises from the fact that the generating complex $\CP^{2n+1}$ of $\Gr_{\SL_{2n+2}}$ is a spin manifold, while the generating complex $\CP^{2n}$ of $\Gr_{\SL_{2n+1}}$ is not a spin manifold; and $\KO$ is $\Spin$-oriented \cite{atiyah-bott-shapiro} (meaning that spin manifolds admit $\KO$-fundamental classes).

    Calculating $\KO_\ast(\Gr_{\SL_3})$ explicitly is somewhat unpleasant, but let us at least indicate (from the perspective of \cref{sec: degenerations}) why there is a $\KO$-module equivalence $\KO[\CP^2] \simeq \KO \oplus \KU$. To do so, let $\tilde{\KU}_\ast(\CP^2)$ denote the reduced $\KU$-homology of $\CP^2$. One then has the homotopy fixed points spectral sequence
    \begin{equation}\label{eq: hfpss KO CP2}
        E_2^{s,\ast} \cong \H^s(\Z/2; \tilde{\KU}_\ast(\CP^2)) \Rightarrow \tilde{\KO}_{\ast-s}(\CP^2),
    \end{equation}
    which we will now calculate. This can be viewed as a special case of the spectral sequence \cref{eq: sseq for coh of sheaf from loc gr}, applied to $k = \KO$ and $\cf$ being the pushforward of the constant sheaf along the map $\CP^2 \to \Gr_{\SL_3}$. To compute \cref{eq: hfpss KO CP2}, one first observed that the action of complex conjugation on $\KU_\ast(\Gr_{\SL_3}) \cong \Z[u^{\pm 1}, a,b]$ is given by $u\mapsto -u$, $a\mapsto -a$, and $b \mapsto b+a$. 
    The action on $a$ and $b$ can also be seen from the equivalence
    $$\Loc^\gr(\Gr_G; \KO) \otimes_\Z F \simeq \QCoh((\cU^\reg/\ld{G})/\pdb{\theta})$$
    derived from \cref{prop: cplx conj KU and B mod B^}, where $\cU^\reg$ is the regular locus in the unipotent cone of $G$ (and $\theta$ acts by inversion on $\cU^\reg$). The action of complex conjugation on $\KU_0(\Gr_{\SL_3})$ identifies with the action of $\theta$ on the centralizer of a regular unipotent element of $\SL_3$, which consists of matrices of the form $\begin{psmallmatrix}
        1 & a & b \\
        0 & 1 & a \\
        0 & 0 & 1
    \end{psmallmatrix}$.
    
    This specifies the action of $\Z/2$ on $\tilde{\KU}_\ast(\CP^2) \cong \Z[u^{\pm 1}]\{a,b\}$, from which we find that
    $$E_2^{\ast,\ast} \cong E_2^{0,\ast} \cong \Z\{\cdots, (2b+a)u^{-2}, au^{-1}, 2b+a, au, (2b+a)u^2, au^3, \cdots\}.$$
    The entire spectral sequence is concentrated in a single line, so it automatically degenerates; this implies that $\tilde{\KO}_\ast(\CP^2)$ is isomorphic to $\Z$ in each even degree (and is zero otherwise). There are canonical maps $\KO \to \KU$ and $\CP^2 \to \KU$, which define a map $\KO \otimes \CP^2 \to \KU$. The above calculation implies that it induces an isomorphism on homotopy groups, and hence is an equivalence.
\end{example}

It is also possible to describe an analogue of \cref{cor: ku reg locus ordinary ABG} with coefficients in the $K(1)$-local sphere $\Lone S^0$ (for some fixed prime $p$). Recall from \cref{def: J graded Loc} that if $A$ is a $p$-power torsion abelian group and $X$ is a (ind-)finite $A$-space with even cells, then the $\infty$-category $\Loc_{A}^\gr(X; \Lone S^0)$ is obtained from $\Loc_{A}^\gr(X; \KU)$ by taking homotopy $\Z_p^\times$-invariants. 
\begin{definition}\label{def: Zpx and grothendieck springer}
    For $n\geq 0$, let $\tilde{\ld{G}}_{p^n}$ denote the (derived) fiber product
    $$\tilde{\ld{G}}_{p^n} := \tilde{\ld{G}} \times_{T} T[p^n].$$
    That is, $\tilde{\ld{G}}_{p^n}/\ld{G} \cong B_{p^n}/\ld{B}$, where $B_{p^n}$ is the subgroup of those elements of $B$ whose eigenvalues are all $p^n$th roots of unity. Similarly, let $\tilde{\ld{G}}_{p^n}^\reg$ denote the fiber product
    $$\tilde{\ld{G}}_{p^n}^\reg := \tilde{\ld{G}}^\reg \times_{T} T[p^n],$$
    There is an action of $\Z_p^\times$ (which factors through an action of $(\Z/p^n)^\times$) on $B_{p^n}$ given by exponentiation; the $\Z_p^\times$-action commutes with the $\ld{B}$-action by conjugation, and hence defines a $\Z_p^\times$-action on the quotient stack $B_{p^n}/\ld{B} \cong \tilde{\ld{G}}_{p^n}/\ld{G}$.
\end{definition}
\begin{prop}\label{prop: imJ reg locus ABG}
    Let $n\geq 0$. The $\Z_p^\times$-action on $\tilde{\ld{G}}_{p^n}/\ld{G}$ restricts to an action on $\tilde{\ld{G}}_{p^n}^\reg/\ld{G}$, and there is an equivalence
    $$\Loc_{T_c[p^n]}^\gr(\Gr_G; \Lone S^0) \otimes_{\Z_p} F \simeq \QCoh((\tilde{\ld{G}}_{p^n}^\reg/\ld{G})/\Z_p^\times).$$
\end{prop}
\begin{proof}
    Base-changing the $\QCoh(T)$-linear equivalence \cref{cor: ku reg locus ordinary ABG} along $\QCoh(T) \to \QCoh(T[p^n])$ gives an equivalence
    $$\Loc_{T_c[p^n]}^\gr(\Gr_G; \KU^\wedge_p) \otimes_\Z F \simeq \QCoh(\tilde{\ld{G}}_{p^n}^\reg/\ld{G}).$$
    Since the $\Z_p^\times$-action on $T[p^n]$ is given by exponentiation, the strategy of \cref{prop: cplx conj KU and B mod B^} shows that the $\Z_p^\times$-action on the left-hand side of the above equivalence via Adams operations on $p$-completed $\KU$ identifies with the $\Z_p^\times$-action on $\tilde{\ld{G}}_{p^n}^\reg/\ld{G}$ described in \cref{def: Zpx and grothendieck springer}. Taking homotopy $\Z_p^\times$-invariants of the displayed equivalence then yields the desired statement.
\end{proof}

The equivalences of \cref{prop: imJ reg locus ABG} are all compatible in $n$, and one finds that there is an equivalence
$$\Loc_{T_c[p^\infty]}^\gr(\Gr_G; \Lone S^0) \otimes_{\Z_p} F \simeq \QCoh((\tilde{\ld{G}}_{p^\infty}^\reg/\ld{G})/\Z_p^\times).$$
\begin{lemma}\label{lem: p-nilpotence and p power torsion}
    There is an isomorphism $B_{p^\infty} \cong B[p^\infty]$ over a $p$-nilpotent ring.
\end{lemma}
\begin{proof}
    Recall that $B_{p^\infty} = B \times_{T} T[p^\infty]$, so there is a canonical map $B[p^\infty] \to B_{p^\infty}$. It suffices to show that if $N$ denotes the unipotent radical of $B$, then $N \cong N[p^\infty]$. This follows by induction on the central series of $N$ (whose quotients are all isomorphic to $\GG_a$), and the fact that $\GG_a \cong \GG_a[p^\infty]$ since we are working over a $p$-nilpotent base.
\end{proof}
It follows that there is an equivalence
$$\Loc_{T_c[p^\infty]}^\gr(\Gr_G; \Lone S^0) \otimes_{\Z_p} F \simeq \QCoh((B[p^\infty]^\reg/\ld{B})/\Z_p^\times);$$
similarly, there is an equivalence
$$\Loc_{\ld{T}_c[p^\infty]}^\gr(\Gr_G; \Lone S^0) \otimes_{\Z_p} F \simeq \QCoh((\ld{B}[p^\infty]^\reg/\ld{B})/\Z_p^\times).$$
Note that $\ld{B}[p^\infty]^\reg/\ld{B}$ is an open substack of $\ld{B}[p^\infty]/\ld{B} \cong \colim_n \Map(B\Z/p^n, B\ld{B})$; one might heuristically view the latter as the stack of $\ld{B}$-bundles on the $p$-adic solenoid.

Finally, let us discuss the question of loop-rotation equivariance. Recall from \cref{def: nil-hecke} the algebra $\cH(\bH, T, W)$ associated to a $1$-dimensional group scheme $\bH$ over a field $F$ and a root system with torus $T$ and Weyl group $W$. In the following discussion, we will set $\bH = \GG_m$, so that $\bH_T = T$; we will also write $q$ to denote the standard character of $S^1_\rot$, so that $\pi_0 \KU_{S^1_\rot} \cong \Z[q^{\pm 1}]$. Exactly the same argument as in \cref{thm: ordinary loop-rot flag} shows the following result; here, $G$ does not need to be simply-laced.
\begin{theorem}\label{thm: ku loop-rot flag}
    There is an isomorphism of associative $\Z[q^{\pm 1}]$-algebras
    \begin{equation}\label{eq: ku comparison to nil hecke}
        \pi_0 \cf_{\tilde{T}_c}(\Fl_G)^\vee \cong \cH(\GG_m, \tilde{T}, \tilde{W}).
    \end{equation}
    Here, $\pi_0 \cf_{\tilde{T}_c}(\Fl_G)^\vee$ is equipped with the associative algebra structure coming from convolution. Moreover, the above isomorphism is also one of (cocommutative) Hopf $\pi_0 \KU_{\tilde{T}_c} \cong \co_{\bH_{\tilde{T}}}$-algebroids.
\end{theorem}
\begin{remark}
    Recall the quotient $\tilde{T}\mmod \tilde{W}$ from \cref{rmk: relationship to t mmod Waff}. The discussion therein combined with \cref{thm: ku loop-rot flag} gives an equivalence of categories
    $$\pi_0 \cf_{\tilde{T}_c}(\Fl_G)^\vee\modc \simeq \cH(\GG_m, \tilde{T}, \tilde{W})\modc \simeq \IndCoh(\tilde{T}\mmod \tilde{W}).$$
    It follows, via the argument of \cref{cor: reg locus quantized satake}, that $\Loc_{\tilde{T}_c}^\gr(\Fl_G; \KU) \otimes_\Z F$ is equivalent to the quotient of $\QCoh(\tilde{T})$ by the action of $\IndCoh(\tilde{T}\mmod \tilde{W})$.
\end{remark}
Assume, that $G$ is simple, simply-connected, and simply-laced.
Just as in \cref{sec: review Q coeff}, one would like to use \cref{thm: ku loop-rot flag} to prove analogues of \cref{cor: reg locus quantized satake} and \cref{eq: gen quantized ABG}. Namely, we expect that $\Loc_{T_c \times S^1_\rot}^\gr(\Gr_G; \KU)$ can be identified with a certain localization of the \textit{quantum} version $\co^{\univ,q}_\ld{G}$ of $\DMod_\hbar(\ld{G} / \ld{N})^{(\ld{G} \times \ld{T}, \weak)}$; the category $\co^{\univ,q}_\ld{G}$ itself is described in \cite[Definition 4.24]{univ-cat-o}. Similarly, we expect that $\Loc_{G_c \times S^1_\rot}^\gr(\Gr_G; \KU)$ can be identified with a certain localization of the \textit{quantum} version $\HC^{q}_\ld{G}$ of $\HC^\hbar_{\ld{G}}$. Here, $\HC^q_\ld{G}$ denotes a slight variant of the category described in \cite[Definition 2.24]{univ-cat-o}: instead of left $\co_q(\ld{G})$-modules in $\Rep_q(\ld{G})$, one needs to consider left $\co_q(G)$-modules in $\Rep_q(\ld{G})$. There is a fully faithful embedding of $\HC^q_\ld{G}$ into the category of $\co_q(G)$-bimodules in $\Rep_q(\ld{G})$.

The strategy of \cref{cor: reg locus quantized satake} very nearly works to prove these expectations: the only sticking point is that we do not have a $q$-analogue of \cite[Theorem 1.2.1]{ginzburg-whittaker} (which would give a $q$-analogue of \cref{cor: loop-rot Gr and biWhit}); I am currently exploring this direction of research. This is related to \cite[Conjecture 3.17]{finkelberg-tsymbaliuk}. 
Despite this complication, one can nevertheless use \cite[Theorem 4.10]{ondrus-whit-uq-sl2} (which gives an analogue of the isomorphism $Z(U(\ld{\g})) \cong U(\ld{\g})/_\psi \ld{N}$ of \cite[Theorem 3.7]{kostant-whittaker}), \cref{lem: hochschild and def to nc}, and the arguments of \cite[Section 2.3, Lemma 4]{bf-derived-satake} to show the following analogue of \cite[Theorem 2]{bf-derived-satake} and \cref{prop: KU full faithful on gr loc}:
\begin{prop}\label{prop: rep q fully faithful}
    Let $[n]_q = \tfrac{q^n - q^{-n}}{q - q^{-1}}$, and let $R$ denote the $(q-1)$-adic completion of $\cc\pw{q-1}[\tfrac{1}{[1]_q}, \tfrac{1}{[2]_q}, \cdots]$. 
    Let $\HC^{q, \free}_{\PGL_2}$ denote the full subcategory of $\HC^{q}_{\PGL_2} = \LMod_{\co_q(\SL_2)}(\Rep_q(\PGL_2))$ spanned by the essential image of the functor $\Rep_q(\PGL_2) \to \HC^{q}_{\PGL_2}$. 
    Finally, define
    $$\Loc_{\SU(2) \times S^1_\rot}^\gr(\Gr_{\SL_2}; \KU)^\wedge = \coLMod_{\pi_0 (\cf_{\SU(2)}(\Gr_{\SL_2})^\vee)^{hS^1_\rot}}(\QCoh(\cM_{\SU(2),0} \times \widehat{\cM_{S^1_\rot,0}})).$$
    Then the standard (homological) $t$-structure on $\QCoh(\cM_{\SU(2),0} \times \widehat{\cM_{S^1_\rot,0}})$ defines a $t$-structure on $\Loc_{\SU(2) \times S^1_\rot}^\gr(\Gr_{\SL_2}; \KU)^\wedge$, and there is a fully faithful functor 
    $$(\HC^{q, \free}_{\PGL_2})^\heartsuit \hookrightarrow \Loc_{\SU(2) \times S^1_\rot}^\gr(\Gr_{\SL_2}; \KU)^{\wedge,\heartsuit} \otimes_{\Z\pw{q-1}} R.$$
\end{prop}
Since it is rather technical, and is mostly a repetition of the arguments of \cite[Section 2.3 and Lemma 4]{bf-derived-satake}, we omit the argument. Note that $\Loc_{\SU(2) \times S^1_\rot}^\gr(\Gr_{\SL_2}; \KU)^\wedge$ is a full subcategory of $\Loc_{\SU(2) \times S^1_\rot}^\gr(\Gr_{\SL_2}; \KU)$: it should be understood as the $(q-1)$-adic completion of $\Loc_{\SU(2) \times S^1_\rot}^\gr(\Gr_{\SL_2}; \KU)$, where we identify $\pi_0 \KU_{S^1_\rot} \cong \Z[q^{\pm 1}]$.
\newpage

\section{The elliptic story}\label{sec: ell coeff}
In this section, we will work over a given algebraically closed field $F$. For the moment, $\ld{G}$ will be a (split) almost-simple group over $F$ with torsion-free fundamental group. Let $E$ be a (smooth) elliptic curve over $k$, let $\Bun_\ld{B}^0(E)$ denote the moduli stack of $\ld{B}$-bundles on $E$ of degree $0$, and let $\Bun_\ld{T}^0(E)$ denote the scheme of $\ld{T}$-bundles on $E$ of degree $0$. We will also make use of the stack $\Bun_\ld{G}^\ss(E)$ of semistable $\ld{G}$-bundles on $E$. Our main references for the structure of $\Bun_\ld{B}^0(E)$ and $\Bun_\ld{G}^\ss(E)$ will be \cite{davis-elliptic-springer, grojnowski-shepherd-barron}.
\begin{definition}\label{B-bundle regular}
    Say that a $\ld{B}$-bundle $\cP_\ld{B}$ on $E$ is \textit{regular} if $\dim \Aut(\cP_\ld{B}) = \rank(\ld{G})$. Let $\Bun_\ld{B}^0(E)^\reg$ denote the open substack of $\Bun_\ld{B}^0(E)$ defined by the regular $\ld{B}$-bundles. Similarly, if $\cP\in \Bun_\ld{G}^\ss(E)$ is a semistable $\ld{G}$-bundle on $E$, we say that $\cP$ is \textit{regular} if $\dim \Aut(\cP) = \rank(\ld{G})$. Let $\Bun_\ld{G}^\ss(E)^\reg \subseteq \Bun_\ld{G}^\ss(E)$ denote the open substack of regular semistable $\ld{G}$-bundles.
\end{definition}
\begin{prop}\label{elliptic-kostant}
    The map $\Bun_\ld{B}^0(E) \to \Bun_\ld{T}^0(E)$ admits a canonical unique section $\kappa: \Bun_\ld{T}^0(E) \to \Bun_\ld{B}^0(E)$ landing in $\Bun_\ld{B}^0(E)^\reg$.
\end{prop}
\begin{proof}
    Let $\cP$ be a semistable $\ld{G}$-bundle on $E$.
    By \cite[Proposition 4.4.5]{davis-elliptic-springer}, the regularity of $\cP$ is equivalent to the condition that for any (or some) $\ld{B}$-reduction $\cP_\ld{B}$ of $\cP$ of degree $0$, the associated $\ld{N}$-bundle $\cP_\ld{B}/\ld{T}$ is induced from an $\ld{N}_\cP$-bundle with nontrivial associated $\ld{N}_\alpha$-bundle for each simple root $\alpha$ in a particular subset of $\Delta$ determined by $\cP$. Moreover, every geometric fiber of the map $\Bun_\ld{G}^\ss(E) \to \Hom(\bX^\ast(\ld{T}), E)\mmod W$ to the coarse moduli space of $\Bun_\ld{G}^\ss(E)$ contains a unique regular semistable $\ld{G}$-bundle. Also see \cite[Proposition 3.9]{friedman-morgan-witten}, where a similar result is stated.

    Following \cite[Definition 3.1.7]{davis-elliptic-springer}, set
    $$\tilde{\Bun}_\ld{G}^\ss(E)^\reg \cong \Bun_\ld{G}^\ss(E)^\reg \times_{\Hom(\bX^\ast(\ld{T}), E)\mmod W} \Hom(\bX^\ast(\ld{T}), E).$$
    Let $\Bun_\ld{B}^0(E)^\reg$ denote the moduli stack of $\ld{B}$-bundles on $E$ of degree $0$.
    It then follows from the isomorphism $\tilde{\Bun}_\ld{G}^\ss(E) \cong \Bun_\ld{B}^0(E)$ of \cite[Proposition 2.1.11]{davis-elliptic-springer} and the equality $\dim \Aut(\cP) = \dim \Aut(\cP_\ld{B})$ that there is an isomorphism $\tilde{\Bun}_\ld{G}^\ss(E)^\reg \cong \Bun_\ld{B}^0(E)^\reg$. In particular, every geometric fiber of the map $\Bun_\ld{B}^0(E) \to \Hom(\bX^\ast(\ld{T}), E) = \Bun_\ld{T}^0(E)$ contains a unique regular $\ld{B}$-bundle of degree $0$. 

    The existence of $\kappa$ is a consequence of \cite[Theorem 4.3.2]{davis-elliptic-springer}, which is a refinement of \cite[Theorem 5.1.1]{friedman-morgan-ii}. Since we will not need the full strength of \cite[Theorem 4.3.2]{davis-elliptic-springer} outside of this proof, we will only briefly recall the necessary notation and statements. In \textit{loc. cit.}, the scheme $\Bun_\ld{T}^0(E)$ is denoted by $Y$. Let $\tilde{\Bun}_\ld{G}(E)$ denote the Kontsevich-Mori compactification of $\tilde{\Bun}_\ld{G}^\ss(E) \cong \Bun_\ld{B}^0(E)$; see \cite[Definition 2.1.2]{davis-elliptic-springer}. Let $\Theta$ denote the theta-line bundle over $\Bun_\ld{T}^0(E)$ of \cite[Corollary 3.2.10]{davis-elliptic-springer}, and let $\tilde{\chi}: \tilde{\Bun}_\ld{G}(E) \to \Theta^{-1}/\GG_m$ denote the map constructed in \cite[Corollary 3.3.2]{davis-elliptic-springer}. Then, \cite[Theorem 4.3.2]{davis-elliptic-springer} shows that there is a map $\Theta^{-1} \to \tilde{\Bun}_\ld{G}^\ss(E)$ landing in $\tilde{\Bun}_\ld{G}^\ss(E)^\reg$ such that the composite 
    $$\Theta^{-1} \to \tilde{\Bun}_\ld{G}^\ss(E) \xar{\tilde{\chi}} \Theta^{-1}/\GG_m$$
    is the canonical map. Composing with the zero section of $\Theta^{-1}$, we obtain a map 
    $$\Bun_\ld{T}^0(E) \cong 0_{\Theta^{-1}} \to \Theta^{-1} \to \tilde{\Bun}_\ld{G}^\ss(E)^\reg \cong \Bun_\ld{B}^0(E).$$
    This is the desired map $\kappa$.
\end{proof}
\begin{definition}
    The map $\kappa: \Bun_\ld{T}^0(E) \to \Bun_\ld{B}^0(E)$ from \cref{elliptic-kostant} will be called the \textit{elliptic Kostant slice}.
\end{definition}
The elliptic Kostant slice builds on work of Friedman-Morgan \cite{friedman-morgan, friedman-morgan-ii, friedman-morgan-iii, friedman-morgan-witten}.

If $E$ is replaced by the constant stack $S^1$ or by $B\GG_a$, the stack $\Bun_\ld{B}^0(E)$ is to be interpreted as $\ld{B}/\ld{B}$ and $\ld{\fr{b}}/\ld{B}$, respectively. The analogue of the elliptic Kostant section is given by the maps $f\cdot \ld{T} \to \ld{B}/\ld{B}$ and $f + \ld{\fr{t}} \to \ld{\fr{b}}/\ld{B}$, respectively.

The following is \cite[Lemma 3.1.11]{davis-elliptic-springer}.
\begin{lemma}\label{lem: vanishing and B-subgroups}
    Let $I\subseteq \Phi^-$ be a subset, and let $\Bun_\ld{T}^0(E)_I$ denote the subscheme of $\Bun_\ld{T}^0(E)$ defined by those bundles $\cP_\ld{T}$ whose $\alpha$-component is trivial precisely for $\alpha\in I$. Let $\ld{N}_I\subseteq \ld{N}$ be the smallest unipotent subgroup which is invariant under $\ld{T}$-conjugation and which contains $\ld{N}_\alpha$ for every $\alpha\in I$. Then the natural map
    $$\Bun_{\ld{T}\ld{N}_I}^0(E) \times_{\Bun_\ld{T}^0(E)} \Bun_\ld{T}^0(E)_I \to \Bun_\ld{B}^0(E) \times_{\Bun_\ld{T}^0(E)} \Bun_\ld{T}^0(E)_I$$
    is an isomorphism.
\end{lemma}
\begin{example}
    Suppose that $I = \emptyset$, so that $\Bun_\ld{T}^0(E)_\emptyset$ denotes the open subscheme of $\ld{T}$-bundles of degree zero whose $\alpha$-component is nontrivial for every negative root $\alpha$. The isomorphism $\tilde{\Bun}_\ld{G}^\ss(E) \cong \Bun_\ld{B}^0(E)$ implies that the map $\tilde{\Bun}_\ld{G}^\ss(E) \to \Bun_\ld{T}^0(E)$ is an isomorphism over $\Bun_\ld{T}^0(E)_\emptyset$. In particular, every point of $\Bun_\ld{T}^0(E)_\emptyset$ has a canonical associated (regular) semistable $\ld{G}$-bundle.
    
    The above results continue to hold if $E$ is replaced by the constant stack $S^1$ or by $B\GG_a$ (in which case  $\Bun_\ld{B}^0(E)$ is to be interpreted as $\ld{B}/\ld{B}$ and $\ld{\fr{b}}/\ld{B}$, respectively). In the case of $S^1$, for instance, the semistable $\ld{G}$-bundles obtained in this way from $\Bun_\ld{T}^0(E)_\emptyset$ are precisely those which lie in the regular \textit{semisimple} locus $\ld{G}^{\rss}/\ld{G}$; similarly for the case of $B\GG_a$.
\end{example}
We now turn to the topology of $G$, so it is connected, almost simple, and simply-laced over $\cc$.  In this setting, $k$ will be an even $2$-periodic $\Eoo$-ring equipped with an oriented group scheme $\GG$ whose underlying classical scheme $\GG_0$ over $\pi_0(k)$ is an elliptic curve $E$. We will continue to fix an algebraically closed field $F$ over $\pi_0(k)$, over which the Langlands dual group $\ld{G}$ will live. As usual, when dealing with the algebraic geometry (as opposed to the topology) of $G$, we will also view it as living over $F$; since $G$ is simply-laced, it is isogenous to $\ld{G}$. 
\begin{definition}
    The \textit{elliptic regular centralizer group scheme} $\tilde{\ld{J}}_\elc$ is defined to be the group scheme over $\Bun_{\ld{T}}^0(E)$ given by the fiber product
    $$\tilde{\ld{J}}_\elc \cong \Bun_{\ld{T}}^0(E) \times_{\Bun_{\ld{B}}^0(E)} \Bun_{\ld{T}}^0(E).$$
    Note that this is very slightly (but importantly) different from the definition of $\tilde{\ld{J}}_\mu$ and $\tilde{\ld{J}}$; the analogues of the fiber product above would instead be $(f \cdot \ld{T}) \times_{\ld{B}/\ld{B}} (f\cdot \ld{T})$ and $(f + \ld{\fr{t}}) \times_{\ld{\fr{b}}/\ld{B}} (f + \ld{\fr{t}})$. 
\end{definition}
In the following discussion, we will consider the $\ld{T}$-equivariant elliptic homology of $\Gr_G$ (instead of the $T$-equivariant elliptic homology); this will capture the minor difference between the definitions of $\tilde{\ld{J}}_\elc$ and $\tilde{\ld{J}}$ mentioned above.
\begin{theorem}\label{thm: elliptic hmlgy reg centr}
    There is an isomorphism of group schemes over $\Bun_{\ld{T}}^0(E) \cong \cM_{\ld{T},0}$:
    $$\spec_{\Bun_{\ld{T}}^0(E)}(\pi_0 \cf_{\ld{T}}(\Gr_G)^\vee) \otimes_{\pi_0(k)} F \cong \Bun_{\ld{T}}^0(E) \times_{\Bun_{\ld{B}}^0(E)} \Bun_{\ld{T}}^0(E).$$
    Here, $\spec_{\Bun_{\ld{T}}^0(E)}(\pi_0 \cf_{\ld{T}}(\Gr_G)^\vee)$ denotes the relative $\spec$ of $\pi_0 \cf_{\ld{T}}(\Gr_G)^\vee$ over $\Bun_{\ld{T}}^0(E)$.
\end{theorem}
As with \cref{thm: ordinary hmlgy reg centr} and \cref{thm: ku hmlgy reg centr}, the proof of \cref{thm: elliptic hmlgy reg centr} relies on two lemmas.
\begin{lemma}\label{lem: ell borel is flat}
    The projection map $\tilde{\ld{J}}_\elc \to \Bun_{\ld{T}}^0(E)$ (onto either factor) is flat.
\end{lemma}
\begin{proof}
    Like in the proof of \cref{lem: kappa for borel is flat}, it suffices, by miracle flatness, to show that the fibers of the map $\tilde{\ld{J}}_\elc \to \Bun_{\ld{T}}^0(E)$ have dimension exactly $\rank(\ld{G})$. But this follows from the fact that the map $\Bun_{\ld{T}}^0(E) \to \Bun_{\ld{B}}^0(E)$ lands in $\Bun_{\ld{B}}^0(E)^\reg$ (see \cref{elliptic-kostant}).
\end{proof}
For a root $\alpha$, let $\Bun_{\ld{T}}^0(E)_{\alpha\dreg} \subseteq \Bun_{\ld{T}}^0(E)$ denote the union of the substacks $\Bun_{\ld{T}}^0(E)_{\{\alpha\}}$ and $\Bun_{\ld{T}}^0(E)_{\emptyset}$. 
The next result follows exactly as in \cref{lem: localization of ordinary J} (using \cref{lem: vanishing and B-subgroups}).
\begin{lemma}\label{lem: ell localization of ordinary J}
    There is an isomorphism
    \begin{equation}\label{eq: ell J of centralizer}
        \tilde{\ld{J}}_\elc(\ld{G})|_{\Bun_{\ld{T}}^0(E)_{\alpha\dreg}} \xar{\sim} \tilde{\ld{J}}_\elc(Z_{\ld{G}}(x)^\circ)|_{\Bun_{\ld{T}}^0(E)_{\alpha\dreg}},
    \end{equation}
    where $Z_{\ld{G}}(x)$ is the centralizer of some $x\in \Bun_{\ld{T}}^0(E)_{\alpha\dreg}$ which lies in $\Bun_{\ld{T}}^0(E)_{\{\alpha\}}$, and $Z_{\ld{G}}(x)^\circ$ denotes the connected component of the identity. 
\end{lemma}
Recall that if $X$ is a scheme with subschemes $V = V(\cI) \subseteq D = V(\cJ)$ (so that $\cJ \subseteq \cI$) where $D$ is locally principal, the affine blowup $\Bl_V^D(X)$ is defined to be the complement of $V_+(\cJ)$ in the blowup $\Bl_V(X)$. That is, it is the relative $\spec$ of the algebra $\co_X[\tfrac{\cI}{\cJ}]$ of weight zero elements in $\Bl_\cI(\co_X)[\tfrac{1}{\cJ}]$, where $\Bl_\cI(\co_X) = \co_X \oplus \cI \oplus \cI^2 \oplus \cdots$ is the Rees algebra.
\begin{proof}[Proof of \cref{thm: elliptic hmlgy reg centr}]
    The argument of \cref{thm: ordinary hmlgy reg centr} reduces us to checking that the isomorphism of \cref{thm: elliptic hmlgy reg centr} holds if $G$ has semisimple rank $1$, i.e., is the product of a torus with one of $\GL_2$, $\SL_2$, or $\PGL_2$. Again, it is easy to match up the contributions from the toral factors, so we will assume that $G$ is either $\GL_2$, $\SL_2$, or $\PGL_2$. In this case, we can even replace $F$ by $\pi_0(k)$. The proofs are all rather uniform (as we have seen in \cref{thm: ordinary hmlgy reg centr} and \cref{thm: ku hmlgy reg centr}), so we will simply illustrate the argument when $G = \SL_2$ and $G = \PGL_2$.

    We begin with the case $G = \SL_2$. Since $\ld{T} = \GG_m$, we may identify $\Bun_{\ld{T}}^0(E) \cong E$; to emphasize that it plays the role of the base of $S^1$-equivariant elliptic cohomology, we will denote it by $\cM$. Let $\infty \in \cM = E$ denote the identity section. Consider the closed subschemes 
    $$V = \{(\infty,1)\} \subseteq D = \{\infty\} \times \GG_m \subseteq \cM \times \GG_m.$$
    Then, as in \cref{thm: ordinary hmlgy reg centr} and \cref{thm: ku hmlgy reg centr}, $\spec_{\Bun_{\ld{T}}^0(E)}(\pi_0 \cf_{\ld{T}}(\Gr_G)^\vee)$ identifies with the affine blowup $\Bl_V^D(\cM \times \GG_m)$. 
    
    Since $\ld{G} = \PGL_2$, an $S$-point of the stack $\Bun_{\ld{B}}^0(E)$ is the data of a degree zero rank $2$ vector bundle $\cV$ over $S \times E$ along with a line subbundle $\cL \subseteq \cV$ and an isomorphism $\cV/\cL \cong \co_{S \times E}$.
    In this language, the elliptic Kostant section $\cM = E \to \Bun_{\ld{B}}^0(E)$ classifies the unique indecomposable extension $\cV$ of $\co_{\cM \times E}$ by the Poincar\'e line bundle $\cP$. (Recall that $\cP$ can be identified, for instance, with the line bundle corresponding to the divisor $\Delta - E \times \{\infty\} - \{\infty\} \times E$.) This extension is classified by a nonzero section of $\ul{\Ext}^1_{\cM \times E}(\co_{\cM \times E}, \cP)$. 
    
    Let us now compute $\tilde{\ld{J}}_\elc$. The fiber product $\cM \times_{\Bun_{\ld{B}}^0(E)} \cM$ is isomorphic (as a group scheme over $\cM$) to the subgroup of the constant group scheme $\ul{\ld{B}} := \cM \times \ld{B}$ of those $b \in \ul{\ld{B}}$ such that $b\cdot \cV = \cV$. First, let $U = (\cM - \{\infty\}) \times E$; then $\cV|_U$ splits as $\co_U \oplus \cP|_U$. Indeed, the restriction $\cP|_U$ is a nontrivial line bundle on $U$, so its pushforward to $\cM - \{\infty\}$ has no cohomology (and hence the extension class is trivial).
    It follows that $\Aut_{\ld{B}}(\cV)|_U = \cM \times_{\Bun_{\ld{B}}^0(E)} U$ can be identified with $U \times \GG_m$.

    On the other hand, let $Z = \{\infty\} \times E$ denote the complement of $U$, so that the formal neighborhood $\hat{Z}$ of $Z$ is isomorphic to $\cM^\wedge_\infty \times E = \hat{\AA}^1 \times E$. Let $t$ denote a coordinate on $\hat{\AA}^1$. Then, the restriction of $\cP$ to $\hat{Z}$ is given by the $1$-parameter family of line bundles $\co_{\hat{Z}}(t - \infty)$ over $\hat{\AA}^1 \times E$. The restriction of $\cV$ to $\hat{Z}$ is classified by a map $\co_{\hat{Z}} \to \co_{\hat{Z}}(t - \infty)[1]$ which vanishes except at the origin of $\hat{\AA}^1$, where it is given by the unique (up to nonzero scalar) nontrivial map $\co_E \to \co_E[1]$.
    
    For instance, $\cV|_Z$ is isomorphic to the Atiyah bundle over $E$ from \cite{atiyah-bundle-elliptic} (i.e., the unique indecomposable rank 2 extension of the structure sheaf by itself), so that it can be realized away from $\infty \in E$ by pairs $(f_1, f_2)$ of regular functions on $E$; and near $\infty$ by pairs $(f_1, f_2)$ such that $f_1$ and $f_1 - zf_2$ are regular, where $z$ is a local coordinate of $E$. Under this description, $\End(\cV|_Z) = \End(\cV)|_Z$ is spanned by the identity and the map $(f_1, f_2) \mapsto (0, f_1)$. That is, $\End(\cV)|_Z$ is isomorphic to the group of matrices $\begin{psmallmatrix}
        a & b\\
        0 & a
    \end{psmallmatrix}$, and so $\Aut_{\ld{B}}(\cV)|_Z$ is isomorphic to $Z \times \GG_a$. It is easy to extend this description to the formal neighborhood of $Z$, and thereby find that $\Aut_{\ld{B}}(\cV)|_{\hat{Z}}$ is isomorphic to the canonical degeneration of $\GG_m$ into $\GG_a$. In other words, there is an isomorphism
    $$\Aut_{\ld{B}}(\cV)|_{\hat{Z}} \cong \spec \pi_0(k)\pw{t}[a^{\pm 1}, \tfrac{a-1}{t}].$$
    Gluing this with the description of $\Aut_{\ld{B}}(\cV)|_U$ from the preceding paragraph, we find that $\Aut_{\ld{B}}(\cV) \cong \cM \times_{\Bun_{\ld{B}}^0(E)} \cM$ is isomorphic to the affine blowup $\Bl_V^D(\cM \times \GG_m)$. We will leave it to the reader to verify that the resulting sequence of isomorphisms
    $$\Aut_{\ld{B}}(\cV) \cong \cM \times_{\Bun_{\ld{B}}^0(E)} \cM \cong \Bl_V^D(\cM \times \GG_m) \cong \spec_{\Bun_{\ld{T}}^0(E)}(\pi_0 \cf_{\ld{T}}(\Gr_G)^\vee)$$
    is one of group schemes over $\cM$.

    The case when $G = \PGL_2$ is very similar; we only indicate the necessary changes. Let $E[2] \subseteq E$ denote the $2$-torsion subgroup, and consider the closed subschemes 
    $$V = E[2] \times \mu_2 \subseteq D = E[2] \times \GG_m \subseteq \cM \times \GG_m.$$
    By arguing as in \cref{thm: ordinary hmlgy reg centr} and \cref{thm: ku hmlgy reg centr}, we find that $\spec_{\Bun_{\ld{T}}^0(E)}(\pi_0 \cf_{\ld{T}}(\Gr_G)^\vee)$ identifies with the affine blowup $\Bl_V^D(\cM \times \GG_m)$. 
    In this case, $\ld{G} = \SL_2$, and the elliptic Kostant section $\cM = E \to \Bun_{\ld{B}}^0(E)$ sends a line bundle $\cL$ to the trivially filtered $\SL_2$-bundle $\co_E \subseteq \co_E \oplus \cL$ if $\cL^2 \neq \co_E$; and to the Atiyah extension of $\cL$ by itself if $\cL^2 \cong \co_E$. This extension is defined by a nontrivial element of $\Ext^1_E(\cL, \cL^{-1}) \cong \H^1(E; \cL^{-2})$. The calculation of $\cM \times_{\Bun_{\ld{B}}^0(E)} \cM$ follows exactly the same path as in the case $G = \SL_2$ studied above.
\end{proof}
\begin{remark}
    The most classical instantiation of the Atiyah bundle $\cA$ is via the Weierstrass functions. The $\GG_a$-torsor over $E$ associated to $\cA$ is the complement of the section at $\infty$ of the projective line $\PP(\cA)$. If we work complex-analytically, $E^\an$ can be identified as the quotient $\cc/\Lambda$ for some rank $2$ lattice $\Lambda\subseteq \cc$. Associated to $\Lambda$ are two Weierstrass functions defined on $\cc$:
    \begin{align*}
        \wp(z; \Lambda) & = \tfrac{1}{z^2} + \sum_{\lambda \in \Lambda-\{0\}} \left(\tfrac{1}{(z-\lambda)^2} - \tfrac{1}{\lambda^2}\right), \\
        \zeta(z; \Lambda) & = \tfrac{1}{z} + \sum_{\lambda \in \Lambda-\{0\}} \left(\tfrac{1}{z-\lambda} + \tfrac{1}{\lambda} + \tfrac{z}{\lambda^2}\right).
    \end{align*}
    Note that $\wp(z; \Lambda)$ is doubly-periodic, i.e., $\wp(z + \lambda; \Lambda) = \wp(z; \Lambda)$ for any $\lambda \in \Lambda$. Alternatively, $\wp$ defines a map $\cc \to \cc$ which factors through a map $\cc/\Lambda = E^\an \to \cc$.

    Although $\zeta(z; \Lambda)$ is not doubly-periodic, an easy calculation shows that $\wp(z; \Lambda) = -\partial_z \zeta(z; \Lambda)$; so if $\lambda \in \Lambda$, then $\zeta(z+\lambda; \Lambda) - \zeta(z; \Lambda) = c(\lambda)$ for some constant $c(\lambda)$. The function $\lambda \mapsto c(\lambda)$ is evidently additive, and defines a homomorphism $\Lambda \to \cc$, which defines a $\cc$-bundle over $E^\an = \cc/\Lambda$. This $\cc$-bundle is precisely the analytification of the $\GG_a$-torsor associated to the Atiyah bundle. It follows that although $\zeta$ is not defined on $E^\an$, this analytification is the universal space over $E^\an$ on which $\zeta$ is well-defined.

    This discussion also describes the total space of the rank $2$-bundle $\cA^\an$ purely analytically. For instance, if $q\in \cc^\times$ is a unit complex number of modulus $<1$, we can identify $\Tot(\cA^\an)$ over the Tate curve $\cc^\times/q^\Z$ with the quotient
    $$\Tot(\cA^\an) = \left(\cc^\times \times \cc^2\right)/\left((z,x) \sim \left(qz, \begin{psmallmatrix}
    1 & 1\\
    0 & 1
    \end{psmallmatrix} x\right)\right).$$
    The appearance of the Jordan block $\begin{psmallmatrix}
    1 & 1\\
    0 & 1
    \end{psmallmatrix}$ is the basic reason why the Atiyah bundle plays the role of the principal nilpotent element $f$ in the proof of \cref{thm: elliptic hmlgy reg centr}.
\end{remark}

\begin{corollary}\label{cor: ell reg locus ordinary ABG}
    There is an $F$-linear equivalence
    $$\Loc_{\ld{T}_c}^\gr(\Gr_G; k) \otimes_{\pi_0(k)} F \simeq \QCoh(\Bun_{\ld{B}}^0(E)^\reg).$$
    Furthermore, the pushforward functor $\Loc_{\ld{T}_c}^\gr(\Gr_G; k) \to \Loc_{\ld{T}_c}^\gr(\ast; k)$ identifies with the pullback functor $\kappa^\ast: \QCoh(\Bun_{\ld{B}}^0(E)) \to \QCoh(\Bun_{\ld{T}}^0(E))$.
\end{corollary}
\begin{proof}
    By definition, $\Loc_{\ld{T}_c}^\gr(\Gr_G; k)$ is equivalent to the category of comodules over $\pi_0 \cf_\ld{T}(\Gr_G)^\vee$ in $\QCoh(\cM_{\ld{T},0}) = \QCoh(\Bun_{\ld{T}}^0(E))$. By \cref{thm: elliptic hmlgy reg centr}, it can be identified the category of quasicoherent sheaves on the quotient stack $\Bun_{\ld{T}}^0(E)/\tilde{\ld{J}}_\elc$.  We may view $\tilde{\ld{J}}_\elc$ as a closed subgroup scheme of the constant group scheme $\ld{B} \times \Bun_{\ld{T}}^0(E)$. This gives an isomorphism
    $$\Bun_{\ld{T}}^0(E)/\tilde{\ld{J}}_\elc \cong \ld{B} \backslash (\ld{B} \times \Bun_{\ld{T}}^0(E))/\tilde{\ld{J}}_\elc.$$
    Let $\Bun_{\ld{B}}^0(E)_\triv$ denote the scheme whose $S$-points are of $\ld{B}$-bundles over $S \times E$  of degree $0$ equipped with a trivialization at $S \times \{\infty\}$, so that there is a natural map $\Bun_{\ld{B}}^0(E)_\triv \to \Bun_{\ld{B}}^0(E)$. Let $\Bun_{\ld{B}}^0(E)_\triv^\reg$ denote the restriction of $\Bun_{\ld{B}}^0(E)_\triv$ to the regular locus $\Bun_{\ld{B}}^0(E)^\reg \subseteq \Bun_{\ld{B}}^0(E)$.
    It follows from Davis' work in \cite{davis-elliptic-springer} that the $\ld{B}$-orbit of $\Bun_{\ld{T}}^0(E)$ inside $\Bun_{\ld{B}}^0(E)_\triv$ is precisely the regular locus $\Bun_{\ld{B}}^0(E)_\triv^\reg$. Since $\tilde{\ld{J}}_\elc$ is by definition the stabilizer of $\kappa: \Bun_{\ld{T}}^0(E) \to \Bun_{\ld{B}}^0(E)$, the quotient $\ld{B} \backslash (\ld{B} \times \Bun_{\ld{T}}^0(E))/\tilde{\ld{J}}_\elc$ is isomorphic to $\Bun_{\ld{B}}^0(E)^\reg$; so there is an isomorphism $\Bun_{\ld{T}}^0(E)/\tilde{\ld{J}}_\mu \cong \Bun_{\ld{B}}^0(E)^\reg$.
\end{proof}
The equivalence of \cref{cor: ell reg locus ordinary ABG} is in fact symmetric monoidal for the convolution tensor structure on $\Loc_{T_c}^\gr(\Gr_G; k)$ (described in \cref{rmk: loc gr convolution tensor}) and the standard tensor product on $\QCoh(\Bun_{\ld{B}}^0(E)^\reg)$.
\begin{remark}\label{rmk: g-equiv regular satake elliptic}
    The work of Gepner and Meier in \cite{gepner-meier, t-equiv-tmf} sets up the theory of $G_c$-equivariant elliptic cohomology for compact Lie groups $G_c$. In particular, they describe a scheme $\cM_G$ over $k$ with underlying scheme $\cM_{G,0}$ over $\pi_0(k)$, such that the global sections of the structure sheaf of $\cM_G$ computes $G_c$-equivariant $k$-cohomology.
    Using this setup (and assuming a slight extension of the results of \cite{davis-elliptic-springer} replacing the simply-connectedness assumption with the condition of having torsion-free fundamental group), it can be shown that if $G$ is almost simple and simply-laced, and has torsion-free fundamental group, there is an $F$-linear equivalence
    $$\Loc_{\ld{G}_c}^\gr(\Gr_G; k) \otimes_{\pi_0(k)} F \simeq \QCoh(\Bun_{\ld{G}}^\ss(E)^\reg).$$
    Here, the left-hand side is defined to be the $\infty$-category $\coLMod_{\pi_0(\cf_G(\Gr_G)^\vee)}(\QCoh(\cM_{G,0}))$, just as in \cref{sec: degenerations}. The proof of the displayed equivalence is quite similar to that of \cref{cor: ell reg locus ordinary ABG}, and in fact can be deduced from it using the observation that $\pi_0(\cf_G(\Gr_G)^\vee) = \pi_0(\cf_T(\Gr_G)^\vee)^W$ and that the natural map $\Bun_{\ld{B}}^0(E)^\reg \to \Bun_{\ld{G}}^\ss(E)^\reg$ is a (ramified) $W$-cover. The first statement uses that $G$ is simply-connected, and the second is the elliptic version of Grothendieck-Springer theory studied in \cite[Proposition 3.1.14]{davis-elliptic-springer}.
\end{remark}

Restriction of a $\ld{B}$-bundle on $E$ to the zero section defines a map $q: \Bun_{\ld{B}}^0(E) \to B\ld{B} \to B\ld{G}$, which in turn defines a functor
\begin{equation}\label{eq: Rep G^ to ell loc}
    \Rep(\ld{G}) \to \QCoh(\Bun_{\ld{B}}^0(E)^\reg) \simeq \Loc_{\ld{T}_c}^\gr(\Gr_G; k) \otimes_{\pi_0(k)} F.
\end{equation}
More generally, the map $q: \Bun_{\ld{B}}^0(E) \to B\ld{B} \to B\ld{G} \times B\ld{T}$ defines a functor
\begin{equation}\label{eq: Rep T^ x G^ to ell loc}
    \Rep(\ld{G} \times \ld{T}) \to \QCoh(\Bun_{\ld{B}}^0(E)^\reg) \simeq \Loc_{\ld{T}_c}^\gr(\Gr_G; k) \otimes_{\pi_0(k)} F.
\end{equation}
If $V \in \Rep(\ld{G})$, let $\cS_k(V)$ denote the corresponding object of $\Loc_{\ld{T}_c}^\gr(\Gr_G; k) \otimes_{\pi_0(k)} F$. The same argument as in \cref{prop: ordinary realizing minuscule reps} shows the following, which says that $\cS_k(V) \in \Loc_{T_c}^\gr(\Gr_G; k)$ is the associated graded of a particular object $\cf_\lambda \in \Loc_{T_c}(\Gr_G; k)$ if $V$ is a minuscule $\ld{G}$-representation. 
\begin{prop}\label{prop: ell realizing minuscule reps}
    Let $\lambda_\bull = (\lambda_1, \cdots, \lambda_n)$ be a tuple of dominant minuscule weights of $\ld{G}$, let $|\lambda_\bull| = \sum_i \lambda_i$, and let $\ol{\Gr_G^{\lambda_\bull}}$ denote the corresponding \textit{convolution variety}. Let $\cf_{\lambda_\bull}$ denote the pushforward of the constant sheaf along the canonical map $q: \ol{\Gr_G^{\lambda_\bull}} \to \ol{\Gr_G^{|\lambda|}} \subseteq \Gr_G$. If $V_{\lambda_i}$ denotes the irreducible representation of $\ld{G}$ with highest weight $\lambda_i$, then there is an isomorphism $\cS_k(\bigotimes_i V_{\lambda_i}) \cong \cf_{\lambda_\bull}^\gr$.
\end{prop}
It would be very interesting to understand whether \cref{prop: ell realizing minuscule reps} can be extended to other non-minuscule irreducible representations. Again, as in \cref{rmk: ordinary action on minuscule}, if $\lambda$ is a dominant minuscule weight of $\ld{G}$, then the coaction of $\pi_0 \cf_T(\Gr_G)^\vee$ on $\pi_0 \cf_T(G/P_\lambda)$ defines a homomorphism 
\begin{equation}
    \spec \pi_0 \cf_T(\Gr_G)^\vee \to \GL(\pi_0 \cf_T(G/P_\lambda))
\end{equation}
of group schemes over $\Bun_T^0(E)$, where $\GL(\pi_0 \cf_T(G/P_\lambda))$ denotes the group scheme of $\co_{\Bun_T^0(E)}$-linear automorphisms of the vector bundle $\pi_0 \cf_T(G/P_\lambda)$. Under the isomorphisms of \cref{thm: elliptic hmlgy reg centr} and \cref{prop: ell realizing minuscule reps}, this homomorphism factors as the composite
\begin{equation}\label{eq: ell factorization action on minuscule}
    \tilde{\ld{J}}_\elc \to \ld{G} \times \Bun_T^0(E) \to \GL(V_\lambda) \times \Bun_T^0(E),
\end{equation}
where the second map describes the $\ld{G}$-action on $V_\lambda$.

\begin{remark}\label{rmk: 1-shifted cartier}
    The statements of \cref{cor: reg locus ordinary ABG}, \cref{cor: ku reg locus ordinary ABG}, and \cref{cor: ell reg locus ordinary ABG} can be packaged into a single statement as follows. Suppose $k$ is a complex-oriented $2$-periodic $\Eoo$-ring, and let $\GG$ be an oriented commutative $k$-group scheme. Let $\GG_0$ denote the underlying commutative group scheme over $\pi_0(k)$, and let $\GG_0^\vee = \Hom(\GG_0, B\GG_m)$ denote its $1$-shifted Cartier dual. Let $F$ be an algebraically closed field over $\pi_0(k)$; then there is an $F$-linear equivalence
    $$\Loc_{\ld{T}_c}^\gr(\Gr_G; k) \otimes_{\pi_0(k)} F \simeq \QCoh(\Bun_{\ld{B}}^0(\GG_0^\vee)^\reg).$$
    Similarly, there is an $F$-linear equivalence
    $$\Loc_{\ld{G}_c}^\gr(\Gr_G; k) \otimes_{\pi_0(k)} F \simeq \QCoh(\Bun_{\ld{G}}^\ss(\GG_0^\vee)^\reg).$$
    In fact, the arguments of \cref{cor: reg locus ordinary ABG}, \cref{cor: ku reg locus ordinary ABG}, and \cref{cor: ell reg locus ordinary ABG} show that these equivalences are monoidal for the convolution tensor products on $\Loc_{\ld{T}_c}^\gr(\Gr_G; k)$ and $\Loc_{\ld{G}_c}^\gr(\Gr_G; k)$ coming from the $\E{2}$-structure on $\Gr_G$, and the ordinary tensor product of quasicoherent sheaves. Moreover, a simple adaptation of the discussion at the end of \cref{sec: KU coeff} (as well as the discussion in \cref{sec: power operations}) shows that the above equivalences are canonical: they respect natural symmetries of $k$ coming from endomorphisms of $\GG_0$. 
    
    The object $\Bun_{\ld{G}}^\ss(\GG_0^\vee)$ has also appeared previously in the literature in connection to equivariant elliptic cohomology; see, for instance, \cite{sibilla-tomasini, toen-hkr}. One could heuristically view $\QCoh(\Bun_{\ld{B}}^0(\GG_0^\vee))$ and $\QCoh(\Bun_{\ld{G}}^\ss(\GG_0^\vee))$ as the ``$\GG_0$-Hochschild homology'' of $\QCoh(B\ld{B})$ and $\QCoh(B\ld{G})$, respectively.
    
    To see that these equivalences do indeed package \cref{cor: reg locus ordinary ABG}, \cref{cor: ku reg locus ordinary ABG}, and \cref{cor: ell reg locus ordinary ABG}, note that if $k = \QQ[u^{\pm 1}]$ and $\GG = \GG_a$, then the $1$-shifted Cartier dual of $\GG_0$ is $B\hat{\GG}_a$, and $\Map(B\hat{\GG}_a, B\ld{B}) \cong \ld{\fr{b}}/\ld{B}$.\footnote{In fact, this works even if $k$ is an $\Eoo$-$\Z$-algebra. Indeed, the $1$-shifted Cartier dual of $\GG_a$ over $\Z$ is the classifying stack of $\Hom(\GG_a, \GG_m) = \widehat{\GG_a^\sharp}$; here, $\widehat{\GG_a^\sharp}$ denotes the formal scheme $\spf(\Z\pdb{x}/I^{[n]})$ where $\Z\pdb{x}$ is the divided power algebra on a class $x$ and $I^{[n]}$ is the ideal generated by elements of $\Z\pdb{x}$ of degree $\geq n$. Then, $\Map(B\widehat{\GG_a^\sharp}, X)$ is isomorphic to the $1$-shifted tangent bundle $T[-1](X)$, so that $\Map(B\widehat{\GG_a^\sharp}, B\ld{B}) \cong \ld{\fr{b}}/\ld{B}$ even over $\Z$.} Similarly, if $k = \KU$ and $\GG = \GG_m$, then the $1$-shifted Cartier dual of $\GG_0$ is $B\Z$, and $\Map(B\Z, B\ld{B}) \cong \ld{B}/\ld{B}$. Finally, if $\GG_0$ is an elliptic curve $E$, then its $1$-shifted Cartier dual is $\Pic^0(E) = E$, so $\Bun_{\ld{B}}^0(\GG_0^\vee) = \Bun_{\ld{B}}^0(E)$. In fact, in this language, the calculations of \cite{ku-rel-langlands} show that the stated equivalence continues to hold if $k = \ku$ (now one must replace $\pi_0(k)$ by $\Z[\beta]$, and $F$ by $F[\beta]$) and $\GG$ is the group scheme $\spec \Z[\beta, x, \tfrac{1}{1+\beta x}]$ with group law $x + y + \beta xy$.
\end{remark}
Observe that if $\cL$ is a degree zero line bundle on $\GG_0^\vee$, then $\H^\ast(\GG_0^\vee; \cL)$ vanishes unless $\cL$ is trivial, in which case it is isomorphic to an exterior algebra over $k$ on a class in degree $1$. Using this, the Kostant slice is straightforward to describe in the semisimple rank $1$ cases. For instance, if $\ld{G} = \PGL_2$, the map $\kappa: \GG_0 \to \Bun_{\ld{B}}^0(\GG_0^\vee)$ can be understood as follows. Since $\GG_0 = \Hom(\GG_0^\vee, B\GG_m)$, a point of $\GG_0$ can be viewed as a degree zero line bundle on $\GG_0^\vee$. Given such a line bundle $\cL$, the map $\kappa$ sends it to the trivial $\ld{B}$-bundle $\cL \subseteq \cL \oplus \co_{\GG_0^\vee} \twoheadrightarrow \co_{\GG_0^\vee}$ if $\cL$ is nontrivial, and to the unique nontrivial extension $\co_{\GG_0^\vee} \subseteq \cA \twoheadrightarrow \co_{\GG_0^\vee}$ if $\cL$ is trivial. This nontrivial extension comes from a nonzero section of $\H^1(\GG_0^\vee; \co)$.
\begin{remark}\label{rmk: morava e-theory}
    \cref{rmk: 1-shifted cartier} suggests that there might be an analogue of \cref{thm: intro omnibus} for other $\Eoo$-rings $k$ which may not be equipped with a $1$-dimensional $k$-group scheme $\GG$, but which may only be equipped with an \textit{$S$-divisible group} for finite set $S$ of primes. (This is the data of a $p$-divisible group for each prime $p\in S$.) We have already seen one example of such an $\Eoo$-ring in \cref{prop: imJ reg locus ABG}: namely, $k = L_{K(1)} S^0 = (\KU^\wedge_p)^{h\Z_p^\times}$ being the $K(1)$-local sphere at a prime $p$. In this case, one only has the $p$-divisible group $\mu_{p^\infty}$ over $\spf(\KU^\wedge_p)/\Z_p^\times$. This was reflected accordingly in \cref{prop: imJ reg locus ABG}, in the sense that we could only consider the $\infty$-category $\Loc_{T_c[p^\infty]}^\gr(\Gr_G; L_{K(1)} S^0)$ of $T_c[p^\infty]$-equivariant (and not $T_c$-equivariant) sheaves on $\Gr_G$.

    To this end, suppose $k$ is an $\Eoo$-ring equipped with an oriented $S$-divisible spectral group scheme $\GG$ in the sense of \cite{elliptic-ii, elliptic-iii} for some set $S$ of primes. Let $\GG_0$ denote the corresponding (classical) $S$-divisible group over $\pi_0(k)$. Let $T_c[S^\infty]$ denote the subgroup of $T_c$ given by the union of $T_c[p^\infty]$ over all $p\in S$. We then expect that one can define the $\infty$-category $\Loc_{T_c[S^\infty]}^\gr(\Gr_G; k)$, and that just as in \cref{rmk: 1-shifted cartier}, there is an $F$-linear equivalence
    \begin{equation}\label{eq: expected S-divisible equivalence}
        \Loc_{\ld{T}_c[S^\infty]}^\gr(\Gr_G; k) \otimes_{\pi_0(k)} F \simeq \QCoh(\Bun_{\ld{B}}^0(\GG_0^\vee)^\reg).
    \end{equation}
    Here, following the lead of \cref{lem: p-nilpotence and p power torsion} and the surrounding discussion, we define
    $$\Bun_{\ld{B}}^0(\GG_0^\vee) = \bigcup_{p\in S} \colim_n \Bun_{\ld{B}}^0(\GG_0[p^n]^\vee).$$
    Both sides of \cref{eq: expected S-divisible equivalence} naturally admit an action of $\Aut(\GG_0)$, and \cref{eq: expected S-divisible equivalence} should be equivariant for this action.
    
    In fact, it should be possible to take $F$ in \cref{eq: expected S-divisible equivalence} to be the $\pi_0(k)$-algebra classifying isomorphisms of $S$-divisible groups between $\GG_0$ and a constant $S$-divisible group. For instance, if $S = \{p\}$ and $\GG_0$ is of height $h \geq 1$, then $\Bun_{\ld{B}}^0(\GG_0^\vee)$ over such an $F$ would be isomorphic to the stack of commuting $h$-tuples of $p$-power torsion elements in $\ld{B}$ modulo simultaneous $\ld{B}$-conjugation. 

    One particularly interesting instance is the case when $k$ is a Lubin-Tate theory associated to a height $n$ formal group over a perfect field of characteristic $p>0$, and $\GG$ denotes the associated ``Quillen $p$-divisible group'' over $k$ (see \cite{elliptic-ii}). Let $\co_{D_{1/n}}^\times$ denote the Morava stabilizer group of automorphisms of the fiber of $\GG_0$ over the residue field of $\pi_0(k)$, so that $\co_{D_{1/n}}^\times$ is the group of units in the ring of integers of the division algebra over $\QQ_p$ of Hasse invariant $1/n$. Then $\co_{D_{1/n}}^\times$ acts on $k$ through $\Eoo$-ring maps, with homotopy fixed points given by the $K(n)$-local sphere. Assuming the discussion surrounding \cref{eq: expected S-divisible equivalence}, we would find that $\Loc_{\ld{T}_c[p^\infty]}^\gr(\Gr_G; L_{K(n)} S^0)$ is well-defined, and is furthermore equivalent (upon an appopriate base-change) to $\QCoh(\Bun_{\ld{B}}^0(\GG_0^\vee)^\reg/\co_{D_{1/n}}^\times)$. Regardless of whether these statements are true, it seems interesting to investigate the stack $\Bun_{\ld{B}}^0(\GG_0^\vee)$ and the action of $\co_{D_{1/n}}^\times$ on it. (When $n=1$, this was done in \cref{prop: imJ reg locus ABG}.)
\end{remark}

Let us now return to some more concrete consequences of \cref{thm: elliptic hmlgy reg centr}. 
Just as with \cref{prop: ordinary gelfand-graev} and \cref{prop: ku gelfand-graev}, the calculation of \cref{thm: elliptic hmlgy reg centr} gives an \textit{elliptic} version of the Gelfand-Graev action on the affine closure $\ol{T^\ast(\ld{G}/\ld{N})}$. Taking the affine closure in the naive sense is very destructive in the case of elliptic cohomology. Nevertheless, one can define $\ol{T^\ast_E(\ld{G}/\ld{N})}$ to be the relative spectrum over $\Bun_{\ld{T}}^0(E)$ of $\pi_0$ of the (\textit{classical}, not derived!) pushforward of the structure sheaf along the quotient morphism
$$(\ld{G} \times \ld{T} \times \Bun_{\ld{T}}^0(E))/\tilde{\ld{J}}_\elc \to \Bun_{\ld{T}}^0(E).$$
\begin{prop}[Elliptic Gelfand-Graev action]\label{prop: ell gelfand-graev}
    The natural action of $\ld{G} \times \ld{T}$ on $\ol{T^\ast_E(\ld{G}/\ld{N})}$ extends to an action of $\ld{G} \times (W \rtimes \ld{T})$, where $W$ is the Weyl group.
\end{prop}
The moment map $\ol{T^\ast_E(\ld{G}/\ld{N})}/\ld{G} \to \Bun_{\ld{G}}^\ss(E)$ is $W$-equivariant for the trivial action on the target. There is a commutative diagram
$$\xymatrix{
\Bun_{\ld{B}}^0(E) \ar@{^(->}[r] \ar[dr] & \ol{T^\ast_E(\ld{G}/\ld{N})}/\ld{T} \ar[d] \\
& \Bun_{\ld{G}}^\ss(E)
}$$
which relates $\ol{T^\ast_E(\ld{G}/\ld{N})}$ to the elliptic Grothendieck-Springer resolution \cite{bzn-elliptic-springer}; and via this diagram, the elliptic Gelfand-Graev action is closely related to the Weyl action in elliptic Springer theory.
\begin{remark}
    The proof of \cref{prop: ell gelfand-graev} generalizes to show that if $\ld{P} \subseteq \ld{G}$ is a parabolic subgroup with Levi quotient $\ld{L}$ and unipotent radical $U_{\ld{P}}$, then the natural action of $\ld{G} \times \ld{L}$ on the affine closure $\ol{T^\ast_E(\ld{G}/U_{\ld{P}})}$ extends to an action of $\ld{G} \times (W_L \rtimes \ld{L})$, where $W_L = N_{\ld{G}}(\ld{L})/\ld{L}$ is the Weyl group.
\end{remark}
As in \cref{rmk: Z/2 symplectic fourier} and \cref{ex: Z/2 multiplicative symplectic fourier}, it is possible to make the action of \cref{prop: ell gelfand-graev} explicit in the case when $\ld{G} = \SL_2$.
\begin{example}\label{ex: Z/2 ell symplectic fourier}
    Let $\co(\infty)$ denote the inverse of the ideal sheaf cutting out the zero section inside $E$, and let $\cf = (\co \oplus \co(\infty))^{\oplus 2}$. As in \cref{ex: Z/2 multiplicative symplectic fourier}, $\ol{T^\ast_E(\SL_2/\GG_a)}$ can be identified with the space of sections $(u,v) = \left(\begin{psmallmatrix}
        u_1 \\
        u_2
    \end{psmallmatrix}, (v_1, v_2)\right)$ of the bundle $\cf \to E$ such that the resulting section $\pdb{u,v} = u_1 v_1 + u_2 v_2$ of $\co(\infty)$ has vanishing locus given by the zero section of $E$. 
    Modifying the analysis of \cref{rmk: Z/2 symplectic fourier} shows that if $[-1]: E \to E$ denotes the inversion map, the $\Z/2$-action of \cref{prop: ell gelfand-graev} sends 
    $$\left(\begin{psmallmatrix}
        u_1 \\
        u_2
    \end{psmallmatrix}, (v_1, v_2)\right) \mapsto \left(\begin{psmallmatrix}
        -u_2 \\
        u_1
    \end{psmallmatrix}, \alpha(v_2, -v_1)\right),$$
    where $\alpha$ is given locally around $\infty$ by multiplication by $-\tfrac{[-1](\pdb{u,v})}{\pdb{u,v}}$.
    (The discussion here makes sense with $E$ replaced by any $1$-dimensional group scheme $\bH$. When $\bH = \GG_a$ or $\GG_m$, the class $-\tfrac{[-1](\pdb{u,v})}{\pdb{u,v}}$ is equal to $1$ or $\tfrac{1}{1 + \pdb{u,v}}$, respectively, as expected from \cref{rmk: Z/2 symplectic fourier} and \cref{ex: Z/2 multiplicative symplectic fourier}.)
\end{example}
One could regard the variety $\ol{T^\ast_E(\SL_2/\GG_a)}$ of \cref{ex: Z/2 ell symplectic fourier} as an elliptic version of Van den Bergh's multiplicative quiver variety $\cB(\AA^1, \AA^2)$ from \cite{van-den-bergh-double-poisson}. Motivated by this observation, we hope to similarly define a notion of ``elliptic quiver varieties'' (generalizing the notion of multiplicative quiver variety from \cite{crawley-boevey-shaw}) in future work.

We also have the following analogue of \cref{prop: ordinary full faithful on gr loc}, whose proof is exactly the same (one only needs to use \cite[Proposition 3.1.16]{davis-elliptic-springer}, which says that $\Bun_{\ld{B}}^0(E)^\reg \hookrightarrow \Bun_{\ld{B}}^0(E)$ has complement of codimension $2$, and similarly for $\Bun_{\ld{G}}^\ss(E)^\reg \hookrightarrow \Bun_{\ld{G}}^\ss(E)$).
\begin{prop}\label{prop: ell full faithful on gr loc}
    Let $\Loc_{\ld{T}_c}^\gr(\Gr_G; k)^\heart$ denote the heart of the $t$-structure on $\Loc_{\ld{T}_c}^\gr(\Gr_G; k) = \coMod_{\pi_0(\cf_\ld{T}(\Gr_G))^\vee}(\QCoh(\Bun_\ld{T}^0(E)))$ coming from the standard (homological truncation) $t$-structure on $\QCoh(\Bun_\ld{T}^0(E))$. 
    Then, the composite functor
    $$\Loc_{\ld{T}_c}^\gr(\Gr_G; k) \otimes_{\pi_0(k)} F \simeq \QCoh(\Bun_{\ld{B}}^0(E)^\reg) \to \QCoh(\ld{G}\backslash \ol{T^\ast_E(\ld{G}/\ld{N})}/\ld{T})$$
    is $t$-exact, and on hearts, it restricts to a fully faithful functor on the essential image of \cref{eq: Rep T^ x G^ to ell loc}. Furthermore, this functor is $W$-equivariant for the natural action of $W = \N_{G_c}(\ld{T}_c)/\ld{T}_c$ on the left-hand side and the Gelfand-Graev action of \cref{prop: ell gelfand-graev} on the right-hand side.

    Similarly, suppose $G$ has torsion-free fundamental group, and let $\Loc_{\ld{G}_c}^\gr(\Gr_G; k)^\heart$ denote the heart of the $t$-structure on $\Loc_{\ld{G}_c}^\gr(\Gr_G; k) = \coMod_{\pi_0(\cf_{\ld{G}}(\Gr_G))^\vee}(\QCoh(\cM_{\ld{G},0}))$ coming from the standard (homological truncation) $t$-structure on $\QCoh(\cM_{\ld{G},0})$. 
    Then, the composite functor
    $$\Loc_{\ld{G}_c}^\gr(\Gr_G; k) \otimes_{\pi_0(k)} F \simeq \QCoh(\Bun_{\ld{G}}^\ss(E)^\reg) \to \QCoh(\Bun_{\ld{G}}^\ss(E))$$
    is $t$-exact, and on hearts, it restricts to a fully faithful functor on the essential image of the functor $\Rep(\ld{G}) \to \Loc_{G_c}^\gr(\Gr_G; k) \otimes_{\pi_0(k)} F$ (analogous to \cref{eq: Rep G^ to ell loc}).
\end{prop}
\cref{prop: ell full faithful on gr loc} gives an analogue of \cite[Theorem 4]{bf-derived-satake}: namely, if $\QCoh_\free(\Bun_{\ld{G}}^\ss(E))$ denotes the essential image of the pullback functor $\Rep(\ld{G}) \to \QCoh(\Bun_{\ld{G}}^\ss(E))$, then there is a fully faithful embedding 
$$\QCoh_\free(\Bun_{\ld{G}}^\ss(E))^\heartsuit \hookrightarrow \Loc_{\ld{G}_c}^\gr(\Gr_G; k)^\heartsuit \otimes_{\pi_0(k)} F.$$ 
Similarly, if $\QCoh_\free(\ld{G}\backslash \ol{T^\ast_E(\ld{G}/\ld{N})}/\ld{T})$ denotes the essential image of the pullback functor $\Rep(\ld{G} \times \ld{T}) \to \QCoh(\ld{G}\backslash \ol{T^\ast_E(\ld{G}/\ld{N})}/\ld{T})$, then there is a fully faithful embedding 
$$\QCoh_\free(\ld{G}\backslash \ol{T^\ast_E(\ld{G}/\ld{N})}/\ld{T})^\heartsuit \hookrightarrow \Loc_{\ld{T}_c}^\gr(\Gr_G; k)^\heartsuit \otimes_{\pi_0(k)} F.$$
This implies the following result.
\begin{corollary}\label{cor: ell minuscule equivalence}
    Let $\QCoh_{\free}(\Bun_{\ld{G}}^\ss(E))^{\min,\heartsuit}$ denote the essential image of $\Rep_\min(\ld{G})$ under the pullback functor $\Rep(\ld{G})^\heartsuit \to \QCoh(\Bun_{\ld{G}}^\ss(E))^\heartsuit$ (so it is the entirety of $\QCoh(\Bun_{\ld{G}}^\ss(E))^\heartsuit$ if $F$ has characteristic zero and $\ld{G}$ is not of type $E_8$). Similarly, let $(\Loc_{\ld{G}_c}^\gr(\Gr_G; \KU)^{\heartsuit} \otimes_{\pi_0(k)} F)^\min$ denote the idempotent completion of the subcategory of $\Loc_{\ld{G}_c}^\gr(\Gr_G; \KU)^\heartsuit \otimes_{\pi_0(k)} F$ spanned by $\cf_{\lambda_\bull}^\gr$ ranging over sequences $\lambda_\bull$ of minuscule highest weights. Then there is an equivalence
    $$\QCoh_\free(\Bun_{\ld{G}}^\ss(E))^{\min,\heartsuit} \simeq (\Loc_{\ld{G}_c}^\gr(\Gr_G; k)^{\heartsuit} \otimes_{\pi_0(k)} F)^\min.$$
\end{corollary}
There is a similar equivalence
$$(\Loc_{\ld{T}_c}^\gr(\Gr_G; k)^{\heartsuit} \otimes_{\pi_0(k)} F)^\min \simeq \QCoh_\free(\ld{G}\backslash \ol{T^\ast_E(\ld{G}/\ld{N})}/\ld{T})^{\min,\heartsuit},$$
where these categories are defined analogously by idempotent completion. 

Note, again, that the category $(\Loc_{\ld{G}_c}^\gr(\Gr_G; k)^{\heartsuit} \otimes_{\pi_0(k)} F)^\min$ is the heart of a degeneration, in the sense of \cref{sec: degenerations}, of the similarly-defined category $(\Loc_{\ld{G}_c}(\Gr_G; k) \otimes_k F[u^{\pm 1}])^\min$. Thus \cref{cor: ell minuscule equivalence} gives an equivalence between the purely algebraically defined category $\QCoh_\free(\Bun_{\ld{G}}^\ss(E))^{\min,\heartsuit}$ and a degeneration of the purely topologically defined category $(\Loc_{\ld{G}_c}(\Gr_G; k) \otimes_k F[u^{\pm 1}])^\min$. Observe, again, that if $\lambda_\bull$ and $\mu_\bull$ are two sequences of dominant minuscule weights of $\ld{G}$, there is an equivalence of $k$-modules
$$\Map_{(\Loc_{\ld{G}_c}(\Gr_G; k) \otimes_k F[u^{\pm 1}])^\min}(\cf_{\lambda_\bull}, \cf_{\mu_\bull}) \simeq \cf_{\ld{G}_c}(\ol{\Gr_G^{\lambda_\bull}} \times_{\Gr_G} \ol{\Gr_G^{\mu_\bull}}),$$
so that the category $(\Loc_{\ld{G}_c}(\Gr_G; k) \otimes_k F[u^{\pm 1}])^\min$ compares to the $k$-analogue of the category from \cite[Section 3.4]{cautis-kamnitzer}.

Let us end this section with a brief comment regarding loop-rotation equivariance.
Recall from \cref{def: nil-hecke} the algebra $\cH(\bH, T, W)$ associated to a $1$-dimensional group scheme $\bH$ over a field $F$ and a root system with torus $T$ and Weyl group $W$. In the following discussion, we will set $\bH = E$, so that $\bH_T = \Bun_T^0(E) = \cM_T$. Exactly the same argument as in \cref{thm: ordinary loop-rot flag} shows the following result; here, $G$ does not need to be simply-laced.
\begin{theorem}\label{thm: ell loop-rot flag}
    There is an isomorphism of sheaves of associative algebras over $\bH_{\GG_m^\rot} = E$:
    \begin{equation}\label{eq: ell comparison to nil hecke}
        \pi_0 \cf_{\tilde{T}_c}(\Fl_G)^\vee \cong \cH(E, \tilde{T}, \tilde{W}).
    \end{equation}
    Here, $\pi_0 \cf_{\tilde{T}_c}(\Fl_G)^\vee$ is equipped with the associative algebra structure coming from convolution. Moreover, the above isomorphism is also one of (cocommutative) Hopf $\co_{\cM_{\tilde{T},0}} \cong \co_{\bH_{\tilde{T}}}$-algebroids.
\end{theorem}
\begin{remark}
    Recall the quotient $\Bun_{\tilde{T}}^0(E)\mmod \tilde{W}$ from \cref{rmk: relationship to t mmod Waff}. The discussion therein combined with \cref{thm: ell loop-rot flag} gives an equivalence of categories
    $$\pi_0 \cf_{\tilde{T}_c}(\Fl_G)^\vee\modc \simeq \cH(E, \tilde{T}, \tilde{W})\modc \simeq \IndCoh(\Bun_{\tilde{T}}^0(E)\mmod \tilde{W}).$$
    It follows, via the argument of \cref{cor: reg locus quantized satake}, that $\Loc_{\tilde{T}_c}^\gr(\Fl_G; k) \otimes_{\pi_0(k)} F$ is equivalent to the quotient of $\QCoh(\Bun_{\tilde{T}}^0(E))$ by the action of $\IndCoh(\Bun_{\tilde{T}}^0(E)\mmod \tilde{W})$.
\end{remark}
Assume, again, that $G$ is simply-laced.
Just as in \cref{sec: review Q coeff}, one would like to use \cref{thm: ell loop-rot flag} to prove analogues of \cref{cor: reg locus quantized satake} and \cref{eq: gen quantized ABG}. However, unlike with \cref{thm: ku loop-rot flag}, we do not even have a putative description for the Langlands dual side. By analogy with the K-theoretic case, it is natural to expect that the dual side will be related to elliptic quantum groups \`a la \cite{felder-elliptic-quantum}; I am currently exploring this direction of research.
\newpage

\section{Power operations under Langlands duality}\label{sec: power operations}
We will momentarily review some of the rich theory of power operations in homotopy theory; these force the existence of additional structures on the Langlands dual side of \cref{thm: intro omnibus}. Our goal in this section is to describe these structures explicitly. This section is motivated by a discussion with David Treumann. 

Before proceeding, we warn the reader of a terminological mismatch. In \cite{lonergan-frob}, Lonergan uses ``Steenrod operators'' to construct new structures on Coulomb branches (and in particular, on $\ld{J}$). These operators, as we will explain in future work, are better viewed as \textit{$\E{3}$-power operations} coming from an $\E{3}$-algebra structure on $C^{G_c}_\ast(\Gr_G; \FF_p)$. While these are related to Steenrod operations in the usual sense of the word (as used by algebraic topologists), they are not the same. More generally, $\E{3}$-power operations on $\cf_G(\Gr_G)^\vee$ are closely related to, but distinct from, the power operations we will describe below. These $\E{3}$-power operations will be described in future work; there, we will prove a generalization to other $\Eoo$-rings of the ``Azumaya property'' of crystalline differential operators in characteristic $p$.

Let $k$ be an $\Eoo$-ring; we will momentarily specialize to the case when $k$ is $2$-periodic \textit{integral} cohomology, complex K-theory, or elliptic cohomology. The theory of power operations describes the additional structure acquired by $k$-cohomology from the $\Eoo$-structure on $k$. As we will see below, it is closely related to the structure of isogenies on the associated $1$-dimensional group scheme. This relationship is not new; we refer the reader to \cite{strickland-symmetric-gps, ando-power-operations, rezk-icm} for some sources.
\begin{construction}\label{cstr: power operations}
    Any $\Eoo$-ring $k$ admits a \textit{Tate-valued Frobenius} $k \to k^{t\Cp}$, which is given by the composite of the Tate diagonal $k \to (k^{\otimes p})^{t\Cp}$ with the $\Cp$-Tate construction of the multiplication map $k^{\otimes p} \to k$. See, e.g., \cite[Definition IV.1.1]{nikolaus-scholze} for a modern reference. 
    
    If $k$ admits additional structure, then this structure can be refined: namely, if $k$ admits a refinement to a normed algebra in the $\infty$-category of genuine $\Cp$-spectra (which will be true in the examples we will study), and $\Phi^{\Cp} k$ is its geometric fixed points, then the Tate-valued Frobenius $k \to k^{t\Cp}$ lifts to an $\Eoo$-map $\varphi: k \to \Phi^{\Cp} k$. This map is given by taking geometric fixed points of the $\Cp$-equivariant norm-multiplication map $N^{\Cp} k \to k$, where $N^{\Cp} k$ is the Hill-Hopkins-Ravenel norm from \cite{hhr}.
    
    If $X$ is any (finite) space, let $\cf_k(X)$ denote the $\Eoo$-$k$-algebra of $k$-cochains on $X$, and let $\cf_k(X)^\vee$ denote the $\Eoo$-$k$-coalgebra of $k$-chains on $X$. Then $\varphi$ induces maps
    $$\cf_k(X) \to \cf_{\Phi^{\Cp} k}(X), \ \cf_k(X)^\vee \to \cf_{\Phi^{\Cp} k}(X)^\vee.$$
    We will denote either of these maps by $\varphi_X$, and call them the \textit{decompleted Frobenius}. Sometimes, we will consider the further composites to $\cf_{k^{t\Cp}}(X)$ and $\cf_{k^{t\Cp}}(X)^\vee$; these composites exist for any $\Eoo$-ring $k$, even if it does not lift to a normed algebra in genuine $\Cp$-spectra.
\end{construction}
In the above context, one should view $\Phi^{\Cp} k$ as a decompletion of $k^{t\Cp}$; we will see this in \cref{ex: examples of power operations} below.
\begin{remark}\label{rmk: power op and tate frob}
    Let $I_\tr$ denote the transfer ideal in $\pi_0 \cf_k(X \times B\Cp)$, given by the image of the map $\pi_0 \cf_k(X) \to \pi_0 \cf_k(X \times B\Cp)$ induced by the transfer. On $\pi_0$, the map $\varphi_X: \cf_k(X) \to \cf_{k^{t\Cp}}(X)$ then factors as a composite
    $$\pi_0 \cf_k(X) \to \pi_0 \cf_k(X \times B\Cp)/I_\tr \to \pi_0 \cf_{k^{t\Cp}}(X).$$
    The first map in this composite is often referred to as the \textit{total power operation}. We will denote it by $\varphi^\tr_X$. It will not be used below in any serious way; we have mentioned it only for completeness.
\end{remark}
\begin{remark}\label{rmk: total power op and phiCp}
    \cref{cstr: power operations} might seem somewhat abstract, but it has very concrete consequences. Suppose, for simplicity, that $k$ is even and $2$-periodic, and that $\pi_0 \cf_k(X \times B\Cp) \cong \pi_0 \cf_k(X) \otimes_{\pi_0(k)} \pi_0 \cf_k(B\Cp)$. Under the assumption on $k$, this happens if, for instance, either $X$ is a finite space with even cells, or $\pi_0 \cf_k(B\Cp)$ is flat over $\pi_0(k)$. The total power operation is then a ring map
    $$\varphi_X^\tr: \pi_0 \cf(X) \to \pi_0 \cf_k(X) \otimes_{\pi_0(k)} \pi_0 \cf_k(B\Cp)/I_\tr.$$
    In fact, this can be upgraded to a map
    \begin{equation}\label{eq: Sigma-p power operation}
        \pi_0 \cf(X) \to \pi_0 \cf_k(X) \otimes_{\pi_0(k)} \pi_0 \cf_k(B\Sigma_p)/I_\tr,
    \end{equation}
    where $\Sigma_p$ is the symmetric group on $p$ letters.
    
    Moreover, under the hypothesis on $k$, there is an isomorphism $\pi_0 \cf_k(B\Cp) \cong \pi_0(k)\pw{t}/[p](t)$, where $[p](t)$ is the $p$-series of the formal group law over $\pi_0 \cf_k(\CP^\infty) \cong \pi_0(k)\pw{t}$.\footnote{Unfortunately, this $t$ is common practice in homotopy theory; but it conflicts with the $t$ which is the coordinate of the formal (punctured) disk used to define the affine Grassmannian. We will use the same symbol $t$ to denote both, and the distinction should be clear from context.} The composite of \cref{rmk: power op and tate frob} can be identified with the map
    $$\pi_0 \cf_k(X) \xrightarrow{\varphi_X^\tr} \pi_0 \cf_k(X) \otimes_{\pi_0(k)} \pi_0 \cf_k(B\Cp)/I_\tr \to \pi_0 \cf_k(X) \otimes_{\pi_0(k)} \pi_0 \cf_k(B\Cp)[1/t].$$
    If $k$ admits the structure of a normed algebra in genuine $\Cp$-spectra, then this composite factors through
    $$\pi_0 \cf_k(X) \xrightarrow{\varphi_X} \pi_0 \cf_k(X) \otimes_{\pi_0(k)} \pi_0 \Phi^{\Cp}(k) \to \pi_0 \cf_k(X) \otimes_{\pi_0(k)} \pi_0 \cf_k(B\Cp)[1/t].$$
    It follows, in particular, that $\varphi_X^\tr$ and $\varphi_X$ together define a map
    $$\pi_0 \cf_k(X) \to \pi_0 \cf_k(X) \otimes_{\pi_0(k)} \left(\pi_0 \Phi^{\Cp}(k) \times_{\pi_0 \cf_k(B\Cp)[1/t]} \pi_0 \cf_k(B\Cp)/I_\tr\right).$$
    The fiber product on the right-hand side does not have any denominators in $t$, and we will see this explicitly in the examples below.
\end{remark}
\begin{example}\label{ex: examples of power operations}
    Let us explicate the preceding remark in two examples.
    \begin{itemize}
        \item Suppose $k = \Z[u^{\pm 1}]$ with $u$ in degree $2$. Then $\pi_0 \cf(B\Cp) \cong \Z\pw{t}/pt$, and the transfer ideal is simply generated by $t$. Therefore, $\pi_0 \cf(B\Cp)/I_\tr \cong \FF_p\pw{t}$. If $X$ is a finite space with even cells, then the map of \cref{rmk: power op and tate frob} can be viewed as an (ungraded) map
        $$\H^\ast(X; \Z) \xrightarrow{\varphi_X^\tr} \H^\ast(X; \FF_p\pw{t}) \to \H^\ast(X; \FF_p\ls{t}).$$
        The decompleted Frobenius is given by an (ungraded) map
        $$\varphi_X: \H^\ast(X; \Z) \to \H^\ast(X; \FF_p[t^{\pm 1}]).$$
        Explicitly, these maps are given on a class $\alpha \in \H^\ast(X; \Z)$ by the formula
        $$\alpha \mapsto \sum_{i\geq 0} (-1)^i P^i(\alpha) t^{(p-1)i}.$$
        Here, $P^i$ is the $i$th Steenrod operation. That is to say, $\varphi_X$ encodes the action of the Steenrod operations on $\H^\ast(X; \Z)$.\footnote{This is perhaps bad terminology, because the Steenrod algebra does give an endomorphism of integral cohomology. Here, however, we are viewing the Steenrod algebra as acting on a class $\alpha$ in integral cohomology through its mod $p$ reduction $\ol{\alpha}$. Our $\varphi_X$ will only see the action of $P^i$ on $\ol{\alpha}$, and not the operations $\beta P^i$ (when $p=2$, this is $\Sq^{2i+1}$). In fact, it turns out that the decompleted Frobenius $\varphi: \Z \to \Phi^{\Cp} \Z$ factors as an $\Eoo$-map through the reduction map $\Z \to \FF_p$ (so that $\varphi_X(\alpha)$ depends only on $\ol{\alpha}$ in a coherently multiplicative way), but proving this is out of the scope of the present article.
        
        Let us mention that only tracking $P^i(\ol{\alpha})$ definitely loses some information about the entire Steenrod algebra action. First, since $\ol{\alpha}$ came from the integral class $\alpha$, its Bockstein $\beta(\ol{\alpha})$ vanishes. It is, however, possible that $\beta P^i(\ol{\alpha})$ be nonzero despite $\ol{\alpha}$ lifting to integral cohomology. For instance, if we identify $\H^\ast(\RP^4 \times \RP^4; \FF_2) = \FF_2[x,y]/(x^5, y^5)$, then the class $\ol{\alpha} = xy(x+y)$ lifts to integral cohomology, but $\Sq^3(\ol{\alpha}) = x^2 y^2 (x^2 + y^2) \neq 0$.} As expected by \cref{rmk: total power op and phiCp}, there are no denominators in $t$ in the above formula.
        For instance, if $X = \CP^n$ for any finite $n$, this map sends $x\in \H^2(\CP^n; \Z)$ to $x - t^{p-1} x^p$.
        \item Suppose $k = \KU$. Then $\pi_0 \cf(B\Cp) \cong \Z\pw{t}/((1+t)^p - 1)$, and the transfer ideal is simply generated by $t$. Therefore, 
        $$\pi_0 \cf(B\Cp)/I_\tr \cong \Z\pw{t}/\tfrac{(1+t)^p - 1}{t} \cong \Z[\zeta_p]^\wedge_t.$$
        Here, $\zeta_p$ is a primitive $p$th root of unity and $t = \zeta_p - 1$. Note that since $t^{p-1}$ is a unit multiple of $p$ in $\Z_p[\zeta_p]$, the $t$-completion above is equivalent to $p$-completion.
        The ring $\Z_p[\zeta_p]$ is flat over $\pi_0(k)^\wedge_p = \Z_p$, and so the composite of \cref{rmk: power op and tate frob} can be viewed as a ring map
        \begin{equation}\label{eq: KU completed tate frob}
            \KU^0(X) \xrightarrow{\varphi_X^\tr} \KU^0(X)[\zeta_p]^\wedge_p \to \KU^0(X)[\zeta_p]^\wedge_p[1/p]
        \end{equation}
        The geometric fixed points $\Phi^{\Cp} \KU$, on the other hand, has homotopy groups given by
        $$\pi_\ast \Phi^{\Cp} \KU \cong \Z[q^{\pm 1}, \beta^{\pm 1}][\tfrac{1}{(q-1)\cdots(q^{p-1}-1)}]/(q^p-1) \cong \Z[\zeta_p, \beta^{\pm 1}][1/p];$$
        the final isomorphism comes from noticing that $(\zeta_p-1)\cdots(\zeta_p^{p-1}-1)$ is $(-1)^{p-1} p$. The decompleted Frobenius is given by a ring map
        $$\varphi_X: \KU^0(X) \to \KU^0(X)[\zeta_p][1/p].$$
        Note that this map is, indeed, a de-$p$-adic completion of \cref{eq: KU completed tate frob}.
        Both $\varphi_X^\tr$ and $\varphi_X$ send a vector bundle $V$ to the $p$th Adams operation $\psi^p(V) \in \KU^0(X)$, viewed as a subalgebra of $\KU^0(X)[\zeta_p]^\wedge_p$ and of $\KU^0(X)[\zeta_p][1/p]$. As expected by \cref{rmk: total power op and phiCp}, there are no denominators in $t = \zeta_p - 1$ in this formula.
    \end{itemize}
\end{example}

In order to understand the interaction between these power operations and \cref{thm: intro omnibus}, we will need to port \cref{cstr: power operations} to the setting of genuine equivariant (co)homology. Namely, we need a decompletion of the map 
$$\varphi_{BS^1}: \cf_k(BS^1) \to \cf_{\Phi^{\Cp} k}(BS^1) \simeq \lim_n \cf_{\Phi^{\Cp} k}(\CP^n).$$
First, observe that this map factors through an $\Eoo$-map
$$\varphi_{BS^1}': \cf_k(BS^1) \to \Phi^{\Cp} k \otimes_k \cf_k(BS^1) \simeq \Phi^{\Cp} k \otimes_k \lim_n \cf_k(\CP^n);$$
the map from the target to $\cf_{\Phi^{\Cp} k}(BS^1)$ generally induces a strict inclusion on homotopy\footnote{For instance, take $k = \KU$. Then the map $\varphi_{BS^1}$ is given on homotopy by the map $\Z\pw{t} \to \Z[\zeta_p][1/p]\pw{t}$ which sends $t\mapsto (1+t)^p - 1$. This factors through a map $\Z\pw{t} \to \Z[\zeta_p]\pw{t}[1/p]$; this is the effect of the map $\varphi_{BS^1}'$ on homotopy. Note that there is a strict inclusion $\Z[\zeta_p]\pw{t}[1/p] \subseteq \Z[\zeta_p][1/p]\pw{t}$.}. Note that $\varphi_{BS^1}'$ can be viewed as a homomorphism
$$\hat{\GG} \times_{\spec k} \spec \Phi^{\Cp} k \to \hat{\GG},$$
where $\hat{\GG} = \spf \cf_k(BS^1)$.

We will now specialize to the case when $k$ is $2$-periodic {integral} cohomology, complex K-theory, or elliptic cohomology, and let $\GG$ denote $\GG_a$, $\GG_m$, or the spectral elliptic curve $E$ over $k$ (respectively). The choice of $\GG$ equips $k$ with a lift to the $\infty$-category of normed rings in genuine $\Cp$-spectra. As usual, let $\GG_0$ denote the underlying group scheme over $\pi_0(k)$. Our desired decompletion will then be given by a particular homomorphism
\begin{equation}\label{eq: phi on genuine S1 equiv}
    \varphi: \GG \times_{\spec k} \spec \Phi^{\Cp} k \to \GG.
\end{equation}
To describe it, we need to give a moduli-theoretic interpretation of $\Phi^{\Cp} k$. Let $\GG[p]$ denote the $p$-torsion subgroup of $\GG$, so that $\GG[p] = \Hom(\Z/p, \GG)$. 

There is a natural action of $\FF_p^\times$ on $\GG[p]$ given by sending $i\in \FF_p^\times$ to the multiplication-by-$i$ map $[i]$. Let $U\subseteq \GG[p]$ denote the open subscheme given by the complement of the closed subscheme
$$\bigcup_{i\in \FF_p^\times} \ker(\GG_0[p] \xrightarrow{[i]} \GG_0[p]) \subseteq \GG_0[p].$$
The following is a straightforward consequence of \cite[Proposition 2.25]{hausmann-meier}.
\begin{lemma}\label{lem: phi Cp and moduli problem}
    The spectral scheme $\spec \Phi^{\Cp} k$ is isomorphic to $U$ over $k$.
\end{lemma}
The spectral scheme $U \subseteq \GG[p]$ is specified by its underlying (classical) scheme $U_0 \subseteq \GG_0[p]$ over $\pi_0(k)$. If $Y$ is a $\pi_0(k)$-scheme, a map $Y \to U_0$ is equivalent to the data of a homomorphism $f: \Z/p \to \GG_Y = \GG \times_{\spec \pi_0(k)} Y$ such that $f(i)$ is not the identity section for $i\in \Z/p - \{0\}$. This implies that $f$ exhibits $\Z/p$ as a closed subgroup scheme of $\GG_Y$ which is isomorphic to the Cartier divisor $\sum_{j\in \FF_p} f(j)$.
\begin{construction}\label{cstr: genuine S1 frobenius}
    Over $U_0$, there is a universal isogeny $q_0: \GG_{0,U_0} \to \GG_{0,U_0}$ given by quotienting by the subgroup scheme $\Z/p \cong \sum_{j\in \FF_p} f(j)$. This isogeny defines an \textit{\'etale} morphism $\co_{\GG_{0,U_0}} \to \co_{\GG_{0,U_0}}$; so \cite[Theorem 7.5.0.6]{HA} implies that the isogeny $q_0$ lifts to a map $q: \GG_U \to \GG_U$ over $\spec \Phi^{\Cp} k$. (In general, $q_0$ is to be understood as an analogue for $\GG_0$ of the Artin-Schreier map on $\GG_a$.) The map \cref{eq: phi on genuine S1 equiv} is then given by the composite
    $$\GG_U \xrightarrow{q} \GG_U \simeq \GG \times_{\spec k} \spec \Phi^{\Cp} k \xrightarrow{\pr} \GG.$$
    We will denote its underlying map by
    $$\varphi_0: \GG_0 \times_{\spec \pi_0(k)} \spec \pi_0(\Phi^{\Cp} k) \to \GG_0.$$
\end{construction}
\begin{example}\label{ex: genuine isogeny examples}
    Let us explicate \cref{cstr: genuine S1 frobenius} in two examples.
    \begin{enumerate}
        \item\label{item: Z isogeny} Let $k = \Z[u^{\pm 1}]$ and $\GG = \GG_a$. Then $U_0 = \spec \FF_p[t^{\pm 1}]$, and the isogeny $q: \GG_{0,U_0} \to \GG_{0,U_0}$ is given by the Artin-Schreier map 
        $$x\mapsto x - t^{p-1} x^p.$$
        \item\label{item: KU isogeny} Let $k = \KU$ and $\GG = \GG_m$ with coordinate $y$. Then $U_0 = \spec \Z[\zeta_p][1/p]$, and $q: \GG_{0,U_0} \to \GG_{0,U_0}$ is given by the map 
        $$y\mapsto 1 + \prod_{j\in \FF_p} (y - \zeta_p^j) = y^p.$$
    \end{enumerate}
\end{example}
\begin{remark}\label{rmk: interpolating artin schreier}
    Let us mention for the sake of completeness that one can interpolate between the two cases in \cref{ex: genuine isogeny examples}, using the group scheme $\GG$ over connective complex K-theory $\ku$ studied in \cite{ku-rel-langlands}. (Using this, the results discussed below can be extended to the case $k = \ku$, too, but we will not address this here.) Let $\GG_\beta := \GG_0$ denote its underlying group scheme. Explicitly, $\GG_\beta$ is the group scheme over $\Z[\beta]$ given by $\spec \Z[\beta, v^{\pm 1}][\tfrac{v-1}{\beta}]$, where the group law is determined by $v \mapsto v \otimes v$. In an abuse of notation, we will also write $\GG_{f(\beta)}$ for an element $f(\beta)\in \Z[\beta]$ to denote the group scheme given by $\spec \Z[\beta, v^{\pm 1}][\tfrac{v-1}{f(\beta)}]$; hopefully this will not cause any confusion to the reader. We will (perhaps unexpectedly) define $t^{-1} := \tfrac{v-1}{\beta}$, and also define the scheme
    $$U_0 = \spec \Z[\beta, v^{\pm 1}][\tfrac{v-1}{\beta}, \tfrac{\beta^{p-1}}{(v-1)\cdots(v^{p-1}-1)}]/\tfrac{v^p-1}{\beta}.$$
    Note that $v = \zeta_p$ is a primitive $p$th root of unity, and $\beta = \tfrac{\zeta_p - 1}{t^{-1}}$. The scheme $U_0$ is rather remarkable: its fiber over the locus where $\beta$ is a unit is precisely $\spec \Z[\zeta_p, \beta^{\pm 1}][1/p]$, while its fiber over $\beta = 0$ is given by $\spec \FF_p[t^{\pm 1}]$. (In homotopy theory, $U_0$ arises as $\spec \pi_\ast (\Phi^{\Cp} \ku)$, where $\ku$ is connective complex K-theory.)
    
    Let $y$ denote the invertible coordinate on $\GG_{\beta,U_0}$, and let $x = \tfrac{y-1}{\beta}$. Then the map $q: \GG_{\beta,U_0} \to \GG_{p\beta,U_0}$ is given by the map $y \mapsto y^p$ and $\beta \mapsto p\beta$, so that it sends
    $$q: x = \tfrac{y-1}{\beta} \mapsto \tfrac{y^p - 1}{p\beta} = \tfrac{(1+\beta x)^p - 1}{p\beta}.$$
    We claim that, as a morphism over $\spec \Z[\beta]$, this map interpolates between the isogenies of Cases \ref{item: Z isogeny} and \ref{item: KU isogeny} in \cref{ex: genuine isogeny examples}. First, it is obvious that when $\beta$ is a unit, we simply recover Case \ref{item: KU isogeny}. Next, let us consider the fiber over $\beta = 0$. Recall that $\beta = (\zeta_p - 1)t$, so the binomial theorem gives
    $$q(x) = \tfrac{1}{p} \sum_{i=1}^p \binom{p}{i} \beta^{i-1} x^i = \sum_{i=1}^p \tfrac{(\zeta_p - 1)^{i-1}}{p} \binom{p}{i} t^{i-1} x^i.$$
    Almost all terms vanish modulo $\beta$, except for the terms $i=1,p$; one is left with
    $$q(x) \equiv x + \tfrac{(\zeta_p - 1)^{p-1}}{p} t^{p-1} x^p = x + \tfrac{t^{p-1} x^p}{[1]_{\zeta_p} \cdots [p-1]_{\zeta_p}} = x - t^{p-1} x^p \pmod{\beta},$$
    as desired. (Here, $[j]_q = \tfrac{q^j - 1}{q-1}$ is the $q$-integer corresponding to $j\in \Z$; we are using the fact that $[1]_{\zeta_p} \cdots [p-1]_{\zeta_p} \equiv -1\pmod{\zeta_p - 1}$, which amounts to the fact that $(p-1)! = -1 \in \FF_p$.) In general, $\ku$ gives a degeneration of power operations/the $p$th Adams operation on $\KU$ to power operations/the Steenrod algebra action on ordinary cohomology; this goes back (albeit not in the form presented above) to \cite[Proposition 6.4 and Theorem 6.5]{atiyah-power-operations}.
\end{remark}
For any compact torus $T_c$, we obtain a map 
$$\varphi_T: \cM_T \times_{\spec k} \spec \Phi^{\Cp} k \to \cM_T,$$
whose underlying map on classical $\pi_0(k)$-schemes will be denoted by $\varphi_{T,0}$.
If $X$ is any (ind-)finite $T_c$-space $X$, we then obtain maps
$$\cf_T(X) \to \varphi_{T,\ast} \varphi_T^\ast \cf_T(X), \ \cf_T(X)^\vee \to \varphi_{T,\ast} \varphi_T^\ast (\cf_T(X)^\vee).$$
We will denote these maps by $\varphi_{T,X}$, and call them the \textit{$T_c$-equivariant decompleted Frobenius}. Note that $\varphi_{T,X}$ on $\cf_T(X)$ is a map of $\Eoo$-algebras in $\QCoh(\cM_T)$, and similarly $\varphi_{T,X}$ on $\cf_T(X)^\vee$ is a map of $\Eoo$-coalgebras in $\QCoh(\cM_T)$. The map $\varphi_{T,X}$ in fact comes from a functor
$$\varphi_{T,X}: \Loc_T(X; k) \to \Loc_T(X; \Phi^{\Cp} k).$$
\begin{remark}
    It is easy to see that if $\cH(\GG_0, T, W)$ denotes the nil-Hecke algebra from \cref{def: nil-hecke} associated to a root system with torus $T$ and Weyl group $W$, then the map $\varphi_{T,0}$ induces a map $\cH(\GG_0, T, W) \to \cH(\GG_0, T, W) \otimes_{\pi_0(k)} \pi_0(\Phi^{\Cp} k)$. This map is very interesting, but we will postpone a detailed study of its combinatorial implications to a future article. When $\GG_0 = \GG_a$, for instance, this map describes the total Steenrod operation on the nil-Hecke algebra; similar ideas are explored in \cite{kitchloo-steenrod, beliakova-cooper}.
\end{remark}

We will now study the case $X = \Gr_G$, where $G$ is connected, almost simple, and simply-laced over $\cc$. For notational simplicity, we will write $\Loc^\gr_T(\Gr_G; \Phi^{\Cp} k)$ to denote the tensor product $\Loc^\gr_T(\Gr_G; k) \otimes_{\pi_0 k} \pi_0 (\Phi^{\Cp} k)$. The $T_c$-equivariant decompleted Frobenius on $\cf_T(\Gr_G)^\vee$ induces a functor 
\begin{equation}\label{eq: frob on loc gr}
    (\varphi_{T, \Gr_G})_\ast: \Loc^\gr_{T_c}(\Gr_G; k) \to \Loc^\gr_{T_c}(\Gr_G; \Phi^{\Cp} k).
\end{equation}
Moreover, the homomorphism $\varphi_{\ld{T},0}$ induces a map in the opposite direction on Cartier duals, and hence a morphism
\begin{equation}\label{eq: frob on Bun B^}
    \varphi_{T,0}: \Bun_{\ld{B}}^0(\GG_0^\vee) \times_{\spec \pi_0(k)} \spec \pi_0(\Phi^{\Cp} k) \to \Bun_{\ld{B}}^0(\GG_0^\vee).
\end{equation}
\begin{remark}
    In fact, it follows from \cref{rmk: total power op and phiCp} that one can replace $\pi_0(\Phi^{\Cp} k)$ above by the fiber product $\pi_0(\Phi^{\Cp} k) \times_{\pi_0 \cf_k(B\Cp)[1/t]} \pi_0 (\cf_k(B\Cp))/I_\tr$, which has the effect of working in a ``$t$-lattice'' inside $\pi_0(\Phi^{\Cp} k)$. For simplicity, we will ignore this point below, and just work with $\pi_0(\Phi^{\Cp} k)$.
\end{remark}
The following says that the map \cref{eq: frob on Bun B^} is precisely the effect of the $T_c$-equivariant decompleted Frobenius under Langlands duality.
\begin{theorem}\label{thm: frobenius and langlands}
    Under the equivalence of \cref{thm: intro omnibus} as rephrased in \cref{rmk: 1-shifted cartier} (which continues to hold true in the case $k = \Z[u^{\pm 1}]$, at least upon inverting enough primes), the functor $(\varphi_{\ld{T}, \Gr_G})_\ast$ of \cref{eq: frob on loc gr} identifies with the functor given by pullback along the map \cref{eq: frob on Bun B^}. That is, the following diagram commutes:
    $$\xymatrix{
    F \otimes_{\pi_0(k)} \Loc^\gr_{\ld{T}_c}(\Gr_G; k) \ar[r]^-{(\varphi_{\ld{T}, \Gr_G})_\ast} \ar[d]_-\sim & F \otimes_{\pi_0(k)} \Loc^\gr_{\ld{T}_c}(\Gr_G; \Phi^{\Cp} k) \ar[d]^-\sim \\
    \QCoh(\Bun_{\ld{B}}^0(\GG_0^\vee)^\reg) \ar[r]_-{\varphi_{\ld{T},0}^\ast} & \QCoh(\Bun_{\ld{B}}^0(\GG_0^\vee)^\reg) \otimes_{\pi_0(k)} \pi_0(\Phi^{\Cp} k).
    }$$
\end{theorem}
\begin{proof}
    The argument is essentially that of \cref{prop: cplx conj KU and B mod B^}, so we only give a sketch. Let us begin by observing that if $\kappa: \cM_{\ld{T}, 0} \to \Bun_{\ld{B}}^0(\GG_0^\vee)$ denotes the Kostant section, there is a commutative diagram
    $$\xymatrix{
    \cM_{\ld{T},0} \times_{\spec \pi_0(k)} \spec \pi_0(\Phi^{\Cp} k) \ar[d]_-\kappa \ar[r]^-{\varphi_{\ld{T},0}} & \cM_{\ld{T},0} \ar[d]^-\kappa \\
    \Bun_{\ld{B}}^0(\GG_0^\vee) \times_{\spec \pi_0(k)} \spec \pi_0(\Phi^{\Cp} k) \ar[r]_-{\varphi_{\ld{T},0}} & \Bun_{\ld{B}}^0(\GG_0^\vee).
    }$$
    The proof of \cref{thm: intro omnibus} shows that it suffices to prove that under the isomorphism
    \begin{equation}\label{eq: iso of reg centr for frob}
        \spec_{\cM_{\ld{T},0}}(\pi_0 \cf_{\ld{T}}(\Gr_G)^\vee) \cong \cM_{\ld{T}, 0} \times_{\Bun_{\ld{B}}^0(\GG_0^\vee)} \cM_{\ld{T},0},
    \end{equation}
    the $\ld{T}_c$-equivariant decompleted Frobenius on $\cf_{\ld{T}}(\Gr_G)^\vee$ identifies with the effect of the map $\varphi_{\ld{T},0}$ on the right-hand side. For brevity, we will phrase this condition as the ``Frobenius-equivariance'' of \cref{eq: iso of reg centr for frob}.
    
    Let $\cM_{\ld{T},0}^\gen \subseteq \cM_{\ld{T},0}$ denote the complement of $\bigcup_{\alpha} \cM_{\ld{T}_\alpha,0}$ as $\alpha$ ranges over the roots of $\ld{G}$, and $\ld{T}_\alpha$ denotes the kernel of the map $\alpha: \ld{T} \to \GG_m$. Since both sides of \cref{eq: iso of reg centr for frob} are flat over $\cM_{\ld{T},0}$, their sheaves of functions inject into the corresponding localizations along the map $\cM_{\ld{T},0}^\gen \subseteq \cM_{\ld{T},0}$. It therefore suffices to show that when restricted to $\cM_{\ld{T},0}^\gen$, the isomorphism of \cref{eq: iso of reg centr for frob} is Frobenius-equivariant.

    By \cref{lem: atiyah localization}, there is an isomorphism 
    $$\pi_0 \cf_{\ld{T}}(\Gr_G)^\vee|_{\cM_{\ld{T},0}^\gen} \cong \pi_0 \cf_{\ld{T}}(\Gr_T)^\vee|_{\cM_{\ld{T},0}^\gen} \cong \co_{\cM_{\ld{T},0}^\gen}[\bX_\ast(T)].$$
    Under this isomorphism, the $\ld{T}_c$-equivariant decompleted Frobenius is given simply by the Frobenius on $\cM_{\ld{T},0}^\gen$, and acts trivially on $\bX_\ast(T)$.
    Similarly, there is an isomorphism
    $$(\cM_{\ld{T}, 0} \times_{\Bun_{\ld{B}}^0(\GG_0^\vee)} \cM_{\ld{T},0}) \times_{\cM_{\ld{T},0}} \cM_{\ld{T},0}^\gen \cong \cM_{\ld{T},0}^\gen \times \ld{T}.$$
    Under this isomorphism, the action of $\varphi_{\ld{T},0}$ is given simply by the Frobenius on $\cM_{\ld{T},0}^\gen$, and acts trivially on $\ld{T}$. It is clear that this matches with the Frobenius on $\pi_0 \cf_{\ld{T}}(\Gr_G)^\vee|_{\cM_{\ld{T},0}^\gen}$, as desired.
\end{proof}
The entire discussion above can be adapted without much difficulty to the setting of $G_c$-equivariant local systems. If $G$ is almost simple and simply-laced, and has torsion-free fundamental group, then the analogue of \cref{thm: frobenius and langlands} states the following. Under the equivalence of \cref{rmk: 1-shifted cartier} (which continues to hold true in the case $k = \Z[u^{\pm 1}]$, at least upon inverting enough primes), the following diagram commutes:
$$\xymatrix{
F \otimes_{\pi_0(k)} \Loc^\gr_{\ld{G}_c}(\Gr_G; k) \ar[r]^-{(\varphi_{\ld{G}, \Gr_G})_\ast} \ar[d]_-\sim & F \otimes_{\pi_0(k)} \Loc^\gr_{\ld{G}_c}(\Gr_G; \Phi^{\Cp} k) \ar[d]^-\sim \\
\QCoh(\Bun_{\ld{G}}^\ss(\GG_0^\vee)^\reg) \ar[r]_-{\varphi_{\ld{G},0}^\ast} & \QCoh(\Bun_{\ld{G}}^\ss(\GG_0^\vee)^\reg) \otimes_{\pi_0(k)} \pi_0(\Phi^{\Cp} k).
}$$
The top and bottom maps are defined just as in \cref{eq: frob on loc gr} and \cref{eq: frob on Bun B^}. 
\begin{remark}\label{rmk: steenrod on G^?}
    Recall from the proof of \cref{thm: intro omnibus} that there is a closed immersion $\spec \pi_0 \cf_{\ld{T}}(\Gr_G)^\vee \hookrightarrow \ld{G} \times \cM_{T,0}$. One can try to extend the action of the decompleted Frobenius to $\ld{G} \times \cM_{T,0}$ itself, but such an extension will \textit{not} be canonical (and seems to be essentially useless in studying $\ld{G}$).
\end{remark}
\begin{remark}
    Let $\lambda$ be a dominant minuscule weight for $\ld{G}$, and let $G/P_\lambda$ denote the corresponding flag variety for $G$ as in \cref{table: minuscule varieties}. The decompleted Frobenius acts on $\pi_0 \cf_{\ld{T}}(G/P_\lambda)$, and does so compatibly with its action on $\pi_0 \cf_{\ld{T}}(\Gr_G)^\vee$, in the sense that the action of $\spec \pi_0 \cf_{\ld{T}}(\Gr_G)^\vee$ on $\pi_0 \cf_{\ld{T}}(G/P_\lambda)$ is equivariant for the decompleted Frobenius. It would be interesting to understand, in some uniform manner, the action of the decompleted Frobenius on $\pi_0 \cf_{\ld{T}}(G/P_\lambda)$. In the case of ordinary cohomology, this amounts to understanding Steenrod operations on $\H^\ast(G/P_\lambda; \Z)$. This is already interesting in the case when $G$ is of type $A$ (i.e., in the case of Grassmannians), where it was studied, for instance, in \cite{borel-serre-steenrod, borel-serre-steenrod-ii, lance-steenrod-dyer-lashof-BU}.
\end{remark}

Let us now explicate \cref{thm: frobenius and langlands} in some examples. Since the description in the case of elliptic cohomology is not much more explicit than the statement of \cref{thm: frobenius and langlands} -- that is, that the decompleted Frobenius on $\Bun_{\ld{G}}^\ss(E)$ is induced by the degree $p$ \'etale isogeny $E \to E$ over $\spec \pi_0(\Phi^{\Cp} k)$ -- we will mostly focus on the cases of ordinary cohomology and complex K-theory below for simplicity. We will also briefly discuss the example of ``Tate K-theory'', where one can also make the decompleted Frobenius explicit at the level of isomorphism classes of objects of $\Bun_{\ld{G}}^\ss(E)$.

Before proceeding, we warn the reader that our discussion above only shows that the decompleted Frobenius is canonically defined on the stack $\Bun_{\ld{G}}^\ss(\GG_0^\vee)$, and not necessarily on a uniformization. For instance, when $\GG_0 = \GG_a$, so that $\Bun_{\ld{G}}^\ss(\GG_0^\vee) = \ld{\g}/\ld{G}$, we will often compute the decompleted Frobenius as a map on $\ld{\g}$; but the resulting formulas are only unique up to $\ld{G}$-conjugation.
\begin{example}\label{ex: steenrod operations langlands}
    Let $k = \Z[u^{\pm 1}]$, $\GG = \GG_a$, and invert $N\gg 0$ so that the equivalence of \cref{cor: reg locus ordinary ABG} continues to hold: that is, so that there is an equivalence $\Loc^\gr_{\ld{T}_c}(\Gr_G; k) \simeq \QCoh(\ld{\fr{b}}^\reg/\ld{B})$. This can be proved by showing that the isomorphism of \cref{thm: ordinary hmlgy reg centr} over $\spec \Z[1/N]$ for some $N\gg 0$. In fact, \cite{homology-langlands} shows that one can take $N$ to be the integer $n_G$ from \cite[Remark 5.8]{homology-langlands}.\footnote{In fact, in the setup at hand, one can take $N = 1$.}
    Under the identification $\Bun_{\ld{B}}^0(\GG_0^\vee) \cong \ld{\fr{b}}/\ld{B}$, the map \cref{eq: frob on Bun B^} is given (for $p\nmid N$) by the map
    $$\varphi_{\ld{T},0}: (\ld{\fr{b}} \times_{\spec \Z[1/N]} \spec \FF_p[t^{\pm 1}])/\ld{B} \to \ld{\fr{b}}/\ld{B}$$
    which is the $\ld{B}$-quotient of the map
    $$\ld{\fr{b}} \times_{\spec \Z[1/N]} \spec \FF_p[t^{\pm 1}] \to \ld{\fr{b}}, \ (x,t) \mapsto x - t^{p-1} x^{[p]}.$$
    Here, $x^{[p]}$ denotes the restricted Lie operation on $\ld{\fr{b}}$. It follows from \cref{thm: frobenius and langlands} that this map implements the action of the decompleted Frobenius/Steenrod operations on $\Loc^\gr_{\ld{T}_c}(\Gr_G; \Z[u^{\pm 1}])$ (upon inverting $N \gg 0$).
    
    Similarly, under the identification $\Bun_{\ld{G}}^\ss(\GG_0^\vee) \cong \ld{\g}/\ld{G}$, the analogue of the map \cref{eq: frob on Bun B^} is given (for $p\nmid N$) by the map
    $$\varphi_{\ld{G},0}: (\ld{\g} \times_{\spec \Z[1/N]} \spec \FF_p[t^{\pm 1}])/\ld{G} \to \ld{\g}/\ld{G}$$
    which is the $\ld{G}$-quotient of the map
    \begin{equation}\label{eq: artin-schreier on ld-g}
        \ld{\g} \times_{\spec \Z[1/N]} \spec \FF_p[t^{\pm 1}] \to \ld{\g}, \ (x,t) \mapsto x - t^{p-1} x^{[p]}.
    \end{equation}
    Again, this map implements the action of the decompleted Frobenius/Steenrod operations on $\Loc^\gr_{\ld{G}_c}(\Gr_G; \Z[1/N,u^{\pm 1}])$ under the equivalence between $\Loc^\gr_{\ld{G}_c}(\Gr_G; \Z[1/N,u^{\pm 1}])$ and $\QCoh(\ld{\g}^\reg/\ld{G})$.

    For instance, suppose $G = \SL_2$, and assume $p>2$. When restricted to the Kostant slice $f + \ld{\g}^e = \left\{\begin{psmallmatrix}
        0 & x \\
        1 & 0
    \end{psmallmatrix}\right\} \subseteq \ld{\g} = \pgl_2$, the map $\varphi_{\ld{G},0}$ sends 
    $$\left(\begin{psmallmatrix}
        0 & x \\
        1 & 0
    \end{psmallmatrix}, t\right) \mapsto \begin{psmallmatrix}
        0 & x - t^{p-1} x^{(p+1)/2} \\
        1 - t^{p-1} x^{(p-1)/2} & 0
    \end{psmallmatrix}.$$
    This is conjugate to the matrix $\begin{psmallmatrix}
        0 & x(1 - t^{p-1} x^{(p-1)/2})^2 \\
        1 & 0
    \end{psmallmatrix}$, so we find that $\varphi_{\ld{G},0}$ is given in coordinates by the map 
    $$\varphi_{\ld{G},0}: x \mapsto x - 2 t^{p-1} x^{(p+1)/2} + t^{2(p-1)} x^p = \prod_{j\in \FF_p} (x - j^2 t^2)$$
    on $f + \ld{\g}^e$. Under the isomorphism $f + \ld{\g}^e \cong \spec \H^\ast_{\SU(2)}(\ast; \Z)$, the coordinate $x$ identifies with the first Pontryagin class $p_1$; and $\varphi_{\ld{G},0}(x)$ is exactly the total Steenrod operation on this class, as expected. Alternatively, one could conjugate the above description of $\varphi_{\ld{G},0}$ to find that the decompleted Frobenius acts on a binary quadratic form $q(x,y)\in \pgl_2 = \Sym^2(\AA^2)$ by
    $$\varphi_{\SL_2,0}: q(x,y) \mapsto (1 - t^{p-1} \det(q)^{(p-1)/2}) q(x,y),$$
    where $\det(q)$ is the discriminant of $q$.
\end{example}

\cref{ex: steenrod operations langlands} has the following algebraic consequence. This result is not new, and can be found in the literature as \cite[Section 4.1]{jantzen-kohomologie};
it also holds in the non-simply-laced case. (The proof below is a thinly veiled topological analogue of Jantzen's argument.)
\begin{prop}\label{prop: pth power zero on nilcone}
    The map $x \mapsto x^{[p]}$ is zero on the nilpotent cone $\ld{\cN} \subseteq \ld{\g}$ if $p$ is at least the Coxeter number of $\ld{G}$.
\end{prop}
\begin{proof}
    The map $\varphi_{\ld{G},0}$ from \cref{eq: artin-schreier on ld-g} is given by taking affine closures of the map 
    $$\ld{\g}^\reg \times_{\spec \Z[1/N]} \spec \FF_p[t^{\pm 1}] \to \ld{\g}^\reg, \ (x,t) \mapsto x - t^{p-1} x^{[p]}.$$
    It suffices to show that $\varphi_{\ld{G},0}|_{\ld{\cN}^\reg}$ sends $x\mapsto x$. If we identify $\ld{\cN}^\reg = \ld{G}/Z_{\ld{G}}(e)$ and $Z_{\ld{G}}(e) = \spec \H_\ast(\Gr_G; \Z[1/N])$ by \cref{thm: ordinary hmlgy reg centr} (or, \cite[Theorem 6.1]{homology-langlands}), then $\varphi_{\ld{G},0}|_{\ld{\cN}^\reg}$ is induced by the map 
    $$Z_{\ld{G}}(e) \times_{\spec \Z[1/N]} \spec \FF_p[t^{\pm 1}] \to Z_{\ld{G}}(e)$$
    coming from the decompleted Frobenius/total Steenrod operation on $\H_\ast(\Gr_G; \Z[1/N])$. It therefore suffices to show that the decompleted Frobenius acts by the identity on $\H_\ast(\Gr_G; \Z[1/N])$.

    This can be proved using the generating complexes from \cite{bott-space-of-loops}, as elaborated upon in \cite{littig-mitchell}. Namely, recall that if $X$ is a homotopy commutative H-space and $f: Y \to X$ is a map from a CW-complex into $X$, then $f$ is said to exhibit $Y$ as a generating complex for $X$ if $f$ induces a surjection $\Sym(\H_\ast(Y; \Z[1/N])) \twoheadrightarrow \H_\ast(X; \Z[1/N])$. In \cite{littig-mitchell}, it was shown that if $\theta$ denotes the highest (short) coroot of $G$, then the Schubert variety $\ol{\Gr_G^{-\theta}}$ corresponding to the antidominant weight $-\theta$ is a generating complex for $\Gr_G$. Since $\H_\ast(\ol{\Gr_G^{-\theta}}; \Z[1/N])$ generates $\H_\ast(\Gr_G; \Z[1/N])$ as a ring, and the decompleted Frobenius is a ring map, it suffices to show that the decompleted Frobenius/total Steenrod operation on $\H_\ast(\ol{\Gr_G^{-\theta}}; \Z[1/N])$ sends $x\mapsto x$. Equivalently, it suffices to show that all Steenrod operations $P^i$ act trivially on $\H_\ast(\ol{\Gr_G^{-\theta}}; \Z[1/N])$ for $i>0$.
    
    To see this, observe that the dimension of $\ol{\Gr_G^{-\theta}}$ is given by $2(h-1)$, where $h$ is the Coxeter number of $\ld{G}$. The operation $P^i$ sends a class in $\H_\ast(\ol{\Gr_G^{-\theta}}; \Z[1/N])$ in homological degree $j$ to a class in $\H_\ast(\ol{\Gr_G^{-\theta}}; \FF_p)$ in homological degree $j - 2i(p-1)$. Since $\H_\ast(\ol{\Gr_G^{-\theta}}; \FF_p)$ is concentrated in nonnegative degrees and $p\geq h$, we see that $P^i$ could only possibly act nontrivially when $p=h$ and $i=1$, and that too only on classes in $\H_{2(h-1)}(\ol{\Gr_G^{-\theta}}; \Z[1/N])$. However, $P^1$ applied to such a class would land in $\H_0(\ol{\Gr_G^{-\theta}}; \FF_p)$. This implies that it is zero: any Steenrod operation landing in $\H_0(X; \FF_p)$ necessarily vanishes if $X$ is a connected space.
\end{proof}
\begin{example}\label{ex: small p and steenrod on SLn}
    Running the argument of \cref{prop: pth power zero on nilcone} backwards tells us that if the map $x \mapsto x^{[p]}$ is not zero on the nilpotent cone $\ld{\cN} \subseteq \ld{\g}$, then the decompleted Frobenius/total Steenrod operation on $\H_\ast(\Gr_G; \Z)$ must be nontrivial. (The following example was shown to me by David Treumann, and was my impetus for more generally exploring the decompleted Frobenius.) Indeed, suppose (for simplicity) that $G = \SL_3$ and $p = 2$. Then the map $x \mapsto x^{[2]}$ is not zero on the nilpotent cone in $\fr{pgl}_3$, and in fact the map $\varphi: x \mapsto x - t^{p-1} x^{[p]}$ sends the principal nilpotent $e = \begin{psmallmatrix}
        0 & 1 & 0\\
        0 & 0 & 1\\
        0 & 0 & 0
    \end{psmallmatrix}$ to the principal nilpotent $\begin{psmallmatrix}
        0 & 1 & t\\
        0 & 0 & 1\\
        0 & 0 & 0
    \end{psmallmatrix}$. This is conjugate to $e$ itself by the matrix $n_e = \begin{psmallmatrix}
        1 & t & 0\\
        0 & 1 & 0\\
        0 & 0 & 0
    \end{psmallmatrix}$. Conjugating the centralizer $Z_{\ld{G}}(e)$ by $n_e$ sends 
    \begin{equation}\label{eq: frob on homology of SL3}
        \begin{psmallmatrix}
        1 & a & b\\
        0 & 1 & a\\
        0 & 0 & 1
        \end{psmallmatrix} \mapsto \begin{psmallmatrix}
        1 & a & b + at\\
        0 & 1 & a\\
        0 & 0 & 1
        \end{psmallmatrix}.
    \end{equation}
    Indeed, this is exactly how the decompleted Frobenius acts on $\H_\ast(\Gr_{\SL_3}; \Z) = \Z[a,b]$. (One can verify this by observing that the generating complex in this case is given by the map $\CP^2 \to \Gr_{\SL_3}$. The $2$- and $4$-cells of $\CP^2$ give the classes $a$ and $b$, respectively, and they are connected by the Steenrod square $\Sq^2$.)
    Note that since the action of the decompleted Frobenius on $Z_{\ld{G}}(e)$ is just conjugation by $n_e$, one can extend it to an action on all of $\ld{G}$. However, the element $n_e$ is not canonical, and a different choice of $n_e$ will act differently on $\ld{G}$.

    Recall from \cref{rmk: steenrod on G^?} that if $\lambda$ denotes a dominant minuscule weight for $\PGL_3$, the action of $\spec \H_\ast(\Gr_G; \Z)$ on $\H^\ast(\PGL_3/P_\lambda; \Z)$ must be equivariant for the decompleted Frobenius. Let us quickly verify this in the case when $\lambda$ is the fundamental weight: in this case, $\PGL_3/P_\lambda = \CP^2$, and if we write $\H^\ast(\PGL_3/P_\lambda; \Z) = \Z\{x,y,z\}$, the decompleted Frobenius sends $y \mapsto y + tz$. (Indeed, $\H^\ast(\CP^2; \Z) \cong \Z[w]/w^3$, and the total Steenrod operation sends $w \mapsto w + tw^2$. Writing $x = w^0$, $y = w$, and $z = w^2$ gives the desired claim.) It is straightforward to see that the action of $Z_{\ld{G}}(e)$ on $\Z\{x,y,z\}$ is equivariant for the decompleted Frobenius as described in \cref{eq: frob on homology of SL3}.
\end{example}

\begin{example}\label{ex: adams operations langlands}
    Let $k = \KU$ and $\GG = \GG_m$. Under the identification $\Bun_{\ld{B}}^0(\GG_0^\vee) \cong \ld{B}/\ld{B}$, the map \cref{eq: frob on Bun B^} is given by the $\ld{B}$-quotient of the $p$th power map on $\ld{B}$. That is, if $F$ is an algebraically closed field, then under the equivalence
    $$\Loc^\gr_{\ld{T}_c}(\Gr_G; \KU) \otimes_\Z F \simeq \QCoh(\ld{B}^\reg/\ld{B})$$
    of \cref{cor: ku reg locus ordinary ABG}, the decompleted Frobenius on the left-hand side (which encodes the $p$th Adams operation on $\KU$) identifies with the $p$th power map on $\ld{B}^\reg$. Similarly, under the equivalence
    $$\Loc^\gr_{\ld{G}_c}(\Gr_G; \KU) \otimes_\Z F \simeq \QCoh(\ld{G}^\reg/\ld{G}),$$
    the decompleted Frobenius on the left-hand side (which encodes the $p$th Adams operation on $\KU$) identifies with the $p$th power map on $\ld{G}^\reg$.

    For instance, suppose $\ld{G} = \SL_2$. When restricted to the Kostant slice inside $\ld{G} = \SL_2$ of matrices of the form $\begin{psmallmatrix}
        x-1 & x-2 \\
        1 & 1
    \end{psmallmatrix}$, the map $\varphi_{\ld{G},0}$ is given by raising to the $p$th power. It turns out that
    $$\begin{psmallmatrix}
        x-1 & x-2 \\
        1 & 0
    \end{psmallmatrix}^p \text{ is conjugate to } \kappa(x) = \begin{psmallmatrix}
        L_p(x)-1 & L_p(x)-2 \\
        1 & 1
    \end{psmallmatrix},$$
    where $L_n(x)$ is the $n$th ``Lucas polynomial'', given by
    $$L_n(x) = \sum_{j=0}^{\lfloor n/2\rfloor} (-1)^j \tfrac{n}{n-j} \binom{n-j}{j} x^{n-2j} = D_n(x,1).$$
    Here, $D_n(x,\alpha)$ is the ``Dickson polynomial'' from \cite{dickson-polynomial}. We therefore find that $\varphi_{\ld{G},0}$ is given on $g\in \SL_2$ by the map 
    $$\varphi_{\ld{G},0}(g) = L_p(g).$$
    Under the isomorphism between the Kostant slice for $\SL_2$ and $\spec \pi_0 \KU_{\SU(2)}$, the coordinate $x$ identifies with the $\KU$-theoretic Pontryagin class; and $\varphi_{\ld{G},0}(x)$ is exactly the $p$th Adams operation on this class. 
\end{example}
\begin{remark}\label{rmk: adams connective ku}
    In fact, one can interpolate between \cref{ex: steenrod operations langlands} and \cref{ex: adams operations langlands} using the results of \cite{ku-rel-langlands} and \cref{rmk: interpolating artin schreier}. To state the result, we will use notation from \cref{rmk: interpolating artin schreier}. Namely, the aforementioned results imply that there are equivalences
    \begin{align}
        \Loc^\gr_{\ld{T}_c}(\Gr_G; \ku) \otimes_\Z F & \simeq \QCoh(\ld{B}_\beta^\reg/\ld{B}) \\
        \Loc^\gr_{\ld{G}_c}(\Gr_G; \ku) \otimes_\Z F & \simeq \QCoh(\ld{G}_\beta^\reg/\ld{G}), \label{eq: G-equiv reg satake connective ku}
    \end{align}
    where, for a group scheme $H$, we define $H_\beta$ to be (the stacky quotient by $\GG_m$ of) the $1$-parameter degeneration of $H$ into its Lie algebra. Explicitly, $H_\beta = \Hom(\DD(\GG_0), H)$, where $\DD(\GG_0)$ is the Cartier dual of the $1$-dimensional group scheme over $\spec(\pi_\ast(\ku))/\GG_m = \spec (\Z[\beta])/\GG_m$ from \cref{rmk: interpolating artin schreier}. For instance, if $H = \SL_n$, then $\SL_{n,\beta}$ consists of (the stacky quotient by $\GG_m$ of) the group scheme of those $n\times n$-matrices $x$ such that $\tfrac{\det(\id + \beta x) - 1}{\beta} = 0$.
    
    For simplicity, let us focus on the equivalence \cref{eq: G-equiv reg satake connective ku} above. The decompleted Frobenius on the topological side of \cref{eq: G-equiv reg satake connective ku} interpolates between the $p$th Adams operation and the total Steenrod operation, and it identifies with pullback along the map $(\ld{G}_\beta \times_{\spec \Z[\beta]} U_0)/\ld{G} \to \ld{G}_{p\beta}/\ld{G}$ given by the $\ld{G}$-quotient of the map on $\ld{G}_\beta$ defined by
    $$x\mapsto \tfrac{(1 + \beta x)^p - 1}{p\beta}.$$
    Using an argument similar to \cref{rmk: interpolating artin schreier}, one finds that when $\beta = 0$, the above map reduces to the Artin-Schreier map on $\ld{\g}$ from \cref{ex: steenrod operations langlands}. In the case $\ld{G} = \SL_2$, for instance, the decompleted Frobenius on the Kostant slice is given by $x\mapsto f_p(x)$, where $f_n(x)$ is the polynomial 
    given by\footnote{For instance, $f_2(x) = 4x - \beta^2 x^2$, $f_3(x) = \beta^4 x^3 - 6\beta^2 x^2 + 9x$, and $f_5(x) = \beta^8 x^5 - 10 \beta^6 x^4 + 35 \beta^4 x^3 - 50 \beta^2 x^2 + 25 x$.}
    $$f_n(x) = \sum_{j=0}^{n - 1} (-1)^j \tfrac{2n}{2n-j} \binom{2n-j}{j} \beta^{2(n-j)-2} x^{n-j} = \beta^{-2} (D_{2n}(\beta x^{1/2}, 1) - 2),$$
    where $D_n(x,\alpha)$ is the ``Dickson polynomial'' from \cite{dickson-polynomial}.
    Elementary arithmetic manipulations confirm that the polynomial $f_p(x)$ indeed computes the decompleted Frobenius on $\spec \pi_0 \ku_{\SU(2)}$, and furthermore that upon writing $\beta = (\zeta_p - 1)t$ in $\co_{U_0}$, we have 
    $$\tfrac{f_p(x)}{(\zeta_p - 1)^{2(p-1)}} = \sum_{j=0}^{p - 1} \tfrac{(-1)^j}{(\zeta_p - 1)^{2j}} \tfrac{2p}{2p-j} \binom{2p-j}{j} t^{2(p-1-j)} x^{p-j}.$$
    Upon reducing modulo $\zeta_p-1$, only the terms indexed by $j=0,\tfrac{p-1}{2}$, and $p-1$ survive. When $j=\tfrac{p-1}{2}$, the coefficient of $t^{p-1} x^{(p+1)/2}$ is 
    $$\tfrac{(-1)^{(p-1)/2}}{(\zeta_p - 1)^{p-1}} \tfrac{4p}{3p+1} \binom{(3p+1)/2}{(p-1)/2} \equiv -2\pmod{(\zeta_p-1)},$$
    so that
    $$\tfrac{f_p(x)}{(\zeta_p - 1)^{2(p-1)}} \equiv x - 2t^{p-1} x^{(p+1)/2} + t^{2(p-1)} x^p \pmod{(\zeta_p - 1)}.$$
    This is exactly as expected from \cref{ex: steenrod operations langlands}.
\end{remark}
\begin{example}\label{ex: small p and adams on SLn}
    For the sake of completeness, let us explain the analogue of the calculation in \cref{ex: small p and steenrod on SLn} for $\KU$, so that $G = \SL_3$ and $p = 2$.  The map $\varphi: x \mapsto x^2$ sends $e = \begin{psmallmatrix}
        1 & 1 & 0\\
        0 & 1 & 1\\
        0 & 0 & 1
    \end{psmallmatrix}$ to $e^2 = \begin{psmallmatrix}
        1 & 2 & 1\\
        0 & 1 & 2\\
        0 & 0 & 1
    \end{psmallmatrix}$. This is conjugate to $e$ itself by the matrix $n_e = \begin{psmallmatrix}
        4 & 1 & 0\\
        0 & 2 & 0\\
        0 & 0 & 1
    \end{psmallmatrix}$. Conjugating the centralizer $Z_{\ld{G}}(e)$ by $n_e$ sends 
    \begin{equation}\label{eq: frob on KU homology of SL3}
        \begin{psmallmatrix}
        1 & a & b\\
        0 & 1 & a\\
        0 & 0 & 1
        \end{psmallmatrix} \mapsto \begin{psmallmatrix}
        1 & 2a & a + 4b\\
        0 & 1 & 2a\\
        0 & 0 & 1
        \end{psmallmatrix}.
    \end{equation}
    Indeed, this is exactly how the decompleted Frobenius acts on $\KU_0(\Gr_{\SL_3}) = \Z[a,b]$.
    (One can verify this by observing that the generating complex in this case is given by the map $\CP^2 \to \Gr_{\SL_3}$. The $2$- and $4$-cells of $\CP^2$ give the classes $a$ and $b$, respectively, and the Adams operation $\psi^2$ sends $a\mapsto 2a$ and $b\mapsto a + 4b$.)
    Again, since the action of the decompleted Frobenius on $Z_{\ld{G}}(e)$ is just conjugation by $n_e$, one can extend it to an action on all of $\ld{G}$; but the element $n_e$ is not canonical, and a different choice of $n_e$ will act differently on $\ld{G}$.

    Recall from \cref{rmk: steenrod on G^?} that if $\lambda$ denotes a dominant minuscule weight for $\PGL_3$, the action of $\spec \KU_0(\Gr_G)$ on $\KU^0(\PGL_3/P_\lambda)$ must be equivariant for the decompleted Frobenius. Let us quickly verify this in the case when $\lambda$ is the fundamental weight: in this case, $\PGL_3/P_\lambda = \CP^2$, and if we write $\KU^0(\PGL_3/P_\lambda) = \Z\{x,y,z\}$, the decompleted Frobenius sends $y \mapsto z + 2y$ and $z \mapsto 4z$. (Indeed, $\KU^0(\CP^2) \cong \Z[w]/w^3$, and the Adams operation $\psi^2$ is given by the ring map sending $w \mapsto w^2 + 2w$. Writing $x = w^0$, $y = w$, and $z = w^2$ gives the desired claim.) It is straightforward to see that the action of $Z_{\ld{G}}(e)$ on $\Z\{x,y,z\}$ is equivariant for the decompleted Frobenius as described in \cref{eq: frob on KU homology of SL3}.
\end{example}

\begin{example}
    Let $k$ denote \textit{Tate K-theory} \cite[Section 2.7]{ando-hopkins-strickland}, so that $k = \KU\ls{q}$ and $\GG$ is a lift to $k$ of the Tate elliptic curve $\GG_0 = \tate(q)$ over $\Z\ls{q} = \pi_0(k)$. (See \cite[Section 4.3]{survey} for a sketch of the construction of $\GG$.) As usual, we will identify $\tate(q)^\vee$ with $\tate(q)$. Take $F = \cc$, and let $q$ be a point in the punctured open unit disk, so that it defines a continuous embedding $\Z\ls{q} \hookrightarrow \cc$. Then there are equivalences
    \begin{align*}
        \Loc^\gr_{\ld{T}_c}(\Gr_G; \KU\ls{q}) \otimes_{\Z\ls{q}} \cc & \simeq \QCoh(\Bun_{\ld{B}}^0(\tate(q))^\reg) \\
        \Loc^\gr_{\ld{G}_c}(\Gr_G; \KU\ls{q}) \otimes_{\Z\ls{q}} \cc & \simeq \QCoh(\Bun_{\ld{G}}^\ss(\tate(q))^\reg).
    \end{align*}
    The ring $\pi_0 \Phi^{\Cp} k$ and the Frobenius $\varphi: \pi_0(k) \to \pi_0 \Phi^{\Cp} k$ can be computed explicitly using \cref{lem: phi Cp and moduli problem}.
    We will not review the precise description here; instead, we only note that $\varphi$ sends $q\mapsto q^p$ on homotopy, and refer the reader to \cite[Section 6.3]{ando-power-operations} and \cite[Theorem 3.5]{huan-finite-subgroup-tate} for a description of the degree $p$-isogeny $\varphi^\ast \tate(q) \to \varphi^\ast \tate(q)$. This isogeny defines a map $\Bun_{\ld{G}}^\ss(\varphi^\ast \tate(q)) \to \Bun_{\ld{G}}^\ss(\tate(q))$, pullback along which identifies (by \cref{thm: frobenius and langlands}) with the decompleted Frobenius $\Loc^\gr_{\ld{G}_c}(\Gr_G; \KU\ls{q}) \to \Loc^\gr_{\ld{G}_c}(\Gr_G; \Phi^{\Cp} \KU\ls{q})$.

    In \cite{baranovsky-ginzburg}, Baranovsky and Ginzburg explicitly describe the set of $\cc$-points of $\Bun_{\ld{G}}^\ss(\tate(q))$. Namely, define the $q$-twisted conjugation action $G\ls{z}$ on itself as follows:
    $$\Ad^q_{h(z)}(g(z)) := h(qz) g(z) h(z)^{-1}.$$
    Then, there is a natural bijection between $\Bun_{\ld{G}}^\ss(\tate(q))(\cc)$ and the set of those $q$-twisted conjugacy classes in $G\ls{z}$ which contain an element of $G\pw{z}$. Under this bijection, one can show that the decompleted Frobenius on $\Bun_{\ld{G}}^\ss(\tate(q))(\cc)$ can be identified with the effect of the map
    $$g(z) \mapsto g(q^{p-1} z) g(q^{p-2} z) \cdots g(qz) g(z)$$
    on $q$-twisted conjugacy classes in $G\ls{z}$.
\end{example}

The structures imposed by \cref{thm: frobenius and langlands} are quite rigid. For instance, there is an action of $\Loc^\gr_{\ld{G}_c}(\Gr_G; k)$ on $\Loc^\gr_{\ld{T}_c}(\Gr_G; k)$ by convolution, which, under the equivalences of \cref{thm: intro omnibus} as rephrased in \cref{rmk: 1-shifted cartier}, defines an action of $\QCoh(\Bun_{\ld{G}}^\ss(\GG_0^\vee)^\reg)$ on $\QCoh(\Bun_{\ld{B}}^0(\GG_0^\vee)^\reg)$. This action is given by pullback along the map $\Bun_{\ld{B}}^0(\GG_0^\vee) \to \Bun_{\ld{G}}^\ss(\GG_0^\vee)$, and it is compatible with power operations.
\begin{example}\label{ex: steenrod on TA2}
    When $k = \Z[u^{\pm 1}]$ and $\GG = \GG_a$ (where we again invert some $N \gg 0$ so that the equivalence of \cref{cor: reg locus ordinary ABG} holds), the action of $\Loc^\gr_{G_c}(\Gr_G; k)$ on $\Loc^\gr_{T_c}(\Gr_G; k)$ by convolution identifies with the action of $\QCoh(\ld{\g}^{\ast,\reg}/\ld{G})$ on $\QCoh(\ld{\fr{n}}^{\perp,\reg}/\ld{B})$ via pullback along the map $\ld{\fr{n}}^{\perp,\reg}/\ld{B} \to \ld{\g}^{\ast,\reg}/\ld{G}$. It follows from \cref{ex: steenrod operations langlands} that this map is compatible with the decompleted Frobenius/Steenrod operations.
    
    The composite map
    $$\ld{\fr{n}}^{\perp,\reg}/\ld{N} \to \ld{\fr{n}}^{\perp,\reg}/\ld{B} \to \ld{\g}^{\ast,\reg}/\ld{G}$$
    can be realized as the $\ld{G}$-quotient of the restriction to regular loci of the moment map $\mu: T^\ast(\ld{G}/\ld{N}) \to \ld{\g}^\ast$. The action of the decompleted Frobenius/Steenrod operations on the regular locus of $T^\ast(\ld{G}/\ld{N})$ in fact extends to all of $T^\ast(\ld{G}/\ld{N})$ itself (and hence on its affine closure $\ol{T^\ast(\ld{G}/\ld{N})}$), and the moment map $\mu$ is equivariant for this action. The action of the decompleted Frobenius on $\ol{T^\ast(\ld{G}/\ld{N})}$ commutes with the Gelfand-Graev action of the Weyl group from \cref{prop: ordinary gelfand-graev}; this can be seen by reducing to the rank $1$ case described below (with a bit of care in keeping track of the difference between $\AA^2$ and $(\AA^2)^\ast$).
    
    An explicit description of this action when $\ld{G} = \SL_2$ is as follows. If we identify $\ol{T^\ast(\ld{G}/\ld{N})} = T^\ast(\AA^2)$ with coordinates $(u,v) \in \AA^2 \oplus (\AA^2)^\ast$, then the total power operation is given by the map
    $$\varphi: (u,v) \mapsto (u, v - t^{p-1} v\pdb{u,v}^{p-1}).$$
    Since the moment map $T^\ast(\AA^2) \to \sl_2^\ast \cong \pgl_2$ sends $(u,v) \mapsto \begin{psmallmatrix}
        u_1 v_1 & u_1 v_2 \\
        u_2 v_1 & u_2 v_2
    \end{psmallmatrix}$, it is easy to check that this map is compatible with the action of the decompleted Frobenius on $\sl_2^\ast$ as described in \cref{ex: steenrod operations langlands}.
\end{example}
\begin{example}\label{ex: adams on mult quiver}
    When $k = \KU$ and $\GG = \GG_m$, the action of $\Loc^\gr_{G_c}(\Gr_G; k)$ on $\Loc^\gr_{T_c}(\Gr_G; k)$ by convolution identifies with the action of $\QCoh(G^\reg/\ld{G})$ on $\QCoh(B^\reg/\ld{B})$ via pullback along the map $B^\reg/\ld{B} \to G^\reg/\ld{G}$. It follows from \cref{ex: steenrod operations langlands} that this map is compatible with the decompleted Frobenius/$p$th Adams operation.
    The composite map
    $$B^\reg/\ld{N} \to B^\reg/\ld{B} \to G^\reg/\ld{G}$$
    can be realized as the $\ld{G}$-quotient of the restriction to regular loci of the multiplicative moment map $\mu: \ld{G}\times^{\ld{N}} B \to G$. The action of the decompleted Frobenius/$p$th Adams operation on the regular locus of $\ld{G}\times^{\ld{N}} B$ in fact extends to all of $\ld{G}\times^{\ld{N}} B$ itself (and hence on its affine closure $\ol{\ld{G}\times^{\ld{N}} B}$), and the moment map $\mu$ is equivariant for this action. The action of the decompleted Frobenius on $\ol{\ld{G}\times^{\ld{N}} B}$ commutes with the Gelfand-Graev action of the Weyl group from \cref{prop: ku gelfand-graev}; this can be seen by reducing to the rank $1$ case described below (with a bit of care in keeping track of the difference between $\AA^2$ and $(\AA^2)^\ast$).
    
    An explicit description of the action of the decompleted Frobenius when $\ld{G} = \SL_2$ is as follows. As in \cref{ex: Z/2 multiplicative symplectic fourier}, we may identify $\ol{\ld{G}\times^{\ld{N}} B}$ with an open subset of $T^\ast(\AA^2)$ with coordinates $(u,v) \in \AA^2 \oplus (\AA^2)^\ast$. The total power operation is then given by the map
    $$\varphi: (u,v) \mapsto \left(u, v\tfrac{(1 + \pdb{u,v})^p - 1}{\pdb{u,v}}\right).$$
    Since the moment map $\ol{\ld{G}\times^{\ld{N}} B} \to \PGL_2$ sends $(u,v) \mapsto \begin{psmallmatrix}
        1 + u_1 v_1 & u_1 v_2 \\
        u_2 v_1 & 1 + u_2 v_2
    \end{psmallmatrix}$, it is easy to check that this map is compatible with the action of the $p$th power map on $\PGL_2$ as described in \cref{ex: adams operations langlands}. In checking that the total power operation is compatible with the Gelfand-Graev action as described in \cref{ex: Z/2 multiplicative symplectic fourier}, the basic input is the identity $q^{-1} [p]_{q^{-1}} = q^{-p} [p]_q$ applied to $q = 1 + \pdb{u,v}$ (where $[p]_q = \tfrac{q^p - 1}{q-1}$).
\end{example}
More generally, (a mild variant of) the relative Langlands program from \cite{bzsv} predicts that if $X$ is an affine spherical $G$-variety, there exists a graded affine Hamiltonian $\ld{G}$-variety $\ld{M}$ over $\Z$ (possibly with an integer $N \gg 0$ inverted) with moment map $\mu: \ld{M} \to \ld{\g}^\ast$ such that there is an equivalence
$$\Shv_{G\pw{t}}^c(X\ls{t}; \Z) \simeq \Perf^{\sh}(\ld{M}/\ld{G}).$$
Here, $\Perf^{\sh}(\ld{M}/\ld{G})$ denotes the $\infty$-category of perfect complexes on the shearing of $\ld{M}$ with respect to its grading. Moreover, under a $\Z$-linear analogue of the derived geometric Satake equivalence, the natural action of $\Shv_{G\pw{t}}^c(\Gr_G; \Z)$ on the left-hand side by convolution should identify with the action of $\Perf(\ld{\g}^\ast[2]/\ld{G})$ on $\Perf^\sh(\ld{M}/\ld{G})$ via pullback along the moment map. This equivalence will restrict (and degenerate) to an equivalence
$$\Loc_{G\pw{t}}^\gr(X\ls{t}; \Z) \simeq \Perf(\ld{M}^\reg/\ld{G})$$
for some open $\ld{M}^\reg \subseteq \ld{M}$, which again satisfies a form of Hecke compatibility. Following the discussion above, the left-hand side will admit an action of the decompleted Frobenius/Steenrod operations, and so one expects the right-hand side to also admit such a structure. That is to say, $\ld{M}$ should admit an action of the decompleted Frobenius, and the moment map $\mu: \ld{M} \to \ld{\g}^\ast$ should be compatible with this action; here, $\ld{\g}^\ast$ is equipped with the action of the decompleted Frobenius described in \cref{ex: steenrod operations langlands}. It is worth remarking that this picture of relative Langlands duality only predicts that the decompleted Frobenius/Steenrod operations only act canonically on the \textit{stack} $\ld{M}/\ld{G}$, and that any formula one writes on $\ld{M}$ will not be canonical. This will be abundantly clear in the examples below, where it is obvious that the formulas we write are not unique (but any other choice will be an $\ld{G}$-translate of our formulas). In any case, these extra symmetries on $\ld{M}/\ld{G}$ are very interesting, and we expect them to play an important role in positive-characteristic analogues of the relative Langlands program.

In \cite{ku-rel-langlands}, we propose a version of this picture for sheaves with coefficients in $\KU$ (and more generally in $\ku$): the main difference is that $\ld{M}$ must be replaced by a \textit{quasi-Hamiltonian} $\ld{G}$-variety in the sense of \cite{amm-qham}, so that its moment map goes from $\ld{M}$ to $\ld{G}$. Again, $\ld{M}$ (or more canonically, $\ld{M}/\ld{G}$) should admit an action of the decompleted Frobenius/$p$th Adams operation on $\KU$, and the multiplicative moment map $\ld{M} \to \ld{G}$ should be compatible with this action, where the action of the decompleted Frobenius on $\ld{G}$ is as described in \cref{ex: adams operations langlands}. Outside of simple cases like \cref{ex: adams on mult quiver}, the quasi-Hamiltonian varieties can be quite complicated; so we will not discuss this case below. 

Let us present two explicit and nontrivial examples of ``Frobenius compatibility'' in the context of relative Langlands. The simplest is perhaps the following example, which generalizes \cref{ex: steenrod on TA2} and \cref{ex: adams on mult quiver}.
\begin{example}[Mirabolic Satake]\label{ex: mirabolic satake}
    This example is concerned with the relative Langlands dual to $G = \GL_n \times \GL_{n-1}$ acting on $G/\GL_{n-1}^\mathrm{diag}$.
    In \cite{mirabolic-satake}, it was shown that there is an equivalence
    $$\Shv_{\GL_{n-1}\pw{t}}^{c,\mathrm{Sat}}(\Gr_{\GL_n}; \QQ) \simeq \Perf^{\sh}(T^\ast \Hom(\AA^n, \AA^{n-1})/(\GL_n \times \GL_{n-1})),$$
    where, if we identify $T^\ast \Hom(\AA^n, \AA^{n-1})$ with $\Hom(\AA^{n-1}, \AA^n) \oplus \Hom(\AA^n, \AA^{n-1})$, the moment map $\mu: T^\ast \Hom(\AA^n, \AA^{n-1}) \to \gl_n^\ast \times \gl_{n-1}^\ast$ sends 
    $$\mu: (f,g)\mapsto (fg, gf).$$
    The equivalence of categories above will continue to hold over $\Z[1/N]$ for some $N \gg 0$, so we may consider the decompleted Frobenius for $p\nmid N$.
    Unwinding the proof of the above equivalence shows that the decompleted Frobenius/Steenrod algebra acts on $T^\ast \Hom(\AA^n, \AA^{n-1})$ via
    $$\varphi: (f,g) \mapsto (f, g - t^{p-1} g (fg)^{p-1}).$$
    It is easy to check that the moment map is indeed Frobenius-equivariant.

    There is also a multiplicative version of this picture. Namely, it follows from \cite[Remark 4.3.4]{ku-rel-langlands} that there is an equivalence
    $$\Loc_{\GL_{n-1}\pw{t}}^\gr(\Gr_{\GL_n}; \KU) \simeq \Perf(\cB(\AA^n, \AA^{n-1})^\reg/(\GL_n \times \GL_{n-1})),$$
    where $\cB(\AA^n, \AA^{n-1})^\reg$ is a particular open subset inside Van den Bergh's variety from \cite{van-den-bergh-double-poisson}:
    $$\cB(\AA^n, \AA^{n-1}) = \{(f,g) \in \Hom(\AA^{n-1}, \AA^n) \oplus \Hom(\AA^n, \AA^{n-1}) | \id + fg \in \GL_n\}.$$
    There is a multiplicative moment map $\mu: \cB(\AA^n, \AA^{n-1}) \to \GL_n \times \GL_{n-1}$ which sends 
    $$\mu: (f,g)\mapsto (\id + fg, \id + gf).$$
    The decompleted Frobenius/$p$th Adams operation acts on $\cB(\AA^n, \AA^{n-1})$ via
    $$\varphi: (f,g) \mapsto (f, f^{-1} ((\id + fg)^p - \id)),$$
    and again, the multiplicative moment map is Frobenius-equivariant.
\end{example}
\begin{example}
    In \cite{mirabolic-satake}, it was also shown that there is an equivalence
    $$\Shv_{\GL_n\pw{t}}^{c,\mathrm{Sat}}(\Gr_{\GL_n} \times \AA^n\ls{t}; \QQ) \simeq \Perf^{\sh}(T^\ast \gl_n/(\GL_n \times \GL_n)),$$
    where, if we identify $T^\ast \gl_n$ with $\gl_n \oplus \gl_n$, the moment map $\mu: T^\ast \gl_n \to \gl_n^\ast \times \gl_n^\ast$ sends 
    $$\mu: (f,g)\mapsto (fg, gf).$$
    Such an equivalence will continue to hold over $\Z[1/N]$ for some $N \gg 0$, so we may consider the decompleted Frobenius for $p\nmid N$. 
    Unwinding the proof of the above equivalence shows that, just as in \cref{ex: mirabolic satake}, the decompleted Frobenius/Steenrod algebra acts on $T^\ast \gl_n$ via
    $$\varphi: (f,g) \mapsto (f, g - t^{p-1} g (fg)^{p-1}).$$
    Again, it is easy to check that the moment map is indeed Frobenius-equivariant.
\end{example}

\begin{example}[Symplectic period]\label{ex: symplectic period}
    The ``quaternionic'' Satake equivalence is concerned with the relative Langlands dual to $G = \GL_{2n}$ acting on $\GL_{2n}/\Sp_{2n}$.
    The main result of \cite{quat-satake} says that there is an equivalence
    $$\Shv_{\GL_{2n}\pw{t}}^c(\GL_{2n}\ls{t}/\Sp_{2n}\ls{t}; \QQ) \simeq \Perf^{\sh}(\ld{M}/\GL_{2n}),$$
    where $\ld{M} \cong \GL_{2n} \times^{\GL_n} \gl_n^\ast[4]$ is equipped with a particular Hamiltonian structure. (Here, $\GL_n$ sits diagonally inside $\GL_{2n}$.) Such an equivalence will continue to hold over $\Z[1/N]$ for some $N \gg 0$ (in fact, one can take $N=1$), so we may consider the decompleted Frobenius for $p\nmid N$. In particular, we will assume $p>2$. The moment map $\ld{M} \to \gl_{2n}^\ast$ is induced by the inclusion $\gl_n^\ast \to \gl_{2n}^\ast$ sending 
    $$\mu: x \mapsto \begin{psmallmatrix}
        0 & \id_n \\
        x & 0
    \end{psmallmatrix}.$$
    Unwinding the proof of \cite{quat-satake} shows that the decompleted Frobenius/Steenrod algebra acts on $\ld{M}$ via the map 
    $$\varphi: x \mapsto x - 2 t^{p-1} x^{(p+1)/2} + t^{2(p-1)} x^p = \prod_{j\in \FF_p} (x - j^2 t^2 \id_n)$$
    on $\gl_n^\ast$. (Observe that the formula for $\varphi$ is a matrix version of the total Steenrod operation on $\H^\ast_{\SU(2)}(\ast; \FF_p)$.) If $x \in \gl_n^\ast$, it is \textit{not} true that $\varphi(\mu(x)) = \mu(\varphi(x))$; but these two elements of $\gl_{2n}^\ast$ are conjugate, from which it follows that the moment map $\ld{M}/\GL_{2n} \to \gl_{2n}^\ast/\GL_{2n}$ is equivariant for the action of the decompleted Frobenius. 

    There is also a multiplicative version of this picture. Namely, using the methods of \cite[Theorem 3.6.4]{ku-rel-langlands}, one obtains a K-theoretic version of the quaternionic Satake equivalence:
    $$\Loc_{\GL_{2n}\pw{t}}^c(\GL_{2n}\ls{t}/\Sp_{2n}\ls{t}; \KU) \simeq \Perf(\ld{M}_\KU^\reg/\GL_{2n}).$$
    Here, $\ld{M}_\KU^\reg$ is an open subset inside $\ld{M}_\KU \cong \GL_{2n} \times^{\GL_n} \gl_n$ (with $\GL_n$ sitting diagonally inside $\GL_{2n}$). The scheme $\ld{M}_\KU$ is equipped with a particular quasi-Hamiltonian structure, which we will now describe. The multiplicative moment map $\mu: \ld{M}_\KU \to \GL_{2n}$ is induced by the inclusion $\gl_n \to \GL_{2n}$ sending 
    $$\mu: x \mapsto \begin{psmallmatrix}
        x + \id_n & \id_n \\
        x & \id_n
    \end{psmallmatrix}.$$
    In this case, the decompleted Frobenius/$p$th Adams operation acts on $\ld{M}_\KU$ via the map $\varphi(x) = L_p(x)$ on $\gl_n$, with $L_p(x)$ as in \cref{ex: adams operations langlands}. For $x \in \gl_n$, the elements $\varphi(\mu(x))$ and $\mu(\varphi(x))$ of $\GL_{2n}$ are not equal. However, they are conjugate, which implies that the moment map $\ld{M}_\KU/\GL_{2n} \to \GL_{2n}/\GL_{2n}$ is equivariant for the action of the decompleted Frobenius.
    
    {In fact, there is also an equivalence between $\Loc_{\GL_{2n}\pw{t}}^\gr(\GL_{2n}\ls{t}/\Sp_{2n}\ls{t}; \KO)$ and $\Perf(\ld{M}_\KU^\reg/\GL_{2n} \times B\Z/2)$. In other words, complex conjugation on $\KU$ acts trivially on $\ld{M}_\KU$. This is easy to see algebraically: using \cref{def: cplx conj on B mod B} and \cref{rmk: cplx conj on G equiv KU}, this follows from the observation that if $x\in \gl_n$, then $\mu(x)^{-1}$ is conjugate to $\mu(x)$. However, the triviality of complex conjugation in this case also has a topological explanation. Namely, the quotient $\GL_{2n}\pw{t} \backslash \GL_{2n}\ls{t}/\Sp_{2n}\ls{t}$ is homotopy equivalent to the quaternionic affine Grassmannian $\Gr_{\GL_n(\bH)}$. The map $\HHP^{n-1} \to \Gr_{\GL_n(\bH)}$ exhibits $\HHP^{n-1}$ as a generating complex for $\Gr_{\GL_n(\bH)}$. Since $\HHP^{n-1}$ is a $\Spin$-manifold, it is $\KO$-oriented \cite{atiyah-bott-shapiro}, which implies that complex conjugation on $\KU$ acts trivially on $\KU^\ast(\HHP^{n-1})$ (and hence on $\Loc_{\GL_{2n}\pw{t}}^\gr(\GL_{2n}\ls{t}/\Sp_{2n}\ls{t}; \KU)$).}
\end{example}
\begin{remark}\label{rmk: ku symplectic}
    There is a $\ku$-theoretic variant of \cref{ex: symplectic period}. Just as in \cite{ku-rel-langlands}, there is an equivalence
    $$\Loc_{\GL_{2n}\pw{t}}^\gr(\GL_{2n}\ls{t}/\Sp_{2n}\ls{t}; \ku) \simeq \Perf(\ld{M}_\beta^\reg/\GL_{2n});$$
    in fact, in the notation of \cite{ku-rel-langlands}, there is an equivalence
    $$\Shv_{\GL_{2n}\pw{t}}^{c,\Sat}(\GL_{2n}\ls{t}/\Sp_{2n}\ls{t}; \ku)^\faux \simeq \Perf(\ld{M}_\beta/\GL_{2n});$$
    Here, $\ld{M}_\beta^\reg$ is an open subset inside $\ld{M}_\beta \cong \GL_{2n} \times^{\GL_n} \gl_n$ (with $\GL_n$ sitting diagonally inside $\GL_{2n}$).
    In fact, there is also an equivalence between $\Loc_{\GL_{2n}\pw{t}}^\gr(\GL_{2n}\ls{t}/\Sp_{2n}\ls{t}; \ko)$ and $\Perf(\gl_n^\reg/\GL_{2n} \times \spev(\ko))$, where we are using notation as in \cref{rmk: connective ko def}.
    
    The scheme $\ld{M}_\beta$ is equipped with a particular $\ku$-Hamiltonian structure (in the sense of \cite{ku-rel-langlands}), which we will now describe. The multiplicative moment map $\mu: \ld{M}_\beta \to \GL_{2n,\beta}$ is induced by the inclusion $\gl_n \to \GL_{2n,\beta}$ sending 
    $$\mu: x \mapsto \begin{psmallmatrix}
        \id_n + \beta^2 x & \beta \id_n \\
        \beta x & \id_n
    \end{psmallmatrix}.$$
    Observe that when $\beta$ is inverted (or, more informally, set to $1$), this is the quasi-Hamiltonian moment map from \cref{ex: symplectic period}; similarly, $\left.\tfrac{\mu(x) - \id_{2n}}{\beta}\right|_{\beta=0}$ is well-defined, and reduces to the moment map $x \mapsto \begin{psmallmatrix}
        0 & \id_n\\
        x & 0
    \end{psmallmatrix}$ from the ordinary symplectic period of \cref{ex: symplectic period}. In this case, the decompleted Frobenius/$p$th Adams operation acts on $\ld{M}_\beta$ via the map $\varphi(x) = f_p(x)$ on $\gl_n$, with $f_p(x)$ as in \cref{rmk: adams connective ku}. For $x \in \gl_n$, the elements $\varphi(\mu(x))$ and $\mu(\varphi(x))$ of $\GL_{2n}$ are conjugate, so the moment map $\ld{M}_\beta/\GL_{2n} \to \GL_{2n,\beta}/\GL_{2n}$ is equivariant for the action of the decompleted Frobenius. 
\end{remark}
In the language of \cite{bzsv, ku-rel-langlands}, the preceding discussion says that the stack which is relative Langlands dual to the Hamiltonian $\GL_{2n}$-space $T^\ast(\GL_{2n}/\Sp_{2n})$ is isomorphic to $\gl_n(4)/\GL_n$ with coefficients in both ordinary cohomology \textit{and} complex/real K-theory. However, this will no longer be true for elliptic cohomology. Geometrically, this is because elliptic cohomology is not $\Spin$-oriented, but is only ``$\String$-oriented'' \cite{koandtmf}; and $\HHP^{n-1}$ is a generating complex for the quaternionic affine Grassmannian $\Gr_{\GL_n(\bH)}$, but it is not a $\String$-manifold.\footnote{If $\TMF$ denotes the universal elliptic cohomology theory \cite{tmf}, then the $\TMF$-homology of $\HHP^n$ is described explicitly in \cite[Proposition 7.5]{meier-tmf-modules}.}

There are some examples where the generating complex for the real Grassmannian $\Gr_{G,\RR}$ \textit{is} orientable for elliptic cohomology, such as the case of (the simply-connected form of) $E_6$ equipped with the involution whose fixed subgroup is $F_4$. (This is the Cartan symmetric space EIV, and in the parlance of relative Langlands duality, it corresponds to the ``octonionic'' Satake equivalence of \cite{octonionic-period} and \cite[Remark 3.6.5]{ku-rel-langlands}.) In this case, the generating complex for $\Gr_{E_6,\RR}$ is given by the octonionic projective plane $\OP^2$, which is indeed a $\String$-manifold (and hence is orientable for elliptic cohomology). It is possible to use this observation to compute the relative Langlands dual to the Hamiltonian $E_6$-space $T^\ast(E_6/F_4)$ with coefficients in elliptic cohomology, but the calculations become very intricate (so we will leave it to future work).

Finally, let us discuss the Frobenius for a non-polarized example. The most famous example of this is the Gan-Gross-Prasad period, which, in the parlance of \cite{bzsv}, is concerned with the relative Langlands dual to the homogeneous spherical $G = \SO_{2n-1} \times \SO_{2n}$-variety given by $G/\SO_{2n-1}^\mathrm{diag}$. This dual is given by the Hamiltonian $\ld{G} = \Sp_{2n-2} \times \SO_{2n}$-space $\std_{2n-2} \otimes \std_{2n}$. It was studied geometrically in \cite{orthosymplectic-satake}. The following is one of the simplest nontrivial cases of the Gan-Gross-Prasad period:
\begin{example}[Triple product period]
    The triple product period, studied geometrically in \cite{pgl2-cubes}, is concerned with the relative Langlands dual to $G = \PGL_2^{\times 3}$ acting on $X = G/\PGL_2^\mathrm{diag}$. (This can be regarded as a special case of the Gan-Gross-Prasad period, because $\PGL_2 \cong \SO_3$ and $\PGL_2^{\times 2} \cong \PSO_4$.) The dual Hamiltonian $\ld{G} = \SL_2^{\times 3}$-variety in this case is given by the $8$-dimensional symplectic vector space $\std^{\otimes 3}$, with each factor of $\SL_2$ in $\ld{G}$ acting on the corresponding tensor factor. Let us quickly summarize a few facts from \cite{pgl2-cubes}: if we identify $\sl_2^\ast \cong \pgl_2$ with the space $\Sym^2(\std)$ of binary quadratic forms, the moment map $\mu: \std^{\otimes 3} \to \Sym^2(\std)^{\times 3}$ is given precisely by Bhargava's construction \cite{bhargava-composition-i} of three quadratic forms from a $2 \times 2 \times 2$-cube $\cC$; furthermore, all three of these quadratic forms have the same discriminant $\det(\cC)$, called the \textit{hyperdeterminant} of the cube \cite{cayley-original, gelfand-hyperdet}.
    
    This information can be used to compute the action of the Frobenius on $\std^{\otimes 3}/\SL_2^{\times 3}$, at least for $p>2$. We will only describe the \textit{completed} Frobenius, i.e., the functor
    $$\Shv^\gr_{G\pw{t}}(X\ls{t}; k) \otimes_{\pi_0(k)} F \to \Shv^\gr_{G\pw{t}}(X\ls{t}; k^{t\Cp}) \otimes_{\pi_0(k)} F,$$
    which, when $k = \Z[u^{\pm 1}]$, identifies with the functor given by pullback along a map
    $$\varphi: \std^{\otimes 3}/\SL_2^{\times 3} \times_{\spec(F)} \spec(F\ls{t}) \to \std^{\otimes 3}/\SL_2^{\times 3}.$$
    To describe it, pick a basis $e_1,e_2\in \std$, and equip $\std$ with the $\Z[1/3]$-grading where $e_1$ has weight $2/3$ and $e_2$ has weight $-1/3$. This equips $\std^{\otimes 3}$ with an $\Z$-grading, and one can then show that the Frobenius map is given by scaling the cube by its natural $\GG_m$-action with respect to the scalar $\delta = 1-t^{p-1} \det(\cC)^{(p-1)/2}$. In other words, it is given by  multiplying each coordinate $\cC_{ijk}$ of a cube $\cC$ by $\delta^{|\cC_{ijk}|}$, where $|\cC_{ijk}|$ is the weight of $\cC_{ijk}$.
    Note that some coordinates will have negative weight, and in this case one must interpret $\delta^{-1}$ as $\sum_{n\geq 0} t^{n(p-1)} \det(\cC)^{n(p-1)/2}$; ensuring convergence of this power series is why we elected to work with the completed Frobenius in the present example.
\end{example}
It might be interesting to describe the (de)completed Frobenius explicitly for the general case of the Gan-Gross-Prasad period, as well as for other non-polarized examples.
\newpage

\section{Comparison to Brylinski-Zhang}\label{sec: brylinski zhang}
In \cite{brylinski-zhang}, Brylinski-Zhang compute the $G_c$-equivariant complex K-theory of $G_c$ for a connected compact Lie group $G_c$ with torsion-free fundamental group as the ring $\Omega^\ast_{\mathrm{RU}(G)/\Z} = \Omega^\ast_{T\mmod W/\Z}$ of K\"ahler differentials on the complex representation ring of $G$. Our goal in this section is to describe the relationship between this calculation and (the proof of) \cref{thm: intro omnibus}. 

We begin by stating an obvious corollary of \cref{thm: intro omnibus}. Recall that if $\GG_0$ is either $\GG_a$, $\GG_m$, or an elliptic curve $E$, and $\cM_{T,0} = \Hom(\bX^\ast(T), \GG_0)$, there is a Kostant section $\kappa: \cM_{T,0} \to \Bun_{\ld{B}}^0(\GG_0^\vee)$ as described in \cref{def: additive kostant slice}, \cref{def: mult kostant slice}, and \cref{elliptic-kostant}. Recall that $F$ is an algebraically closed field of characteristic zero containing $\pi_0(k)$.
\begin{theorem}\label{thm: T-equiv loc on G}
    Let $G$ be a connected almost simple simply-laced group. Let $k$ denote either $\QQ[u^{\pm 1}]$, $\KU$, or elliptic cohomology, and let $\GG_0$ be either $\GG_a$, $\GG_m$, or an elliptic curve $E$ over $\pi_0(k)$, respectively.
    Then there is an equivalence 
    $$\Loc_{\ld{T}_c}^\gr(G_c; k) \otimes_{\pi_0(k)} F \simeq \QCoh(\cM_{\ld{T},0} \times_{\Bun_{\ld{B}}^0(\GG_0^\vee)} \cM_{\ld{T},0}),$$
    where the right-hand side denotes the self-intersection of the Kostant slice.
\end{theorem}
\begin{proof}
    Recall from \cref{def: graded Loc} that
    $$\Loc_{\ld{T}_c}^\gr(G_c; k) = \LMod_{\pi_0(\cf_{\ld{T}}(\Gr_G)^\vee)}(\QCoh(\cM_{\ld{T},0})).$$
    In \cref{thm: ordinary hmlgy reg centr}, \cref{thm: ku hmlgy reg centr}, and \cref{thm: elliptic hmlgy reg centr}, we showed that $\spec_{\cM_{\ld{T},0}}(\pi_0(\cf_{\ld{T}}(\Gr_G)^\vee))$ is isomorphic to the self-intersection $\cM_{\ld{T},0} \times_{\Bun_{\ld{B}}^0(\GG_0^\vee)} \cM_{\ld{T},0}$, so the desired equivalence follows.
\end{proof}
In the same way, if $G$ is further assumed to have torsion-free fundamental group, and $\cM_{G,0}$ denotes the moduli \textit{space} of semistable $G$-bundles on $\GG_0^\vee$, there is a Kostant section $\kappa: \cM_{G,0} \to \Bun_{\ld{G}}^\ss(\GG_0^\vee)$. In the additive and multiplicative cases, this follows from \cref{def: additive kostant slice}, \cref{def: mult kostant slice}, and in the elliptic case, it can be deduced from \cite{davis-elliptic-springer} as in \cref{elliptic-kostant}. Just as in \cref{thm: T-equiv loc on G}, there is an equivalence 
\begin{equation}\label{eq: G-equiv loc on G}
    \Loc_{\ld{G}_c}^\gr(G_c; k) \otimes_{\pi_0(k)} F \simeq \QCoh(\cM_{\ld{G},0} \times_{\Bun_{\ld{G}}^\ss(\GG_0^\vee)} \cM_{\ld{G},0})
\end{equation}
where the right-hand side denotes the self-intersection of the Kostant slice. Under this equivalence, the ``constant sheaf'' in $\Loc_{\ld{G}_c}^\gr(G_c; k)$ is sent to the pushforward of the structure sheaf under the relative diagonal
$$\delta: \cM_{\ld{G},0} \to \cM_{\ld{G},0} \times_{\Bun_{\ld{G}}^\ss(\GG_0^\vee)} \cM_{\ld{G},0}.$$

In the remainder of this section, we will explain how \cref{eq: G-equiv loc on G} implies the calculation of \cite{brylinski-zhang}, as well as the relationship to the Hochschild-Kostant-Rosenberg theorem. (This, of course, is a triple of authors distinct from Hopkins-Kuhn-Ravenel with initials ``HKR''!)
For simplicity, we will only focus on the case when $k$ is $\QQ[u^{\pm 1}]$ or $\KU$ (so $\GG_0$ is either $\GG_a$ or $\GG_m$, and $\Bun_{\ld{G}}^\ss(\GG_0^\vee)$ is either $\ld{\g}/\ld{G}$ or $\ld{G}/\ld{G}$). With a little bit of elbow grease, one can show that most of the results below continue to work for elliptic cohomology, too.

Recall that $\Loc_{{G}_c}^\gr(G_c; k)$ is intended to be an approximation to a $k$-linear $\infty$-category of $G_c$-equivariant local systems on $G_c$. The algebra of endomorphisms of the constant sheaf in this $\infty$-category is given by the equivariant cochains $\cf_G(G_c)$. This is a quasicoherent sheaf over the spectral $k$-scheme $\cM_G$, and it can be described explicitly as follows. If $\fr{X}_k$ is a spectral prestack over $k$, let $\cL \fr{X}_k$ denote the free loop space of $\fr{X}_k$, i.e., the mapping prestack $\Map(B\Z, \fr{X}_k)$. Here, $\Z$ is viewed as a constant stack over $k$. The global sections of the structure sheaf of $\cL \fr{X}_k$ computes the Hochschild homology $\HH(\fr{X}_k/k)$.
\begin{prop}\label{prop: G-equiv coh of G and HH}
    Assume (for simplicity) that $k$ is either $\QQ[u^{\pm 1}]$ or $\KU$. If $G$ is connected, then there is an isomorphism of spectral $k$-schemes
    $$\spec_{\cM_G}(\cf_G(G_c)) \cong \cL \cM_G.$$
    In particular, there is an isomorphism of $\Eoo$-$k_{G_c}$-algebras
    $$\cf_G(G_c) \cong \HH(\cM_G/k).$$
\end{prop}
\begin{proof}
    Recall that $B\Z$ is isomorphic to the constant $k$-stack $S^1$, which can be written as the pushout $\ast \sqcup_{\ast \sqcup \ast} \ast$. Therefore, since $\cM_G = \spec k_{G_c}$ is affine (because $k$ is either $\QQ[u^{\pm 1}]$ or $\KU$), we may wite $\cL \cM_G = \spec (k_{G_c} \otimes_{k_{G_c} \otimes_k k_{G_c}} k_{G_c})$. Since the functor $\cf_G: \Top(G_c)^\op \to \Mod_{k_{G_c}}$ sends finite products of connected finite $G$-spaces to tensor products, we find that 
    $$k_{G_c} \otimes_{k_{G_c} \otimes_k k_{G_c}} k_{G_c} \cong \cf_G(\ast) \otimes_{\cf_{G \times G}(\ast)} \cf_G(\ast) \cong \cf_G(G_c),$$
    since there is an isomorphism of orbispaces
    $$\ast/G_c\times_{\ast/(G_c \times G_c)} \ast/G_c \cong G_c/G_c.\qedhere$$
\end{proof}
\begin{remark}\label{rmk: equiv coh of G mod K and HH}
    The approach of \cref{prop: G-equiv coh of G and HH} can be used to compute the equivariant cohomology $\cf_G(\Omega G_c)$, too. Namely, observe that there is an isomorphism of orbispaces
    $$\ast/G_c\times_{\ast/G_c\times_{\ast/(G_c \times G_c)} \ast/G_c} \ast/G_c \cong (\Omega G_c)/G_c,$$
    so that there is an isomorphism
    $$\cf_G(\Omega G_c) = k_{G_c} \otimes_{k_{G_c} \otimes_{k_{G_c} \otimes_k k_{G_c}} k_{G_c}} k_{G_c}.$$
    The right-hand side can be expressed more succinctly as the factorization homology $\int_{S^2}(k_{G_c}/k)$.

    More generally, observe that if $K_c \subseteq G_c$ is a closed subgroup such that $G_c/K_c$ is a finite $K_c$-space (where $K_c$ acts on the left by multiplication), and $L(G_c/K_c)$ denotes the (topological) free loop space of $G_c/K_c$, then
    $$G_c\backslash L(G_c/K_c) \simeq K_c \backslash \Omega(G_c/K_c) \simeq (\ast \times_{\ast \times_{\ast/G_c} \ast/K_c} \ast)/K_c \simeq \ast/K_c \times_{\ast/K_c \times_{\ast/G_c} \ast/K_c} \ast/K_c.$$
    It follows that there is an isomorphism
    $$\cf_G(\cL(G_c/K_c)) = k_{K_c} \otimes_{k_{K_c} \otimes_{k_{G_c}} k_{K_c}} k_{K_c}.$$
    The right-hand side can be expressed more succinctly as the relative Hochschild homology $\HH(\cM_K/\cM_G)$, so that there is an isomorphism of spectral $k$-schemes
    $$\spec_{\cM_G}(\cf_G(G_c/K_c)) \cong \cL(\cM_K/\cM_G) \cong \cL(\cM_K) \times_{\cL(\cM_G)} \cM_G.$$
    The discussion above computing $\cf_K(\Omega K_c)$ is the special case of the above calculation when $G_c = K_c \times K_c$, with $K_c$ embedded diagonally.
\end{remark}
\begin{example}
    Let $k = \QQ[u^{\pm 1}]$. Then the preceding discussion shows that $C^\ast_{G_c}(\Omega G_c; \QQ[u^{\pm 1}])$ is isomorphic to the factorization homology $\int_{S^2}(k_{G_c}/k) = \HH(k_{G_c}/k_{G_c} \otimes_k k_{G_c})$. The latter has a Hochschild-Kostant-Rosenberg filtration whose associated graded is given by the $2$-periodification $L\Omega^\ast_{\fr{t}\mmod W/(\fr{t}\mmod W \times_{\spec \QQ} \fr{t}\mmod W)}[u^{\pm 1}]$ of the derived Hodge complex of $\fr{t}\mmod W$ embedded diagonally into $\fr{t}\mmod W \times_{\spec \QQ} \fr{t}\mmod W$.
    Since we are working rationally, the Hochschild-Kostant-Rosenberg filtration splits, and so there is an isomorphism
    $$\int_{S^2}(k_{G_c}/k) \cong L\Omega^\ast_{\fr{t}\mmod W/(\fr{t}\mmod W \times_{\spec \QQ} \fr{t}\mmod W)}[u^{\pm 1}].$$
    Note that if $X$ (like $\fr{t}\mmod W$) is an affine space over a commutative ring $R$, then $L\Omega^\ast_{X/(X \times_{\spec(R)} X)} \cong \Gamma^\ast(\Omega^1_{X/R})$; so the above isomorphism could instead be stated as
    $$\int_{S^2}(k_{G_c}/k) \cong \Sym_{\co_{\fr{t}\mmod W}}(\Omega^1_{\fr{t}\mmod W})[u^{\pm 1}] = \co_{T(\fr{t}\mmod W)}[u^{\pm 1}],$$
    where $T(\fr{t}\mmod W)$ is the tangent bundle of $\fr{t}\mmod W$. It follows that there is an isomorphism
    $$\spec C^\ast_{G_c}(\Omega G_c; \QQ[u^{\pm 1}]) \cong T(\fr{t}\mmod W) \times_{\spec(\QQ)} \spec(\QQ[u^{\pm 1}]).$$
    This recovers the $\hbar = 0$ case of \cite[Theorem 1]{bf-derived-satake}. The case with loop-rotation equivariance included, i.e., when $\hbar$ need not be zero, follows from \cref{lem: hochschild and def to nc} below, which recovers the description of $\spec C^\ast_{G_c \times S^1_\rot}(\Omega G_c; \QQ[u^{\pm 1}])$ as the deformation to the normal cone of the diagonal embedding $\fr{t}\mmod W \hookrightarrow \fr{t}\mmod W \times \fr{t}\mmod W$.
\end{example}
The following statement is essentially Koszul dual to the usual Hochschild-Kostant-Rosenberg theorem describing Hochschild homology with its circle action via the de Rham complex:
\begin{lemma}\label{lem: hochschild and def to nc}
    Let $X = \spec(A)$ be a smooth affine scheme over $\QQ$, and let $\Def_\hbar^\Delta(X)$ denote the deformation to the normal cone of the diagonal embedding $X \hookrightarrow X \times X$. Then there is an isomorphism
    $$\spec \pi_\ast \left(\int_{S^2}(A/\QQ)\right)^{hS^1} \cong \Def_\hbar^\Delta(X)$$
    of $\pi_\ast \QQ^{hS^1} \cong \QQ\pw{\hbar}$-algebras.
\end{lemma}
\begin{proof}
    By standard arguments, it suffices to check the claim when $A$ is a finitely generated polynomial algebra. Let us demonstrate the claim when $A$ is a polynomial algebra on a single class; an easy modification of this argument will prove the claim in general when $A = \QQ[V]$ for some finite-dimensional $\QQ$-vector space $V$.
    Let us identify $\QQ[x] \otimes \QQ[x] = \QQ[x,y]$, so that the standard resolution of $\QQ[x]$ as a $\QQ[x,y]$-algebra identifies 
    $$\QQ[x] \otimes_{\QQ[x] \otimes \QQ[x]} \QQ[x] \cong \QQ[x, \sigma(x-y)]/(\sigma(x-y)^2),$$
    where $\sigma(x-y)$ is in degree $1$. This implies that 
    $$\int_{S^2}(A/\QQ) \simeq \QQ[x] \otimes_{\QQ[x] \otimes_{\QQ[x] \otimes \QQ[x]} \QQ[x]} \QQ[x] \cong \QQ[x, \sigma^2(x-y)],$$
    with $\sigma^2(x-y)$ in degree $2$. Note that $\int_{S^2}(A/\QQ)$ is an $S^1$-equivariant $\Eoo$-$\QQ[x,y]$-algebra, so that $\pi_\ast\left(\int_{S^2}(A/\QQ)\right)^{hS^1}$ is an $\Eoo$-$\QQ[x,y]$-algebra; let us now determine this algebra structure. Since this ring is concentrated in even degrees, the homotopy fixed point spectral sequence computing $\pi_\ast\left(\int_{S^2}(A/\QQ)\right)^{hS^1}$ degenerates, and we find that $\pi_\ast\left(\int_{S^2}(A/\QQ)\right)^{hS^1} \cong \QQ\pw{\hbar}[x, \sigma^2(x-y)]$ with $\sigma^2(x-y)$ in weight $2$ and $\hbar$ in weight $-2$. The $\QQ[x,y]$-algebra structure is given by the observation that 
    $$x - y = \hbar \sigma^2(x-y);$$
    this relation is true for abstract reasons (as explained, for instance, in \cite[Appendix A]{hahn-wilson-bpn}). It follows that there is an isomorphism
    $$\pi_\ast\left(\int_{S^2}(A/\QQ)\right)^{hS^1} \cong \QQ\pw{\hbar}[x, y, \tfrac{x-y}{\hbar}]$$
    of $\QQ[x,y]$-algebras. The spectrum of the right-hand side identifies with $\Def_\hbar^\Delta(\AA^1)$, as desired.
\end{proof}
\begin{remark}
    More generally, suppose $A$ is a commutative $\Z$-algebra, and let $X = \spec(A)$. Let $\Def_\hbar^\Delta(X)$ denote the divided power deformation to the normal cone of the diagonal embedding $X \hookrightarrow X \times X$, so that its fiber over $\hbar = 0$ is the PD-hull $T_X^\sharp$ of the (derived) tangent bundle of $X$. Then there is a filtration on $\left(\int_{S^2}(A/\Z)\right)^{hS^1}$ whose associated graded is given by $\co_{\Def_\hbar^\Delta(X)}$. The proof is exactly as in \cref{lem: hochschild and def to nc}: the desired filtration is given by left Kan extending the Postnikov filtration $\tau_{\geq \star} \left(\int_{S^2}(A/\Z)\right)^{hS^1}$ from polynomial $\Z$-algebras to all commutative $\Z$-algebras.
\end{remark}

Let us now discuss the relationship between \cref{prop: G-equiv coh of G and HH} and \cref{eq: G-equiv loc on G}. Although the cases $k = \QQ[u^{\pm 1}]$ and $k = \KU$ can be treated simultaneously, we will present the discussion separately for both for the sake of clarity. The upshot of this discussion is that the approximation to $\cf_G(G_c)$ afforded by the degeneration of $\Loc_{G_c}(G_c; k)$ to $\Loc_{{G}_c}^\gr(G_c; k)$ identifies, under \cref{prop: G-equiv coh of G and HH} and \cref{eq: G-equiv loc on G}, with the Hochschild-Kostant-Rosenberg spectral approximation of $\pi_\ast \HH(\cM_G/k)$ by $\Omega^\ast_{\cM_{G,0}/\pi_0(k)}$.
\begin{lemma}\label{lem: endomorphisms of delta sheaf}
    Let $H$ be a smooth affine group scheme over an affine scheme $S = \spec(R)$, let $\delta: S \to H$ denote the zero section, and let $\fr{h}$ denote its Lie algebra (viewed as a vector bundle over $S$). Then $\End_{\QCoh(H)}(\delta_\ast \co_S)$ has a filtration whose associated graded is isomorphic to $\co_{\fr{h}^\ast[1]}$. If $R$ is a $\QQ$-algebra, this filtration splits.
\end{lemma}
\begin{proof}
    The endomorphism algebra $\End_{\QCoh(H)}(\delta_\ast \co_S)$ is isomorphic to the $R$-linear dual of $\co_{S \times_H S}$. The derived scheme $S \times_H S$ depends only on the formal completion $\hat{H}$. Note that $\hat{H}$ admits a filtration (coming from powers of the ideal sheaf of the zero section of $\hat{H}$) whose associated graded is isomorphic to $\hat{\fr{h}}$; furthermore, the exponential map defines a splitting of this filtration when $R$ is a $\QQ$-algebra. This defines a filtration on $S \times_H S$ whose associated graded is isomorphic to $S \times_{\fr{h}} S = \fr{h}[-1]$. Therefore, the $R$-linear dual of $\co_{S \times_H S}$ is isomorphic to $\co_{\fr{h}^\ast[1]}$.
\end{proof}
\begin{example}\label{ex: rational coh of G mod G}
    Suppose $k = \QQ[u^{\pm 1}]$, and let $\ld{J} = \fr{t}\mmod W \times_{\ld{\g}^\ast/\ld{G}} \fr{t}\mmod W$. Then \cref{eq: G-equiv loc on G} states that there is an equivalence 
    $$\Loc_{{G}_c}^\gr(G_c; k) \otimes_{\QQ} F \simeq \QCoh(\ld{J}),$$
    and the ``constant sheaf'' $\ul{k}^\gr$ in $\Loc_{\ld{G}_c}^\gr(G_c; k)$ is sent to the pushforward of the structure sheaf under the identity section $\delta: \fr{t}\mmod W \to \ld{J}$. Taking endomorphisms, we find that
    $$\End_{\Loc_{{G}_c}^\gr(G_c; k)}(\ul{k}^\gr) \otimes_{\QQ} F \cong \End_{\QCoh(\ld{J})}(\delta_\ast \co_{\fr{t}\mmod W}).$$
    By \cref{lem: endomorphisms of delta sheaf}, the right-hand side admits a (split) filtration whose associated graded is isomorphic to the algebra of functions on $\Lie_{\fr{t}\mmod W}(\ld{J})^\ast[1]$. By \cite[Theorem 3.4.2]{riche}, one finds that the Lie algebra $\Lie_{\fr{t}\mmod W}(\ld{J})$ is isomorphic to the cotangent bundle $T^\ast(\fr{t}\mmod W)$, so that $\Lie_{\fr{t}\mmod W}(\ld{J})^\ast[1]$ is isomorphic to $T[1](\fr{t}\mmod W)$. Its ring of functions is precisely the Hodge cohomology $\Omega^\ast_{\fr{t}\mmod W/F} = \bigoplus (\Omega^i_{\fr{t}\mmod W/F})[-i]$ of $\fr{t}\mmod W$. Summarizing, we have found that there is an isomorphism
    $$\End_{\Loc_{{G}_c}^\gr(G_c; k)}(\ul{k}^\gr) \otimes_{\QQ} F \cong \Omega^\ast_{\fr{t}\mmod W/F}.$$
    
    On the other hand, it follows from the constructions in \cref{sec: degenerations} that there is a filtration on $\cf_G(G_c) = \End_{\Loc_{{G}_c}(G_c; k)}(\ul{k})$ whose associated graded is $\End_{\Loc_{{G}_c}^\gr(G_c; k)}(\ul{k}^\gr)[u^{\pm 1}]$. By the above discussion, the latter is $\Omega^\ast_{\fr{t}\mmod W/F}$. \cref{prop: G-equiv coh of G and HH} shows that $\cf_G(G_c) \otimes_k F[u^{\pm 1}]$ is isomorphic to the Hochschild homology $\HH(\fr{t}\mmod W/F)[u^{\pm 1}]$. There is therefore a filtration on $\HH(\fr{t}\mmod W/F)[u^{\pm 1}]$ whose associated graded is $\Omega^\ast_{\fr{t}\mmod W/F}[u^{\pm 1}]$. This filtration is precisely the Hochschild-Kostant-Rosenberg filtration on Hochschild homology (see, e.g., \cite{antieau-filtration-HP, raksit, toen-hkr} for modern references).
\end{example}
\begin{example}\label{ex: KU of G mod G}
    Suppose $k = \KU$, and assume $G$ is simply-laced and has torsion-free fundamental group. Let $\ld{J}_\mu = T\mmod W \times_{G/\ld{G}} T\mmod W$. Then \cref{eq: G-equiv loc on G} states that there is an equivalence 
    $$\Loc_{{G}_c}^\gr(G_c; \KU) \otimes_{\Z} F \simeq \QCoh(\ld{J}_\mu),$$
    and the ``constant sheaf'' $\ul{\KU}^\gr$ in $\Loc_{\ld{G}_c}^\gr(G_c; \KU)$ is sent to the pushforward of the structure sheaf under the identity section $\delta: T\mmod W \to \ld{J}_\mu$. Taking endomorphisms, we find that
    $$\End_{\Loc_{{G}_c}^\gr(G_c; \KU)}(\ul{\KU}^\gr) \otimes_\Z F \cong \End_{\QCoh(\ld{J}_\mu)}(\delta_\ast \co_{T\mmod W}).$$
    By \cref{lem: endomorphisms of delta sheaf}, the right-hand side admits a (split) filtration whose associated graded is isomorphic to the algebra of functions on $\Lie_{T\mmod W}(\ld{J}_\mu)^\ast[1]$. There is a multiplicative analogue of \cite[Theorem 3.4.2]{riche} which states the Lie algebra $\Lie_{T\mmod W}(\ld{J}_\mu)$ is isomorphic to the cotangent bundle $T^\ast(T\mmod W)$. In particular, $\Lie_{T\mmod W}(\ld{J}_\mu)^\ast[1]$ is isomorphic to $T[1](T\mmod W)$. Its ring of functions is precisely the Hodge cohomology $\Omega^\ast_{T\mmod W/F} = \bigoplus (\Omega^i_{T\mmod W/F})[-i]$ of $T\mmod W$. Summarizing, we have found that there is an isomorphism
    $$\End_{\Loc_{{G}_c}^\gr(G_c; k)}(\ul{k}^\gr) \otimes_{\Z} F \cong \Omega^\ast_{T\mmod W/F}.$$

    On the other hand, it follows from the constructions in \cref{sec: degenerations} that there is a filtration on $\cf_G(G_c) = \End_{\Loc_{{G}_c}(G_c; \KU)}(\KU)$ with associated graded given by $\End_{\Loc_{{G}_c}^\gr(G_c; \KU)}(\ul{\KU}^\gr)[u^{\pm 1}]$. By the above discussion, the latter is $\Omega^\ast_{T\mmod W/F}$. \cref{prop: G-equiv coh of G and HH} shows that $\cf_G(G_c) \otimes_{\KU} F[u^{\pm 1}]$ is isomorphic to the Hochschild homology $\HH(T\mmod W/F)[u^{\pm 1}]$. There is therefore a filtration on $\HH(T\mmod W/F)[u^{\pm 1}]$ whose associated graded is $\Omega^\ast_{T\mmod W/F}[u^{\pm 1}]$. Again, this filtration is precisely the Hochschild-Kostant-Rosenberg filtration on Hochschild homology.
\end{example}
\begin{remark}
    While we are on the topic of the equivariant K-theory of $G_c$, let us note the relationship between \cref{eq: G-equiv loc on G} and the work of Freed-Hopkins-Teleman \cite{fht-i, fht-ii, fht-iii, fht-complex, fht-categorified}.\footnote{Nearly the same perspective can also be found in some of Teleman's talks; e.g., \cite{teleman-slides}.} We will be brief, since we will not use these results below. Associated to a class $\tau \in \H^4(BG_c; \Z)$ is the ``twisted equivariant K-homology'' $\KU^G_\tau(G_c)$. When $\tau$ is sufficiently nondegenerete, Freed-Hopkins-Teleman computed that $\pi_\ast \KU^G_\tau(G_c)$ is isomorphic to $\RU(G)/I^\tau$ for a particular ideal $I^\tau$ (called the ``Verlinde ideal''). The categorification of this isomorphism from \cite{fht-categorified} shows that, associated to $\tau$, there is a map $f: T\mmod W \cong \spec \pi_0 \KU_G \to \AA^1$ such that (under certain hypotheses on $\tau$), there is an isomorphism between $\pi_\ast \KU^G_\tau(G_c) \otimes_\Z F$ and the Jacobian ring of $f$.
    
    This is related to \cref{eq: G-equiv loc on G} in the following manner. Below, we will implicitly base-change all rings from $\Z$ to $F$, to avoid cumbersome notation. Recall from \cite[Equation 3]{fht-i} that there is a spectral sequence
    \begin{equation}\label{eq: twisted k-theory sseq}
        E_1^{\ast,\ast} \cong \pi_\ast \KU_G \otimes_{\pi_\ast \cf_G(\Gr_G)^\vee} \pi_\ast \KU_G \Rightarrow \pi_\ast \KU^G_\tau(G_c).
    \end{equation}
    The tensor product is derived; moreover, the class $\tau$ defines a particular $\pi_\ast \cf_G(\Gr_G)^\vee$-module structure on $\pi_\ast \KU_G$, and one of the tensor factors is given this module structure. (The other tensor factor is given the module structure coming from the augmentation.) Using \cref{thm: ku hmlgy reg centr}, let us view $\spec \pi_\ast \cf_G(\Gr_G)^\vee$ as the ($2$-periodification of) $\ld{J}_\mu$. Then $\tau$ defines a particular closed subscheme $T\mmod W \cong L_\tau \hookrightarrow \ld{J}_\mu$ (which is in fact a Lagrangian), and the $E_1$-page of this spectral sequence can be identified with (the $2$-periodification of) the ring of functions on $L_\tau \times_{\ld{J}_\mu} T\mmod W$. If $L_\tau$ lies in the formal neighborhood of $\ld{J}_\mu$, then we may replace $\ld{J}_\mu$ in this fiber product by its formal completion $\hat{\ld{J}}_\mu$ at the zero section. Since we have implicitly base-changed everything to the characteristic zero field $F$, the argument of \cref{lem: endomorphisms of delta sheaf} further lets us replace $\hat{\ld{J}}_\mu$ by its Lie algebra, which (as mentioned in \cref{ex: KU of G mod G}) is given by $T^\ast(T\mmod W)$. Under this replacement, the map $T\mmod W \cong L_\tau \to \hat{\ld{J}}_\mu$ becomes identified with the map $df: T\mmod W \to T^\ast(T\mmod W)$, where $f: T\mmod W \to \AA^1$ is the map from \cite{fht-categorified}. The derived fiber product $L_\tau \times_{T^\ast(T\mmod W)} T\mmod W$ is precisely the Jacobian ring of $f$; that is to say, the $E_1$-page of the spectral sequence \cref{eq: twisted k-theory sseq} identifies with the Jacobian ring of $f$. If the spectral sequence \cref{eq: twisted k-theory sseq} degenerates at the $E_1$-page, then we conclude that $\pi_\ast \KU^G_\tau(G_c)$ is isomorphic to the Jacobian ring of $f$, as desired.
\end{remark}

In fact, the Hochschild-Kostant-Rosenberg filtrations on $\HH(\fr{t}\mmod W/F)$ and $\HH(T\mmod W/F)$ from \cref{ex: rational coh of G mod G} and \cref{ex: KU of G mod G} both split, since $F$ is of characteristic zero and $\fr{t}\mmod W$ and $T\mmod W$ are smooth schemes. We therefore conclude that there are isomorphisms
\begin{align*}
    \H^\ast_{G_c}(G_c; F[u^{\pm 1}]) & \cong \Omega^\ast_{\fr{t}\mmod W/F}[u^{\pm 1}], \\
    \KU^\ast_{G_c}(G_c) \otimes_\Z F & \cong \Omega^\ast_{T\mmod W/F}[u^{\pm 1}],
\end{align*}
the latter for $G_c$ being simply-laced. (This assumption can be removed with further work.)
The final isomorphism above recovers (the base-change to $F$ of) the isomorphism of Brylinski-Zhang. Arguing as above, one also finds that if $k$ is an elliptic cohomology theory, $G_c$ is simply-laced and has torsion-free fundamental group, and $i: \spec(F[u^{\pm 1}]) \to \cM_G$ is a map with $F$ being an algebraically closed field of characteristic zero, there is an isomorphism of quasicoherent sheaves over $\spec(F[u^{\pm 1}])$:
\begin{equation}\label{eq: kahler diff and G mod G}
    \pi_\ast i^\ast \cf_G(G_c) \cong \Omega^\ast_{\cM_{G,0}/F}[u^{\pm 1}].
\end{equation}
As stated, \cref{eq: kahler diff and G mod G} holds if $k$ is $\QQ[u^{\pm 1}]$, complex K-theory, or elliptic cohomology. One could ask whether \cref{eq: kahler diff and G mod G} holds over the sphere spectrum.
\begin{remark}\label{rmk: miller splitting}
    At least in the case of classical groups, additive isomorphisms of the form discussed in this section follow from stronger statements about splittings of the suspension spectrum $(G_c)_+$. Such statements were proved in \cite{miller-stable-splittings}; let us illustrate this when $G = \GL_n$. For $j\leq n$, let $\Gr_j(\cc^n) = \U(n)/(\U(j) \times \U(n-j))$, and let $\Gr_j(\cc^n)^{\fr{u}(j)}$ denote the Thom spectrum of the vector bundle over $\Gr_j(\cc^n)$ given by the pulling back the adjoint representation of $\U(j)$ along the map $\Gr_j(\cc^n) \to \BU(j)$. Then there is a $\U(n)$-equivariant splitting
    $$(G_c)_+ \simeq \bigoplus_{j=0}^n \Gr_j(\cc^n)^{\fr{u}(j)}.$$
    This induces a splitting of $\cf_G(G_c)$, and hence of $i^\ast \cf_G(G_c)$ for any map $i: \spec(F[u^{\pm 1}]) \to \cM_G$ with $F$ being an algebraically closed field of characteristic zero. One can show that there is an isomorphism
    $$\pi_\ast i^\ast \cf_G(\Gr_j(\cc^n)^{\fr{u}(j)}) \cong \Omega^j_{\cM_{\GL_n,0}/F}[u^{\pm 1}],$$
    so taking the direct sum over $j = 0, \cdots, n$ gives an additive equivalence of the form \cref{eq: kahler diff and G mod G}.
\end{remark}
Although such splittings of $(G_c)_+$ were proved in \cite{miller-stable-splittings} only for classical groups, they can also be extended with some work to the exceptional groups, too.
However, one encounters an important difficulty in trying to extend \cref{eq: kahler diff and G mod G} to a statement about the stable homotopy type of $G_c$ itself. Namely, suppose that \cref{eq: kahler diff and G mod G} holds for $F$ of arbitrary characteristic; in fact, let us even suppose that the \textit{non}equivariant version of \cref{eq: kahler diff and G mod G} holds, i.e., that there is an isomorphism
\begin{equation}\label{eq: ansatz all char}
    \pi_\ast i^\ast \cf(G_c) \cong \Omega^\ast_{\cM_{G,0}/F}[u^{\pm 1}] \otimes_{\co_{\cM_{G,0}}} F
\end{equation}
for any map $i: \spec(F[u^{\pm 1}]) \to \cM_G$.
Motivated by the example of \cref{rmk: miller splitting}, it is natural to wonder whether this putative splitting could arise from a(n additive) splitting of $(G_c)_+$ itself, where the summands of $(G_c)_+$ realize the individual summands $\Omega^j_{\cM_{G,0}/F}[u^{\pm 1}] \otimes_{\co_{\cM_{G,0}}} F$. Unfortunately, this turns out to be impossible, at least if interpreted naively.
\begin{example}
    Suppose $G = \G_2$. Since $\G_2$ is a framed manifold, its top cell stably splits, and so there is a splitting
    $$(\G_2)_+ \simeq S^0 \oplus X \oplus S^{\g_2},$$
    where $S^{\g_2}$ is the one-point compactification of the adjoint representation of $\G_2$, and $X$ is a finite CW-complex with partial cell diagram shown in \cref{fig: cell diagram G2}.
    \begin{figure}[H]
    \begin{tikzpicture}[scale=0.75]
    \draw [fill] (0, 0) circle [radius=0.05];
    \draw [fill] (2, 0) circle [radius=0.05];
    \draw [fill] (3, 0) circle [radius=0.05];
    \draw [fill] (5, 0) circle [radius=0.05];
    \draw [fill] (6, 0) circle [radius=0.05];
    \draw [fill] (8, 0) circle [radius=0.05];
    
    \draw (0,0) to[out=-90,in=-90] node[below] {\footnotesize{$\Sq^2$}} (2,0);
    \draw (2,0) to node[below] {\footnotesize{$\Sq^1$}} (3,0);
    \draw (5,0) to node[below] {\footnotesize{$\Sq^1$}} (6,0);
    \draw (6,0) to[out=-90,in=-90] node[below] {\footnotesize{$\Sq^2$}} (8,0);
    \end{tikzpicture}
    \caption{A partial cell diagram for the stable summand $X$ of $(\G_2)_+$. The dots indicate the cells; starting from the left, the cells lie in dimensions $3,5,6,8,9$, and $11$. The labels represent the action of the Steenrod operations in mod $2$ cohomology.}
    \label{fig: cell diagram G2}
    \end{figure}
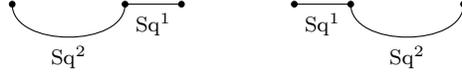
    Let us now take $F$ to be a field of characteristic $2$. Then $\H^\ast_{\G_2}(\ast; F) \cong F[w_4, w_6, w_7]$, where the subscript indicates the cohomological degree. This implies that 
    $$\Omega^\ast_{\H^\ast_{\G_2}(\ast; F)/F} \otimes_{\H^\ast_{\G_2}(\ast; F)} F \cong \Lambda(dw_4, dw_6, dw_7),$$ 
    where $\Lambda$ denotes the exterior algebra on the classes $dw_4$, $dw_6$, and $dw_7$. According to \cref{eq: ansatz all char}, these classes would contribute to $\H^\ast(\G_2; F)$ in cohomological degrees $3,5$, and $6$ respectively. First of all, let us observe that the above ring is \textit{not} isomorphic to $\H^\ast(\G_2; F)$: instead, there is an isomorphism
    $$\H^\ast(\G_2; F) \cong F[dw_4, dw_6]/((dw_4)^4, (dw_6)^2).$$
    Nevertheless, there is an additive isomorphism between $\Omega^\ast_{\H^\ast_{\G_2}(\ast; F)/F} \otimes_{\H^\ast_{\G_2}(\ast; F)} F$ and $\H^\ast(\G_2; F)$, so we can still ask for a stable splitting of $(\G_2)_+$ which realizes the individual summands $\Omega^j_{\H^\ast_{\G_2}(\ast; F)/F} \otimes_{\H^\ast_{\G_2}(\ast; F)} F$. This already fails for $j=1$. Indeed, the $6$-skeleton of $X$ provides a CW-complex $Y$ with a map $Y \to \G_2$ which realizes the inclusion of the subspace $\Omega^1_{\H^\ast_{\G_2}(\ast; F)/F} \otimes_{\H^\ast_{\G_2}(\ast; F)} F \cong F\{dw_4, dw_6, dw_7\}$ into $\H^\ast(\G_2; F)$. However, the map $Y \to \G_2$ cannot stably split (in particular, the $6$-skeleton of $X$ does not stably split off $X$). This was proved ``by hand'' in \cite[Theorem 1.10]{cohen-peterson} using Dyer-Lashof operations.
\end{example}
The preceding example is not special to the non-simply-laced case; one can show that a similar result holds for $G$ of type $E$, too.
\newpage

\appendix

\section{(Not) Lifting $\SL_2$}\label{sec: lifting SL2}
In this brief section, we study the question of lifting $\SL_2$ as a group scheme over $\Z$ to other $\Eoo$-rings like K-theory $\KU$ or the sphere spectrum $S^0$. To set up the question, let us first make the notion of ``lifting'' precise: if $X$ is a scheme over $\Z$ with structure sheaf $\co_X$ and $k$ is an $\Eoo$-ring equipped with an $\Eoo$-map $k \to \Z$, a flat lifting of $X$ to $k$ as an $\E{n}$-scheme (see \cite{francis-thesis}) will mean the data of a sheaf $\co_X^\top$ of $\E{n}$-$k$-algebras on $X$ along with an isomorphism $\co_X^\top \otimes_k \Z \xrightarrow{\simeq} \co_X$. 
It is easy to lift $\GL_n$ to the sphere as an $\Eoo$-scheme (i.e., a spectral scheme in the sense of \cite{SAG}), because it is an open subset in $\AA^{n^2}$. However, we will see in a moment that a simple calculation proves that $\SL_2$ itself cannot be lifted in a natural way (and slight variants of this question lead to very subtle issues that we have been unable to resolve).

Observe that many schemes associated to $\SL_2$ lift all the way to the sphere spectrum. For instance, each choice of Borel subgroup $B\subseteq \SL_2$ (with unipotent radical $U$) defines a surjection $\SL_2 \to \AA^2 - \{0\}$, given by quotienting on the left or the right by $U$. The scheme $\AA^2$ admits a flat lift to a spectral scheme $(\AA^2)_{S^0}$ over $S^0$, given by $\spec S^0[\Z_{\geq 0} \times \Z_{\geq 0}]$. The simple observation is the following:
\begin{prop}\label{prop: E4-lifting SL2}
    There is no flat lifting $(\SL_2)_{S^0}$ of $\SL_2$ to $S^0$ (or even to connective complex K-theory $\ku$) as an $\E{4}$-scheme along with a lifting $(\SL_2)_{S^0} \to (\AA^2)_{S^0}$ of the maps $\SL_2 \to \AA^2 - \{0\} \subseteq \AA^2$.
\end{prop}
\begin{proof}
    Fix a prime $p$, and let $n\geq 1$.
    A flat lifting to $\ku$ of an affine (say) scheme $X = \spec(R)$ over $\Z$ defines power operations on $R$. Indeed, if $\tilde{R}$ is the $\E{n}$-$\ku$-algebra lifting $R$, then $R^\wedge_p \cong \pi_0(L_{K(1)} \tilde{R})$. If $A$ is a $K(1)$-local $\E{2n+1}$-$\KU$-algebra, then $\pi_0(A)$ admits the structure of a ``weak $\delta_n$-ring'', in the sense that there is a map $\delta: \pi_0(A) \to \pi_0(A/p^n)$ of sets (where $A/p^n$ denotes the derived quotient) such that 
    \begin{align*}
        \delta(x+y) & = \delta(x) + \delta(y) - \tfrac{1}{p}((x+y)^p - x^p - y^p) \pmod{p^{n-1}}, \\
        \delta(xy) & = \delta(x) y^p + \delta(y) x^p + p\delta(x) \delta(y) \pmod{p^{n-1}}.
    \end{align*}
    If $A$ refines to an $\E{2n+2}$-$\KU$-algebra, then $\pi_0(A)$ further admits the structure of a ``$\delta_n$-ring'', meaning that the above relations hold modulo $p^n$. (I am grateful to Ishan Levy for a discussion about this.) In this case, the map $\psi: \pi_0(A) \to \pi_0(A/p^{n+1})$ sending $x\mapsto x^p + p\delta(x)$ is a ring map lifting the Frobenius. Observe that
    $$\delta(-x) = \begin{cases}
        -\delta(x) - x^2 & p=2, \\
        -\delta(x) & p>2.
    \end{cases}$$
    The operation $\delta$ is furthermore natural in maps of $K(1)$-local $\E{2n+1}$-$\KU$-algebras. When $n = \infty$, the power operation $\delta$ is constructed in \cite{k1local}, and its construction for finite $n$ is nearly identical.

    The $\Z$-algebra $R = \Z[x]$ admits a canonical lifting to $S^0$ as an $\Eoo$-ring (via $S^0[\Z_{\geq 0}] = S^0[x]$). The corresponding $\delta$-operation on $R$ is simply given by $\delta(x) = 0$. By choosing $U \subseteq \SL_2$ to be the subgroup of upper or lower triangular matrices, one obtains two maps $\SL_2 \to \AA^2$ which send a matrix $\begin{psmallmatrix}
        a & b \\
        c & d
    \end{psmallmatrix}$ to $(a,b)$ and $(d,b)$. The resulting map $f: \co_{\AA^2} \otimes_\Z \co_{\AA^2} \to \co_{\SL_2}$ is a surjection, with kernel given by the determinant ideal $(ad - bc - 1)$. If $\SL_2$ admits a lift to an $\E{4}$-$\ku$-algebra compatibly with the two maps $\SL_2 \to \AA^2$, then the map $f$ must be one of $\delta_1$-rings. It follows that $\delta$ vanishes on the generators $a,b,c,d\in \co_{\SL_2}$. In particular, $\delta(ad) = \delta(a) d^p + \delta(d) a^p + p\delta(a) \delta(d)$ must vanish in $\co_{\SL_2}/p$; similarly for $\delta(bc)$.

    If $p=2$, then
    \begin{align*}
        \delta(ad - bc) & = \delta(ad) + \delta(-bc) + adbc \\
        & = \delta(ad) - \delta(bc) - (-bc)^2 + adbc = bc,
    \end{align*}
    where the final equality is because $\delta(ad) = \delta(bc) = 0$ and $ad - bc = 1$. Similarly, if $p>2$, then
    \begin{align*}
        \delta(ad - bc) & = \delta(ad) + \delta(-bc) - \tfrac{1}{p} ((ad - bc)^p - (ad)^p - (-bc)^p) \\
        & = \tfrac{1}{p} ((bc+1)^p - b^p c^p - 1),
    \end{align*}
    again because $\delta(ad) = \delta(bc) = 0$. The fact that $\delta(ad - bc) = \delta(1) = 0$ implies that for any commutative $\FF_p$-algebra $R$ and a matrix $\begin{psmallmatrix}
        a & b \\
        c & d
    \end{psmallmatrix} \in \SL_2(R)$, the polynomial $\tfrac{1}{p} ((bc+1)^p - b^p c^p - 1)$ must vanish in $R$. This is clearly false: take $R = \FF_p[x]$ and the matrix $\begin{psmallmatrix}
        x+1 & x \\
        1 & 1
    \end{psmallmatrix}$. (One could of course use any prime $p$ to obtain this contradiction; but we allow flexibility in the choice of $p$ to assuade any worries about $\SL_2$ being liftable upon localization at some primes but not others.)
\end{proof}
\begin{remark}\label{rmk: delta and PD}
    Since $\delta(x)$ behaves like the $p$th divided power $-\gamma_p(x)$, the argument of \cref{prop: E4-lifting SL2} can alternatively be interpreted as showing that the ideal which cuts out $\SL_2 \hookrightarrow \GL_2$ does not have a divided power structure, even over $\Z/p^2$.
\end{remark}
Since a weak $\delta_1$-ring structure on a commutative ring $R$ is just a map of sets $\delta: R \to R/p$ satisfying no relations, the above argument does not prove the analogue of \cref{prop: E4-lifting SL2} with $\E{4}$ replaced by $\E{3}$ or $\E{2}$. In particular, \cref{prop: E4-lifting SL2} intriguingly leaves open the possibility that $\SL_2$ admits a lift as an $\E{2}$- or $\E{3}$-scheme to $S^0$ or $\ku$ compatibly with its natural actions on $\AA^2$. Note that the same argument in \cref{prop: E4-lifting SL2} shows that $\SL_n$ also cannot be lifted as an $\E{4}$-scheme to to $S^0$ (or even to connective complex K-theory $\ku$) for any $n\geq 2$ compatibly with its natural actions on $\AA^n$.
\begin{remark}
    The argument of \cref{prop: E4-lifting SL2} is very robust. It can be used to show, for instance, that if $1<k<n-1$, then the Grassmannian $\Gr_k(\AA^n)$ over $\Z$ cannot be lifted as an $\E{4}$-scheme to $S^0$, or even to $\ku$, compatibly with its Pl\"ucker embedding into $\PP(\wedge^k \AA^n) \cong \PP^{\binom{n}{k}-1}$ (which is lifted via the flat projective space of \cite[Construction 5.4.1.3]{SAG}). In fact, an even stronger statement is true: \cite[Theorem 6]{frobenius-on-toric} implies that for a semisimple algebraic group $G$, the flag variety $G/P$ over $\Z/p^2$ does not have a lift of Frobenius as long as the parabolic subgroup $P$ is contained in one of the maximal parabolics enumerated in \cite[Examples 4.3.1-4.3.7]{frobenius-on-toric}. This, in particular, recovers the statement about Grassmannians above. The proof of the general claim uses Bott vanishing, which is more sophisticated than the hands-on approach of \cref{prop: E4-lifting SL2}.
    
    For concreteness, let us demonstrate this non-liftability for $\Gr_2(\AA^4)$, which is cut out inside $\PP^5$ (with coordinates $[x_0: x_1: x_2: y_0: y_1: y_2]$) via the formula $x_0 y_0 - x_1 y_1 + x_2 y_2 = 0$. Let $a = x_0 y_0$, $b = x_1 y_1$, and $c = x_2 y_2$, so that $b = a+c$. Again, $\delta(a) = \delta(b) = \delta(c) = 0$. When $p=2$, for instance, this implies that
    \begin{align*}
        \delta(a - b + c) & = \delta(a) + \delta(-b) + \delta(c) + ab + bc - ac\\
        & = -b^2 + ab + bc - ac = -ac.
    \end{align*}
    For a general prime, one has $\delta(a - b + c) = \tfrac{1}{p} (a^p + c^p - (a+c)^p)$.
    Since $\delta(a-b+c) = \delta(0) = 0$, this implies that $\tfrac{1}{p} (a^p + c^p - (a+c)^p) = 0$; since this function is not identically zero on $\Gr_2(\AA^4)_{\FF_p}$, we obtain the desired contradiction. Note that, just as in \cref{rmk: delta and PD}, this argument says that the ideal cut out by $x_0 y_0 - x_1 y_1 + x_2 y_2$ in $\Z[x_0, \cdots, y_2]/p^2$ does not admit a divided power structure; from this perspective, the above observation should be attributed to Koblitz \cite[Section 3.3(4)]{berthelot-ogus}. 
\end{remark}
\begin{corollary}
    Let $(\GL_2)_{S^0}$ denote the spectral scheme $(\AA^4)_{S^0}[\tfrac{1}{ad- bc}]$, and let $(\GL_1)_{S^0} = \spec S^0[x^{\pm 1}]$. Then the map $\det: \GL_2 \to \GG_m$ over $\Z$ does not lift to a map $(\GL_2)_{S^0} \to (\GL_1)_{S^0}$ exhibiting $(\GL_2)_{S^0}$ as an $\E{4}$-scheme over $(\GL_1)_{S^0}$; in fact, such a lifting is prohibited even over $\ku$.
\end{corollary}
\begin{proof}
    If there was a lifting $(\GL_2)_\ku \to (\GL_1)_\ku$ which exhibits $(\GL_2)_\ku$ as an $\E{4}$-scheme over $(\GL_1)_\ku$, then there would be a map $\ku[x^{\pm 1}] \to \co_{(\GL_2)_\ku} = \ku[a,b,c,d,\tfrac{1}{ad - bc}]$ sending $x\mapsto ad - bc$ which exhibits $\co_{(\GL_2)_\ku}$ as an $\E{4}$-$\ku[x^{\pm 1}]$-algebra. Base-changing along the map $\ku[x^{\pm 1}] \to \ku$ sending $x\mapsto 1$ would then produce an $\E{4}$-$\ku$-algebra $\co_{(\GL_2)_\ku} \otimes_{\ku[x^{\pm 1}]} \ku$ which lifts $\co_{\SL_2}$ to $\ku$. Such a lifting is prohibited by \cref{prop: E4-lifting SL2}. 
\end{proof}
If $(\GL_1^\free)_{S^0} = \spec S^0\{x\}[1/x]$ where $S^0\{x\}$ denotes the free $\Eoo$-ring on one generator, then there \textit{is} a map $(\GL_2)_{S^0} \to (\GL_1^\free)_{S^0}$ exhibiting $(\GL_2)_{S^0}$ as an $\Eoo$-scheme over $(\GL_1^\free)_{S^0}$. However, $(\GL_1^\free)_{S^0}$ is not a flat lift of $\GL_1$ to $S^0$. In other words, there is no reasonable way to construct ``strict'' determinants over the sphere spectrum (or even over $\ku$), at least in the setting of spectral algebraic geometry of $\E{4}$-schemes. 
\begin{remark}
    There \textit{is} a lifting of $\det$ to a map $(\GL_2)_\MU \to (\GL_1)_\MU$ of $\E{2}$-$\MU$-schemes\footnote{However, it does \textit{not} necessarily exhibit $(\GL_2)_\MU$ as an $\E{2}$-scheme over $(\GL_1)_\MU$; so the fiber product $(\GL_2)_\MU \times_{(\GL_1)_\MU} \spec(\MU)$ may not exist in $\E{2}$-$\MU$-schemes. This is one of the quirks of spectral algebraic geometry with $\E{n}$-rings for finite $n$: even if $X$ and $Y$ are spectral schemes with sheaves of $\Eoo$-rings, a map $f: X \to Y$ of $\E{n}$-schemes may not exhibit $X$ as an $\E{n}$-$Y$-scheme. This is because an $\E{n}$-map $A \to B$ between $\Eoo$-rings need not exhibit $B$ as an $\E{n}$-$A$-algebra.}. In fact, any $\E{1}$-$\MU$-algebra map from $\MU[x^{\pm 1}]$ to an even $\E{2}$-$\MU$-algebra $A$ can be refined to an $\E{2}$-$\MU$-algebra map. Indeed, an $\E{1}$-$\MU$-algebra map $f: \MU[x^{\pm 1}] \to A$ can be viewed as the map $g: S^1 \to \BGL_1(A)$ which detects $f(x) \in \pi_1(\BGL_1(A)) = \pi_0(A)^\times$. The map $f$ can be upgraded to an $\E{2}$-$\MU$-algebra map if and only if $g$ can be delooped once. Since $BS^1 = \CP^\infty$ has an even cell structure and $B^2\GL_1(A)$ has even homotopy, there are no obstructions to extending the map $S^2 \to B^2\GL_1(A)$ detecting $f(x)$ along the inclusion $S^2 \to BS^1$. It follows from this discussion that there is an $\E{2}$-$\MU$-algebra map $\det_\MU: \MU[x^{\pm 1}] \to \co_{(\GL_2)_\MU} = \MU[a,b,c,d,\tfrac{1}{ad - bc}]$ which sends $x\mapsto ad - bc$. This, however, only exhibits $\co_{(\GL_2)_\MU}$ as an $\E{1}$-$\MU[x^{\pm 1}]$-algebra. To exhibit it as an $\E{2}$-$\MU[x^{\pm 1}]$-algebra, one would need to upgrade $\det_\MU$ to an $\E{3}$-$\MU$-algebra map, but this seems quite difficult.
\end{remark}

The argument for \cref{prop: E4-lifting SL2} does not require any knowledge of the coalgebra structure on $\co_{\SL_2}$, so it is possible that $\SL_2$ lifts as an $\E{3}$-scheme to $S^0$ or $\ku$, but does not lift as an $\E{1}$-group object therein. The group scheme $\GG_a$ provides a simple example of this phenomenon of lifting as a spectral scheme, but not as a group scheme:
\begin{prop}\label{prop: lifting Ga as group scheme}
    There is no flat lifting $(\GG_a)_{S^0}$ of $\GG_a$ to $S^0$ (or even to connective complex K-theory $\ku$) as an $\E{1}$-group object in $\E{4}$-schemes.
\end{prop}
\begin{proof}
    The $\E{4}$-$\ku$-algebra of functions on the flat lifting $(\GG_a)_\ku$ must be given by $\ku[x]$. The group law on $\GG_a$ is given by the coproduct $\Z[x] \to \Z[x,y]$ sending $x\mapsto x + y$. We therefore need to show that there is no $\E{4}$-$\ku$-algebra map $\Delta:\ku[x] \to \ku[x,y]$ given by $\Delta(x) =  x+y$ on $\pi_0$. This follows from the $\delta_1$-ring structure: we need $\Delta(\delta(x)) = \delta(\Delta(x))$ in $\FF_p[x]$. But $\delta(x) = \delta(y) = 0$, so $\delta(\Delta(x)) = \tfrac{1}{p} (x^p + y^p - (x+y)^p)$ must vanish in $\FF_p[x,y]$, which is a contradiction.
\end{proof}
It was already shown in \cite[Proposition 1.6.20]{elliptic-ii} that there is no flat lifting $(\GG_a)_{S^0}$ of $\GG_a$ to the truncation $\tau_{\leq 1}(S^0)$ as a group scheme; this of course prohibits such a lifting to $S^0$, too. The proof in \textit{loc. cit.} uses the nontriviality of the Hopf element $\eta \in \pi_1(S^0)$. Since this element vanishes in $\ku$, the proof therein cannot be directly adapted to prove \cref{prop: lifting Ga as group scheme}. 
However, let us note that using \cite[Proposition 5.4.9]{rotinv}, one can show that the additive group over $\Z$ admits a flat lifting to a group object in $\E{2}$-schemes over $S^0$.
\newpage

\section{Comparison to Hopkins-Kuhn-Ravenel}\label{sec: comparison hkr}
The calculations of this article (more precisely, the perspective of \cref{rmk: 1-shifted cartier}) were motivated by the work of Hopkins-Kuhn-Ravenel \cite{hkr}, who study the case of finite groups.
In this section, we will describe a relationship to their work. Our discussion will be rather heuristic, and we will sweep a few details under the rug to keep the exposition readable.

Before proceeding, the first thing to note is that while the present article only discusses \textit{connected} compact Lie groups, Hopkins-Kuhn-Ravenel only study \textit{discrete} compact Lie groups (that is, finite groups).  Next, the work of \cite{hkr} only deals with Borel-equivariant cohomology. This means that one does \textit{not} need to assume that the complex-oriented $2$-periodic $\Eoo$-ring $k$ is equipped with an oriented commutative $k$-group $\GG$; recall from \cref{sec: equiv coh} that the purpose of $\GG$ is to provide a decompletion of Borel-equivariant cohomology for compact abelian Lie groups. All that is needed is the formal completion $\hat{\GG}$ of $\GG$ at the identity section. Note that this is not extra data associated to $k$, since $\hat{\GG} = \spf k^{\CP^\infty_+}$. Let $\hat{\GG}_0$ denote the underlying $1$-dimensional formal group over $\pi_0(k)$.

In fact, an even more stringent condition is required of $k$ in \cite{hkr}: not only is it required to be complex-oriented and $2$-periodic, but $\pi_0(k)$ is required to be a complete local Noetherian domain with maximal ideal $\fr{m}$ whose residue field $\pi_0(k)/\fr{m}$ is of characteristic $p>0$, such that $p$ is not nilpotent in $\pi_0(k)$. Let $n$ denote the height of the formal group $\hat{\GG}_0$ base-changed along $\pi_0(k) \to \pi_0(k)/\fr{m}$. In the following discussion, we will simply write $k^0(X)$ to denote $\pi_0$ of the the $k$-cochains on $X$ (instead of the more cumbersome notation $\pi_0 \cf(X)$).

Let $\cc_p$ denote the completion of the algebraic closure of $\QQ_p$, and choose a continuous embedding $\pi_0(k) \to \co_{\cc_p}$. The base-change of $\hat{\GG}_0$ to $\co_{\cc_p}$ defines a formal group law on the maximal ideal of $\co_{\cc_p}$; assume that the base-change of $\hat{\GG}_0$ along the map $\pi_0(k) \to \pi_0(k)/\fr{m}$ has finite height. Then, there exists an exponential isomorphism 
\begin{equation}\label{eq: exponential iso}
    e: (\QQ_p/\Z_p)^n \xrightarrow{\sim} (\fr{m}_{\co_{\cc_p}}, +_{\hat{\GG}_0}),
\end{equation}
where $n$ is the height of $\hat{\GG}_0$. The basic calculation driving the results of \cite{hkr} is the following.
\begin{prop}\label{prop: coh of BZpj}
    There is an isomorphism 
    $$k^0(B\Z/p^j) \cong \pi_0(k)\pw{t}/[p^j](t),$$
    where $[p^j](t) \in \pi_0(k)\pw{t}$ is the $p^j$-series of the formal group law $\hat{\GG}_0$, and $t$ is the first Chern class of the standard character $\Z/p^j \cong \mu_{p^j} \subseteq S^1$. That is, there is an isomorphism $\spf k^0(B\Z/p^j) \cong \hat{\GG}_0[p^j]$.
\end{prop}
\begin{construction}\label{cstr: hkr}
    \cref{prop: coh of BZpj} and the discussion preceding it gives an isomorphism
    $$\spf(k^0(B\Z/p^j)) \otimes_{\spf \pi_0(k)} \spec \cc_p \cong \tfrac{1}{p^j}\Z/\Z,$$
    where the right-hand side denotes the constant group scheme over $\cc_p$. A choice of generator (e.g., $\tfrac{1}{p^j}$) of this group therefore gives a map $k^0(B\Z/p^j) \to \cc_p$.
    Now let $F$ be a finite group, and let $f: \Z_p^n \to F$ be a homomorphism. Then $f$ factors as a map $\Z_p^n \to (\Z/p^j)^n \to F$ for some $j$, so there is a map $k^0(BF) \to k^0(B(\Z/p^j)^n)$. Taking the product of the maps $k^0(B\Z/p^j) \to \cc_p$ described above gives a map $k^0(B(\Z/p^j)^n) \to \cc_p$, which finally defines a composite map
    $$k^0(BF) \to k^0(B(\Z/p^j)^n) \to \cc_p.$$
    This composite depends only on the conjugacy class of $f$, and so this construction defines a map $\Hom(\Z_p^n, F)\mmod F \to \Map(k^0(BF), \cc_p)$, whose adjoint is a map $k^0(BF) \to \Map(\Hom(\Z_p^n, F)\mmod F, \cc_p)$. Here, $F$ acts on $\Hom(\Z_p^n, F)$ by conjugation.
\end{construction}
In the discussion below, $F$ will be a finite group. For simplicity, we will further assume that $k^\ast(BF)$ is concentrated in even degrees (so, by the $2$-periodicity of $k$, it is completely determined by $k^0(BF)$). If $X$ is an $F$-space, the homotopy orbits of $X$ will be denoted $X_{hF}$, while the ordinary quotient of $X$ by the $F$-action will be denoted $X\mmod F$.
\begin{theorem}[Hopkins-Kuhn-Ravenel]\label{thm: hkr}
    The map from \cref{cstr: hkr} defines an isomorphism 
    $$k^0(BF) \otimes_{\pi_0(k)} \cc_p \xrightarrow{\cong} \Map(\Hom(\Z_p^n, F)\mmod F, \cc_p).$$
    The quotient $\Hom(\Z_p^n, F)\mmod F$ can be replaced by the homotopy orbits $\Hom(\Z_p^n, F)_{hF}$, since $F$ is a finite group and its order is invertible in $\cc_p$.
\end{theorem}
Note that the homotopy orbits $\Hom(\Z_p^n, F)_{hF}$ can be identified with $\Map(BT_p^n, BF)$, where $T_p^n = (\QQ_p/\Z_p)^n$ is the $p$-adic $n$-torus.
One can use a ring smaller than $\cc_p$ in \cref{thm: hkr}; essentially, one only needs to extend scalars to the rationalization of the smallest ring containing $\pi_0(k)$ over which the exponential isomorphism \cref{eq: exponential iso} holds.

In \cite{elliptic-iii}, Lurie observes that the isomorphism of \cref{thm: hkr} can be categorified, at least if one assumes the data of a decompletion $\GG$ of $\hat{\GG}$. (We refer the reader to \cite{elliptic-iii} for further details, since the specific setup will not concern us much below.) Namely, if $F$ is a finite group, Lurie defines an $\infty$-category $\Loc_F(\ast; k)$ (denoted by $\mathrm{LocSys}_\GG(BF)$ in \textit{loc. cit.}), and proves the following as (a consequence of) \cite[Theorem 6.4.1]{elliptic-iii}:
\begin{theorem}[Lurie]\label{thm: lurie tempered}
    Fix an embedding $\pi_0(k) \to \cc_p$, so it defines an $\Eoo$-map $k \to \cc_p[u^{\pm 1}]$.
    There is a symmetric monoidal fully faithful embedding
    $$\Loc_F(\ast; k) \otimes_k \cc_p[u^{\pm 1}] \hookrightarrow \Loc(\Map(BT_p^n, BF); \cc_p[u^{\pm 1}]).$$
\end{theorem}
The essential image of the above embedding is described in \cite[Theorem 6.5.13]{elliptic-iii}.

Let us examine the isomorphism \cref{thm: hkr} and the embedding \cref{thm: lurie tempered} further; we will rephrase the right-hand sides of both results as algebro-geometric objects. To do this, note that the exponential isomorphism between $\hat{\GG}_0 \otimes_{\pi_0(k)} \cc_p$ and $(\QQ_p/\Z_p)^n$ defines an isomorphism between $\bD(\hat{\GG}_0) \otimes_{\pi_0(k)} \cc_p$ and $\Z_p^n$. Here, $\bD(\hat{\GG}_0) = \Hom(\hat{\GG}_0, \GG_m)$ is the Cartier dual of $\hat{\GG}_0$.  Note that the $1$-shifted Cartier dual $\hat{\GG}_0^\vee$ can be identified with the classifying stack of $\bD(\hat{\GG}_0)$.

View the finite group $F$ as defining a constant group scheme $\ul{F}$ over $\cc_p$. Since $\hat{\GG}_0^\vee \otimes_{\pi_0(k)} \cc \cong \Z_p^n$, the mapping stack $\Map(\hat{\GG}_0^\vee, B\ul{F})$ is the quotient of the discrete scheme $\ul{\Hom(\Z_p^n, F)}$ by the constant group scheme $\ul{F}$ acting by conjugation. It follows that the $\cc_p$-algebra $\Map(\Hom(\Z_p^n, F)\mmod F, \cc_p)$ is the algebra of functions on the mapping stack $\Map(\hat{\GG}_0^\vee, B\ul{F})$. (Not that since the order of $F$ is invertible in $\cc_p$, the derived and classical algebras of functions agree.) Similarly, $\Loc(\Map(BT_p^n, BF); \cc_p)$ can be viewed as the category of quasicoherent sheaves on the mapping stack $\Map(\hat{\GG}_0^\vee, B\ul{F})$. Therefore, \cref{thm: hkr} and \cref{thm: lurie tempered} can be restated as:
\begin{align}
    \pi_0 k_F \otimes_{\pi_0(k)} \cc_p & \xrightarrow{\cong} \Gamma(\Map(\hat{\GG}_0^\vee, B\ul{F}); \co), \label{eq: restated hkr} \\
    \Loc_F(\ast; k) \otimes_k \cc_p[u^{\pm 1}] & \hookrightarrow \QCoh(\Map(\hat{\GG}_0^\vee, B\ul{F})) \otimes_{\Mod_{\cc_p}} \Mod_{\cc_p[u^{\pm 1}]}. \label{eq: restated lurie tempered}
\end{align}
One can even replace $\hat{\GG}_0$ in the above by $\GG_0$. Observe, now, that $\Map(\GG_0^\vee, B\ul{F})$ is simply the stack $\Bun_{\ul{F}}(\GG_0^\vee)$. 

We can now compare \cref{eq: restated hkr} and \cref{eq: restated lurie tempered} to the discussion in the body of this article. Assume now that $k$ is either $\QQ[u^{\pm 1}]$, $\KU$, or elliptic cohomology. If $G_c$ was instead  a connected compact Lie group, the analogue of \cref{eq: restated hkr} states that $\pi_0 k_{G_c} \otimes_{\pi_0(k)} \cc$ is the ring of (classical, not derived!) global sections of the structure sheaf on $\Bun_G^\ss(\GG_0^\vee)$, where $G$ is the complex reductive group corresponding to $G_c$. This is clear when $\GG_0$ is $\GG_a$ (and $k = \QQ[u^{\pm 1}]$) or $\GG_m$ (and $k = \KU$). In the case when $\GG_0$ is an elliptic curve, this is essentially part of the definition of equivariant elliptic cohomology as sketched in \cite{survey} and constructed in \cite{gepner-meier, t-equiv-tmf}.

Let us continue to assume that $G_c$ is a connected compact Lie group, and further impose that it is simply-laced and almost simple. We will now a give a heuristic argument suggesting that \cref{thm: intro omnibus} -- or rather, its variant from \cref{rmk: g-equiv regular satake elliptic} describing $\Loc_{G_c}^\gr(\Gr_G; k)$ -- can be viewed as an analogue of \cref{eq: restated lurie tempered}. 

Indeed, the rephrasing of \cref{rmk: g-equiv regular satake elliptic} from \cref{rmk: 1-shifted cartier} states that there is an equivalence
\begin{equation}\label{eq: appendix 1-shifted cartier}
    \Loc_{\ld{G}_c}^\gr(\Gr_G; k) \otimes_{\pi_0(k)} \cc \simeq \QCoh(\Bun_{\ld{G}}^\ss(\GG_0^\vee)^\reg).
\end{equation}
The regular locus $\Bun_{\ld{G}}^\ss(\GG_0^\vee)^\reg$ is an open substack of $\Bun_{\ld{G}}^\ss(\GG_0^\vee)$ (whose complement has codimension $\geq 2$, as proved in \cite[Proposition 3.1.16]{davis-elliptic-springer}), and so there is a fully faithful embedding $\QCoh(\Bun_{\ld{G}}^\ss(\GG_0^\vee)^\reg) \hookrightarrow \QCoh(\Bun_{\ld{G}}^\ss(\GG_0^\vee))$. That is, there is a fully faithful embedding
\begin{equation}\label{eq: rephrased-fully faithful}
    \Loc_{\ld{G}_c}^\gr(\Gr_G; k) \otimes_{\pi_0(k)} \cc \hookrightarrow \QCoh(\Bun_{\ld{G}}^\ss(\GG_0^\vee)).
\end{equation}
Assume for the moment that \cref{eq: rephrased-fully faithful} holds if $\ld{G}$ is a finite group $\ld{F}$ (and replace $\cc$ above by $\cc_p$). Of course, it is not clear what the Langlands dual $F$ of $\ld{F}$ should mean; but it is reasonable to believe that, whatever it is, $F$ should be a finite group (or perhaps a finite group scheme). In any case, $\Gr_F$ will just be a point, so the left-hand side of \cref{eq: rephrased-fully faithful} is simply $\Loc_{\ld{F}}^\gr(\ast; k) \otimes_{\pi_0(k)} \cc$. It is reasonable to expect that, thanks to a formality-type statement, the $2$-periodification of the category $\Loc_{\ld{F}}^\gr(\ast; k) \otimes_{\pi_0(k)} \cc$ is equivalent to $\Loc_{\ld{F}}(\ast; k) \otimes_{k} \cc[u^{\pm 1}]$.

Turning to the right-hand side of \cref{eq: rephrased-fully faithful}, note that because $\ld{F}$ is a finite group, there is no meaningful notion of semistability, and so $\Bun_{\ld{F}}^\ss(\GG_0^\vee) = \Bun_{\ld{F}}(\GG_0^\vee)$. With these translations made (so the left-hand side of \cref{eq: rephrased-fully faithful} is replaced by $\Loc_{\ld{F}}(\ast; k) \otimes_{k} \cc[u^{\pm 1}]$, and the right-hand side by the $2$-periodification of $\QCoh(\Bun_{\ld{F}}(\GG_0^\vee))$), \cref{eq: rephrased-fully faithful} is precisely of the form \cref{eq: restated lurie tempered}, as claimed.

\begin{remark}\label{rmk: Bun S2 finite group}
    The above comparison between the quotient $\Gr_G/G\pw{t}$ for a connected compact Lie group $G_c$ and the classifying space $BF$ for a finite group $F$ can be made more precise by noting that $\Gr_G/G\pw{t}$ is homotopy equivalent to the mapping space $\Map(S^2, BG_c) = \Bun_{G_c}(S^2)$, and that if $F$ is a finite group, then $\Bun_F(S^2) = BF$. 
\end{remark}

The work of Hopkins-Kuhn-Ravenel in fact proves a statement which is much more general than \cref{thm: hkr} (and similarly, Lurie's work in \cite{elliptic-iii} yields a much stronger statement than \cref{thm: lurie tempered}). Namely, they prove the following.
\begin{theorem}[Hopkins-Kuhn-Ravenel]\label{thm: upgraded hkr}
    Let $F$ be a finite group, and let $X$ be a finite $F$-space. For each homomorphism $\alpha: \Z_p^n \to F$, let $X^\alpha$ denote the fixed locus of $\im(\alpha)$. Then there is an isomorphism
    $$k^\ast(X_{hF}) \otimes_{\pi_0(k)} \cc_p \xrightarrow{\cong} \H^\ast\left(\left(\sqcup_{\alpha\in \Hom(\Z_p^n, F)} X^\alpha\right)\mmod F;  \cc_p[u^{\pm 1}]\right).$$
\end{theorem}
The isomorphism of \cref{thm: hkr} is the special case when $X$ is a point. In \cite{elliptic-iii}, Lurie shows that \cref{thm: upgraded hkr} is a consequence of a more general statement. If $X$ is a finite $F$-space, Lurie defines an $\infty$-category $\Loc_F(X; k)$ (denoted by $\mathrm{LocSys}_\GG(X//F)$ in \textit{loc. cit.}), and proves the following as \cite[Theorem 6.4.1]{elliptic-iii}:
\begin{theorem}[Lurie]\label{thm: upgraded lurie tempered}
    There is a symmetric monoidal fully faithful embedding
    $$\Loc_F(X; k) \otimes_k \cc_p[u^{\pm 1}] \hookrightarrow \Loc(\Map(BT_p^n, X_{hF}); \cc_p[u^{\pm 1}]).$$
\end{theorem}
The essential image of the above embedding is described in \cite[Theorem 6.5.13]{elliptic-iii}. For the reader interested in chasing down references: specifically, \cref{thm: upgraded lurie tempered} generalizes \cite[Theorem 4.3.2]{elliptic-iii}; the latter implies \cref{thm: upgraded hkr} by \cite[Corollary 4.3.4]{elliptic-iii}. The basic observation is that the mapping space $\Map(BT_p^n, X_{hF})$ is equivalent to $\left(\sqcup_{\alpha\in \Hom(\Z_p^n, F)} X^\alpha\right)_{hF}$. Note that the homotopy quotient $\Hom(\Z_p^n, F)_{hF}$ can be written as a disjoint union $\sqcup_{[\alpha]} BZ(\alpha)$ ranging over conjugacy classes of homomorphisms $\alpha: \Z_p^n \to F$; here $Z(\alpha)$ denotes the centralizer of the image of $\alpha$. Similarly, the homotopy orbits $\left(\sqcup_{\alpha\in \Hom(\Z_p^n, F)} X^\alpha\right)_{hF}$ can be rewritten as the disjoint union $\sqcup_{[\alpha]} X^\alpha_{hZ(\alpha)}$.

\begin{remark}
    One could contemplate a variant of \cref{thm: upgraded hkr} and \cref{thm: upgraded lurie tempered} which replaces $\cc_p$ by other $\Eoo$-$k$-algebras (e.g., over which the base-change of $\hat{\GG}_0$ is not necessarily isomorphic to $(\QQ_p/\Z_p)^n$, but over which it has $(\QQ_p/\Z_p)^j$ as a summand for some $j<n$). The analogues of \cref{thm: upgraded hkr} and \cref{thm: upgraded lurie tempered} in this generality were proved in \cite{stapleton-hkr, stapleton-2} and \cite{elliptic-iii}. 
\end{remark}

Given the analogy between \cref{thm: lurie tempered} and \cref{thm: intro omnibus}, it is natural to ask for an analogue of \cref{thm: upgraded lurie tempered} for connected compact Lie groups. In the following discussion, we suggest an analogy: namely, one could view the $k$-theoretic variant (described for $k = \ku$ in \cite{ku-rel-langlands}) of the local unramified relative Langlands conjecture of \cite{bzsv} as an analogue of the aforementioned results.

To understand this, let us again massage \cref{thm: upgraded hkr} and \cref{thm: upgraded lurie tempered} to a form more suited to algebro-geometric considerations.  We will continue to assume for simplicity that $k^\ast(BF)$ is concentrated in even degrees.
\cref{thm: upgraded hkr} describes how, under the isomorphism of \cref{thm: hkr}, the $k^0(BF) \otimes_{\pi_0(k)} \cc_p$-module $k^\ast(X_{hF}) \otimes_{\pi_0(k)} \cc_p$ decomposes as a module over $\Gamma(\Map(\hat{\GG}_0^\vee, B\ul{F}); \co)$. Similarly, \cref{thm: upgraded lurie tempered} says that there is an explicit $\QCoh(\Map(\hat{\GG}_0^\vee, B\ul{F}))$-module category $\tilde{\cC}_X$ and a fully faithful $\Loc_F(\ast; k) \otimes_k \cc_p[u^{\pm 1}]$-linear embedding
$$\Loc_F(X; k) \otimes_k \cc_p[u^{\pm 1}] \hookrightarrow \tilde{\cC}_X \otimes_{\cc_p} \cc_p[u^{\pm 1}].$$ 
Note that one source of $\QCoh(\Map(\hat{\GG}_0^\vee, B\ul{F}))$-module categories are maps $\ld{L} \to \Map(\hat{\GG}_0^\vee, B\ul{F})$: namely, $\QCoh(\ld{L})$ is a $\QCoh(\Map(\hat{\GG}_0^\vee, B\ul{F}))$-module category. That is, one could imagine that $\tilde{\cC}_X$ is of the form $\QCoh(\ld{L})$ for some such $\ld{L}$ as above which is associated to $X$. (While one can give a somewhat \textit{ad hoc} definition of $\ld{L}$ in terms of the fixed point spaces $X^\alpha$ and their (co)homology\footnote{For instance, take $\ld{L}$ to be the stack $\sqcup_{[\alpha]} \spec(\H^\ast(X^\alpha; \cc_p))/\ul{Z(\alpha)}$. }, it should be rather interesting to intrinsically understand the algebro-geometric properties of $\ld{L}$ directly.)

More generally, recall that the data of a $k$-linear $\infty$-category with $F$-action is just a $\Fun(BF, \Mod_k)$-module category. Since $\Fun(BF, \Mod_k)$ is a completion of the $\infty$-category $\Loc_F(\ast; k)$, one might view the data of a $\Loc_F(\ast; k)$-module category $\cC$ as a decompletion of the notion of a $k$-linear $\infty$-category with $F$-action. One example of such a category is $\Loc_F(X; k)$ for a finite $F$-space $X$. If $\unit_\cC$ is a distinguished object of $\cC$, then $\End_\cC(\unit_\cC) \otimes_{\pi_0(k)} \cc_p$ is a $k^0(BF) \otimes_{\pi_0(k)} \cc_p$-module, and hence a $\Gamma(\Map(\hat{\GG}_0^\vee, B\ul{F}); \co)$-module. One could now ask for a description of this module structure; when $\cC = \Loc_F(X; k)$ and $\unit_\cC$ is the constant sheaf therein, this is precisely answered by \cref{thm: upgraded hkr}.  Similarly, one could ask for an analogue of \cref{thm: upgraded lurie tempered} in this generalized context.
Summarizing, both \cref{thm: upgraded hkr} and \cref{thm: upgraded lurie tempered} can be understood as describing how a $\Loc_F(\ast; k)$-module category decomposes over the mapping stack $\Map(\hat{\GG}_0^\vee, B\ul{F})$. 

Let $G_c$ be a connected, almost simple, simply-laced compact Lie group. Then, as discussed above, the analogue of $\Loc_F(\ast; k)$ is the $\infty$-category $\Loc_{\ld{G}_c}(\Gr_G; k)$. Moreover, the analogue of the tensor product on $\Loc_F(\ast; k)$ is the \textit{convolution} tensor product on $\Loc_{\ld{G}_c}(\Gr_G; k)$ coming, for instance, from the $\ld{G}_c$-equivariant $\E{2}$-space structure on $\Gr_G \cong \Omega G_c$. As mentioned in \cref{rmk: 1-shifted cartier}, the equivalence of \cref{eq: appendix 1-shifted cartier} is monoidal for the convolution tensor product on $\Loc_{\ld{G}_c}(\Gr_G; k)$ and the ordinary tensor product of quasicoherent sheaves on $\Bun_{\ld{G}}^\ss(\GG_0^\vee)^\reg$.

Based on the discussion above, one can interpret the following question as an analogue of \cref{thm: upgraded hkr} and \cref{thm: upgraded lurie tempered}: how does a $\Loc_F(\ast; k)$-module category decompose over $\Bun_{\ld{G}}^\ss(\GG_0^\vee)^\reg$? More precisely, any finite $G_c$-space $X$ should:
\begin{enumerate}
    \item define a $\Loc_{\ld{G}_c}(\Gr_G; k)$-module category $\cC_X$; this is the analogue of the $\Loc_F(\ast; k)$-module category $\Loc_F(X; k)$.
    \item define a fully faithful embedding $\cC_X \hookrightarrow \tilde{\cC}_X$ into an explicit $\QCoh(\Bun_{\ld{G}}^\ss(\GG_0^\vee))$-module category $\tilde{\cC}_X$; this is the analogue of the fully faithful embedding $\Loc_F(X; k) \hookrightarrow \bigoplus_{[\alpha]} \Loc(X^\alpha_{hZ(\alpha)}; \cc_p)$ from \cref{thm: upgraded lurie tempered}.
\end{enumerate}

In the following discussion, we will quietly replace $\Loc_{\ld{G}_c}(\Gr_G; k)$ by $\Loc_{G_c}(\Gr_G; k)$ for conceptual simplicity; this, of course, changes the quasicoherent side, but to avoid getting into more detail than is necessary, we will pretend that the dual side remains unchanged\footnote{If $k = \QQ[u^{\pm 1}]$ and $\GG = \GG_a$, the object $\Bun_{\ld{G}}^\ss(\GG_0^\vee) = \ld{\g}/\ld{G}$ must be replaced by $\ld{\g}^\ast/\ld{G} = \g/\ld{G}$; and similarly, if $k = \KU$ and $\GG = \GG_m$, the object $\Bun_{\ld{G}}^\ss(\GG_0^\vee) = \ld{G}/\ld{G}$ must be replaced by $G/\ld{G}$.}. To describe a candidate for $\cC_X$, recall that the quotient $\Gr_G/G\pw{t}$ is homotopy equivalent to the mapping space $\Map(S^2, BG_c) = \Bun_{G_c}(S^2)$. This, in turn, can be described as the double coset stack $G_c\backslash (LG_c)/G_c$, where $LG_c$ denotes the (topological) free loop space of $G_c$.  Any $G_c$-space $X$ defines an $LG_c$-space $LX$, and the stack $G_c\backslash (LG_c)/G_c$ acts on $(LX)/G_c$ by convolution. That is, the $\infty$-category $\Loc_{G_c}(\Gr_G; k)$ with its convolution tensor product acts on $\Loc_{G_c}(LX; k)$. One could therefore regard the latter category as a candidate for $\cC_X$, and further ask for the following strengthening of (a) and (b) above:
\begin{itemize}
    \item there should be a stack $\ld{L}^\reg$ equipped with a map $\ld{L}^\reg \to \Bun_{\ld{G}}^\ss(\GG_0^\vee)^\reg$ such that there is an equivalence
    $$\cC_X = \Loc_{G_c}(LX; k) \simeq \QCoh(\ld{L}^\reg).$$
    \item the stack $\ld{L}^\reg$ should be an open substack of a larger stack $\ld{L}$, and the map $\ld{L}^\reg \to \Bun_{\ld{G}}^\ss(\GG_0^\vee)^\reg$ extends to a map $\ld{L} \to \Bun_{\ld{G}}^\ss(\GG_0^\vee)$. This gives a fully faithful embedding 
    $$\cC_X \hookrightarrow \tilde{\cC}_X := \QCoh(\ld{L}).$$
\end{itemize}
Note that $\Bun_{\ld{G}}^\ss(\GG_0^\vee)$ is the quotient of $\Bun_{\ld{G}}^\ss(\GG_0^\vee)_\triv$ by $\ld{G}$, so one could equivalently view $\ld{L}$ as the data of a $\ld{G}$-stack $\ld{M}$ equipped with a $\ld{G}$-equivariant map 
$$\mu: \ld{M} \to \Bun_{\ld{G}}^\ss(\GG_0^\vee)_\triv.$$
The relation between $\ld{L}$ and $\ld{M}$ is that $\ld{L} = \ld{M}/\ld{G}$. (There is more to say, regarding shifted symplectic structures \cite{ptvv}, but we refer the reader to \cite[Section 5.2]{ku-rel-langlands} for further discussion.)
\begin{example}
    If $k = \QQ[u^{\pm 1}]$ and $\GG_0 = \GG_a$, then $\ld{M}$ is simply a $\ld{G}$-stack equipped with a $\ld{G}$-equivariant map $\mu: \ld{M} \to \ld{\g}^\ast$. Similarly, if $k = \KU$ and $\GG_0 = \GG_m$, then $\ld{M}$ is simply a $\ld{G}$-stack equipped with a $\ld{G}$-equivariant map $\mu: \ld{M} \to G$.
\end{example}

Suppose $X$ is the analytification of an affine $G$-variety $X_\cc$. In \cite{bzsv}, Ben-Zvi--Sakellaridis--Venkatesh study (under certain additional conditions on $X_\cc$) the full $\infty$-category $\Shv_{G\pw{t}}(X_\cc\ls{t}; \cc)$ as a module over $\Shv_{G\pw{t}}(\Gr_G; \cc)$. The local unramified geometric conjecture of \cite{bzsv} (see \cite[Conjecture 7.5.1]{bzsv}) says -- up to the issue of shearing, which we will ignore here -- that associated to $X_\cc$ is a Hamiltonian $\ld{G}$-stack $\ld{M}$ such that there is an equivalence of categories $\Shv_{G\pw{t}}(X_\cc\ls{t}; \cc) \simeq \QCoh(\ld{M}/\ld{G})$. The data of a Hamiltonian $\ld{G}$-structure on $\ld{M}$ gives, in particular, an $\ld{G}$-equivariant moment map $\ld{M} \to \ld{\g}^\ast$ which makes $\QCoh(\ld{M}/\ld{G})$ into a $\QCoh(\ld{\g}^\ast/\ld{G})$-module category. Moreover, under certain assumptions on $X_\cc$, there is a fully faithful embedding $\Loc_{G_c}(LX; \cc) \hookrightarrow \Shv_{G\pw{t}}(X_\cc\ls{t}; \cc)$. Putting this together, we find a picture exactly like the one described in the preceding paragraph: namely, assuming \cite[Conjecture 7.5.1]{bzsv}, there is a fully faithful embedding
$$\Loc_{G_c}(LX; \cc) \hookrightarrow \Shv_{G\pw{t}}(X_\cc\ls{t}; \cc) \simeq \QCoh(\ld{M}/\ld{G})$$
of $\Loc_{G_c}(LX; \cc)$ into an explicit $\QCoh(\ld{\g}^\ast/\ld{G})$-module category. Therefore, one could view (the $2$-periodification of) \cite[Conjecture 7.5.1]{bzsv} as a conjectural analogue for connected compact Lie groups and $k = \cc[u^{\pm 1}]$ of \cref{thm: upgraded hkr} and \cref{thm: upgraded lurie tempered}.\footnote{Of course, since $F$ is a finite group, \cref{thm: upgraded hkr} and \cref{thm: upgraded lurie tempered} are contentless if $k = \cc[u^{\pm 1}]$; so what we mean by the analogy between \cite[Conjecture 7.5.1]{bzsv} and \cref{thm: upgraded lurie tempered} is that the latter admits a conjectural generalization to connected compact Lie groups, and that the resulting statement specialized to $k = \cc[u^{\pm 1}]$ is still interesting and bears analogy to \cite[Conjecture 7.5.1]{bzsv}.} Motivated by this discussion, we propose in \cite{ku-rel-langlands} that there should be a variant of \cite[Conjecture 7.5.1]{bzsv} for sheaves with coefficients in other $\Eoo$-rings (like connective complex K-theory $\ku$ or elliptic cohomology).
\newpage

\section{Coulomb branches of pure supersymmetric gauge theories}\label{sec: coulomb}
In this brief appendix, we explain some motivation for the results of this article from the perspective of Coulomb branches of $4$d $\cN=2$ and $5$d $\cN=1$ gauge theories with a generic choice of complex structure. The goal here is not to be precise, but instead explain some motivation for the ideas in this article. While reading this appendix, the reader should keep in mind that I know very little physics!
In \cite{bfn-ii, nakajima-coulomb} (see also \cite{nakajima-intro}), it is argued that the Coulomb branch of $3$d $\cN=4$ pure gauge theory on $\RR^3$ can be modeled by the algebraic symplectic variety $\cM_C := \spec \H^{G_c}_\ast(\Gr_G; \cc)$ over $\cc$. This is in turn isomorphic by \cite[Theorem 3]{bf-derived-satake} (reproved here as \cref{cor: loop-rot Gr and biWhit}) to the phase space of the Toda lattice for $\ld{G}$, as well as (by \cite[Theorem A.1]{bfn-ii}) to the moduli space of solutions of Nahm's equations on $[-1,1]$ for a compact form of $\ld{G}$ with an appropriate boundary condition.
The \textit{quantized} Coulomb branch of $3$d $\cN=4$ pure gauge theory on $\RR^3$ is then modeled by $\cA_\epsilon := \H^{G_c\times S^1_\rot}_\ast(\Gr_G; \cc)$. Note that $\cA_\epsilon$ is isomorphic to the algebra of operators of the quantized Toda lattice for $\ld{G}$.

The physical reason for the definition of $\cA_\epsilon$ is the ``$\Omega$-background'' (introduced in \cite{nek-shat}); we refer the reader to \cite{ben-zvi-susy, teleman-icm} for helpful expositions on this topic. The essential idea is as follows: the equivariant homology $C^G_\ast(\Gr_G; \cc)$ admits the structure of an $\Efr{3}$-algebra. In particular, the $\E{3}$-algebra structure on $C^G_\ast(\Gr_G; \cc)$ is equivariant for the action of $S^1$ on $C^G_\ast(\Gr_G; \cc)$ via loop rotation, and the action of $S^1$ on $\E{3}$ via rotation about a line $\ell\subseteq \RR^3$. Using the fact that the fixed points of the $S^1$-action on $\RR^3$ are given by the line $\ell$, it is argued in \cite{ben-zvi-susy} that the homotopy fixed points of $C^G_\ast(\Gr_G; \cc)$ admits the structure of an $\E{1}$-$C_{S^1}^\ast(\ast; \cc)$-algebra. Furthermore, the associative multiplication on $C^{G_c\times S^1_\rot}_\ast(\Gr_G; \cc)$ degenerates to the $2$-shifted Poisson bracket on $\H^{G_c}_\ast(\Gr_G; \cc)$ obtained from the $\E{3}$-algebra structure. The ``$\Omega$-background'' is supposed to refer to the compatibility of the $S^1$-action on $C^G_\ast(\Gr_G; \cc)$ with the $S^1$-action on the $\E{3}$-operad.

From the mathematical perspective, the idea that $S^1$-actions can be viewed as deformation quantizations has been made precise by \cite{preygel, toen-icm}, and more recently in \cite{butson-i, butson-ii}, at least in characteristic zero.  Although often not said explicitly, the idea has been a cornerstone of the development of Hochschild homology and its relatives. (The reader can skip the following discussion, since it will not be necessary in the remainder of this section; we only include it for completeness.) 

Consider a smooth $\cc$-scheme $X$, so that the Hochschild-Kostant-Rosenberg theorem gives an isomorphism $\HH(X/\cc) \simeq \Sym(\Omega^1_{X/\cc}[1])$. There is an isomorphism $\Sym(\Omega^1_{X/\cc}[1]) \simeq \bigoplus_{n\geq 0} (\wedge^n \Omega^1_{X/\cc})[n]$, so $\Sym(\Omega^1_{X/\cc}[1])$ can be understood as a shearing of the algebra $\Omega^\ast_{X/\cc} = \bigoplus_{n\geq 0} (\wedge^n \Omega^1_{X/\cc})[-n]$ of differential forms. The Hochschild-Kostant-Rosenberg theorem further states that the $S^1$-action on $\HH(X/\cc)$ is a shearing of the de Rham differential on $\Omega^\ast_{X/\cc}$. 
    
The Koszul dual of the algebra $\HH(X/\cc) \simeq \Sym(\Omega^1_{X/\cc}[1])$ is $\Sym(T_{X/\cc}[-2]) \simeq \co_{T^\ast[2] X}$; in the same way, the sheaf of differential operators on $X$ is Koszul dual to the de Rham complex of $X$. This can be drawn pictorially as follows:
$$\xymatrix{
\Sym(T_{X/\cc}[-2]) \simeq \co_{T^\ast[2] X} \ar@{~>}[r]^-{\text{def. quant}} \ar@{~>}[d]_-{\text{Koszul dual}} & \cd^\hbar_{X/\cc} \ar@{~>}[d]^-{\text{Koszul dual}} \\
\Sym_{\co_X}(\Omega^1_{X/\cc}[1]) \simeq \HH(X/\cc) \ar@{~>}[r]_-{S^1\text{-action}} & \text{shearing of }(\Omega^\ast_{X/k}, d_\dR).
}$$
Since the algebra $\cd_X^\hbar$ of differential operators is a quantization of $T^\ast[2] X$, this diagram illustrates the idea that the $S^1$-action on Hochschild homology plays the role of a Koszul dual to deformation quantization.

\begin{example}\label{ex: 3d-sl2}
    We will keep $G = \PGL_2$ as a running example in discussing Coulomb branches (see also \cite[Section 2]{seiberg-witten-coulomb}), so that $\ld{G} = \SL_2$. In this case,
    $$\cM_C \cong \spec \cc[x, a^{\pm 1}, \tfrac{a-a^{-1}}{x}]^{\Z/2} \cong \spec \cc[x^2, a+a^{-1}, \tfrac{a-a^{-1}}{x}]$$
    by \cref{thm: ordinary hmlgy reg centr}, where $\Z/2$ acts on $\cc[x, a^{\pm 1}, \tfrac{a-a^{-1}}{x}]$ by $x\mapsto -x$ and $a\mapsto a^{-1}$. 
    This is the regular centralizer group scheme of $\SL_2$. 
    Let us denote by 
    \begin{align*}
        \Phi & = x^2,  \\
        U & = a + a^{-1}, \\
        V & = \tfrac{a-a^{-1}}{x}.
    \end{align*}
    Then we have the single relation
    $$U^2 - \Phi V^2 = (a + a^{-1})^2 - (a - a^{-1})^2 = 4,$$
    so $\cM_C$ is isomorphic to the subvariety of $\AA^3_\cc$ cut out by the above equation. 
    This is known as the \textit{Atiyah-Hitchin manifold}, and was studied in great detail in \cite{atiyah-hitchin} (see \cite[Page 20]{atiyah-hitchin} for the definition). In \cite[Theorem A.1]{bfn-ii}, it was shown that the Atiyah-Hitchin manifold is isomorphic to the moduli space of solutions of Nahm's equations on $[-1,1]$ for $\mathrm{SU}(2)$ with an appropriate boundary condition.

    Since a normal vector to the defining equation of $\cM_C$ is $2U\partial_U - V^2 \partial_\phi - 2V\Phi \partial_V$, the standard holomorphic $3$-form $dU \wedge d\Phi \wedge dV$ on $\AA^3_\cc$ induces a holomorphic symplectic form $\tfrac{d\Phi \wedge dV}{2U}$ on $\cM_C$. (This can also be written as $\tfrac{dU \wedge dV}{V^2}$ or as $\tfrac{d\Phi \wedge dU}{2\Phi V}$.) The associated Poisson bracket on $\co_{\cM_C} \cong \H^{G_c}_\ast(\Gr_G; \cc)$ agrees with the $2$-shifted Poisson bracket arising from the $\E{3}$-structure on $C^{G_c}_\ast(\Gr_G; \cc)$.

    The quantized algebra $\cA_\epsilon$ can be described explicitly as follows. Let us write $\theta = \tfrac{1}{x}(s-1)$, where $s$ is the simple reflection generating the Weyl group of $\SL_2$. Then $\cA_\epsilon$ is generated as an algebra over $\cc\pw{\hbar}$ by $\Z/2$-invariant polynomials in $x$, $a^{\pm 1}$, and $\theta$, where $x$ is to be viewed as $a\partial_a$. Moreover, under the isomorphism $\cA_\epsilon/\hbar \cong \co_{\cM_C}$, the class $x$ is sent to $x$, and $\theta$ is sent to $\tfrac{a-1}{x}$. We then have the commutation relation $[x,a^{\pm 1}] = \pm \hbar a^{-1}$, induced by $[\partial_a, a] = \hbar$; see \cref{ex: ordinary quantized diffop}. This implies that 
    $$[x^2, a^{\pm 1}] = \hbar^2 a^{\pm 1} \pm 2\hbar a^{\pm 1} x,$$
    which in turn implies that $\cA_\epsilon$ is the quotient of the free associative $\cc\pw{\hbar}$-algebra on $\Phi$, $U$, and $V = \tfrac{1}{x}(a-a^{-1})$ subject to the relations
    \begin{align*}
        [\Phi,V] & = 2\hbar U - \hbar^2 V,\\
        [\Phi, U] & = 2\hbar \Phi V - \hbar^2 U, \\
        [U,V] & = \hbar V^2,\\
        U^2 - 4 & = \Phi V^2 - \hbar UV.
    \end{align*}
    Note that the commutation relations for $[\Phi, U]$ and $[U,V]$ in \cite[Equation B.3]{dimofte-garner} have typos, but it is stated correctly in \cite[Equation 5.51]{bullimore-dimofte-gaiotto}. 
\end{example}
\begin{example}
    When $G = \SL_2$, we can identify $\cM_C$ with the quotient of the scheme of \cref{ex: 3d-sl2} by the free $\Z/2$-action sending $U\mapsto -U$ and $V\mapsto -V$; so
    $$\cM_C \cong \spec \cc[x^2, (a+a^{-1})^2, \left(\tfrac{a-a^{-1}}{2x}\right)^2, \tfrac{(a+a^{-1})(a-a^{-1})}{2x}].$$
    This is the regular centralizer group scheme for $\PGL_2$. Note that if we denote
    \begin{align*}
        \Phi & = x^2,\\
        A & = (a+a^{-1})^2,\\
        B & = 4\left(\tfrac{a-a^{-1}}{2x}\right)^2 = \tfrac{(a-a^{-1})^2}{x^2},\\
        C & = 2\tfrac{(a+a^{-1})(a-a^{-1})}{2x} = \tfrac{(a+a^{-1})(a-a^{-1})}{x},
    \end{align*}
    then we have relations
    \begin{align*}
        AB & = C^2,\\
        A - \Phi B & = 4.
    \end{align*}
    In particular, $\cM_C$ is cut out in $\AA^3_\cc$ (with coordinates $\Phi$, $B$, and $C$) via the equation
    $$C^2 - \Phi B^2 = 4B.$$
    Note the similarity to the manifold from \cref{ex: 3d-sl2}: in fact, it is the quotient of the aforementioned manifold by the free $\Z/2$-action sending $U\mapsto -U$ and $V\mapsto -V$. In terms of these coordinates, $B = V^2$ and $C = UV$. (Sometimes, this quotient is also referred to as the Atiyah-Hitchin manifold.) It is also possible to describe $\cA_\epsilon$; we leave this to the reader, since it is rather tedious.
\end{example}

\begin{heuristic}\label{heuristic: 4d-n2}
    An unpublished conjecture of Gaiotto (which I learned about from Nakajima) says that the Coulomb branch of $4$d $\cN=2$ pure gauge theory over $\RR^3 \times S^1$ with a generic choice of complex structure can be modeled by $\cM_C^\fourd := \spec \KU^{G_c}_0(\Gr_G) \otimes_\Z \cc$. Although I do not know Gaiotto's motivation for this conjecture (it is probably inspired by \cite{seiberg-witten-coulomb}), my attempt at heuristically justifying it goes as follows. (In \cite[Appendix C(b)]{ku-rel-langlands}, I suggest that it might be slightly better to consider $\spec \ku^{G_c}_\ast(\Gr_G) \otimes_\Z \cc$ instead, where $\ku$ denotes \textit{connective} complex K-theory. The Bott class generating $\pi_2(\ku)$ plays the role of the radius of the circle $S^1$.)
    
    Recall that $\Gr_G/G\pw{t}$ can be viewed as $\Bun_G(S^2)$. It is reasonable to view $\KU_0(\Bun_G(S^2)) \otimes \cc$ as closely related to $\H_\ast(L\Bun_G(S^2); \cc)$, where $L\Bun_G(S^2)$ denotes the topological free loop space of $\Bun_G(S^2)$. Since $LBG \simeq BLG$, we have $L\Bun_G(S^2) \simeq \Bun_{LG}(S^2)$, so one might view $\H_\ast(L\Bun_G(S^2); \cc)$ as the ring of functions on the ``Coulomb branch of $3$d $\cN=4$ pure gauge theory on $\RR^3$ with gauge group $LG$''.

    Making precise sense of this phrase seems difficult, but one possible workaround could be the following. It is often useful to view gauge theory with gauge group $LG$ as ``finite temperature'' gauge theory with gauge group $G$. Recall that Wick rotation relates $(3+1)$-dimensional quantum field theory at a finite temperature $T$ to statistical mechanics over $\RR^3 \times S^1$ where the circle has radius $\tfrac{1}{2\pi T}$. This suggests that $\H_\ast(L\Bun_G(S^2); \cc)$ (which is more precisely to be replaced by $\KU^{G_c}_0(\Gr_G) \otimes \cc$) can be viewed as the ring of functions on the ``Coulomb branch of $4$d $\cN=2$ pure gauge theory on $\RR^3 \times S^1$ with gauge group $G$''. See \cite[Remark 3.14]{bfn-ii}. In \cite{bfm}, $\spec \KU^{G_c}_0(\Gr_G) \otimes \cc$ was identified with the phase space of the relativistic Toda lattice for $\ld{G}$.

    One can also define a quantization of $\cM_C^\fourd$ via $\cA_\epsilon^\fourd := \KU^{G_c \times S^1_\rot}_0(\Gr_G) \otimes \cc$; this can be viewed as a model for the quantized Coulomb branch of $4$d $\cN=2$ pure gauge theory on $\RR^3\times S^1$. The algebra $\cA_\epsilon^\fourd$ can be identified with the algebra of operators of the quantized relativistic Toda lattice for $\ld{G}$.
\end{heuristic}
\begin{example}\label{ex: 4d-sl2}
    When $G = \PGL_2$, \cref{thm: ku hmlgy reg centr} tells us that 
    $$\cM_C^\fourd \cong \spec \cc[x^{\pm 1}, a^{\pm 1}, \tfrac{a-a^{-1}}{x-1}]^{\Z/2} \cong \spec \cc[x+x^{-1}, a+a^{-1}, \tfrac{(a-a^{-1})(x+1)}{x-1}],$$
    where $\Z/2$ acts by $x\mapsto x^{-1}$ and $a\mapsto a^{-1}$.
    For simplicity, let us consider instead a slight variant of $\cM_C^\fourd = \spec \KU^{\mathrm{PSU}(2)}_0(\Gr_{\PGL_2}) \otimes_\Z \cc$, given by ${\cM'}_C^\fourd = \spec \KU^{\SU(2)}_0(\Gr_{\PGL_2}) \otimes_\Z \cc$. Then
    $${\cM'}_C^\fourd \cong \spec \cc[x^{\pm 1}, a^{\pm 1}, \tfrac{a-a^{-1}}{x-x^{-1}}]^{\Z/2} \cong \spec \cc[x+x^{-1}, a+a^{-1}, \tfrac{a-a^{-1}}{x-x^{-1}}].$$
    Let us write $\Psi = x + x^{-1}$, $W = a+a^{-1}$, and $Z = \tfrac{a-a^{-1}}{x-x^{-1}}$. Then, one easily verifies that ${\cM'}_C^\fourd$ is the subvariety of $\AA^3_\cc$ cut out by the equation
    $$W^2 - (\Psi^2 - 4)Z^2 = 4.$$
    This may be regarded as a multiplicative analogue of the Atiyah-Hitchin manifold. It would be very interesting to understand a relationship between this manifold and the moduli space of solutions to some analogue of Nahm's equations for $\mathrm{PSU}(2)$ with an appropriate boundary condition. The complex manifold ${\cM'}_C^\fourd$ has a holomorphic symplectic form given by $\tfrac{d\Psi \wedge dZ}{W}$, which can also be written as $\tfrac{d\Psi \wedge dW}{(\Psi^2-4)Z}$ or as $\tfrac{dZ \wedge dW}{\Psi Z^2}$.

    It is also possible to explicitly describe the quantized algebra $\cA_\epsilon^\fourd$. The resulting description is not very enlightening, so we will only indicate how one reaches the answer. In this case, instead of the relation $[\partial_a, a] = \hbar$ which appeared in \cref{ex: 3d-sl2}, we have the relation $xa = qax$ (i.e., $xax^{-1} a^{-1} = q$); see \cref{ex: q quantized diffop}. In particular, $xa^{-1} = q^{-1} a^{-1} x$, $x^{-1} a = q^{-1} a x^{-1}$, and $x^{-1} a^{-1} = qa^{-1}x^{-1}$. It follows after some tedious calculation that $\cA_\epsilon^\fourd$ is the quotient of the free associative $\cc[q^{\pm 1}]$-algebra on $\Psi$, $W$, and $\tfrac{x+1}{x-1} (a-a^{-1})$ subject to four relations.

    Suppose we consider instead the variant of $\cA_\epsilon^\fourd$ defined by ${\cA'}_\epsilon^\fourd = \KU^{\SU(2) \times S^1_\rot}_0(\Gr_{\PGL_2}) \otimes \cc$.  Then ${\cA'}_\epsilon^\fourd$ is the quotient of the free associative $\cc[q^{\pm 1}]$-algebra on $\Psi$, $W$, and $Z = \tfrac{1}{x-x^{-1}} (a-a^{-1})$ subject to the relations
    \begin{align*}
        [\Psi, W] & = (q-1)(\Psi^2-4)Z - \tfrac{(q-1)^2}{2q} ((\Psi^2-4)Z + \Psi W), \\
        [\Psi, Z] & = (q-1)W - \tfrac{(q-1)^2}{2q}(\Psi Z + W),\\
        [Z,W] & = (q-1) \Psi Z^2 - \tfrac{(q-1)^2}{2q} (\Psi Z + W)Z, \\
        W^2 - 4 & = (\Psi^2-4)Z^2 - \tfrac{(q-1)^2}{2q} (\Psi^2-4)Z^2 + \tfrac{q^2-1}{2q} \Psi WZ.
    \end{align*}
\end{example}
\begin{remark}
    As always, one can also replace $\KU$ by connective complex K-theory $\ku$. This introduces a new ``Bott'' parameter $\beta$ which recovers $\KU$ when $\beta$ is set to $1$ (informally speaking), and recovers ordinary cohomology when $\beta$ is killed. In \cite[Appendix C(b)]{ku-rel-langlands}, I have suggested that working with $\ku$ instead of $\KU$ in \cref{heuristic: 4d-n2}  should produce the ``Coulomb branch of $4$d $\cN=2$ pure gauge theory on $\RR^3 \times S^1$ with gauge group $G$'' where the Bott parameter $\beta$ identifies with the radius of the circle $S^1$.
    
    In the case of \cref{ex: 4d-sl2}, the resulting manifold ${\cM'}_C^\beta = \spec \ku^{\SU(2)}_\ast(\Gr_{\PGL_2}) \otimes_{\Z[\beta]} \cc[\beta]$ is given by
    $${\cM'}_C^\beta \cong \spec \cc[\beta, x, \tfrac{1}{1+\beta x}, a^{\pm 1}, \tfrac{a-a^{-1}}{x-\ol{x}}]^{\Z/2} \cong \spec \cc[x+\ol{x}, a+a^{-1}, \tfrac{a-a^{-1}}{x-\ol{x}}].$$
    Here, $\ol{x}$ is the inverse of $x$ in the group law $x + y + \beta xy$, so that $\ol{x} = -\tfrac{x}{1+\beta x}$. 
    Let us write $\Psi' = x + \ol{x}$, $W = a+a^{-1}$, and $Z = \tfrac{a-a^{-1}}{x-\ol{x}}$; then ${\cM'}_C^\beta$ is the subvariety of $\AA^3_\cc \times \AA^1_\beta$ cut out by the equation
    $$W^2 - (\beta^2 \Psi' - 4)\Psi' Z^2 = 4.$$
    The complex manifold ${\cM'}_C^\beta$ has a holomorphic symplectic form given by $\tfrac{d\Psi' \wedge dZ}{W}$, which can also be written as $\tfrac{d\Psi' \wedge dW}{2Z\Psi' (\beta^2 \Psi'-4)}$ or as $\tfrac{dZ \wedge dW}{2Z^2(\beta^2 \Psi' - 2)}$.
    When $\beta$ is set to $1$, one can identify $\Psi'$ with $\Psi + 2$ with $\Psi$ as in \cref{ex: 4d-sl2}; then ${\cM'}_C^\beta$ recovers the multiplicative Atiyah-Hitchin manifold. When $\beta$ is killed, ${\cM'}_C^\beta$ is the usual Atiyah-Hitchin manifold. Again, it would be very interesting to understand a relationship between ${\cM'}_C^\beta$ and the moduli space of solutions to some analogue of Nahm's equations for $\mathrm{PSU}(2)$ with an appropriate boundary condition.  It is also possible to compute the loop-rotation equivariant version of ${\cM'}_C^\beta$, but we leave this to the reader.
\end{remark}
\begin{example}
    When $G = \SL_2$, one can view $\cM_C^\fourd$ with the quotient of the scheme ${\cM'}_C^\fourd$ of \cref{ex: 4d-sl2} by the free $\Z/2$-action sending $W\mapsto -W$ and $Z\mapsto -Z$; so
    $$\cM_C^\fourd \cong \spec \cc[x + x^{-1}, (a + a^{-1})^2, \left(\tfrac{a - a^{-1}}{x - x^{-1}}\right)^2, \tfrac{(a - a^{-1})(a + a^{-1})}{x - x^{-1}}].$$
    This is the regular centralizer group scheme for $\PGL_2$. Note that if we denote
    \begin{align*}
        \Psi & = x + x^{-1},\\
        A & = (a+a^{-1})^2 = a^2 + a^{-2} + 2,\\
        B & = \left(\tfrac{a-a^{-1}}{x-x^{-1}}\right)^2 = \tfrac{a^2+a^{-2}-2}{x^2 + x^{-2} - 2},\\
        C & = \tfrac{(a+a^{-1})(a-a^{-1})}{x - x^{-1}} = \tfrac{a^2-a^{-2}}{x-x^{-1}},
    \end{align*}
    then we have relations
    \begin{align*}
        AB & = C^2,\\
        A - (\Psi^2 - 4) B & = 4.
    \end{align*}
    In particular, $\cM_C^\fourd$ is cut out in $\AA^3_\cc$ (with coordinates $\Psi$, $B$, and $C$) via the equation
    $$C^2 - (\Phi^2 - 4) B^2 = 4B.$$
    Note the similarity to \cref{ex: 4d-sl2}. It is also possible to describe $\cA_\epsilon^\fourd$; again, we leave this to the reader, since it is rather tedious.

    Let us again note the variant involving connective complex K-theory $\ku$: in this case, 
    $$\cM_C^\beta \cong \spec \cc[\beta, x + \ol{x}, (a + a^{-1})^2, \left(\tfrac{a - a^{-1}}{x - \ol{x}}\right)^2, \tfrac{(a - a^{-1})(a + a^{-1})}{x - \ol{x}}].$$
    If we denote
    \begin{align*}
        \Psi' & = x + \ol{x},\\
        A & = (a+a^{-1})^2 = a^2 + a^{-2} + 2,\\
        B' & = \left(\tfrac{a-a^{-1}}{x-\ol{x}}\right)^2 = \tfrac{a^2+a^{-2}-2}{x^2 + \ol{x}^2 - 2x\ol{x}},\\
        C' & = \tfrac{(a+a^{-1})(a-a^{-1})}{x - \ol{x}} = \tfrac{a^2-a^{-2}}{x-\ol{x}},
    \end{align*}
    then we have relations
    \begin{align*}
        AB' & = {C'}^2,\\
        A - (\beta^2 \Psi' - 4)\Psi' B & = 4.
    \end{align*}
    In particular, $\cM_C^\beta$ is cut out in $\AA^3_\cc \times \AA^1_\beta$ (with coordinates $\Psi$, $B$, $C$, and $\beta$) via the equation
    $$C^2 - (\beta^2 \Psi' - 4)\Psi' B^2 = 4B.$$
\end{example}

Consider an elliptic curve $E(\cc)$ over $\cc$. Motivated by \cref{heuristic: 4d-n2} and \cite{nakajima-yoshioka}, one might expect that (in some specific complex structure) the Coulomb branch of $5$d $\cN=1$ pure gauge theory over $\RR^3 \times E(\cc)$ can be modeled by the complexification of the $G_c$-equivariant $k$-homology of $\Gr_G$, where $k$ is an elliptic cohomology theory associated to a putative integral lift of $E$. Unfortunately, a classical result of Tate says that there are no smooth elliptic curves over $\Z$ (see \cite{ogg-ell-curve-over-Z} for an elementary proof); so $E(\cc)$ cannot literally lift to $\Z$ (i.e., $\pi_0(k)$ cannot be $\Z$).

As a fix, one can more generally simultaneously consider all possible ``Coulomb branches'' $\cM_C^\fived := \spec \pi_0 \cf_G(\Gr_G)^\vee \otimes \cc$ associated to every complex-oriented $2$-periodic $\Eoo$-ring $k$ equipped with an oriented elliptic curve (this is very nearly the same as considering the universal example $\spec \tmf^{G_c}_0(\Gr_G) \otimes \cc$). We have described $\spec \pi_0 \cf_T(\Gr_G)^\vee \otimes \cc$ in \cref{thm: elliptic hmlgy reg centr}, from which one can calculate $\cM_C^\fived$. Similarly, one can even use \cref{thm: ell loop-rot flag} to calculate $\pi_0 \cf_{T\times \GG_m^\rot}(\Gr_G)^\vee \otimes \cc$ and $\cA_\epsilon^\fived := \pi_0 \cf_{G\times \GG_m^\rot}(\Gr_G)^\vee \otimes \cc$, but this is already incredibly complicated for $G = \SL_2$.

It would be very interesting to give a physical interpretation to $\pi_0 \cf_G(\Gr_G)^\vee \otimes \cc$ and $\pi_0 \cf_{G\times \GG_m^\rot}(\Gr_G)^\vee \otimes \cc$ for other $2$-periodic $\Eoo$-rings $k$, although we expect this to be very difficult. Indeed, most other chromatically interesting generalized equivariant cohomology theories only exist after profinite or $p$-adic completion, and do not admit transcendental analogues; but see \cref{rmk: morava e-theory}. It would also be very interesting to describe the analogue of our calculations for the ind-schemes $\cR_{G,\mathbf{N}}$ introduced in \cite{bfn-ii}. By adapting the methods of \cite[Section 4]{bfn-ii}, this is easy when $G$ is a torus. We expect it to lead to interesting geometry for nonabelian $G$. In \cite{ku-rel-langlands}, we extend the discussion of this paper (at least, the parts concerning ordinary cohomology and K-theory) to connective K-theory, and suggest an analogue of the relative Langlands program of \cite{bzsv} in this setting; as mentioned in the introduction of \cite{ku-rel-langlands}, the story therein also admits an elliptic variant.
\newpage

\bibliographystyle{alphanum}
\bibliography{main}

\newcommand{\etalchar}[1]{$^{#1}$}
\begin{thebibliography}{{Gan}3}

\bibitem[AB]{atiyah-bott-localization}
M.~{Atiyah} and R.~{Bott}.
\newblock {The moment map and equivariant cohomology}.
\newblock {\em {Topology}}, 23(1):1--28, 1984.

\bibitem[ABG]{abg-iwahori-satake}
S.~{Arkhipov}, R.~{Bezrukavnikov}, and V.~{Ginzburg}.
\newblock {Quantum groups, the loop {G}rassmannian, and the {S}pringer resolution}.
\newblock {\em {J. Amer. Math. Soc.}}, 17(3):595--678, 2004.

\bibitem[ABS]{atiyah-bott-shapiro}
M.~F. {Atiyah}, R.~{Bott}, and A.~{Shapiro}.
\newblock Clifford modules.
\newblock {\em Topology}, 3(suppl. 1):3--38, 1964.

\bibitem[{Ada}1]{adams-vector-fields}
J.~F. {Adams}.
\newblock {Vector fields on spheres}.
\newblock {\em {Ann. of Math. (2)}}, 75:603--632, 1962.

\bibitem[{Ada}2]{adams-jx-iv}
J.~F. {Adams}.
\newblock {On the groups {$J(X)$}. {IV}}.
\newblock {\em {Topology}}, 5:21--71, 1966.

\bibitem[{Ada}3]{adams-lectures-coh}
J.~F. {Adams}.
\newblock {Lectures on generalised cohomology}.
\newblock In {\em {Category {T}heory, {H}omology {T}heory and their {A}pplications, {III} ({B}attelle {I}nstitute {C}onference, {S}eattle, {W}ash., 1968, {V}ol. {T}hree)}}, pages 1--138. {Springer, Berlin}, 1969.

\bibitem[AG]{arinkin-gaitsgory-singsupp}
D.~{Arinkin} and D.~{Gaitsgory}.
\newblock {Singular support of coherent sheaves and the geometric {L}anglands conjecture}.
\newblock {\em {Selecta Math. (N.S.)}}, 21(1):1--199, 2015.

\bibitem[AH]{atiyah-hitchin}
M.~{Atiyah} and N.~{Hitchin}.
\newblock {\em {The geometry and dynamics of magnetic monopoles}}.
\newblock {M. B. Porter Lectures}. {Princeton University Press, Princeton, NJ}, 1988.

\bibitem[AHR]{koandtmf}
M.~{Ando}, M.~{Hopkins}, and C.~{Rezk}.
\newblock {Multiplicative orientations of $KO$-theory and of the spectrum of topological modular forms}.
\newblock \url{http://www.math.uiuc.edu/~mando/papers/koandtmf.pdf}, May 2010.

\bibitem[AHS]{ando-hopkins-strickland}
M.~{Ando}, M.~{Hopkins}, and N.~{Strickland}.
\newblock {Elliptic spectra, the {W}itten genus and the theorem of the cube}.
\newblock {\em {Invent. Math.}}, 146(3):595--687, 2001.

\bibitem[AMM]{amm-qham}
A.~{Alekseev}, A.~{Malkin}, and E.~{Meinrenken}.
\newblock {Lie group valued moment maps}.
\newblock {\em {J. Differential Geom.}}, 48(3):445--495, 1998.

\bibitem[{And}]{ando-power-operations}
M.~{Ando}.
\newblock {Power operations in elliptic cohomology and representations of loop groups}.
\newblock {\em {Trans. Amer. Math. Soc.}}, 352(12):5619--5666, 2000.

\bibitem[{Ant}]{antieau-filtration-HP}
B.~{Antieau}.
\newblock {Periodic cyclic homology and derived de {R}ham cohomology}.
\newblock {\em {Ann. K-Theory}}, 4(3):505--519, 2019.

\bibitem[AS]{atiyah-segal-original}
M.~{Atiyah} and G.~{Segal}.
\newblock {Equivariant {$K$}-theory and completion}.
\newblock {\em {J. Differential Geometry}}, 3:1--18, 1969.

\bibitem[{Ati}1]{atiyah-bundle-elliptic}
M.~{Atiyah}.
\newblock {Vector bundles over an elliptic curve}.
\newblock {\em {Proc. London Math. Soc. (3)}}, 7:414--452, 1957.

\bibitem[{Ati}2]{atiyah-power-operations}
M.~{Atiyah}.
\newblock {Power operations in {$K$}-theory}.
\newblock {\em {Quart. J. Math. Oxford Ser. (2)}}, 17:165--193, 1966.

\bibitem[BBB{\etalchar{+}}]{ben-zvi-susy}
C.~{Beem}, D.~{Ben-Zvi}, M.~{Bullimore}, T.~{Dimofte}, and A.~{Neitzke}.
\newblock {Secondary products in supersymmetric field theory}.
\newblock {\em Annales Henri Poincare}, 21(4):1235--1310, 2020.

\bibitem[BBD]{bbdg}
A.~{Beilinson}, J.~{Bernstein}, and P.~{Deligne}.
\newblock {Faisceaux pervers}.
\newblock In {\em {Analysis and topology on singular spaces, {I} ({L}uminy, 1981)}}, volume 100 of {\em {Ast\'erisque}}, pages 5--171. {Soc. Math. France, Paris}, 1982.

\bibitem[BC]{beliakova-cooper}
A.~{Beliakova} and B.~{Cooper}.
\newblock {Steenrod structures on categorified quantum groups}.
\newblock {\em {Fund. Math.}}, 241(2):179--207, 2018.

\bibitem[BDG]{bullimore-dimofte-gaiotto}
M.~{Bullimore}, T.~{Dimofte}, and D.~{Gaiotto}.
\newblock {The {C}oulomb branch of 3d {$\mathcal{N}=4$} theories}.
\newblock {\em {Comm. Math. Phys.}}, 354(2):671--751, 2017.

\bibitem[{Bei}]{beilinson-glue-perverse}
A.~{Beilinson}.
\newblock {How to glue perverse sheaves}.
\newblock In {\em {{$K$}-theory, arithmetic and geometry ({M}oscow, 1984--1986)}}, volume 1289 of {\em {Lecture Notes in Math.}}, pages 42--51. {Springer, Berlin}, 1987.

\bibitem[BF]{bf-derived-satake}
R.~{Bezrukavnikov} and M.~{Finkelberg}.
\newblock {Equivariant {S}atake category and {K}ostant-{W}hittaker reduction}.
\newblock {\em {Mosc. Math. J.}}, 8(1):39--72, 183, 2008.

\bibitem[BFGT]{mirabolic-satake}
A.~{Braverman}, M.~{Finkelberg}, V.~{Ginzburg}, and R.~{Travkin}.
\newblock {Mirabolic {S}atake equivalence and supergroups}.
\newblock {\em {Compos. Math.}}, 157(8):1724--1765, 2021.

\bibitem[BFM]{bfm}
R.~{Bezrukavnikov}, M.~{Finkelberg}, and I.~{Mirkovi\'{c}}.
\newblock {Equivariant homology and {$K$}-theory of affine {G}rassmannians and {T}oda lattices}.
\newblock {\em Compos. Math.}, 141(3):746--768, 2005.

\bibitem[BFN]{bfn-ii}
A.~{Braverman}, M.~{Finkelberg}, and H.~{Nakajima}.
\newblock {Towards a mathematical definition of Coulomb branches of $3$-dimensional $\mathcal{N} = 4$ gauge theories, II}.
\newblock {\em Adv. Theor. Math. Phys.}, 22:1071--1147, 2018.

\bibitem[BFT]{orthosymplectic-satake}
A.~{Braverman}, M~{Finkelberg}, and R.~{Travkin}.
\newblock {Orthosymplectic {S}atake equivalence}.
\newblock {\em {Commun. Number Theory Phys.}}, 16(4):695--732, 2022.

\bibitem[BG]{baranovsky-ginzburg}
V.~{Baranovsky} and V.~{Ginzburg}.
\newblock {Conjugacy classes in loop groups and {$G$}-bundles on elliptic curves}.
\newblock {\em {Internat. Math. Res. Notices}}, (15):733--751, 1996.

\bibitem[{Bha}]{bhargava-composition-i}
M.~{Bhargava}.
\newblock {Higher composition laws. {I}. {A} new view on {G}auss composition, and quadratic generalizations}.
\newblock {\em {Ann. of Math. (2)}}, 159(1):217--250, 2004.

\bibitem[BHS]{bhs-artin-tate}
R.~{Burklund}, J.~{Hahn}, and A.~{Senger}.
\newblock {Galois reconstruction of Artin-Tate $\mathbb{R}$-motivic spectra}.
\newblock \url{https://arxiv.org/abs/2010.10325}, 2020.

\bibitem[BM]{borho-macpherson}
W.~{Borho} and R.~{MacPherson}.
\newblock {Partial resolutions of nilpotent varieties}.
\newblock In {\em {Analysis and topology on singular spaces, {II}, {III} ({L}uminy, 1981)}}, volume 101-102 of {\em {Ast\'erisque}}, pages 23--74. {Soc. Math. France, Paris}, 1983.

\bibitem[BN]{bzn-elliptic-springer}
D.~{Ben-Zvi} and D.~{Nadler}.
\newblock {Elliptic {S}pringer theory}.
\newblock {\em {Compos. Math.}}, 151(8):1568--1584, 2015.

\bibitem[BO]{berthelot-ogus}
P.~{Berthelot} and A.~{Ogus}.
\newblock {\em Notes on crystalline cohomology}.
\newblock Princeton University Press, Princeton, N.J.; University of Tokyo Press, Tokyo, 1978.

\bibitem[{Bot}]{bott-space-of-loops}
R.~{Bott}.
\newblock {The space of loops on a {L}ie group}.
\newblock {\em {Michigan Math. J.}}, 5:35--61, 1958.

\bibitem[{Bri}]{brion-poincare-dual}
M.~{Brion}.
\newblock {Poincar\'e duality and equivariant (co)homology}.
\newblock {\em {Michigan Math. J.}}, 48:77--92, 2000.
\newblock {Dedicated to William Fulton on the occasion of his 60th birthday}.

\bibitem[BS1]{borel-serre-steenrod-ii}
A.~{Borel} and J.-P. {Serre}.
\newblock {D\'etermination des $p$-puissances r\'eduites de Steenrod dans la cohomologie des groupes classiques. Applications}.
\newblock {\em {C. R. Acad. Sci. Paris}}, 233:680--682.

\bibitem[BS2]{borel-serre-steenrod}
A.~{Borel} and J.-P. {Serre}.
\newblock {Groupes de {L}ie et puissances r\'eduites de {S}teenrod}.
\newblock {\em Amer. J. Math.}, 75:409--448, 1953.

\bibitem[BTLM]{frobenius-on-toric}
A.~{Buch}, J.~{Thomsen}, N.~{Lauritzen}, and V.~{Mehta}.
\newblock {The {F}robenius morphism on a toric variety}.
\newblock {\em {Tohoku Math. J. (2)}}, 49(3):355--366, 1997.

\bibitem[{But}1]{butson-i}
D.~{Butson}.
\newblock {Equivariant localization in factorization homology and applications in mathematical physics I: Foundations}.
\newblock \url{https://arxiv.org/abs/2011.14988}, 2020.

\bibitem[{But}2]{butson-ii}
D.~{Butson}.
\newblock {Equivariant localization in factorization homology and applications in mathematical physics I: Gauge theory applications}.
\newblock \url{https://arxiv.org/abs/2011.14978}, 2020.

\bibitem[BZ]{brylinski-zhang}
J.-L. {Brylinski} and B.~{Zhang}.
\newblock {Equivariant {$K$}-theory of compact connected {L}ie groups}.
\newblock {\em {$K$-Theory}}, 20(1):23--36, 2000.
\newblock {Special issues dedicated to Daniel Quillen on the occasion of his sixtieth birthday, Part I}.

\bibitem[BZSV]{bzsv}
D.~{B}en {Z}vi, Y.~{Sakellaridis}, and A.~{Venkatesh}.
\newblock {Relative Langlands duality}.
\newblock \url{https://www.math.ias.edu/~akshay/research/BZSVpaperV1.pdf}, 2023.

\bibitem[{Cay}]{cayley-original}
A.~{Cayley}.
\newblock {\em On the Theory of Linear Transformations}, page 80–94.
\newblock Cambridge Library Collection - Mathematics. Cambridge University Press, 2009.

\bibitem[CD]{chen-dhillon}
H.~{Chen} and G.~{Dhillon}.
\newblock {A Langlands dual realization of coherent sheaves on the nilpotent cone}.
\newblock \url{https://arxiv.org/abs/2310.10539}, 2023.

\bibitem[CG]{chriss-ginzburg}
N.~{Chriss} and V.~{Ginzburg}.
\newblock {\em Representation theory and complex geometry}.
\newblock Modern Birkh\"{a}user Classics. Birkh\"{a}user Boston, Ltd., Boston, MA, 2010.
\newblock Reprint of the 1997 edition.

\bibitem[CK]{cautis-kamnitzer}
S.~{Cautis} and J.~{Kamnitzer}.
\newblock {Quantum {K}-theoretic geometric {S}atake: the {$\SL_n$} case}.
\newblock {\em {Compos. Math.}}, 154(2):275--327, 2018.

\bibitem[CMNO]{quat-satake}
T.-H. {Chen}, M.~{Macerato}, D.~{Nadler}, and J.~{O'Brien}.
\newblock {Quaternionic Satake equivalence}.
\newblock \url{https://arxiv.org/abs/2207.04078}, 2022.

\bibitem[CO]{octonionic-period}
T.-H. {Chen} and J.~{O'Brien}.
\newblock {Lorentzian and Octonionic Satake equivalence}.
\newblock \url{https://arxiv.org/abs/2409.03969}, 2024.

\bibitem[CP]{cohen-peterson}
F.~R. {Cohen} and F.~{Peterson}.
\newblock {Suspensions of {S}tiefel manifolds}.
\newblock {\em {Quart. J. Math. Oxford Ser. (2)}}, 35(138):115--119, 1984.

\bibitem[CR]{campbell-raskin-satake}
J.~{Campbell} and S.~{Raskin}.
\newblock {Langlands duality on the Beilinson-Drinfeld Grassmannian}.
\newblock Available at \url{https://gauss.math.yale.edu/~sr2532/}, 2023.

\bibitem[CS]{crawley-boevey-shaw}
W.~{{C}rawley-{B}oevey} and P.~{Shaw}.
\newblock {Multiplicative preprojective algebras, middle convolution and the {D}eligne-{S}impson problem}.
\newblock {\em {Adv. Math.}}, 201(1):180--208, 2006.

\bibitem[{Dav}]{davis-elliptic-springer}
D.~{Davis}.
\newblock {The elliptic Grothendieck-Springer resolution as a simultaneous log resolution of algebraic stacks}.
\newblock \url{https://arxiv.org/abs/1908.04140}, 2019.

\bibitem[{Dev}1]{thh-xn}
S.~{Devalapurkar}.
\newblock {Topological Hochschild homology, truncated Brown-Peterson spectra, and a topological Sen operator}.
\newblock \url{https://sanathdevalapurkar.github.io/files/thh-Xn.pdf}, 2023.

\bibitem[{Dev}2]{pgl2-cubes}
S.~{Devalapurkar}.
\newblock {Derived geometric Satake for $\PGL_2^{\times 3}/\PGL_2^\mathrm{diag}$}.
\newblock \url{https://sanathdevalapurkar.github.io/files/PGL2-cubes.pdf}, 2024.

\bibitem[{Dev}3]{ku-rel-langlands}
S.~{Devalapurkar}.
\newblock {$\ku$-theoretic spectral decompositions for spheres and projective spaces}.
\newblock \url{https://sanathdevalapurkar.github.io/files/hyperboloid_spectral_decomp.pdf}, 2024.

\bibitem[DFHH]{tmf}
C.~{Douglas}, J.~{Francis}, A.~{Henriques}, and M.~{Hill}.
\newblock {\em {Topological Modular Forms}}, volume 201 of {\em Mathematical Surveys and Monographs}.
\newblock American Mathematical Society, 2014.

\bibitem[DG]{dimofte-garner}
T.~{Dimofte} and N.~{Garner}.
\newblock {Coulomb branches of star-shaped quivers}.
\newblock {\em {J. High Energy Phys.}}, (2):{004, front matter+87}, 2019.

\bibitem[{Dic}]{dickson-polynomial}
L.~E. {Dickson}.
\newblock {The analytic representation of substitutions on a power of a prime number of letters with a discussion of the linear group}.
\newblock {\em {Ann. of Math.}}, 11(1-6):65--120, 1896/97.

\bibitem[DM]{generalized-n-series}
S.~{Devalapurkar} and M.~{Misterka}.
\newblock {Generalized $n$-series and de Rham complexes}.
\newblock \url{https://sanathdevalapurkar.github.io/files/fgls-and-dR-complexes.pdf}, 2023.

\bibitem[{Dri}]{drinfeld-formal-group}
V.~{Drinfeld}.
\newblock {A 1-dimensional formal group over the prismatization of $\spf \Z_p$}.
\newblock \url{http://arxiv.org/abs/2107.11466}, 2021.

\bibitem[{Fel}]{felder-elliptic-quantum}
G.~{Felder}.
\newblock {Elliptic quantum groups}.
\newblock In {\em {X{I}th {I}nternational {C}ongress of {M}athematical {P}hysics ({P}aris, 1994)}}, pages {211--218}. {Int. Press, Cambridge, MA}, 1995.

\bibitem[FG]{frenkel-gaitsgory-local-langlands}
E.~{Frenkel} and D.~{Gaitsgory}.
\newblock {Local geometric {L}anglands correspondence and affine {K}ac-{M}oody algebras}.
\newblock In {\em {Algebraic geometry and number theory}}, volume 253 of {\em {Progr. Math.}}, pages 69--260. {Birkh\"auser Boston, Boston, MA}, 2006.

\bibitem[FGKV]{geometric-casselman-shalika-i}
E.~{Frenkel}, D.~{Gaitsgory}, D.~{Kazhdan}, and K.~{Vilonen}.
\newblock {Geometric realization of {W}hittaker functions and the {L}anglands conjecture}.
\newblock {\em {J. Amer. Math. Soc.}}, 11(2):451--484, 1998.

\bibitem[FGV]{geometric-casselman-shalika-ii}
E.~{Frenkel}, D.~{Gaitsgory}, and K.~{Vilonen}.
\newblock {Whittaker patterns in the geometry of moduli spaces of bundles on curves}.
\newblock {\em {Ann. of Math. (2)}}, 153(3):699--748, 2001.

\bibitem[FHT1]{fht-complex}
D.~{Freed}, M.~{Hopkins}, and C.~{Teleman}.
\newblock {Twisted equivariant {$K$}-theory with complex coefficients}.
\newblock {\em {J. Topol.}}, 1(1):16--44, 2008.

\bibitem[FHT2]{fht-i}
D.~{Freed}, M.~{Hopkins}, and C.~{Teleman}.
\newblock {Loop groups and twisted {$K$}-theory {I}}.
\newblock {\em {J. Topol.}}, 4(4):737--798, 2011.

\bibitem[FHT3]{fht-iii}
D.~{Freed}, M.~{Hopkins}, and C.~{Teleman}.
\newblock {Loop groups and twisted {$K$}-theory {III}}.
\newblock {\em {Ann. of Math. (2)}}, 174(2):947--1007, 2011.

\bibitem[FHT4]{fht-ii}
D.~{Freed}, M.~{Hopkins}, and C.~{Teleman}.
\newblock {Loop groups and twisted {$K$}-theory {II}}.
\newblock {\em {J. Amer. Math. Soc.}}, 26(3):595--644, 2013.

\bibitem[FM1]{friedman-morgan}
R.~{Friedman} and J.~{Morgan}.
\newblock {Holomorphic principal bundles over elliptic curves}.
\newblock \url{https://arxiv.org/abs/math/9811130}, 1998.

\bibitem[FM2]{friedman-morgan-ii}
R.~{Friedman} and J.~{Morgan}.
\newblock {Holomorphic principal bundles over elliptic curves. {II}. {T}he parabolic construction}.
\newblock {\em {J. Differential Geom.}}, 56(2):301--379, 2000.

\bibitem[FM3]{friedman-morgan-iii}
R.~{Friedman} and J.~{Morgan}.
\newblock {Holomorphic Principal Bundles Over Elliptic Curves III: Singular Curves and Fibrations}.
\newblock \url{https://arxiv.org/abs/math/0108104}, 2001.

\bibitem[FMW]{friedman-morgan-witten}
R.~{Friedman}, J.~{Morgan}, and E.~{Witten}.
\newblock {Principal {$G$}-bundles over elliptic curves}.
\newblock {\em {Math. Res. Lett.}}, 5(1-2):97--118, 1998.

\bibitem[{Fra}]{francis-thesis}
J.~{Francis}.
\newblock {\em {Derived algebraic geometry over $\mathcal{E}_n$-rings}}.
\newblock ProQuest LLC, Ann Arbor, MI, 2008.
\newblock Thesis (Ph.D.)--Massachusetts Institute of Technology.

\bibitem[FT1]{finkelberg-tsymbaliuk}
M.~{Finkelberg} and A.~{Tsymbaliuk}.
\newblock {Multiplicative slices, relativistic {T}oda and shifted quantum affine algebras}.
\newblock In {\em {Representations and nilpotent orbits of {L}ie algebraic systems}}, volume 330 of {\em {Progr. Math.}}, pages 133--304. {Birkh\"{a}user/Springer, Cham}, 2019.

\bibitem[FT2]{fht-categorified}
D.~{Freed} and C.~{Teleman}.
\newblock {Dirac families for loop groups as matrix factorizations}.
\newblock {\em {C. R. Math. Acad. Sci. Paris}}, 353(5):415--419, 2015.

\bibitem[{Gan}1]{gannon-thesis}
T.~{Gannon}.
\newblock {Classification of nondegenerate $G$-categories}.
\newblock \url{https://arxiv.org/abs/2206.11247}, 2022.

\bibitem[{Gan}2]{gannon-tmmodw}
T.~{Gannon}.
\newblock {The coarse quotient for affine Weyl groups and pseudo-reflection groups}.
\newblock \url{https://arxiv.org/abs/2206.00175}, 2022.

\bibitem[{Gan}3]{affine-closure-G-UP}
T.~{Gannon}.
\newblock {The cotangent bundle of $G/U_P$ and Kostant-Whittaker descent}.
\newblock \url{https://arxiv.org/abs/2407.16844}, 2024.

\bibitem[GG]{gannon-ginzburg}
T.~{Gannon} and V.~{Ginzburg}.
\newblock {Quantization of the universal centralizer and central D-modules}.
\newblock \url{https://arxiv.org/abs/2409.18054}, 2024.

\bibitem[{Gin}]{ginzburg-whittaker}
V.~{Ginzburg}.
\newblock {Nil-{H}ecke algebras and {W}hittaker {$\mathcal{D}$}-modules}.
\newblock In {\em {Lie groups, geometry, and representation theory}}, volume 326 of {\em {Progr. Math.}}, pages 137--184. {Birkh\"{a}user/Springer, Cham}, 2018.

\bibitem[GK]{ginzburg-kazhdan}
V.~{Ginzburg} and D.~{Kazhdan}.
\newblock {Differential operators on {$G / U$} and the {G}elfand-{G}raev action}.
\newblock {\em {Adv. Math.}}, 403:{Paper No. 108368, 48}, 2022.

\bibitem[GKM]{gkm-original}
M.~{Goresky}, R.~{Kottwitz}, and R.~{MacPherson}.
\newblock {Equivariant cohomology, {K}oszul duality, and the localization theorem}.
\newblock {\em {Invent. Math.}}, 131(1):25--83, 1998.

\bibitem[GKV1]{ginzburg-kapranov-vasserot}
V.~{Ginzburg}, M.~{Kapranov}, and E.~{Vasserot}.
\newblock {Elliptic algebras and equivariant elliptic cohomology}.
\newblock \url{https://arxiv.org/abs/q-alg/9505012}, 1995.

\bibitem[GKV2]{ginzburg-kapranov-vasserot-residue-hecke}
V.~{Ginzburg}, M.~{Kapranov}, and E.~{Vasserot}.
\newblock {Residue construction of {H}ecke algebras}.
\newblock {\em {Adv. Math.}}, 128(1):1--19, 1997.

\bibitem[GKZ]{gelfand-hyperdet}
I.~{Gelfand}, M.~{Kapranov}, and A.~{Zelevinsky}.
\newblock {\em {Discriminants, resultants, and multidimensional determinants}}.
\newblock {Mathematics: Theory \& Applications}. {Birkh\"{a}user Boston, Inc., Boston, MA}, {1994}.

\bibitem[GL]{gaitsgory-lysenko-kirillov}
D.~{Gaitsgory} and S.~{Lysenko}.
\newblock {Metaplectic Whittaker category and quantum groups : the ``small'' FLE}.
\newblock \url{https://arxiv.org/abs/1903.02279}, 2019.

\bibitem[GM1]{t-equiv-tmf}
D.~{Gepner} and L.~{Meier}.
\newblock {On equivariant topological modular forms}.
\newblock \url{https://arxiv.org/abs/2004.10254}, 2020.

\bibitem[GM2]{gepner-meier}
D.~{Gepner} and L.~{Meier}.
\newblock {Equivariant elliptic cohomology with integral coefficients}.
\newblock Forthcoming, 2023.

\bibitem[GPS]{ganatra-pardon-shende}
S.~{Ganatra}, J.~{Pardon}, and V.~{Shende}.
\newblock {Microlocal {M}orse theory of wrapped {F}ukaya categories}.
\newblock {\em {Ann. of Math. (2)}}, 199(3):943--1042, 2024.

\bibitem[GR1]{gr-i}
D.~{Gaitsgory} and N.~{Rozenblyum}.
\newblock {\em A study in derived algebraic geometry. {V}ol. {I}. {C}orrespondences and duality}, volume 221 of {\em Mathematical Surveys and Monographs}.
\newblock American Mathematical Society, Providence, RI, 2017.

\bibitem[GR2]{garland-raghunathan}
H.~{Garland} and M.~S. {Raghunathan}.
\newblock {A {B}ruhat decomposition for the loop space of a compact group: a new approach to results of {B}ott}.
\newblock {\em {Proc. Nat. Acad. Sci. U.S.A.}}, 72(12):4716--4717, 1975.

\bibitem[GR3]{ginzburg-riche}
V.~{Ginzburg} and S.~{Riche}.
\newblock {Differential operators on {$G/U$} and the affine {G}rassmannian}.
\newblock {\em {J. Inst. Math. Jussieu}}, 14(3):493--575, 2015.

\bibitem[{Gre}]{gregoric-synthetic}
R.~{Gregoric}.
\newblock {Moduli stack of oriented formal groups and cellular motivic spectra over $\mathbf{C}$}.
\newblock \url{https://arxiv.org/abs/2111.15212}, 2021.

\bibitem[{Gro}]{gross-minuscule}
B.~{Gross}.
\newblock {On minuscule representations and the principal {${\SL}_2$}}.
\newblock {\em {Represent. Theory}}, 4:225--244, 2000.

\bibitem[GSB]{grojnowski-shepherd-barron}
I.~{Grojnowski} and N.~{S}hepherd {B}arron.
\newblock {Del {P}ezzo surfaces as {S}pringer fibres for exceptional groups}.
\newblock {\em {Proc. Lond. Math. Soc. (3)}}, 122(1):1--41, 2021.

\bibitem[GWX]{gwx-special-fiber}
B.~{Gheorghe}, G.~{Wang}, and Z.~{Xu}.
\newblock {The special fiber of the motivic deformation of the stable homotopy category is algebraic}.
\newblock {\em {Acta Math.}}, 226(2):319--407, 2021.

\bibitem[HHH]{generalized-gkm}
M.~{Harada}, A.~{Henriques}, and T.~{Holm}.
\newblock {Computation of generalized equivariant cohomologies of {K}ac-{M}oody flag varieties}.
\newblock {\em {Adv. Math.}}, 197(1):198--221, 2005.

\bibitem[HHR]{hhr}
M.~{Hill}, M.~{Hopkins}, and D.~{Ravenel}.
\newblock On the nonexistence of elements of {K}ervaire invariant one.
\newblock {\em Ann. of Math. (2)}, 184(1):1--262, 2016.

\bibitem[HKR]{hkr}
M.~{Hopkins}, N.~{Kuhn}, and D.~{Ravenel}.
\newblock {Generalized group characters and complex oriented cohomology theories}.
\newblock {\em {J. Amer. Math. Soc.}}, 13(3):553--594, 2000.

\bibitem[HL]{ho-li-mixed}
Q.~{Ho} and P.~{Li}.
\newblock {Revisiting mixed geometry}.
\newblock \url{https://arxiv.org/abs/2202.04833}, 2022.

\bibitem[HM]{hausmann-meier}
M.~{Hausmann} and L.~{Meier}.
\newblock {Invariant prime ideals in equivariant Lazard rings}.
\newblock \url{https://arxiv.org/abs/2309.00850v1}, 2023.

\bibitem[{Hop}]{k1local}
M.~{Hopkins}.
\newblock {$K(1)$-local $\Eoo$-ring spectra}.
\newblock In {\em {Topological Modular Forms}}, volume 201 of {\em Mathematical Surveys and Monographs}. American Mathematical Society, 2014.

\bibitem[HRW]{even-filtr}
J.~{Hahn}, A.~{Raksit}, and D.~{Wilson}.
\newblock {A motivic filtration on the topological cyclic homology of commutative ring spectra}.
\newblock \url{https://arxiv.org/abs/2206.11208}, 2022.

\bibitem[{Hua}]{huan-finite-subgroup-tate}
Z.~{Huan}.
\newblock {Universal Finite Subgroup of the Tate Curve}.
\newblock \url{https://arxiv.org/abs/1708.08637}, 2017.

\bibitem[HW]{hahn-wilson-bpn}
J.~{Hahn} and D.~{Wilson}.
\newblock {Redshift and multiplication for truncated Brown-Peterson spectra}, 2022.

\bibitem[HY]{hahn-yuan}
J.~{Hahn} and A.~{Yuan}.
\newblock Multiplicative structure in the stable splitting of {$\Omega SL_n(\cc)$}.
\newblock {\em Adv. Math.}, 348:412--455, 2019.

\bibitem[{Jan}]{jantzen-kohomologie}
J.~{Jantzen}.
\newblock {Kohomologie von {$p$}-{L}ie-{A}lgebren und nilpotente {E}lemente}.
\newblock {\em {Abh. Math. Sem. Univ. Hamburg}}, 56:191--219, 1986.

\bibitem[{Kit}]{kitchloo-steenrod}
N.~{Kitchloo}.
\newblock {Cohomology operations and the Nil-Hecke ring}.
\newblock \url{https://math.jhu.edu/~nitu/papers/NH.pdf}, 2013.

\bibitem[KK1]{kostant-kumar-2}
B.~{Kostant} and S.~{Kumar}.
\newblock {The nil {H}ecke ring and cohomology of {$G/P$} for a {K}ac-{M}oody group {$G$}}.
\newblock {\em {Proc. Nat. Acad. Sci. U.S.A.}}, 83(6):1543--1545, 1986.

\bibitem[KK2]{kostant-kumar}
B.~{Kostant} and S.~{Kumar}.
\newblock {{$T$}-equivariant {$K$}-theory of generalized flag varieties}.
\newblock {\em {J. Differential Geom.}}, 32(2):549--603, 1990.

\bibitem[{Kos}1]{kostant-lie-group-reps}
B.~{Kostant}.
\newblock {Lie group representations on polynomial rings}.
\newblock {\em {Amer. J. Math.}}, 85:327--404, 1963.

\bibitem[{Kos}2]{kostant-whittaker}
B.~{Kostant}.
\newblock {On {W}hittaker vectors and representation theory}.
\newblock {\em {Invent. Math.}}, 48(2):101--184, 1978.

\bibitem[KS1]{univ-cat-o}
A.~{Kalmykov} and P.~{Safronov}.
\newblock {A categorical approach to dynamical quantum groups}.
\newblock {\em {Forum Math. Sigma}}, 10:Paper No. e76, 57, 2022.

\bibitem[KS2]{kashiwara-schapira}
M.~{Kashiwara} and P.~{Schapira}.
\newblock {\em {Sheaves on manifolds}}, volume 292 of {\em {Grundlehren der mathematischen Wissenschaften [Fundamental Principles of Mathematical Sciences]}}.
\newblock {Springer-Verlag, Berlin}, 1990.
\newblock {With a chapter in French by Christian Houzel}.

\bibitem[{Kum}]{kumar-kac-moody}
S.~{Kumar}.
\newblock {\em {Kac-{M}oody groups, their flag varieties and representation theory}}, volume 204 of {\em {Progress in Mathematics}}.
\newblock {Birkh\"auser Boston, Inc., Boston, MA}, 2002.

\bibitem[KW]{kapustin-witten}
A.~{Kapustin} and E.~{Witten}.
\newblock {Electric-magnetic duality and the geometric {L}anglands program}.
\newblock {\em {Commun. Number Theory Phys.}}, 1(1):1--236, 2007.

\bibitem[{Lan}]{lance-steenrod-dyer-lashof-BU}
T.~{Lance}.
\newblock {Steenrod and {D}yer-{L}ashof operations on {$BU$}}.
\newblock {\em {Trans. Amer. Math. Soc.}}, 276(2):497--510, 1983.

\bibitem[LM]{littig-mitchell}
P.~{Littig} and S.~{Mitchell}.
\newblock {Generating varieties for affine {G}rassmannians}.
\newblock {\em {Trans. Amer. Math. Soc.}}, 363(7):3717--3731, 2011.

\bibitem[{Lon}1]{lonergan-descent}
G.~{Lonergan}.
\newblock {A remark on descent for Coxeter groups}.
\newblock \url{https://arxiv.org/abs/1707.01156v2}, 2017.

\bibitem[{Lon}2]{lonergan-fourier}
G.~{Lonergan}.
\newblock {A {F}ourier transform for the quantum {T}oda lattice}.
\newblock {\em {Selecta Math. (N.S.)}}, 24(5):4577--4615, 2018.

\bibitem[{Lon}3]{lonergan-slides}
G.~{Lonergan}.
\newblock {Geometric Satake over KU}.
\newblock {Online talk, available at \url{https://www.youtube.com/watch?v=Aazmxfb2i_A}}, 2021.

\bibitem[{Lon}4]{lonergan-frob}
G.~{Lonergan}.
\newblock {Steenrod operators, the {C}oulomb branch and the {F}robenius twist}.
\newblock {\em {Compos. Math.}}, 157(11):2494--2552, 2021.

\bibitem[{Lur}1]{survey}
J.~{Lurie}.
\newblock {A survey of elliptic cohomology}.
\newblock In {\em {Algebraic Topology}}, volume~4 of {\em Abel. Symp.}, pages 219--277. {Springer}, 2009.

\bibitem[{Lur}2]{lurie-icm}
J.~{Lurie}.
\newblock {Moduli problems for ring spectra}.
\newblock In {\em {Proceedings of the {I}nternational {C}ongress of {M}athematicians. {V}olume {II}}}, pages 1099--1125. {Hindustan Book Agency, New Delhi}, 2010.

\bibitem[{Lur}3]{rotinv}
J.~{Lurie}.
\newblock {Rotation invariance in algebraic K-theory}.
\newblock \url{https://www.math.ias.edu/~lurie/papers/Waldhaus.pdf}, 2015.

\bibitem[{Lur}4]{HA}
J.~{Lurie}.
\newblock {Higher Algebra}.
\newblock \url{http://www.math.harvard.edu/~lurie/papers/HA.pdf}, 2016.

\bibitem[{Lur}5]{SAG}
J.~{Lurie}.
\newblock {Spectral Algebraic Geometry}.
\newblock \url{http://www.math.harvard.edu/~lurie/papers/SAG-rootfile.pdf}, 2017.

\bibitem[{Lur}6]{elliptic-i}
J.~{Lurie}.
\newblock {Elliptic Cohomology I: Spectral Abelian Varieties}.
\newblock \url{http://www.math.harvard.edu/~lurie/papers/Elliptic-I.pdf}, 2018.

\bibitem[{Lur}7]{elliptic-ii}
J.~{Lurie}.
\newblock {Elliptic Cohomology II: Orientations}.
\newblock \url{http://www.math.harvard.edu/~lurie/papers/Elliptic-II.pdf}, 2018.

\bibitem[{Lur}8]{elliptic-iii}
J.~{Lurie}.
\newblock {Elliptic Cohomology III: Tempered Cohomology}.
\newblock \url{https://www.math.ias.edu/~lurie/papers/Elliptic-III-Tempered.pdf}, 2019.

\bibitem[{Mei}]{meier-tmf-modules}
L.~{Meier}.
\newblock {Relatively free $\TMF$-modules}.
\newblock \url{https://webspace.science.uu.nl/~meier007/RelativelyFree4.pdf}, 2017.

\bibitem[{Mil}]{miller-stable-splittings}
H.~{Miller}.
\newblock {Stable splittings of {S}tiefel manifolds}.
\newblock {\em {Topology}}, 24(4):411--419, 1985.

\bibitem[{Mit}]{mitchell-buildings}
S.~{Mitchell}.
\newblock {Quillen's theorem on buildings and the loops on a symmetric space}.
\newblock {\em {Enseign. Math. (2)}}, 34(1-2):123--166, 1988.

\bibitem[{Mou}]{moulinos-loop}
T.~{Moulinos}.
\newblock Filtered formal groups, cartier duality, and derived algebraic geometry.
\newblock \url{https://arxiv.org/abs/2101.10262v1}, 2021.

\bibitem[MRT]{toen-hkr}
T.~{Moulinos}, M.~{Robalo}, and B.~{To\"en}.
\newblock {A universal {H}ochschild-{K}ostant-{R}osenberg theorem}.
\newblock {\em {Geom. Topol.}}, 26(2):777--874, 2022.

\bibitem[MV]{mirkovic-vilonen}
I.~{Mirkovi{c}} and K.~{Vilonen}.
\newblock {Geometric {L}anglands duality and representations of algebraic groups over commutative rings}.
\newblock {\em {Ann. of Math. (2)}}, 166(1):95--143, 2007.

\bibitem[{Nak}1]{nakajima-coulomb}
H.~{Nakajima}.
\newblock {Towards a mathematical definition of {C}oulomb branches of 3-dimensional {$\mathcal{N}=4$} gauge theories, {I}}.
\newblock {\em {Adv. Theor. Math. Phys.}}, 20(3):595--669, 2016.

\bibitem[{Nak}2]{nakajima-intro}
H.~{Nakajima}.
\newblock {Introduction to a provisional mathematical definition of Coulomb branches of $3$-dimensional $\mathcal N=4$ gauge theories}.
\newblock \url{https://arxiv.org/abs/1706.05154}, 2017.

\bibitem[NP]{ngo-polo}
B.~C. {Ng\^o} and P.~{Polo}.
\newblock {R\'esolutions de {D}emazure affines et formule de {C}asselman-{S}halika g\'eom\'etrique}.
\newblock {\em {J. Algebraic Geom.}}, 10(3):515--547, 2001.

\bibitem[NS1]{nek-shat}
N.~{Nekrasov} and S.~{Shatashvili}.
\newblock {Quantization of Integrable Systems and Four Dimensional Gauge Theories}.
\newblock In {\em {16th International Congress on Mathematical Physics}}, pages 265--289, 8 2009.

\bibitem[NS2]{nikolaus-scholze}
T.~{Nikolaus} and P.~{Scholze}.
\newblock On topological cyclic homology.
\newblock {\em Acta Math.}, 221(2):203--409, 2018.

\bibitem[NY]{nakajima-yoshioka}
H.~{Nakajima} and K.~{Yoshioka}.
\newblock {Instanton counting on blowup. {II}. {$K$}-theoretic partition function}.
\newblock {\em {Transform. Groups}}, 10(3-4):489--519, 2005.

\bibitem[NZ]{nadler-zaslow}
D.~{Nadler} and E.~{Zaslow}.
\newblock {Constructible sheaves and the {F}ukaya category}.
\newblock {\em {J. Amer. Math. Soc.}}, 22(1):233--286, 2009.

\bibitem[{Ogg}]{ogg-ell-curve-over-Z}
A.~{Ogg}.
\newblock {Abelian curves of {$2$}-power conductor}.
\newblock {\em {Proc. Cambridge Philos. Soc.}}, 62:143--148, 1966.

\bibitem[{Ond}]{ondrus-whit-uq-sl2}
M.~{Ondrus}.
\newblock {Whittaker modules for {$U_q(\sl_2)$}}.
\newblock {\em {J. Algebra}}, 289(1):192--213, 2005.

\bibitem[{Pre}1]{premet}
A.~{Premet}.
\newblock {Special transverse slices and their enveloping algebras}.
\newblock {\em {Adv. Math.}}, 170(1):1--55, 2002.
\newblock {With an appendix by Serge Skryabin}.

\bibitem[{Pre}2]{preygel}
A.~{Preygel}.
\newblock Ind-coherent complexes on loop spaces and connections.
\newblock In {\em Stacks and categories in geometry, topology, and algebra}, volume 643 of {\em Contemp. Math.}, pages 289--323. Amer. Math. Soc., Providence, RI, 2015.

\bibitem[{Pst}1]{piotr-even-filtr}
P.~{Pstragowski}.
\newblock {Perfect even modules and the even filtration}.
\newblock \url{https://arxiv.org/abs/2304.04685}, 2023.

\bibitem[{Pst}2]{piotr-synthetic}
P.~{Pstragowski}.
\newblock {Synthetic spectra and the cellular motivic category}.
\newblock {\em {Invent. Math.}}, 232(2):553--681, 2023.

\bibitem[PTVV]{ptvv}
T.~{Pantev}, B.~{To\"{e}n}, M.~{Vaqui\'{e}}, and G.~{Vezzosi}.
\newblock Shifted symplectic structures.
\newblock {\em Publ. Math. Inst. Hautes \'{E}tudes Sci.}, 117:271--328, 2013.

\bibitem[{Rak}]{raksit}
A.~{Raksit}.
\newblock {Hochschild homology and the derived de Rham complex revisited}.
\newblock \url{https://arxiv.org/abs/2007.02576}, 2020.

\bibitem[{Rav}]{ravenel-loc}
D.~{Ravenel}.
\newblock Localization with respect to certain periodic homology theories.
\newblock {\em Amer. J. Math.}, 106(2):351--414, 1984.

\bibitem[{Rez}]{rezk-icm}
C.~{Rezk}.
\newblock {Isogenies, power operations, and homotopy theory}.
\newblock In {\em {Proceedings of the {I}nternational {C}ongress of {M}athematicians---{S}eoul 2014. {V}ol. {II}}}, pages {1125--1145}. {Kyung Moon Sa, Seoul}, 2014.

\bibitem[{Ric}]{riche}
S.~{Riche}.
\newblock {Kostant section, universal centralizer, and a modular derived Satake equivalence}.
\newblock {\em {Math. Z.}}, 286(1):223--261, 2017.

\bibitem[{Rog}]{rognes}
J.~{Rognes}.
\newblock {\em Galois Extensions of Structured Ring Spectra/Stably Dualizable Groups}, volume 192 of {\em Mem. Amer. Math. Soc.}
\newblock American Mathematical Society, 2008.

\bibitem[{Sch}]{scholze-q-def}
P.~{Scholze}.
\newblock Canonical {$q$}-deformations in arithmetic geometry.
\newblock {\em {Ann. Fac. Sci. Toulouse Math. (6)}}, 26(5):1163--1192, 2017.

\bibitem[{Seg}]{segal-equiv-KU}
G.~{Segal}.
\newblock {Equivariant {$K$}-theory}.
\newblock {\em {Inst. Hautes \'{E}tudes Sci. Publ. Math.}}, (34):129--151, 1968.

\bibitem[ST]{sibilla-tomasini}
N.~{Sibilla} and P.~{Tomasini}.
\newblock {Equivariant Elliptic Cohomology and Mapping Stacks I}.
\newblock \url{https://arxiv.org/abs/2303.10146}, 2023.

\bibitem[{Sta}1]{stapleton-hkr}
N.~{Stapleton}.
\newblock {Transchromatic generalized character maps}.
\newblock {\em {Algebr. Geom. Topol.}}, 13(1):171--203, 2013.

\bibitem[{Sta}2]{stapleton-2}
N.~{Stapleton}.
\newblock {Transchromatic twisted character maps}.
\newblock {\em {J. Homotopy Relat. Struct.}}, 10(1):29--61, 2015.

\bibitem[{Ste}]{steinberg-slice}
R.~{Steinberg}.
\newblock {Regular elements of semisimple algebraic groups}.
\newblock {\em {Inst. Hautes \'{E}tudes Sci. Publ. Math.}}, (25):49--80, 1965.

\bibitem[{Str}]{strickland-symmetric-gps}
N.~{Strickland}.
\newblock {Morava {$E$}-theory of symmetric groups}.
\newblock {\em {Topology}}, 37(4):757--779, 1998.

\bibitem[SW]{seiberg-witten-coulomb}
N.~{Seiberg} and E.~{Witten}.
\newblock {Gauge dynamics and compactification to three dimensions}.
\newblock In {\em {The mathematical beauty of physics ({S}aclay, 1996)}}, volume~24 of {\em {Adv. Ser. Math. Phys.}}, pages 333--366. {World Sci. Publ., River Edge, NJ}, 1997.

\bibitem[{Tel}1]{teleman-icm}
C.~{Teleman}.
\newblock {Gauge theory and mirror symmetry}.
\newblock In {\em {Proceedings of the {I}nternational {C}ongress of {M}athematicians---{S}eoul 2014. {V}ol. {II}}}, pages {1309--1332}. {Kyung Moon Sa, Seoul}, 2014.

\bibitem[{Tel}2]{teleman-slides}
C.~{Teleman}.
\newblock {Topological Gauge Theory in low dimensions}.
\newblock \url{https://math.berkeley.edu/~teleman/math/Auckland.pdf}, 2018.

\bibitem[{Toe}]{toen-icm}
B.~{Toen}.
\newblock {Derived algebraic geometry and deformation quantization}.
\newblock In {\em {Proceedings of the {I}nternational {C}ongress of {M}athematicians---{S}eoul 2014. {V}ol. {II}}}, pages 769--792. {Kyung Moon Sa, Seoul}, 2014.

\bibitem[{Van}]{van-den-bergh-double-poisson}
M.~{Van den Bergh}.
\newblock {Double {P}oisson algebras}.
\newblock {\em {Trans. Amer. Math. Soc.}}, 360(11):5711--5769, 2008.

\bibitem[{Woj}]{hopf-algebroid-nil-hecke}
Z.~{Wojciechowski}.
\newblock {A study of nil Hecke algebras via Hopf algebroids}.
\newblock \url{https://arxiv.org/abs/2410.08061v3}, 2024.

\bibitem[{Woo}]{wood-banach}
R.~{Wood}.
\newblock {Banach algebras and {B}ott periodicity}.
\newblock {\em {Topology}}, 4:371--389, 1965/66.

\bibitem[YZ1]{yang-zhao-e-thy-quantum-group}
Y.~{Yang} and G.~{Zhao}.
\newblock {Frobenii on Morava E-theoretical quantum groups}.
\newblock \url{https://arxiv.org/abs/2105.14681}, 2021.

\bibitem[YZ2]{homology-langlands}
Z.~{Yun} and X.~{Zhu}.
\newblock Integral homology of loop groups via {L}anglands dual groups.
\newblock {\em Represent. Theory}, 15:347--369, 2011.

\bibitem[{Zho}]{zhong-equiv-homology-of-Gr}
C.~{Zhong}.
\newblock {Equivariant oriented homology of the affine Grassmannian}.
\newblock \url{https://arxiv.org/abs/2301.13056}, 2023.

\bibitem[{Zhu}]{zhu-grass}
X.~{Zhu}.
\newblock An introduction to affine {G}rassmannians and the geometric {S}atake equivalence.
\newblock In {\em Geometry of moduli spaces and representation theory}, volume~24 of {\em IAS/Park City Math. Ser.}, pages 59--154. Amer. Math. Soc., Providence, RI, 2017.

\end{thebibliography}
\end{document}